\documentclass{amsart}
\usepackage[dvipsnames]{xcolor}
\usepackage{amscd,amsmath,xypic,amssymb,combelow,tikz-cd,etoolbox,calligra,mathrsfs,enumitem,mathtools,hyperref,comment,graphicx}
\usepackage{stmaryrd}
\usetikzlibrary{quotes}

\makeatletter
\patchcmd{\@settitle}{\uppercasenonmath\@title}{}{}{}
\makeatother

\newtheorem{theorem}{Theorem}[section]

\newtheorem{proposition}[theorem]{Proposition}
\newtheorem{lemma}[theorem]{Lemma}
\newtheorem{corollary}[theorem]{Corollary}

\newtheorem{definition}[theorem]{Definition}
\newtheorem{claim}[theorem]{Claim}

\newtheorem{remark}[theorem]{Remark}

\def\fg{{\mathfrak{g}}}

\def\fsl{{\mathfrak{sl}}}
\def\fgl{{\mathfrak{gl}}}
\def\fosp{{\mathfrak{osp}}}

\def\hsl{{\widehat{\fsl}}}

\def\BC{{\mathbb{C}}}
\def\BK{{\mathbb{K}}}

\def\BN{{\mathbb{N}}}

\def\BP{{\mathbb{P}}}
\def\BR{{\mathbb{R}}}
\def\BQ{{\mathbb{Q}}}
\def\BZ{{\mathbb{Z}}}

\def\CO{{\mathcal{O}}}

\def\CS{{\mathcal{S}}}
\def\CT{{\mathcal{T}}}

\def\CV{{\mathcal{V}}}

\def\ph{\varphi}

\def\sym{\textrm{sym}}
\def\Sym{\textrm{Sym}}

\def\U{\mathbf{U}}
\def\Up{\mathbf{U}^+}
\def\Upm{\mathbf{U}^\pm}

\def\Um{\mathbf{U}^-}

\def\usl{U_q(L\fg_{A_\mathbf{s}})}

\def\uslG{U_q(L\fg_{A_\Gamma})}
\def\uslGp{U_q^+(L\fg_{A_\Gamma})}
\def\uslGm{U_q^-(L\fg_{A_\Gamma})}

\def\Usl{U_{q,d}(L\fg_{\wA_\mathbf{s}})}

\def\UslGBM{U^{BM}_{q,d}(L\fg_{\wA_\Gamma})}
\def\UslG{U_{q,d}(L\fg_{\wA_\Gamma})}
\def\UslGp{U_{q,d}^+(L\fg_{\wA_\Gamma})}
\def\UslGm{U_{q,d}^-(L\fg_{\wA_\Gamma})}

\def\tuslGp{\widetilde{U}_q^+(L\fg_{A_\Gamma})}

\def\tUslGp{\widetilde{U}_{q,d}^+(L\fg_{\wA_\Gamma})}

\def\ub{U_q(L\fg_{B_\mathbf{s}})}

\def\ubG{U_q(L\fg_{B_\Gamma})}
\def\ubGp{U_q^+(L\fg_{B_\Gamma})}

\def\tubG{\widetilde{U}_q(L\fg_{B_\Gamma})}
\def\tubGp{\widetilde{U}_q^+(L\fg_{B_\Gamma})}

\def\uc{U_q(L\fg_{C_\mathbf{s}})}

\def\ucG{U_q(L\fg_{C_\Gamma})}
\def\ucGp{U_q^+(L\fg_{C_\Gamma})}

\def\tucG{\widetilde{U}_q(L\fg_{C_\Gamma})}
\def\tucGp{\widetilde{U}_q^+(L\fg_{C_\Gamma})}

\def\ud{U_q(L\fg_{D_\mathbf{s}})}

\def\udG{U_q(L\fg_{D_\Gamma})}
\def\udGp{U_q^+(L\fg_{D_\Gamma})}

\def\tudG{\widetilde{U}_q(L\fg_{D_\Gamma})}
\def\tudGp{\widetilde{U}_q^+(L\fg_{D_\Gamma})}

\def\tUp{\mathbf{\widetilde{U}}^+}

\def\tUm{\mathbf{\widetilde{U}}^-}

\def\br{{\mathbf{r}}}

\def\bsi{\boldsymbol{\varsigma}}

\def\nn{{\mathbb{N}^I}}
\def\zz{{\mathbb{Z}^I}}

\def\bpsi{{\boldsymbol{\psi}}}

\def\bn{{\boldsymbol{n}}}

\def\b0{{\boldsymbol{0}}}

\def\loccit{\emph{loc.~cit.~}}
\def\loccitt{\emph{loc.~cit.}}

\def\hdeg{\text{hdeg }}
\def\vdeg{\text{vdeg }}

\def\oCS{\mathring{\CS}}

\def\Sym{\text{Sym}}

\def\wCV{\CV}
\def\wCS{\CS}

\def\tUpsilon{\widetilde{\Upsilon}}
\def\twUpsilon{\tUpsilon}

\def\tQ{\widetilde{Q}}
\def\wA{\widehat{A}}

\def\bs{\mathbf{s}}

\newcommand{\twist}{\underline{z}^{\bn}}
\newcommand{\iso}{\,\vphantom{j^{X^2}}\smash{\overset{\sim}{\vphantom{\rule{0pt}{0.20em}}\smash{\longrightarrow}}}\,}

\def\pI{I_N}
\def\pnn{{\BN^{\pI}}}

\def\wI{\widehat{I}_N}
\def\wnn{\BN^{\wI}}
\def\wzz{\BZ^{\wI}}

\def\sI{I_{N-1}}
\def\snn{{\BN^{\sI}}}
\def\szz{{\BZ^{\sI}}}

\def\tI{I_N^{\mathrm{ext}}}
\def\tnn{{\BN^{\tI}}}

\def\wIthree{\widehat{I}_3}
\def\wItwo{\widehat{I}_2}

\def\Gb{\Gamma_{\bullet}}
\def\Gh{\Gamma_{\circ}}
\def\Ga{\Gamma}
\def\Gar{\Gamma'}
\def\Gal{\Gamma''}

\def\wXYG{\widehat{XY}_{\Gamma}}

\def\relR{{\mathrm{relations}\, }_{Y_{\Gamma'}}}
\def\relL{{\mathrm{relations}\, }_{X_{\Gamma''}}}

\begin{document}

\title[Scary Wheels (and Super Shrubs)]{\Large{\textbf{Scary Wheels (and Super Shrubs)}}}

\author[Andrei Negu\cb t]{Andrei Negu\cb t}
\address{École Polytechnique Fédérale de Lausanne (EPFL), Lausanne, Switzerland \newline 
         \text{ } \ \ Simion Stoilow Institute of Mathematics (IMAR), Bucharest, Romania} 
\email{andrei.negut@gmail.com}

\author[Alexander Tsymbaliuk]{Alexander Tsymbaliuk}
\address{Purdue University, Department of Mathematics, West Lafayette, IN, USA}
\email{sashikts@gmail.com}

\maketitle

\begin{abstract} 
We prove the shuffle realization for quantum affine superalgebras of types $B,C,D$ in the loop realization, in that we construct an isomorphism between each of these algebras and a suitably defined double shuffle algebra. We explicitly describe the latter shuffle algebra using certain vanishing conditions that generalize the Feigin-Odesskii wheel conditions; a novel feature is the appearance of a so-called scary wheel, which is a particular order $2$ vanishing condition involving $7$ variables. Along the way, we fully develop the theory of shrubs in super types $A$ (affine and toroidal), which are important combinatorial tools in the study of the shuffle algebras associated to toric Calabi-Yau threefolds. Our techniques also allow us to define quantum toroidal superalgebras of types $B,C,D$. More importantly, we provide a general framework to formulate and prove shuffle realizations in the wide generality of quivers with parameters.
\end{abstract}


\bigskip
    
\section{Introduction}
\label{sec:intro}


\medskip

\subsection{Motivation} 
\label{sub:motivation}

Let $\fg$ be any finite-dimensional simple complex Lie algebra, with a henceforth fixed set $I$ of simple roots, and Cartan matrix 
\begin{equation}
\label{eq:nonsymm-Cartan}
  \left( c_{ij} = \frac {2d_{ij}}{d_{ii}} \in \BZ \right)_{i,j \in I}
\end{equation}
where $\{d_{ij} = (\alpha_i,\alpha_j)\}_{i,j \in I}$ denotes the symmetrized Cartan matrix. It is customary to represent the Dynkin diagram of $\fg$ by using simple, double or triple edges between vertices $i,j \in I$. However, in the present paper, we will find it more convenient to work instead with the so-called \emph{decorated Dynkin diagram} of $\fg$
\begin{equation}
\label{eqn:g}
\begin{tikzcd}
  i && j
  \arrow["q^{d_{ij}}"', no head, from=1-3, to=1-1]
  \arrow["q^{d_{ii}}"', no head, from=1-1, to=1-1, looseness=4.5]
  \arrow["q^{d_{jj}}"', no head, from=1-3, to=1-3, looseness=4]
\end{tikzcd}
\end{equation}
for a henceforth fixed $q \in \BC^\times$ that is not a root of unity (thus, the decorated Dynkin diagram equally well encodes the entries of the symmetrized Cartan matrix).

\medskip 
\noindent 
If one applies the construction of \cite{Dr 0,J} to the particular case of the affine Dynkin diagram associated to $\fg$, one obtains the quantum affine algebra with trivial central charge \footnote{The case of arbitrary central charge can easily be accommodated by suitably twisting the product in the doubles of Subsection \ref{sub:double}, but we choose to not do so for conciseness.} $U_q(L\fg)$. Drinfeld proposed in \cite{Dr} a new realization of this algebra  as
\begin{equation*}
  U_q(L\fg) = U_q^+(L\fg) \otimes (\text{loop Cartan subalgebra}) \otimes U_q^-(L\fg)
\end{equation*}
(which was proved in full in \cite{B}). In \cite{E, FO}, a homomorphism was constructed
\begin{equation}
\label{eqn:intro}
  U_q^+(L\fg) \rightarrow \CV^+_{\fg} = 
  \bigoplus_{\bn = (n_i)_{i \in I} \in \nn} \BC[z_{i1}^{\pm 1},\dots,z_{in_i}^{\pm 1}]^{\sym}_{i \in I} 
\end{equation}
(in the present paper, $\BN$ contains $0$) where the right-hand side is made into an algebra using the trigonometric version of the Feigin-Odesskii shuffle product, cf.~\eqref{eqn:shuffle product}. It was conjectured in \loccit that the image of the homomorphism above coincides with the following particular subset of $\CV^+_{\fg}$: 
\begin{equation*}
  \CS^+_{\fg} = \Big\{ E(z_{ia})_{i \in I, a \geq 1} 
  \text{ which vanishes at the wheels \eqref{eqn:classic wheels}} \text{ for all } i\neq j \Big\} 
\end{equation*}
where we represent vanishing conditions by the following diagram (cf.~\eqref{eqn:g}): 
\begin{equation}
\label{eqn:classic wheels}
\begin{tikzcd}
  &&& {z_{j1}} &&& \\ \\
  {z_{i1}} && {z_{i2}} && \dots && {z_{i,1-c_{ij}}}
  \arrow["{q^{d_{ji}}}"', from=1-4, to=3-1]
  \arrow["{q^{d_{ii}}}", from=3-1, to=3-3]
  \arrow["{q^{d_{ii}}}", from=3-3, to=3-5]
  \arrow["{q^{d_{ii}}}", from=3-5, to=3-7]
  \arrow["{q^{d_{ij}}}"', from=3-7, to=1-4]
\end{tikzcd}
\end{equation}
If we specialize a Laurent polynomial $E(z,z',\dots)$ according to $z'=zt$ for all arrows
$$
  z \xrightarrow{t} z'
$$
that appear in the diagram \eqref{eqn:classic wheels}, then the phrase ``$E$ vanishes at the wheel'' simply means that the aforementioned specialization should be $0$; due to the graphical representation \eqref{eqn:classic wheels}, this vanishing was called a \emph{wheel condition} in \cite{FO}. We proved in \cite{NT} that the algebra homomorphism \eqref{eqn:intro} is injective and induces an isomorphism
\begin{equation}
\label{eqn:upsilon intro 1}
  \Upsilon^+_{\fg} \colon U_q^+(L\fg) \iso \CS^+_{\fg}
\end{equation} 
Using the fact that $\Upsilon^+_{\fg}$ can be naturally upgraded to a topological bialgebra isomorphism, we may further upgrade it to an algebra isomorphism
\begin{equation}
\label{eqn:upsilon intro 2}
  \Upsilon_{\fg} \colon  U_q(L\fg) \iso \Big( \text{Drinfeld double of extended } \CS^+_{\fg} \Big) 
\end{equation}
This yields a complete description of the quantum affine algebra $U_q(L\fg)$ in terms of explicit shuffle algebras (the central charge can also be incorporated into $\Upsilon_{\fg}$ in a straightforward manner, but we will not do it for simplicity). We call isomorphisms such as \eqref{eqn:upsilon intro 1} and \eqref{eqn:upsilon intro 2} \emph{shuffle realizations} of the corresponding quantum loop algebras. In recent years, such shuffle realizations also appeared in the following situations:
\begin{itemize}[leftmargin=0.7cm]

\item 
type $A$ (including super and affine types, as well as $\widehat{\widehat{\fgl}}_1$) in \cite{FHHSY, N Shuffle, N Toroidal, Ts};

\medskip 
		
\item 
preprojective $K$-theoretic Hall algebras (\cite{N Wheel});

\medskip 
        
\item 
quantum loop algebras associated to symmetric Cartan matrices (\cite{N Symmetric});
		
\medskip 
		
\item 
BPS algebras associated to toric Calabi-Yau threefolds (\cite{N Reduced});
		
\medskip
		
\item 
twisted quantum affine algebras (\cite{NW}).
		
\end{itemize}

\medskip 
\noindent 
It is very interesting that the shuffle realizations in each of the situations above involve unique formulas and challenges, and thus the combinatorics and algebra behind them differ on a case-by-base basis. In the present paper, we add one more instance of shuffle realization to the list above, which involves novel combinatorics.


\medskip 
	
\subsection{Quantum affine superalgebras}
\label{sub:super}

We fix $N>0$, a so-called parity sequence 
$$
  \bs=(s_1,\dots,s_N) \in \{+1,-1\}^N
$$
and let $m$ and $n$ be the number of $+1$'s and $-1$'s occurring in $\bs$. The goal of the present paper is to provide the shuffle realization for the following quantum affine superalgebras (see \cite{Y} for their Drinfeld-Jimbo presentation and Drinfeld new  presentation in type $A$, and \cite{BFK, XZ} for their Drinfeld new presentation in types $B,C,D$), which we classify according to their decorated Dynkin diagrams as follows:

\medskip

\begin{itemize}[leftmargin=0.5cm]
		
\item[$(a)$]
$\usl$ for finite type $A$, corresponding to the decorated Dynkin diagram $\Gamma$
\begin{equation}
\label{eqn:dynkin a intro}
\begin{tikzcd}
  1 && 2 && \cdots && {N-2} && {N-1}
  \arrow["q^{-s_{N-1}}"', no head, from=1-9, to=1-7]
  \arrow["q^{-s_{N-2}}"', no head, from=1-7, to=1-5]
  \arrow["q^{-s_3}"', no head, from=1-5, to=1-3]
  \arrow["q^{-s_2}"', no head, from=1-3, to=1-1]
  \arrow[dashed, no head, from=1-7, to=1-7, looseness=4]
  \arrow[dashed, no head, from=1-9, to=1-9, looseness=4]
  \arrow[dashed, no head, from=1-3, to=1-3, looseness=4.5]
  \arrow[dashed, no head, from=1-1, to=1-1, looseness=4.5]
\end{tikzcd}
\end{equation}
We note that the assignment $\bs \mapsto \Gamma$ is bijective and switching $\bs \leadsto -\bs$ results in inverting all edge parameters. Explicitly, such $\Gamma$ can be described as follows: 
\begin{itemize}[leftmargin=0.7cm]

\item 
the vertex set is $\{1,\dots,N-1\}$ 

\item 
there is an edge between $k-1$ and $k$ with parameter $q^{-s_k}\in \{q^{\pm1 }\}$, for $1< k \leq N-1$ 

\item 
a vertex $1<k<N-1$ has a loop iff $s_k = s_{k+1}$, and its parameter is $q^{2s_k}$ 

\item 
the vertex $1$ may have a loop only with parameter $q^{2s_2}$ 

\item 
the vertex $N-1$ may have a loop only with parameter $q^{2s_{N-1}}$ 

\end{itemize}
Because of the above bijection, we will often write $\uslG = \usl$.

\medskip
\noindent
We also have the superalgebra $U_{q}(L\fg_{\wA_\mathbf{s}}) = U_{q}(L\fg_{\wA_\Gamma})$ of affine type $\wA$, whose decorated Dynkin diagram is obtained from~\eqref{eqn:dynkin a intro} by adding an extra vertex $0$ connected to both $N-1$ and $1$ with respective edge parameters $q^{-s_N}$ and $q^{-s_1}$, and having a loop at $0$ with parameter $q^{2s_1}$ iff $s_1=s_N$. This is the $d=1$ version of the more general quantum toroidal superalgebra $\Usl=\UslG$ that we recall in Subsection~\ref{sub:quantum toroidal sl}, with (non-symmetric) quiver depicted in~\eqref{eqn:quiver affine type A}.

\medskip 
		
\item[$(b)$] 
$\ub$, i.e.\ type $B$ as in Subsection \ref{sub:b} (this covers the classical $B(m,n)$ super types, that is $\fosp(2m+1|2n)$, for all possible Dynkin diagrams). We prefer to denote them by $\ubG$, where $\Gamma$ is one of two decorated diagrams~below: 

\medskip 

\begin{itemize}[leftmargin=*]

\item 
Case 1: 
\begin{equation*}
\label{eqn:dynkin b 1 intro}
\begin{tikzcd}
  \cdots && {N-2} && {N-1} && N
  \arrow["q^{\pm 1}"', no head, from=1-7, to=1-5]
  \arrow[no head, from=1-5, to=1-3]
  \arrow[no head, from=1-3, to=1-1]
  \arrow["q^{\mp 1}"', no head, from=1-7, to=1-7, looseness=4]
  \arrow[dashed, no head, from=1-5, to=1-5, looseness=4]
  \arrow[dashed, no head, from=1-3, to=1-3, looseness=4]
\end{tikzcd}
\end{equation*}
If $q^{\pm 1} = q^{-1}$, this construction recovers $\ub$ associated to $\bs$ satisfying $s_N=1$, in the loop (Drinfeld new) realization.

\medskip

\item 
Case 2: 
\begin{equation*}
\label{eqn:dynkin b 2 intro}
\begin{tikzcd}
  \cdots && {N-2} && {N-1} && N 
  \arrow["q^{\pm 1}"', no head, from=1-7, to=1-5]
  \arrow[no head, from=1-5, to=1-3]
  \arrow[no head, from=1-3, to=1-1]
  \arrow["q^{\pm 1}", no head, from=1-7, to=1-7, looseness=4]
  \arrow["q^{\mp 2}"', no head, from=1-7, to=1-7, looseness=8]
  \arrow[dashed, no head, from=1-5, to=1-5, looseness=4]
  \arrow[dashed, no head, from=1-3, to=1-3, looseness=4]
\end{tikzcd}
\end{equation*}
If $q^{\pm 1} = q$, this construction recovers $\ub$ associated to $\bs$ satisfying $s_N=-1$, in the loop (Drinfeld new) realization.

\end{itemize}

\medskip
\noindent
In both cases, the subdiagram to the left of $N$ is a type $A$ decorated diagram as in~\eqref{eqn:dynkin a intro}, and the dotted loop at $N-1$ means no loop or a loop with parameter $q^{\mp 2}$.

\medskip 

\item[$(c)$] 
$\uc$, i.e.\ type $C$ as in Subsection \ref{sub:c} (this covers the classical $C(n)$ and $D(m,n)$ super types, that is $\fosp(2m|2n)$, for all possible Dynkin diagrams that end with the longer even simple root, see~\cite[\S 3.1, diagrams ($C_N$)]{Y1}; $s_N=-1$). We prefer to denote them by $\ucG$, where $\Gamma$ is one of the following diagrams:

\medskip 

\begin{itemize}[leftmargin=*]

\item 
Case 1: 
\begin{equation*}
\label{eqn:dynkin c 1 intro}
\begin{tikzcd}
  \cdots &&  {N-2} && {N-1} && N
  \arrow["q^{\pm 2}"', no head, from=1-7, to=1-5]
  \arrow["q^{\pm 1}"', no head, from=1-5, to=1-3] 
  \arrow[no head, from=1-3, to=1-1]
  \arrow["q^{\mp 4}"', no head, from=1-7, to=1-7, looseness=4]
  \arrow["q^{\mp 2}"', no head, from=1-5, to=1-5, looseness=4]
  \arrow[dashed, no head, from=1-3, to=1-3, looseness=4]
\end{tikzcd}
\end{equation*}
If $q^{\pm 2} = q^{2}$, this construction recovers $\uc$ for the parity sequence $\bs$ with $s_{N-1}=-1$  (recall that $s_N=-1$), in the loop (Drinfeld new) realization.

\medskip 

\item 
Case 2: 
\begin{equation*}
\label{eqn:dynkin c 2 intro}
\begin{tikzcd}
  \cdots && N-3 && {N-2} && {N-1} && N
  \arrow["q^{\pm 2}"', no head, from=1-9, to=1-7]
  \arrow["q^{\mp 1}"', no head, from=1-7, to=1-5]
  \arrow["q^{\pm 1}"', no head, from=1-5, to=1-3]
  \arrow[no head, from=1-3, to=1-1]
  \arrow["q^{\mp 4}"', no head, from=1-9, to=1-9, looseness=4]
  \arrow[dashed, no head, from=1-3, to=1-3, looseness=4]
\end{tikzcd}
\end{equation*}
If $q^{\pm 2} = q^{2}$, this construction recovers $\uc$ when $s_{N-1}=1$ and $s_{N-2}=-1$ (recall that $s_N=-1$), 
in the loop (Drinfeld new) realization.

\medskip 

\item 
Case 3: 
\begin{equation*}
\label{eqn:dynkin c 3 intro}
\begin{tikzcd}
  \cdots && N-3 && {N-2} && {N-1} && N
  \arrow["q^{\pm 2}"', no head, from=1-9, to=1-7]
  \arrow["q^{\mp 1}"', no head, from=1-7, to=1-5]
  \arrow["q^{\mp 1}"', no head, from=1-5, to=1-3]
  \arrow[no head, from=1-3, to=1-1]
  \arrow["q^{\mp 4}"', no head, from=1-9, to=1-9, looseness=4]
  \arrow["q^{\pm 2}"', no head, from=1-5, to=1-5, looseness=4]
  \arrow[dashed, no head, from=1-3, to=1-3, looseness=4]
\end{tikzcd}
\end{equation*}
If $q^{\pm 2} = q^{2}$, this construction recovers $\uc$ with $s_{N-1}=1$ and $s_{N-2}=1$ (recall that $s_N=-1$), in the loop (Drinfeld new) realization. 

\end{itemize}

\medskip
\noindent
In all cases, the part to the left of $N$ is a type $A$ decorated diagram as in~\eqref{eqn:dynkin a intro}.

\medskip 
		
\item[$(d)$] 
$\ud$, i.e.\ type $D$ as in Subsection \ref{sub:d} (this covers the classical $C(n)$ and $D(m,n)$ super types, that is $\fosp(2m|2n)$, for all possible Dynkin diagrams that end with a fork-type pattern as in~\cite[\S 3.1, diagrams ($D_N$)]{Y1}; $s_N=1$). We prefer to denote them by $\udG$, where $\Gamma$ is one of the following diagrams:

\medskip

\begin{itemize}

\item 
Case 1: 
\begin{equation*}
\label{eqn:dynkin d 1 intro}
\begin{tikzcd}
  &&&& {N-1} \\
  \cdots && {N-2} \\
  &&&& N
  \arrow[no head, from=2-1, to=2-3]
  \arrow["q^{\pm 1}", no head, from=2-3, to=1-5]
  \arrow["q^{\mp 2}", swap, no head, from=1-5, to=1-5, looseness=4]
  \arrow["q^{\pm 1}"', no head, from=2-3, to=3-5]
  \arrow["q^{\mp 2}", swap, no head, from=3-5, to=3-5, looseness=4]
  \arrow[dashed, no head, from=2-3, to=2-3, looseness=4]
\end{tikzcd}
\end{equation*}
If $q^{\pm 1} = q^{-1}$, this construction recovers $\ud$ for the parity sequence $\bs$ with $s_{N-1}=1$ (recall that $s_N=1$), in the loop (Drinfeld new) realization.

\medskip 

\item 
Case 2: 
\begin{equation*}
\label{eqn:dynkin d 2 intro}
\begin{tikzcd}
  &&&& {N-1} \\
  \cdots && {N-2} \\
  &&&& N
  \arrow[no head, from=2-1, to=2-3]
  \arrow["q^{\pm 1}", no head, from=2-3, to=1-5]
  \arrow["q^{\pm 1}"', no head, from=2-3, to=3-5]
  \arrow["{q^{\mp 2}}", no head, from=1-5, to=3-5]
  \arrow[dashed, no head, from=2-3, to=2-3, looseness=4]
\end{tikzcd}
\end{equation*}
If $q^{\pm 1} = q$, this construction recovers $\ud$ with $s_{N-1}=-1$ (recall that $s_N=1$), in the loop (Drinfeld new) realization. 

\end{itemize}

\medskip
\noindent
In both cases, the part to the left of $N-1,N$ is a type $A$ decorated diagram, and the dotted loop at $N-2$ means either no loop or a loop with parameter~$q^{\mp 2}$. 
   
\end{itemize}

\medskip
\noindent
A vertex $i$ of a decorated Dynkin diagram as in the list above is called \emph{even} (denoted by $|i|=0$) if it has a loop and \emph{odd} (denoted by $|i|=1$) otherwise; the only exception is type $B$ Case 2, where the vertex $N$ with two loops is declared to be odd ($|N|=1$). For a type $X \in \{B,C,D\}$ decorated Dynkin diagram, let its

\medskip 

\begin{itemize}[leftmargin=0.7cm]

\item 
\emph{body} be the type $A$ decorated Dynkin diagram with vertices $1,\dots,N-1-\delta_{XD}$;

\medskip 

\item 
\emph{head} be the subdiagram with vertices $N-1-\delta_{XD},\dots,N$;

\medskip 

\item 
\emph{neck} be the vertex $N-1-\delta_{XD}$, and its \emph{neck parameter} be the decoration $q^{\pm 1}$ or $q^{\pm 2}$ of the horizontal/slanted edge connecting the neck with vertex $N$.

\end{itemize}

\medskip 
\noindent
Inverting all edge parameters in a decorated Dynkin diagram of type $X \in \{B,C,D\}$ leads to the quantum affine superalgebra $U_{q^{-1}}(L\fg_{X_{-\bs}})$.


\medskip 

\subsection{Shuffle realizations}
\label{sub:shuffle intro}

Recall the bijection $\bs \leftrightarrow \Gamma$ between parity sequences and decorated Dynkin diagrams of type $A$. As shown in \cite{Ts}, we have the following analogue of the shuffle realization isomorphism \eqref{eqn:upsilon intro 1} 
\begin{equation*}
\uslGp \simeq \CS^+_{A_{\Gamma}} \subset \bigoplus_{\{n_i \geq 0\}_{i \in \{1,\dots,N-1\}}} 
\BC[z_{i1}^{\pm 1},\dots,z_{in_i}^{\pm 1}]^{\sym}_{i \in \{1,\dots,N-1\}} 
\end{equation*}
where $\CS^+_{A_{\Gamma}}$ is the subset of Laurent polynomials which vanish at length $3$ wheels 
\begin{equation}
\label{eqn:wheel finite even}
		\begin{tikzcd}
			{z_{i-1,1}} && {z_{i2}} &&&& {z_{i2}} \\ \\
			{z_{i1}} &&&& {z_{i1}} && {z_{i+1,1}}
			\arrow["{q^{-s_{i}}}"', from=1-1, to=3-1]
			\arrow["{q^{-s_i}}"', from=1-3, to=1-1]
			\arrow["{q^{-s_{i}}}", from=1-7, to=3-7]
			\arrow["{q^{2s_i}}"', from=3-1, to=1-3]
			\arrow["{q^{2s_i}}", from=3-5, to=1-7]
			\arrow["{q^{-s_i}}", from=3-7, to=3-5]
		\end{tikzcd}
\end{equation}
for each $i$ having a loop with parameter $q^{2s_i}$ (so that $|i|=0$) and length $4$ wheel
\begin{equation}
\label{eqn:wheel finite odd}
		\begin{tikzcd}
			{z_{i-1,1}} && {z_{i2}} \\ \\
			{z_{i1}} && {z_{i+1,1}}
			\arrow["{q^{-s_i}}", from=1-1, to=1-3]
			\arrow["{q^{-s_{i+1}}}", from=1-3, to=3-3]
			\arrow["{q^{-s_i}}", from=3-1, to=1-1]
			\arrow["{q^{-s_{i+1}}}", from=3-3, to=3-1]
		\end{tikzcd}
\end{equation}
for each $i$ without a loop (so that $|i|=1$).

\medskip
\noindent
In affine type $A$ such that $m\neq n$, the following result was proved in \cite{N Reduced}, when the Calabi-Yau threefold is the generalized conifold (see also \cite{FJM} for a conjectured shuffle presentation of the quantum toroidal $\fgl_{m|n}$)
\begin{equation}
\label{eqn:shuffle realization affine}
  \UslGp \simeq \wCS^+_{\wA_{\Gamma}} \subset \bigoplus_{\{n_i \geq 0\}_{i \in \BZ/N\BZ}} 
  \BC[z_{i1}^{\pm 1},\dots,z_{in_i}^{\pm 1}]^{\sym}_{i \in \BZ/N\BZ} 
\end{equation}
where $\wCS^+_{\wA_{\Gamma}}$ is the subset of Laurent polynomials which vanish at the length $3$ wheels 
\begin{equation}
\label{eqn:wheel affine even}
		\begin{tikzcd}
			{z_{i-1,1}} && {z_{i2}} &&&& {z_{i2}} \\ \\
			{z_{i1}} &&&& {z_{i1}} && {z_{i+1,1}}
			\arrow["{q^{-s_{i}}d^{-s_{i}}}"', from=1-1, to=3-1]
			\arrow["{q^{-s_i}d^{s_i}}"', from=1-3, to=1-1]
			\arrow["{q^{-s_i}d^{-s_i}}", from=1-7, to=3-7]
			\arrow["{q^{2s_i}}"', from=3-1, to=1-3]
			\arrow["{q^{2s_i}}", from=3-5, to=1-7]
			\arrow["{q^{-s_i}d^{s_i}}", from=3-7, to=3-5]
		\end{tikzcd}
\end{equation}
for each $i$ having a loop with parameter $q^{2s_i}$ (so that $|i|=0$) and length $4$ wheel
\begin{equation}
\label{eqn:wheel affine odd}
		\begin{tikzcd}
			{z_{i-1,1}} && {z_{i2}} \\ \\
			{z_{i1}} && {z_{i+1,1}}
			\arrow["{q^{-s_i}d^{-s_i}}", from=1-1, to=1-3]
			\arrow["{q^{-s_{i+1}}d^{-s_{i+1}}}", from=1-3, to=3-3]
			\arrow["{q^{-s_i}d^{s_i}}", from=3-1, to=1-1]
			\arrow["{q^{-s_{i+1}}d^{s_{i+1}}}", from=3-3, to=3-1]
		\end{tikzcd}
\end{equation}
for each $i$ without a loop (so that $|i|=1$). When $m=n$, we conjecture that \eqref{eqn:shuffle realization affine} holds with $\CS^+_{\wA_\Gamma}$ defined as the subset of Laurent polynomials which vanish at the wheels 
\begin{equation}
\label{eqn:wheel special 1}
\begin{tikzcd}
  && {z_{N-1,1}} &&& {z_{01}} && \\ \\
  \vdots &&&&&&& {z_{11}} \\ \\
  && {z_{31}} &&& {z_{21}}
  \arrow["{q^{-s_N}d^{-s_N}}", from=1-3, to=1-6]
  \arrow["{q^{-s_1}d^{-s_1}}", from=1-6, to=3-8]
  \arrow[from=3-1, to=1-3]
  \arrow["{q^{-s_2}d^{-s_2}}", from=3-8, to=5-6]
  \arrow[from=5-3, to=3-1]
  \arrow["{q^{-s_3}d^{-s_3}}", from=5-6, to=5-3]
\end{tikzcd}
\end{equation}
\begin{equation}
\label{eqn:wheel special 2}
\begin{tikzcd}
  && {z_{N-1,1}} &&& {z_{01}} && \\ \\
  \vdots &&&&&&& {z_{11}} \\ \\
  && {z_{31}} &&& {z_{21}}
  \arrow["{q^{-s_N}d^{s_N}}", swap, from=1-6, to=1-3]
  \arrow["{q^{-s_1}d^{s_1}}", swap, from=3-8, to=1-6]
  \arrow[from=1-3, to=3-1]
  \arrow["{q^{-s_2}d^{s_2}}", swap, from=5-6, to=3-8]
  \arrow[from=3-1, to=5-3]
  \arrow["{q^{-s_3}d^{s_3}}", swap, from=5-3, to=5-6]
\end{tikzcd}
\end{equation}
on top of \eqref{eqn:wheel affine even}-\eqref{eqn:wheel affine odd}. All the results above will be reviewed in Section \ref{sec:gl}. Then, in Section~\ref{sec:osp} we will prove the following results on the shuffle realizations of the quantum affine algebras in the cases $(b), (c), (d)$ above (let $\pI = \{1,\dots,N\}$).

\medskip

\begin{theorem}
\label{thm:d intro}
We have the following shuffle realization:
\begin{equation}
\label{eqn:shuffle d intro}
  \udGp \simeq \CS^+_{D_{\Ga}} 
\end{equation}
$$
= \Big \{ E(z_{ia})_{i \in \pI, a \geq 1} 
  \text{ which vanishes at the wheels \eqref{eqn:wheel finite even}, \eqref{eqn:wheel finite odd}, \eqref{eqn:wheel d 1}} \Big\} 
$$
where \eqref{eqn:wheel finite even} applies to all indices $i,i\pm 1 < N$ as well as to $(i,i\pm 1)$ replaced by $(N-2,N)$ and $(N,N-2)$, while \eqref{eqn:wheel finite odd} applies to all indices $i-1,i,i+1 < N$ as well as to $(i-1,i,i+1)$ replaced by $(N-3,N-2,N)$. Additionally, the following length $3$ wheels are imposed in Case 2, where we write $q^{\pm 1}$ for the neck parameter: 
\begin{equation}
\label{eqn:wheel d 1}
		\begin{tikzcd}
	&& {z_{N-1,1}} &&&& {z_{N-1,1}} \\
	{z_{N-2,1}} &&&& {z_{N-2,1}} \\
	&& {z_{N1}} &&&& {z_{N1}}
	\arrow["{q^{\mp 2}}", from=1-3, to=3-3]
	\arrow["q^{\pm 1}"', from=1-7, to=2-5]
	\arrow["q^{\pm 1}", from=2-1, to=1-3]
	\arrow["q^{\pm 1}"', from=2-5, to=3-7]
	\arrow["q^{\pm 1}", from=3-3, to=2-1]
	\arrow["{q^{\mp 2}}"', from=3-7, to=1-7]
\end{tikzcd}
\end{equation}
\end{theorem}

\medskip

\begin{theorem}
\label{thm:b intro}
We have the following shuffle realization
\begin{equation}
\label{eqn:shuffle b intro}
  \ubGp \simeq \CS^+_{B_{\Ga}} 
\end{equation} 
$$
  = \Big \{ E(z_{ia})_{i \in \pI, a \geq 1} 
  \text{ which vanishes at the wheels \eqref{eqn:wheel finite even}, \eqref{eqn:wheel finite odd}, \eqref{eqn:wheel b 1}, \eqref{eqn:wheel b 2}} \Big\} 
$$
where \eqref{eqn:wheel finite even}-\eqref{eqn:wheel finite odd} apply to all indices with the exception of \eqref{eqn:wheel finite even} for $(i,i-1)=(N,N-1)$ in Case 1, when we impose instead the length $4$ wheel (cf.~\eqref{eqn:classic wheels} for non-super type~$B$)
\begin{equation}
\label{eqn:wheel b 1}
  \begin{tikzcd}
  && {z_{N-1,1}}  && \\ \\ 
  {z_{N 1}} && {z_{N 2}} && {z_{N 3}}
  \arrow["{q^{\pm 1}}"', swap, from=3-1, to=1-3]
  \arrow["{q^{\mp 1}}"', swap, from=3-3, to=3-1]
  \arrow["{q^{\mp 1}}"', swap, from=3-5, to=3-3]
  \arrow["{q^{\pm 1}}"', swap, from=1-3, to=3-5]
  \end{tikzcd}
\end{equation}
and finally the following length $3$ wheel is imposed only in Case 2
\begin{equation}
\label{eqn:wheel b 2}
			\begin{tikzcd}
				&& {z_{N1}} \\ \\ \\
				{z_{N2}} &&&& {z_{N3}}
				\arrow["q^{\pm 1}"', swap, from=4-1, to=1-3]
				\arrow["{q^{\mp 2}}"', swap, from=4-5, to=4-1]
				\arrow["q^{\pm 1}"', swap, from=1-3, to=4-5]
			\end{tikzcd}
\end{equation}
(in the pictures above, $q^{\pm 1}$ denotes the neck parameter).
\end{theorem}

\medskip

\begin{theorem}
\label{thm:c intro}
We have the following shuffle realization: 
\begin{equation}
\label{eqn:shuffle c intro}
  \ucGp \simeq \CS^+_{C_{\Ga}}
\end{equation}
$$
  = \Big \{ E(z_{ia})_{i \in \pI, a \geq 1} 
  \text{ which vanishes at the wheels \eqref{eqn:wheel finite even}, \eqref{eqn:wheel finite odd}, \eqref{eqn:wheel c 1}, \eqref{eqn:wheel c 2}, \eqref{eqn:wheel c 3}} \Big\} 
$$
where \eqref{eqn:wheel finite even} applies to all indices $i,i\pm 1 < N$ as well as to $(i,i-1)=(N,N-1)$, while \eqref{eqn:wheel finite odd} applies to all indices $i-1,i,i+1 < N$. Additionally, if we let $q^{\pm 2}$ denote the neck parameter, we impose in Case 1 the following length $4$ wheel: 
\begin{equation}
\label{eqn:wheel c 1}
  \begin{tikzcd}
				&&&& {z_{N-1,2}} \\ \\
				&& {z_{N-1,3}} \\ &&&  \\
				{z_{N-1,1}} &&&& {z_{N1}}
				\arrow["{q^{\pm 2}}"', from=5-1, to=5-5]
				\arrow["{q^{\pm 2}}"', from=5-5, to=1-5]
				\arrow["{q^{\mp 2}}"', from=1-5, to=3-3]
				\arrow["{q^{\mp 2}}"', from=3-3, to=5-1]
	\end{tikzcd}
\end{equation}
(cf.~\eqref{eqn:classic wheels} for non-super type $C$), we impose in Case 2 the following length $6$ wheel: 
\begin{equation}
\label{eqn:wheel c 2}
	\begin{tikzcd}
				&& {z_{N-2,1}} && {z_{N-1,2}} \\ \\
				{z_{N-2,2}} && {z_{N-1,3}} \\ &&&  \\
				{z_{N-1,1}} &&&& {z_{N1}}
				\arrow["{q^{\pm 2}}"', from=5-1, to=5-5]
				\arrow["{q^{\pm 2}}"', from=5-5, to=1-5]
				\arrow["{q^{\mp 1}}"', from=1-5, to=1-3]
				\arrow["{q^{\mp 1}}"', from=1-3, to=3-3]
				\arrow["{q^{\mp 1}}"', from=3-3, to=3-1]
				\arrow["{q^{\mp 1}}"', from=3-1, to=5-1]
	\end{tikzcd}
\end{equation}
and we impose in Case 3 the following length $7$ wheel: 
\begin{equation}
			\label{eqn:wheel c 3}
			\begin{tikzcd}
				z_{N-3,1} && {z_{N-2,1}} && {z_{N-1,2}} \\ \\
				{z_{N-2,2}} && {z_{N-1,3}} \\ &&& \boxed{2} \\
				{z_{N-1,1}} &&&& {z_{N1}}
				\arrow["{q^{\pm 2}}"', from=5-1, to=5-5]
				\arrow["{q^{\pm 2}}"', from=5-5, to=1-5]
				\arrow["{q^{\mp 1}}"', from=1-5, to=1-3]
				\arrow["{q^{\mp 1}}"', swap, from=1-3, to=3-3]
				\arrow["{q^{\mp 1}}"', from=3-1, to=5-1]
				\arrow["{q^{\mp 1}}"', from=1-1, to=3-1]
				\arrow["{q^{\mp 1}}"', from=1-3, to=1-1]
				\arrow["{q^{\mp 1}}"', swap, from=3-3, to=3-1]
				\arrow["{q^{\pm 2}}", from=3-1, to=1-3]
		\end{tikzcd}
\end{equation}
The meaning of the $\boxed{2}$ in \eqref{eqn:wheel c 3} is the following: any Laurent polynomial $E$ which satisfies \eqref{eqn:wheel finite even} for $i=N-2$ already vanishes at the wheel \eqref{eqn:wheel c 3}, due to the two triangles in the top-left corner. Instead, we further require such a Laurent polynomial $E$ to vanish to order $2$ at the wheel \eqref{eqn:wheel c 3}, in a sense which will be made precise in~\eqref{eqn:rigorous}.
\end{theorem}

\medskip 
\noindent 
Due to the higher order vanishing condition, the wheel condition \eqref{eqn:wheel c 3} is significantly more complicated than the straightforward specializations \eqref{eqn:wheel finite even}-\eqref{eqn:wheel finite odd} or even \eqref{eqn:wheel special 1}-\eqref{eqn:wheel special 2}, so we dub it a \emph{scary wheel}. We hope that in the future, we will be able to understand such wheels as part of a more general framework.


\medskip 
	
\subsection{Shrubs}
\label{sub:intro shrub}
	
In the present paper, we also classify shrubs in super-types $A$ and $\wA$. Let us give an overview of this important technical notion, which is instrumental in the proofs of Theorems \ref{thm:d intro}, \ref{thm:b intro}, \ref{thm:c intro}: inspired by $4$-dimensional superconformal field theories and later $1$-dimensional quantum mechanics, \cite{GLY, LY, NW1, NW2} have studied the so-called \emph{quiver quantum toroidal algebras} associated to certain quivers $Q$ drawn on the torus $S^1 \times S^1$. These quivers are quite special, as the arrows form oriented loops around every face of $Q$. More importantly, these quivers correspond to toric Calabi-Yau threefolds, with many of the geometric features of the latter being captured by the combinatorics of the former. For example, the quiver in Figure~\ref{fig:c3} corresponds to the simplest toric Calabi-Yau threefold, namely~$\BC^3$ (the torus $S^1 \times S^1$ is represented by the square, with opposite sides identified).

\begin{figure}[h]
  \includegraphics[scale=2]{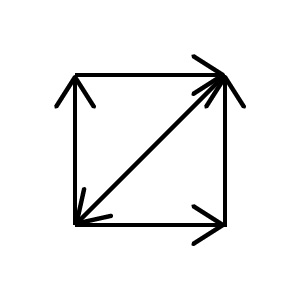} 
  \caption{The quiver $Q$ corresponding to $\BC^3$.}
  \label{fig:c3}
\end{figure}

\medskip
\noindent 
A \emph{shrub} is a generalization of the notion of sub-tree of the universal cover $\tQ \subset \BR^2$ of the quiver $Q \subset S^1 \times S^1$, that was introduced by the first-named author in \cite{N Reduced}; see Figure \ref{fig:C3 cover} for an example when $Q$ is the quiver of Figure \ref{fig:c3}. This notion allowed us to calculate in \loccit the \emph{reduced} version of quiver quantum toroidal algebras, i.e.\ to impose precisely the correct higher order relations between generators that ensure the shuffle realizations of these algebras (and also ensure faithful actions of said algebras on vector spaces of crystal configurations, which are important combinatorial models for the BPS states that arise in physics, see \cite{LY}).

\begin{figure}[h]
  \includegraphics[scale=2]{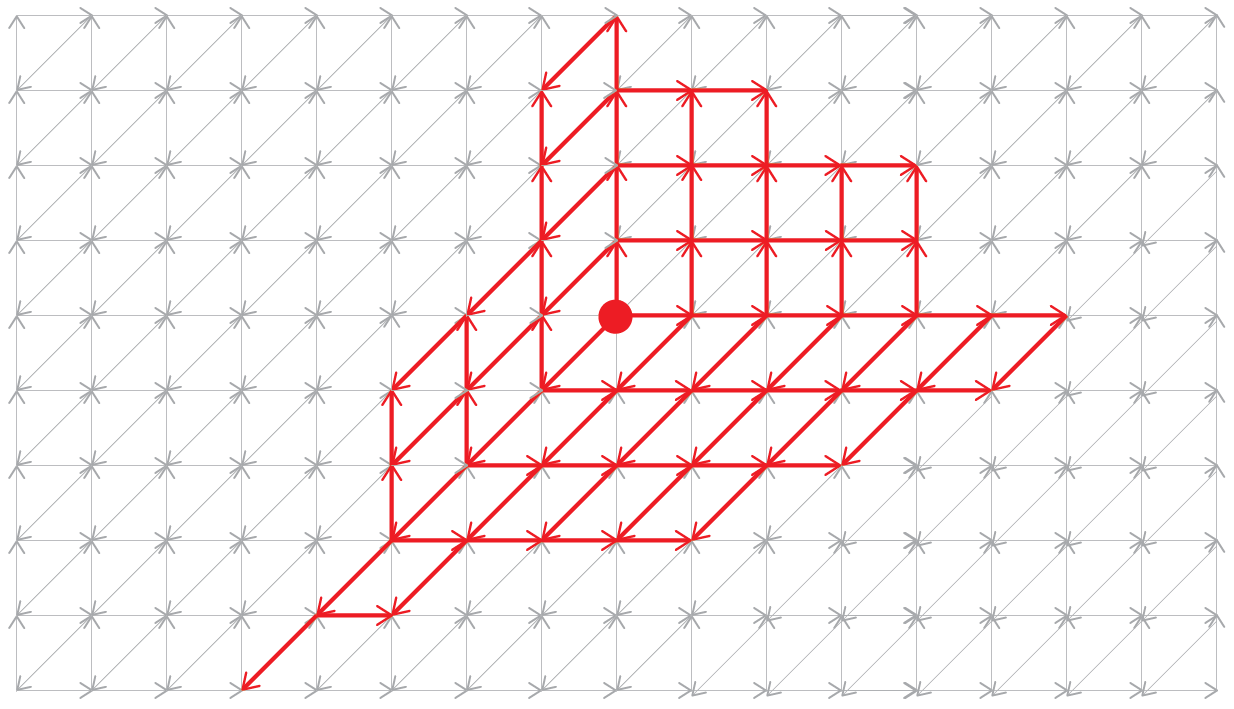} 
  \caption{The quiver $\tQ$ corresponding to $\BC^3$, and a shrub (in red).}
  \label{fig:C3 cover}
\end{figure}

\medskip 
\noindent 
In Subsection \ref{sub:shrubs sl}, we will define a certain generalization of the quiver from Figure~\ref{fig:C3 cover}, which corresponds to type $\wA_\Gamma$ with an arbitrary (periodic) parity sequence $\bs$ of $\pm 1$'s. For instance, Figure \ref{fig:grid} shows the corresponding quiver $\tQ$ for the parity sequence indicated on the sidelines. For such quivers, the main result of \cite{N Reduced} states:
\begin{equation*}
\Big(\text{reduced half quantum toroidal algebra of the quiver of Figure \ref{fig:grid}}\Big) \simeq \wCS^+_{\wA_{\Gamma}}
\end{equation*}
Using the isomorphism \eqref{eqn:shuffle realization affine}, this formula implies that the reduced quiver quantum toroidal algebra for the quiver associated to $\bs$ in Figure \ref{fig:grid} is isomorphic to
$$
  \UslG
$$
which was long known to mathematical physicists (\cite{GLY, LY, NW1, NW2}).

\medskip
\noindent
In the present paper, we completely classify shrubs for finite and affine type $A$ decorated Dynkin diagrams. For finite and affine type $B,C,D$ decorated Dynkin diagrams, we study certain generalizations of shrubs that we call \emph{asps (acceptable specialization patterns)}. Once we do so, we will be able to prove that the so-called Drinfeld-Serre-like relations discovered in \cite{Y} (type $A$) and \cite{BFK, XZ} (types $B,C,D$) are precisely dual to the wheel conditions in Subsection \ref{sub:super}. As we will explain in Section \ref{sec:arbitrary}, this is the key fact needed to prove Theorems~\ref{thm:d intro},~\ref{thm:b intro},~\ref{thm:c intro}. In Subsection \ref{sub:double}, we explain how to prove the following.

\medskip

\begin{corollary}
\label{cor:intro-nondegeneracy}
When $U$ is one of the quantum affine superalgebras 
$$
  \uslG,\ \ \ubG,\ \ \ucG,\ \ \udG,
$$
the natural pairing between its half subalgebras
$$
  U^+ \otimes U^- \rightarrow \BC
$$
is non-degenerate (see \cite{Y} for an analogue in the Drinfeld-Jimbo presentation).
\end{corollary}


\medskip

\subsection{Application to quantum toroidal superalgebras}
\label{sub:toroidal intro}

As we will explain in Section \ref{sec:arbitrary}, the principles above can be applied even beyond finite types $B,C,D$. We recall from Subsection \ref{sub:super} that a finite type $B,C,D$ decorated Dynkin diagram $\Gamma$ is obtained by gluing one of the following patterns, denoted $\Gh$ (\emph{head} of $\Gamma$)
\begin{equation}
\label{eqn:type b 1 head}
  \text{Type $B$ Case 1:} 
\qquad 
  \begin{tikzcd}
    N-1 && N
    \arrow["q^{\pm 1}"', no head, from=1-3, to=1-1]
    \arrow["q^{\mp 1}"', no head, from=1-3, to=1-3, looseness=4]
    \arrow[dashed, no head, from=1-1, to=1-1, looseness=4]
  \end{tikzcd} 
\end{equation}
\begin{equation}
\label{eqn:type b 2 head}
  \text{Type $B$ Case 2:} 
\quad  
  \begin{tikzcd}
    N-1 && N 
    \arrow["q^{\pm 1}"', no head, from=1-3, to=1-1]
    \arrow["q^{\pm 1}", no head, from=1-3, to=1-3, looseness=4]
    \arrow["q^{\mp 2}"', no head, from=1-3, to=1-3, looseness=8]
    \arrow[dashed, no head, from=1-1, to=1-1, looseness=4]
  \end{tikzcd}
\end{equation}
\begin{equation}
\label{eqn:type c head}
  \text{Type $C$:} 
\qquad \qquad
  \begin{tikzcd}
    N-1 && N
    \arrow["q^{\pm 2}"', no head, from=1-3, to=1-1]
    \arrow["q^{\mp 4}"', no head, from=1-3, to=1-3, looseness=4]
    \arrow[dashed, no head, from=1-1, to=1-1, looseness=4]
  \end{tikzcd} 
\end{equation}
\begin{equation}
\label{eqn:type d 1 head}
  \text{Type $D$ Case 1:} \
\quad 
  \begin{tikzcd}
    && {N-1} \\
    {N-2} \\
    && N
    \arrow["q^{\pm 1}", no head, from=2-1, to=1-3]
    \arrow["q^{\mp 2}", swap, no head, from=1-3, to=1-3, looseness=4]
    \arrow["q^{\pm 1}"', no head, from=2-1, to=3-3]
    \arrow["q^{\mp 2}", swap, no head, from=3-3, to=3-3, looseness=4]
    \arrow[dashed, no head, from=2-1, to=2-1, looseness=4]
  \end{tikzcd} 
\end{equation}
\begin{equation}
\label{eqn:type d 2 head}
  \text{Type $D$ Case 2:} \
\qquad 
  \begin{tikzcd}
    && {N-1} \\
    {N-2} \\
    && N
    \arrow["q^{\pm 1}", no head, from=2-1, to=1-3]
    \arrow["q^{\pm 1}"', no head, from=2-1, to=3-3]
    \arrow["{q^{\mp 2}}", no head, from=1-3, to=3-3]
    \arrow[dashed, no head, from=2-1, to=2-1, looseness=4]
  \end{tikzcd}
\end{equation}
to the rightmost end of a type $A$ decorated Dynkin diagram $\Gb$ (\emph{body} of $\Gamma$) as in~\eqref{eqn:dynkin a intro}. The dashed loops at $N-1$ (types $B,C$) or $N-2$ (type $D$) should coincide in $\Gb$ and $\Gh$, both denoting either no loop or a loop with parameter $q^{\mp 2}$.

\medskip
\noindent
An \emph{extended decorated Dynkin diagram} of type $B,C,D$ will refer to the result $\Gamma$ of gluing two \emph{heads} $X,Y \in \{B,C,D\}$ as in diagrams \eqref{eqn:type b 1 head}-\eqref{eqn:type d 2 head} to the leftmost and rightmost (respectively) vertices of a type $A$ decorated Dynkin diagram $\Gb$, which is called the \emph{body} of $\Gamma$; see Subsection \ref{sub:extended} for pictures. We assume that $|\Gb|\geq 2$ (and $|\Gb|\geq 3$ if $X=Y=C$). We note that if we remove the left (respectively right) head $X$ (respectively $Y$), the resulting shape $\Gamma'$ (respectively $\Gamma''$) is a decorated Dynkin diagram of type $Y$ (respectively $X$) $\in \{B,C,D\}$. To this data, we associate
\begin{equation}
\label{eqn:toroidal intro}
  \text{a quantum algebra }U^+_q(L\fg_{\wXYG}) \text{ and a shuffle algebra } \CS^+_{\wXYG}
\end{equation}
(see Subsections~\ref{sub:toroidal}-\ref{sub:wheel toroidal} for details; in short, we only impose those relations on the quantum algebra and those wheel conditions on the shuffle algebra which come either from the body or from the two heads separately). The doubles of the algebras above are defined as in Subsection \ref{sub:double}, and in particular we refer to the double of the quantum algebra in \eqref{eqn:toroidal intro} as a type $XY$ \emph{quantum toroidal superalgebra}.

\medskip

\begin{theorem}
\label{thm:toroidal intro}
We have the following shuffle realization
\begin{equation}
\label{eqn:shuffle toroidal intro}
  U^+_q(L\fg_{\wXYG}) \simeq \CS^+_{\wXYG} 
\end{equation} 
for any extended decorated Dynkin diagram $\Gamma$ that is non-special (see Definition~\ref{def:special}). 
\end{theorem}

\medskip
\noindent
Moreover, Corollary~\ref{cor:intro-nondegeneracy} holds for any $U_q(L\fg_{\wXYG})$ with non-special $\Gamma$, and for any $\UslG$ with $\Gamma$ such that $m \ne n$ (we will also call the latter $\Gamma$ \emph{non-special}).


\medskip

\subsection{Application to representation theory}
\label{sub:q-char intro}

A very useful consequence of shuffle realizations as in Theorems \ref{thm:d intro}, \ref{thm:b intro}, \ref{thm:c intro}, \ref{thm:toroidal intro}, \ref{thm:shuffle sl} arises in the study of simple modules of quantum affine/toroidal superalgebras of types $A,B,C,D$. In what follows, $\CS^+$ will refer to any one of the shuffle algebras in the right-hand sides of \eqref{eqn:shuffle d intro}, \eqref{eqn:shuffle b intro}, \eqref{eqn:shuffle c intro} or in Definitions \ref{def:wheel sl} and \ref{def:extended shuffle} (in fact, the discussion below applies even beyond classical types, such as the type $D(2,1;\alpha)$ shuffle algebra studied in \cite{FH}).

\medskip 
\noindent 
We recall that finite-dimensional modules in the non-super setting were studied by \cite{CP}, and a wider framework of category $\CO$ modules was introduced in \cite{H Shifted, HJ}. All these modules admit a notion of $q$-characters, following the seminal work of~\cite{FR}. The generalization of this theory to quantum affine superalgebras is currently underway, see for instance \cite{Lee} for the study of $q$-characters in super type $A$ and \cite{DM} for connections with $R$-matrices and intertwiners in all types.

\medskip 
\noindent 
The first-named author showed in \cite{N Cat} that shuffle algebras provide a good approach for the construction of simple modules in category $\CO$, and we will review this framework here. For simplicity, in what follows we will discuss the simple modules 
\begin{equation}
\label{eqn:simple module}
  L^{\text{sh}}(\bpsi), \quad \forall\, \bpsi = (\psi_i(z) \in \BC(z)^\times)_{i \in I}
\end{equation}
of the shifted quantum loop algebra (see Subsection \ref{sub:double} and especially Remark \ref{rem:shifted} for the definition of shifted doubles), defined by analogy with~\cite{H Shifted}. The version of this theory where the word ``shifted" is replaced by ``Borel subalgebra" is treated as in \cite{HJ, N Cat}, and its equivalence with the discussion above was shown in \cite{HN}. Let
\begin{equation}
\label{eqn:j}
  J^{\text{sh}}(\bpsi) = \left\{E \in \CS^+ \text{ s.t. }  \int_{z_1} \dots \int_{z_n} \frac {E (z_1,\dots,z_n)(\text{any monomial})}{\prod_{1\leq a < b \leq n} \zeta_{i_bi_a} \left(\frac {z_b}{z_a} \right)} \prod_{a=1}^n \psi_{i_a}(z_a) = 0 \right\}
\end{equation}
In the formula above, $\int_z$ denotes the difference between the residue at $\infty$ and the residue at 0 of the rational function in the integrand. For any $E \in \CS^+$ and any
\begin{equation}
\label{eqn:sequence of i}
  i_1,\dots,i_n \in I
\end{equation}
we use in the right-hand side of \eqref{eqn:j} the notation 
$$
  E(z_1,\dots,z_n)
$$
to mean that we plug each variable $z_a$ into one of the variables $z_{i_a\bullet_a}$ of $E$ (where $\bullet_1,\dots,\bullet_n$ denote the minimal positive integers such that $\bullet_a < \bullet_b$ whenever $a<b$ and $i_a = i_b$). Therefore, in the right-hand side of \eqref{eqn:j} we require the vanishing of certain residues of $E$ times arbitrary monomials, for all possible sequences \eqref{eqn:sequence of i}. By analogy with \cite{N Cat}, the underlying vector space of the simple module \eqref{eqn:simple module} satisfies
\begin{equation*}
  L^{\text{sh}}(\bpsi) \simeq \CS^+ / J^{\text{sh}}(\bpsi)
\end{equation*}
The isomorphism above immediately implies that the $q$-character of $L^{\text{sh}}(\bpsi)$ can be expressed in terms of dimensions of certain quotients of shuffle algebras. More categorically, it shows that shuffle algebras give natural settings for the underlying vector spaces of simple modules of (shifted) quantum affine/toroidal superalgebras.


\medskip 

\subsection{Historical analysis}

The theory of Lie superalgebras and their quantizations is more interesting than the non-super case, particularly due to the following (here, we shall use $\bs$ instead of $\Gamma$ in order to better match the existing literature):

\medskip

\begin{itemize}[leftmargin=0.7cm]

\item[-] 
every Lie superalgebra $\fg$ has several Dynkin diagrams; for the classical types $X \in \{A,B,C,D\}$ treated in this paper, such Dynkin diagrams are parameterized by parity sequences $\bs$; we shall use $\fg_{X_{\bs}}$ for the associated Lie superalgebra 

\medskip

\item[-] 
the Chevalley-Serre presentation of each $\fg_{X_\bs}$ features new higher order relations, depending in a case-by-case way on the respective $\bs$ (see~\cite{Y,Z})

\medskip

\item[-] 
the Lie algebras $\fg_{X_\bs}$ associated with different parity sequences $\bs,\bs'$ are non-trivially isomorphic ($\fg_{X_\bs} \simeq \fg_{X_{\bs'}}$) through the Weyl groupoid action (with simple reflections in odd isotropic roots introduced in~\cite[Appendix]{LSS})

\medskip

\item[-] 
each $\fg_{X_\bs}$ and its affinization $\fg_{\widehat{X}_\bs}$ admit quantizations $U^{DJ}_q(\fg_{X_\bs})$ and $U^{DJ}_q(\fg_{\widehat{X}_\bs})$ in the sense of Drinfeld and Jimbo, featuring highly nontrivial higher order relations that one needs to impose in order to have an expected  Drinfeld double realization of these algebras, i.e.\ have a non-degenerate bialgebra pairing between their two halves (this is a key result of~\cite{Y})

\medskip

\item[-]
For any parity sequences $\bs,\bs'$, the isomorphisms of the corresponding finite and affine quantum supergroups $U^{DJ}_q(\fg_{X_\bs})\simeq U^{DJ}_q(\fg_{X_{\bs'}})$, $U^{DJ}_q(\fg_{\widehat{X}_\bs})\simeq U^{DJ}_q(\fg_{\widehat{X}_{\bs'}})$ are highly nontrivial and are provided through the braid group action on the total direct sums $\oplus_{\bs} U^{DJ}_q(\fg_{X_\bs})$ and $\oplus_{\bs} U^{DJ}_q(\fg_{\widehat{X}_\bs})$, respectively (this constitutes another main result of \cite{Y}, see also \cite{BFK} for the computation of key constants). 

\end{itemize}

\medskip
\noindent
Furthermore, there is no uniform loop (Drinfeld new) realization for quantum affine superalgebras akin to the one invoked in Subsection~\ref{sub:motivation}. In type $A$, such an algebra $U_q(L\fg_{X_\bs})$ was defined in~\cite[\S8]{Y}, together with an epimorphism 
$$
  U_q(L\fg_{X_\bs}) \twoheadrightarrow U_q^{DJ}(\fg_{\widehat{X}_\bs}) |_{c=1}
$$
($c$ denotes the central element) which was shown to be an isomorphism in~\cite{LYZ}. A similar technique was used in~\cite{HSTY} to provide a new loop realization $U_q(L\fg_{X_\bs})$ for the exceptional $1$-parametric family $D(2,1;\alpha)$. For $X\in \{B,C,D\}$, the corresponding algebra $U_q(L\fg_{X_\bs})$ endowed with a surjective  homomorphism to $U_q^{DJ}(\fg_{\widehat{X}_\bs})|_{c=1}$ appeared recently in~\cite{BFK} (also cf.~\cite{WLZ}). It is these algebras whose shuffle realization is established in Section~\ref{sec:osp}. 
The shuffle realizations of the respective positive halves $U^+_q(L\fg_{X_\bs})$ were established in~\cite{Ts} for type $A$ and in~\cite{FH} for type $D(2,1;\alpha)$. We also note that~\cite{DF} suggested such a shuffle realization for $\fg=\fosp(1|2)$; while Theorem~4.8 of \loccit is presented without any proof, it served as one of initial motivations for the present work. Concurrently and independently of our work, \cite{BKL} develops a proof of Corollary \ref{cor:intro-nondegeneracy} for type $B,C,D$ quantum affine superalgebras, using PBW bases. It would be very interesting to understand the latter from the point of view of our shuffle realization.


\medskip
    
\subsection{Outline of the paper} 

The structure of this paper is as follows:

\medskip

\begin{itemize}[leftmargin=0.7cm]

\item 
In Section~\ref{sec:arbitrary}, we recall the general framework of quantum loop and shuffle algebras associated to arbitrary quivers with parameters, and discuss the general principles behind shuffle realizations. We believe that these principles will find future applications to types beyond those studied in the present paper.

\medskip 

\item 
In Section~\ref{sec:gl}, we recall the shuffle realization of quantum affine and toroidal superalgebras in type $A$. A key ingredient is Lemma~\ref{lem:pairing sl}, and to prove it, we classify all affine type shrubs. We conclude with a notion of \emph{asps} that mimics the combinatorial features of shrubs and is more suitable for types $B,C,D$.

\medskip 

\item 
In Section~\ref{sec:osp}, we prove Theorems~\ref{thm:d intro},~\ref{thm:b intro},~\ref{thm:c intro} by means of classifying the corresponding asps in types $D$, $B$, $C$, respectively.

\medskip 

\item 
In Section~\ref{sec:osp-toroidal}, we prove Theorem \ref{thm:toroidal intro} concerning the shuffle realization in non-special extended types, closely following the proofs of Theorems~\ref{thm:d intro},~\ref{thm:b intro},~\ref{thm:c intro}. 

\medskip 

\item 
In Appendix~\ref{sec:small-N}, we present an analysis of small rank cases (rank $\leq 2$ as well as super case of rank $3$) of $\UslG$ arising uniformly in the present approach. 

\medskip 

\item 
In Appendix~\ref{non-zero pairings}, we present explicit calculations of various pairing constants that are used in the proofs of Theorems~\ref{thm:d intro},~\ref{thm:b intro},~\ref{thm:c intro}. 

\medskip 

\item
In Appendix~\ref{sec:delta-equalities}, we present in a more straightforward way the duality between the wheel conditions and the corresponding ``higher order'' relations used to define quantum affine superalgebras (our analysis does not cover the degree $6$ and $7$ relations dual to the wheels~\eqref{eqn:wheel c 2} and~\eqref{eqn:wheel c 3}, as the formulas are  unwieldy).

\end{itemize}


\medskip
    
\subsection{Acknowledgements} 
A.N.\ gratefully acknowledges the support of the Swiss National Science Foundation grant 10005316. A.T.\ gratefully acknowledges the support by NSF Grant DMS-2302661 and Simons Fellowship MPS-SFM-00020373, and is indebted to the MPIM Bonn for the hospitality and support in July 2026, where part of this project was carried out.


\bigskip
    
\section{Shuffle realizations for arbitrary quivers}
\label{sec:arbitrary}

\medskip 
In the present Section, we recall quantum loop algebras and Feigin-Odesskii shuffle algebras in the generality of arbitrary quivers with parameters, and provide a general blueprint for proving isomorphisms between these two kinds of algebras. This blueprint will be used in Sections \ref{sec:gl}, \ref{sec:osp}, \ref{sec:osp-toroidal} in finite and affine super types $A,B,C,D$.


\medskip 

\subsection{Quantum loop and shuffle algebras}
\label{sub:quiver}

We follow the framework of \cite{N Arbitrary}, though many ideas herein have their origins in the seminal works \cite{FO} and \cite{E} (see also \cite{FHHSY, FT, N Shuffle}). Consider any quiver with vertex set $I$, whose arrows
\begin{equation}
\label{eqn:parameters}
  \alpha \colon i \xrightarrow{t_\alpha} j
\end{equation}
are all decorated with non-zero elements in a certain field $\BK$ of characteristic $0$.

\medskip

\begin{remark}
The decorated Dynkin diagrams of Subsection \ref{sub:super} are particular examples of quivers with parameters, if we posit that any undirected decorated edge
\[
\begin{tikzcd}
  i & j
  \arrow["t", no head, from=1-1, to=1-2]
\end{tikzcd}
\]
represents a pair of decorated arrows $i \xrightarrow{t} j$ and $j \xrightarrow{t} i$ for $i \neq j$, and a single decorated loop $i \xrightarrow{t} i$ for $i=j$. In all these examples, we have $\BK = \BC$ and $t \in q^{\BZ}$. 
\end{remark}

\medskip 
\noindent 
Any quiver with parameters \eqref{eqn:parameters} encodes \emph{zeta functions} 
\begin{equation}
\label{eqn:zeta intro}
  \zeta_{ij}(x) \sim \frac {\prod_{\alpha : i \rightarrow j} (1-xt_{\alpha})}{(1-x)^{\delta_{ij}}}, \quad \forall\, i,j\in I
\end{equation}
where $\sim$ means that the two sides are equal up to multiplication by $\gamma x^\ell$, with the precise values of $\gamma \in \BK^\times$ and $\ell \in \BQ$ only becoming important in Section \ref{sub:double}. To the data \eqref{eqn:zeta intro}, one associates a (half) \emph{pre-quantum loop algebra} 
\begin{equation}
\label{eqn:quad intro}
  \tUp = \BK \langle e_{i,d} \rangle_{i \in I, d\in \BZ} \Big/ \left(e_i(x)e_j(y) \zeta_{ji}\left(\frac yx\right) = e_j(y) e_i(x) \zeta_{ij}\left(\frac xy\right) \right)_{i,j \in I}
\end{equation}
where $e_i(x) = \sum_{d \in \BZ} \frac {e_{i,d}}{x^d}$ and one interprets the quadratic relation above by clearing the denominators of the form $(x-y)^{\delta_{ij}}$ and then equating all monomials in $x,y$ in the two sides of the equation.  By analogy with \cite{E, FO}, we may also consider 
\begin{equation}
\label{eqn:shuf intro}
  \CV^+ = \bigoplus_{\bn \in \nn} \BK[z_{i1}^{\pm 1}, \dots, z_{in_i}^{\pm 1}]^{\sym}_{i \in I}
\end{equation}
with the superscript ``sym" referring to color-symmetric Laurent polynomials, i.e. those which are symmetric in $z_{i1},\dots,z_{in_i}$ for each individual $i \in I$. This terminology is due to the fact that $i\in I$ is called the \emph{color} of the variable $z_{ia}$. We make the vector space \eqref{eqn:shuf intro} into an associative algebra via the \emph{shuffle product} 
\begin{equation}
\label{eqn:shuffle product}
  E(z_{i1},\dots,z_{in_i})_{i \in I} * E'(z_{i1},\dots,z_{in'_i})_{i \in I} = 
\end{equation}
$$
  \Sym\, \left[ \frac {E(z_{i1},\dots,z_{in_i})_{i \in I} E'(z_{i,n_i+1},\dots,z_{i,n_i+n_i'})_{i \in I}}{\bn!\bn'!} \prod_{i,j \in I} \prod_{a=1}^{n_i} \prod_{b=n_j+1}^{n_j+n_j'} \zeta_{ij} \left(\frac {z_{ia}}{z_{jb}} \right) \right] 
$$
where ``Sym" denotes summation over the $(\bn+\bn')! = \prod_{i \in I} (n_i+n_i')!$ permutations of the variables of each color taken separately. The relation between the algebras \eqref{eqn:quad intro} and \eqref{eqn:shuf intro} is that there exists an algebra homomorphism
\begin{equation}
\label{eqn:tilde upsilon}
  \tUpsilon^+ \colon \tUp \rightarrow \CV^+
\end{equation}
given by $e_{i,d} \mapsto z_{i1}^d$ for all $i \in I$ and $d \in \BZ$. One is interested in computing 
\begin{equation*}
  K^+ = \text{Ker }\tUpsilon^+ \quad \text{and} \quad \oCS^+ = \text{Im }\tUpsilon^+
\end{equation*}
since tautologically, $\tUpsilon^+$ descends to an algebra isomorphism
\begin{equation}
\label{eqn:iso plus}
  \Upsilon^+ \colon \Up \iso \oCS^+
\end{equation}
where the (half) \emph{quantum loop algebra} is $\Up = \tUp / K^+$. The algebras $\CV^+$ and $\oCS^+$ are called the \emph{big} and \emph{small shuffle algebras}, respectively. By definition, the latter is the linear span of all color-symmetric Laurent polynomials of the form
\begin{equation}
\label{eqn:span}
  \oCS^+ = \text{span}_{\BK} \left\{ \Sym \left[z_1^{d_1}\dots z_k^{d_k} \prod_{1 \leq a < b \leq k} \zeta_{i_ai_b} \left(\frac {z_a}{z_b} \right)\right] \right\}_{\substack{i_1,\dots,i_k \in I\\d_1,\dots,d_k \in \BZ}}
\end{equation}
In the right-hand side of the expression above, we make the convention that each variable $z_a$ stands for $z_{i_a\bullet_a}$, where $\bullet_1,\dots,\bullet_k$ are the minimal positive integers such that $\bullet_a < \bullet_b$ whenever $a<b$ and $i_a = i_b$.

\medskip 
\noindent 
(Pre-)quantum loop algebras have the \emph{horizontal} and the \emph{vertical} grading given by
$$
  \hdeg e_{i,d} = \bsi^i \quad \text{and} \quad \vdeg e_{i,d} = d
$$
where $\bsi^i \in \nn$ denotes the $I$-tuple consisting of a single $1$ on position $i$ and $0$ everywhere else. Similarly, the assignment
$$
  \hdeg E = \bn \quad \text{and} \quad \vdeg E = d
$$
for any homogeneous degree $d$ Laurent polynomial $E$ in $\bn = (n_i \geq 0)_{i \in I}$ variables of each color, induces a $\nn \times \BZ$ grading on big and small shuffle algebras. We write
$$
  \CV^+ = \bigoplus_{\bn \in \nn} \CV_{\bn} = \bigoplus_{\bn \in \nn} \bigoplus_{d \in \BZ} \CV_{\bn,d}
$$
for the graded components, and similarly for the small shuffle algebra.


\medskip 

\subsection{The pairing}
\label{sub:duals}

Consider the opposite algebras $\CV^- = \CV^{+,\text{op}}$ and
\begin{equation*}
  \tUm = \BK \langle f_{i,d} \rangle_{i \in I, d\in \BZ} \Big/ 
  \left(f_i(x)f_j(y) \zeta_{ij}\left(\frac xy\right) = f_j(y) f_i(x) \zeta_{ji}\left(\frac yx\right) \right)_{i,j \in I}
\end{equation*}
where $f_i(x) = \sum_{d \in \BZ} \frac {f_{i,d}}{x^d}$. Define $\Um$ and $\oCS^-$ to be the opposites of the analogous algebras (with superscript $+$) from Subsection \ref{sub:quiver}. The horizontal and vertical degrees of elements in the opposite algebras are set by convention to be
$$
  \hdeg f_{i,d} = -\bsi^i \quad \text{and} \quad \vdeg f_{i,d} = d
$$
and
$$
  \hdeg F = -\bn \quad \text{and} \quad \vdeg F = d
$$
for any homogeneous degree $d$ Laurent polynomial $F$ in $\bn$ variables. We will write
$$
  \CV^- = \bigoplus_{\bn \in \nn} \CV_{-\bn} = \bigoplus_{\bn \in \nn} \bigoplus_{d \in \BZ} \CV_{-\bn,d}
$$
for the graded components, and similarly for the small shuffle algebra.

\medskip 
\noindent 
The relation between the objects introduced above is the existence of a pairing
\begin{equation}
\label{eqn:pairing}
  \tUp \otimes \CV^- \xrightarrow{\langle \cdot, \cdot \rangle} \BK
\end{equation}
given by the following formula for all $i_1,\dots,i_k \in I$, $d_1,\dots,d_k \in \BZ$ and homogeneous $F \in \CV^-$:
\begin{equation}
\label{eqn:pairing explicit}
  \Big \langle e_{i_1,d_1} \cdots e_{i_k,d_k}, F \Big \rangle = \int_{|z_1| \gg \cdots \gg |z_k|} 
  \frac {z_1^{d_1}\cdots z_k^{d_k} F(z_1,\dots,z_k)}{\prod_{1\leq a < b \leq k} \zeta_{i_bi_a} \left(\frac {z_b}{z_a} \right)} \prod_{a=1}^k Dz_a 
\end{equation}
if $\bsi^{i_1}+\dots+\bsi^{i_k} = \bn$ and $d_1+\dots+d_k$ is equal to minus the homogeneous degree of $F$, and $0$ otherwise; the meaning of the variables $z_1,\dots,z_k$ is the same as in the paragraph following~\eqref{eqn:span}. In the right-hand side of formula \eqref{eqn:pairing explicit}, the notation 
\begin{equation*}
  \int_{|z_1| \gg \cdots \gg |z_k|} G(z_1,\dots,z_k) \prod_{a=1}^k Dz_a
\end{equation*}
refers to the constant term in the expansion of any (homogeneous of degree 0) rational function $G$ as a power series in 
$$
  \frac {z_2}{z_1}, \dots, \frac {z_k}{z_{k-1}}
$$
One of the main results of \cite{N Arbitrary} is that the pairing \eqref{eqn:pairing} descends to a pairing $\langle \cdot, \cdot \rangle'$ as in the diagram below, which is non-degenerate in both arguments:
\begin{equation}
\label{eqn:diagram intro}
\begin{tikzcd}
	{\mathbf{\widetilde{U}}^+} & \otimes & {\CV^-} &&&& \\
	&&&&&& {\BK} \\
	{\mathbf{U}^+} & \otimes & {\oCS^-}
	\arrow[two heads, from=1-1, to=3-1]
	\arrow["{\langle \cdot, \cdot \rangle}", from=1-3, to=2-7]
	\arrow[hook, from=3-3, to=1-3]
	\arrow["{\langle \cdot, \cdot \rangle'}"', from=3-3, to=2-7]
\end{tikzcd}
\end{equation}
In other words, the aforementioned descent statement claims that
\begin{equation}
\label{eqn:other words}
  \langle K^+, F \rangle = 0 \quad \text{for} \quad F \in \CV^- \quad \Leftrightarrow \quad F \in \oCS^-
\end{equation}
or equivalently, every linear condition that cuts out the vector subspace $\oCS^- \subset \CV^-$ is given by pairing with a specific element of $K^+$.


\medskip 

\subsection{Shuffle realizations}
\label{sub:plan}

In a wide variety of examples (which include all items in the bulleted list of Subsection \ref{sub:motivation}), the algebras $\Upm$ are isomorphic to algebras of great interest in representation theory, algebraic geometry, and mathematical physics. In light of the isomorphism \eqref{eqn:iso plus} and its natural opposite analogue
\begin{equation*}
  \Upsilon^- \colon \Um \iso \oCS^-
\end{equation*}
the small shuffle algebras provide models of (half) quantum loop algebras. Therefore, an important problem becomes to understand small shuffle algebras as explicitly as possible. To this end, suppose one constructs an explicit pair of opposite subalgebras (slightly abusively called \emph{shuffle algebras})
\begin{equation}
\label{eqn:shuffle algebra}
  \CS^\pm \subset \CV^\pm
\end{equation}
which contain all horizontal degree $\{\pm \bsi^i\}_{i \in I}$ Laurent polynomials, thus implying 
$$
  \oCS^\pm \subseteq \CS^\pm 
$$
\cite[Theorem 2.11]{N Arbitrary} provides an explicit criterion for the opposite inclusion.

\medskip

\begin{proposition}
\label{prop:criterion 1}
Suppose that the pairing \eqref{eqn:pairing} has the property that
\begin{equation*}
  \langle K^+, \CS^- \rangle = 0
\end{equation*}
Then we have $\CS^\pm = \oCS^\pm$.
\end{proposition}

\medskip 
\noindent 
Suppose now that one is given a two-sided ideal
\begin{equation}
\label{eqn:j and k}
  J^+ \subseteq K^+
\end{equation}
i.e.\ a collection of elements of $\tUp$ which map to zero inside $\CV^+$ under the map \eqref{eqn:tilde upsilon}. The following result of \cite{N Arbitrary} (contained in Proposition~3.23 and Subsection~3.24 of \loccitt) provides an explicit criterion for the opposite inclusion.

\begin{proposition}
\label{prop:criterion 2}
Suppose that the pairing \eqref{eqn:pairing} has the property that
\begin{equation}
\label{eqn:criterion 2}
  \langle J^+, F \rangle = 0 \quad \text{for} \quad F \in \CV^- \quad \Rightarrow \quad F \in \CS^-
\end{equation}
Then we have $J^+ = K^+$. 
\end{proposition}

\medskip 
\noindent 
Note that the opposite implication to \eqref{eqn:criterion 2} is a consequence of Proposition \ref{prop:criterion 1} and equations \eqref{eqn:other words} and \eqref{eqn:j and k}. If big shuffle algebras were finite-dimensional vector spaces, Propositions \ref{prop:criterion 1} and \ref{prop:criterion 2} would be immediate consequences of the fact that the complement of the complement of a vector subspace (with respect to a non-degenerate pairing) is the vector subspace itself. Thus, the main challenge in the aforementioned Propositions is the infinite-dimensional nature of shuffle algebras.

\medskip

\begin{definition}
\label{def:shuffle realization}
Assume we have a shuffle algebra \eqref{eqn:shuffle algebra} satisfying the hypothesis of Proposition \ref{prop:criterion 1} and an ideal \eqref{eqn:j and k} satisfying the hypothesis of Proposition \ref{prop:criterion 2}. Let
$$
  U^+ = \tUp / J^+
$$
and $U^- = U^{+,\emph{op}}$. Under these circumstances, we call
$$
  U^+ \stackrel{\text{Prop.\ref{prop:criterion 2}}}{=} \Up \iso \oCS^+ \stackrel{\text{Prop.\ref{prop:criterion 1}}}{=} \CS^+
$$
a \emph{shuffle realization} (more precisely, the right-most algebra is said to provide a shuffle realization of the left-most algebra).
\end{definition}


\medskip 

\subsection{Wheel conditions}
\label{sub:wheel}

The small shuffle algebra $\oCS^+$ is, by its very definition, only as explicit as the linear span \eqref{eqn:span}. On the other hand, shuffle algebras $\CS^+$ as in \eqref{eqn:shuffle algebra} are often defined by very explicit properties; see Theorems~\ref{thm:d intro},~\ref{thm:b intro},~\ref{thm:c intro} for specific examples. With the exception of the scary wheel \eqref{eqn:wheel c 3}, the shuffle algebras in all the aforementioned examples are of the following nature: consider the set of Laurent polynomials $F$ which vanish when a certain subset $z_1,\dots,z_k$ of their variables of colors $i_1,\dots,i_k \in I$ are specialized according to the so-called \emph{length $k$ wheel}
\begin{equation}
\label{eqn:general wheel}
  \bigcirc = \Big\{ z_k = z_{k-1} t_k, \dots, z_2 = z_1 t_2, z_1 = z_k t_1 \Big\}
\end{equation}
for various quiver arrows $i_{a-1} \xrightarrow{t_a} i_a$, whose parameters satisfy $t_1\cdots t_k = 1$ and 
\begin{equation}
\label{eq:assumption}
  t_{a+1}\dots t_b \neq 1 \qquad \forall \, 1\leq a < b\leq k \text{   with   } i_a = i_b
\end{equation}
Above, we make the convention that $z_a$ actually denotes the variable $z_{i_a\bullet_a}$ of $F$, for various $i_a \in I$ and positive integers $\bullet_a$ (the latter of which will not be important to us, but we do require $\bullet_a \neq \bullet_b$ if $i_a = i_b, a \neq b$). Such a wheel may be represented graphically as
\begin{equation}
\label{eqn:wheel}
\begin{tikzcd}
  && {z_{k-1}} &&& {z_{k}} && \\ \\
  \vdots &&&&&&& {z_{1}} \\ \\
  && {z_{3}} &&& {z_{2}}
  \arrow["{t_k}", from=1-3, to=1-6]
  \arrow["{t_1}", from=1-6, to=3-8]
  \arrow[from=3-1, to=1-3]
  \arrow["{t_2}", from=3-8, to=5-6]
  \arrow[from=5-3, to=3-1]
  \arrow["{t_3}", from=5-6, to=5-3]
\end{tikzcd}
\end{equation}

\medskip

\begin{lemma}
\label{lem:easy}
The set of Laurent polynomials which vanish at the length $k$ wheel~\eqref{eqn:general wheel} is closed under the shuffle product.
\end{lemma}

\medskip

\begin{proof} 
We must show that if $E$ and $E'$ vanish at the wheel \eqref{eqn:general wheel}, so does $E * E'$. From formula \eqref{eqn:shuffle product}, it suffices to prove the more general statement that
\begin{equation}
\label{eqn:lem shuf prod}
  E(z_{i1},\dots,z_{in_i})_{i \in I} E'(z_{i,n_{i}+1},\dots,z_{i,n_{i}+n'_{i}})_{i \in I} 
  \prod_{i,j \in I} \mathop{\prod_{1 \leq a \leq n_i}}_{n_j < b \leq n_j+n_j'} \prod_t (z_{jb} - z_{ia} t) 
\end{equation}
vanishes when an arbitrary subset 
\begin{equation*}
  \{z_{i_1\bullet_1},\dots,z_{i_k\bullet_k} \} \subseteq \{z_{i1},\dots,z_{i,n_i+n'_i}\}_{i \in I}
\end{equation*}
is specialized according to $\{z_{i_\ell\bullet_\ell} = z_{i_{\ell-1} \bullet_{\ell-1}}t_\ell\}_{\ell \in \{1,\dots,k\}}$. The linear factors which appear in the right-most product in \eqref{eqn:lem shuf prod} depend on the particular zeta functions, and the only thing that will matter to us is that they include the linear factors
$$
  (z_{i_\ell \bullet_\ell} - z_{i_{\ell-1}\bullet_{\ell-1}} t_{\ell}), \quad \text{if  } \bullet_\ell > n_{i_{\ell}} 
  \text{   and   } \bullet_{\ell-1} \leq n_{i_{\ell-1}} 
$$
This implies that the $z_{i_\ell\bullet_\ell} = z_{i_{\ell-1} \bullet_{\ell-1}}t_\ell$ specialization of \eqref{eqn:lem shuf prod} will vanish, because there will always be some $\ell \in \{1,\dots,k\}$ such that $\bullet_\ell > n_{i_{\ell}}$ and $\bullet_{\ell-1} \leq n_{i_{\ell-1}}$ (with the indices $\ell,\ell-1$ taken modulo $k$) unless we have either $\bullet_\ell \leq n_{i_\ell}$ for all $\ell$ or $\bullet_\ell > n_{i_\ell}$ for all $\ell$. However, in the latter two cases, the condition that both $E$ and $E'$ vanish at the wheel \eqref{eqn:general wheel} implies the vanishing of the specialization of \eqref{eqn:lem shuf prod}. 
\end{proof}

\medskip 
\noindent
Recall the following series from \cite[(3.3)]{N Reduced} (that we refer to as a \emph{clunky relation}):
\begin{equation}
\label{eqn:clunky}
  e_{\bigcirc} = \sum_{\ell = 1}^k P_{\ell}(x_1,\dots,x_k) e_{i_\ell}(x_{\ell}) \cdots e_{i_1}(x_1) e_{i_k}(x_k) \cdots e_{i_{\ell+1}}(x_{\ell+1})
\end{equation}
where if we write $\widetilde{\zeta}_{ij}(x) = \zeta_{ij}(x) (1-x)^{\delta_{ij}}$, then we set
\[
  P_{\ell}(x_1,\dots,x_k) = \frac {\frac {x_1 t_2 \cdots t_{\ell}}{x_{\ell}} \prod^{1 \leq a < b \leq \ell \text{ or } \ell < a < b \leq k}_{\text{or } 1 \leq b \leq \ell < a \leq k} \widetilde{\zeta}_{i_ai_b}\left(\frac {x_a}{x_b} \right) \left(-\frac {x_b}{x_a}\right)^{\delta_{b<a} \delta_{i_bi_a}}}{\left(1- \frac {x_k t_{1}}{x_{1}}\right)^{\delta_{\ell \neq k}} \prod^{1 \leq a < \ell}_{\text{ or } \ell < a < k}\left(1- \frac {x_a t_{a+1}}{x_{a+1}}\right)}
\]
The reason for this choice of $P_{\ell}$'s is that they are Laurent polynomials and that we have the following formula for power series expansions of rational functions (see \cite[Proposition 3.5]{N Reduced})
\begin{equation}
\label{eqn:delta long wheel}
\begin{split}
  \sum_{\ell=1}^k \text{ev}_{|x_{\ell}| \gg \cdots \gg |x_1| \gg |x_k| \gg \cdots \gg |x_{\ell+1}|} \left[ \frac {P_{\ell}(x_1,\dots,x_k)}{\prod^{1 \leq a < b \leq \ell \text{ or } \ell < a < b \leq k}_{\text{or } 1 \leq b \leq \ell < a \leq k} \zeta_{i_ai_b}\left(\frac {x_a}{x_b} \right)} \right] \\ = \delta \left(\frac {x_k}{x_{k-1}t_k} \right) \cdots  \delta \left(\frac {x_2}{x_1t_2} \right) \prod_{1 \leq a < b \leq k} \left(1-\frac {x_a}{x_b}\right)^{\delta_{i_ai_b}}
\end{split}
\end{equation}
where $\delta (x) = \sum_{d \in \BZ} x^d$ is the formal delta function. It has the key property that
\begin{equation}
\label{eqn:key}
\delta \left(\frac xy \right) P(x) = \delta \left(\frac xy \right) P(y)
\end{equation}
for any Laurent polynomial $P$. As shown in \loccitt, formula \eqref{eqn:delta long wheel} ensures that
\begin{equation}
\label{eqn:compare 1}
  \Big \langle e_{\bigcirc}, F \Big \rangle = 0 \quad \Leftrightarrow \quad 
  \Big(F \text{ vanishes at the wheel} \bigcirc \text{of \eqref{eqn:general wheel}} \Big)
\end{equation}
for any shuffle element $F \in \CV_{-\bsi^{i_1}-\dots-\bsi^{i_k}}$.


\medskip

\subsection{Equivalent ideals of relations}
\label{sub:ideals}

More generally, suppose we define a shuffle algebra $\CS^\pm \subset \CV^\pm$ as in \eqref{eqn:shuffle algebra} to consist of all color-symmetric Laurent polynomials which vanish at several wheels $\bigcirc_1,\dots,\bigcirc_k$. In this situation, \cite{N Arbitrary, NSS} showed that
\[
  K^+  = \Big(e_{\bigcirc_1},\dots,e_{\bigcirc_k}\Big) 
\]
cf.~Proposition~\ref{prop:criterion 2}. 
For instance, this is the case for the quantum affine algebras associated to finite type Lie algebras $\fg$ that we discussed in Subsection \ref{sub:motivation}: indeed, the wheel \eqref{eqn:classic wheels} is a particular case of the wheel \eqref{eqn:wheel}, so one would expect 
\begin{equation}
\label{eqn:classic clunky}
  e_{\bigcirc \text{ of \eqref{eqn:classic wheels}}} = 0
\end{equation}
to be a relation in the half quantum affine algebra $U^+_q(L\fg)$. However, the aforementioned half quantum affine algebra is known to be determined by the quadratic relations in \eqref{eqn:quad intro} and the well-known Drinfeld-Serre relations (\cite{Dr})
\begin{multline}
\label{eqn:drinfeld-serre}
  \sum_{k=0}^{1-c_{ij}} (-1)^k \binom{1-c_{ij}}{k}_{q^{\frac {d_{ii}}2}} \text{Sym} \Big[ e_i(x_1) \cdots e_i(x_k) e_j(y) e_i(x_{k+1}) \cdots e_i(x_{1-c_{ij}}) \Big] \\ = 0
\end{multline}
with $c_{ij}$ as in \eqref{eq:nonsymm-Cartan} and $\text{Sym}$ denoting symmetrization in the $x$-variables. The left-hand sides of \eqref{eqn:classic clunky} and \eqref{eqn:drinfeld-serre} are manifestly not identical, but they determine the same quantum affine algebra. This is because the ideal
$$
  J^+_{\fg} = \Big(\text{LHS of \eqref{eqn:drinfeld-serre}} \Big)
$$
satisfies the hypothesis of Proposition \ref{prop:criterion 2}, due to a combinatorial identity that was proved in~\cite{DJ}. Therefore, we conclude that $J^+_{\fg}$ matches
$$
  K^+_{\fg} = \Big(\text{LHS of \eqref{eqn:classic clunky}}\Big) 
$$
which reveals a phenomenon that we will encounter repeatedly throughout the present paper: different collections of relations in quantum loop algebras actually generate one and the same ideal, and are thus equivalent to each other.

\medskip 
\noindent 
Even more so, the particular shape of the aforementioned relations depends quite strongly on the zeta functions \eqref{eqn:zeta intro}, which in turn depend on the quiver \eqref{eqn:parameters}. In particular, one often encounters similarly looking wheels associated to different decorated Dynkin diagrams, such as

\begin{itemize}[leftmargin=0.7cm]

\item[-] \eqref{eqn:wheel finite even} in type $A$ 

\item[-] \eqref{eqn:wheel finite even} for $(i,i-1)=(N,N-1)$ in type $B$ Case 2 

\item[-] \eqref{eqn:wheel b 2} in type $B$ Case 2 

\item[-] \eqref{eqn:wheel d 1} in type $D$ Case 2 

\end{itemize}
that are dual to substantially different formal series of elements of $K^+$: the LHS of \eqref{eqn:rel quantum affine 8}, \eqref{eqn:rel b 2}, \eqref{eqn:rel b 1},  \eqref{eqn:rel cd 1}, respectively. Another example is provided by the wheels
\begin{itemize}[leftmargin=0.7cm]

\item[-] \eqref{eqn:wheel affine even} in type $\wA$ with $|i|=0$

\item[-] its toroidal $\fgl_1$-analogue of~\eqref{eq:wheel-pic-gl1} 

\item[-] its toroidal $\fgl_2$-analogue of~\eqref{eq:wheel-pic-gl2} 

\end{itemize}
that are dual to significantly different formal series of elements of $K^+$: the LHS of \eqref{eqn:rel quantum affine 8}, \eqref{eq:Miki}, \eqref{eq:3Serre-sl2}, respectively. Sometimes, a single formal series of elements in $K^+$ is dual to several linear conditions on the shuffle algebra elements, such as 
\begin{itemize}[leftmargin=0.7cm]

\item[-] \eqref{eqn:rel cd 1} is dual to both wheels in \eqref{eqn:wheel d 1}

\item[-] \eqref{eq:Miki} is dual to both wheels in~\eqref{eq:wheel-pic-gl1}

\item[-] \eqref{eq:3Serre-sl2} is dual to both wheels in~\eqref{eq:wheel-pic-gl2}

\end{itemize}
according to \eqref{eq:D-fork-odd-delta}-\eqref{eq:D-fork-odd-conclusion}, \eqref{eq:delta-gl1}-\eqref{eq:wheel-gl1}, \eqref{eq:delta-gl2}-\eqref{eq:wheel-gl2}, respectively.


\medskip

\subsection{Doubles}
\label{sub:double}

In the generality of Subsection \ref{sub:quiver}, assume that we have
\begin{equation*}
  \lim_{x \rightarrow \infty} \frac {\zeta_{ij}(x)}{\zeta_{ji}(x^{-1})} < \infty, \quad \forall\, i,j \in I
\end{equation*}
Then we define the \emph{quantum loop algebra} as the Drinfeld double
\begin{equation}
\label{eqn:double intro}
  \U = \Up \otimes \BK[\ph_{i,0}^\pm, \ph_{i,1}^{\pm}, \ph_{i,2}^{\pm}, \dots ]_{i \in I} \otimes \Um 
\end{equation}
In more detail, the algebra structure on \eqref{eqn:double intro} is governed by the following commutation relations for all $i,j \in I$  (let $\ph^\pm_i(x)=\sum_{d\geq 0} \ph^\pm_{i,d} x^{\mp d}$):
\begin{align}
  & \ph^\pm_i(x) e_j(y)  = e_j(y) \ph^\pm_i(x) 
    \frac {\zeta_{ij} \left(\frac xy \right)}{\zeta_{ji} \left(\frac yx \right)}  \label{eqn:e phi} \\
  & f_j(y) \ph^\pm_i(x) = \ph^\pm_i(x) f_j(y) 
    \frac {\zeta_{ij} \left(\frac xy \right)}{\zeta_{ji} \left(\frac yx \right)} \label{eqn:f phi}
\end{align}
where the ratios of zeta functions in the right-hand sides are expanded as power series in non-positive powers of $x^{\pm 1}$. Finally, we impose the relation 
\begin{equation}
\label{eqn:plus minus comm}
  [e_{i,k}, f_{j,\ell} ] = 
  \delta_{ij} \cdot \text{constant}_{i}\Big( \ph^+_{i,k+\ell} \delta_{k +\ell\geq 0} - \ph^-_{i,-k-\ell} \delta_{k+\ell\leq 0} \Big)
\end{equation}
for all $i,j \in I$ and $k,\ell \in \BZ$ (here $\text{constant}_{i}\in \BK^\times$ are arbitrary).

\medskip

\begin{remark}
\label{rem:super}
In the context of quantum loop superalgebras that will be discussed in Sections~\ref{sec:gl},~\ref{sec:osp},~\ref{sec:osp-toroidal}, every $i \in I$ carries a parity $|i| \in \BZ/2\BZ$. In this case, one defines the Drinfeld super-double \eqref{eqn:double intro} by multiplying the right-hand sides of \eqref{eqn:e phi} and \eqref{eqn:f phi} by $(-1)^{|i||j|}$, and replacing the commutator $[e_{i,k}, f_{j,\ell} ]$ in \eqref{eqn:plus minus comm} by the super-commutator
$$
  e_{i,k} f_{j,\ell} - (-1)^{|i||j|} f_{j,\ell} e_{i,k} 
$$
\end{remark}

\medskip

\begin{remark}
\label{rem:shifted} 
One can also define \emph{shifted quantum loop algebras} by the following natural analogue of the construction of \cite{FiT}. Pick any $\br = (r_i)_{i \in I} \in \zz$ and define
\begin{equation*}
  \U^{\br} = \Up \otimes \BK[\ph_{i,0}^\pm, \ph_{i,1}^{\pm}, \ph_{i,2}^{\pm}, \dots ]_{i \in I} \otimes \Um 
\end{equation*}
where the algebra structure is governed by relations \eqref{eqn:e phi}-\eqref{eqn:f phi} and the following alternative version of relation \eqref{eqn:plus minus comm}:
\begin{equation*}
  [e_{i,k} , f_{j,\ell} ]  = 
  \delta_{ij} \cdot \emph{constant}_{i}\Big( \ph^+_{i,k+\ell} \delta_{k +\ell\geq 0} - \ph^-_{i,-k-\ell+r_i} \delta_{k+\ell\leq r_i} \Big)
\end{equation*}
The generalization to superalgebras is straightforward, following Remark~\ref{rem:super}.
\end{remark}

\medskip 
\noindent 
Assume that we have a shuffle realization as in Definition \ref{def:shuffle realization}. Since $\Upm \simeq U^\pm$ and $\oCS^\pm \simeq \CS^\pm$, the diagram \eqref{eqn:diagram intro} implies that the induced pairing
\begin{equation}
\label{eqn:non degenerate}
  U^+ \otimes U^- \xrightarrow{\langle\cdot,\cdot\rangle'} \BK 
\end{equation}
is non-degenerate. This is the general framework behind Corollary \ref{cor:intro-nondegeneracy}.


\bigskip
   	
\section{Quantum affine/toroidal superalgebras of type $A$}
\label{sec:gl}

\medskip

In this Section, we develop the shuffle realization for type $A$ and $\wA$ Lie superalgebras, namely Theorem \ref{thm:shuffle sl}. While the result is likely known to experts, the techniques developed herein (particularly the classification of shrubs, following \cite{N Reduced}) will be very important when we will discuss types $B,C,D$ in the next Sections. The notion of \emph{asps} introduced in Subsection \ref{sub:specialization patterns} is of more general interest.


\medskip 
	
\subsection{Lie superalgebras}
\label{sub:super lie}

Fix any $N\geq 2$ and a type $A$ decorated Dynkin diagram $\Gamma$ corresponding to a \emph{parity sequence} (see Subsection~\ref{sub:super})
\begin{equation}
\label{eqn:signs}
  \bs=(s_1,\dots,s_{N}) \in \{+1,-1\}^{N}
\end{equation}
We shall use $m$ and $n$ to denote the number of $+1$'s and $-1$'s in $\bs$. Let $\sI=\{1,\ldots,N-1\}$ and consider the lattice $P_{A_\Gamma}=\bigoplus_{k=1}^{N} \BZ \varepsilon_k$. We endow $P_{A_\Gamma}$ with the symmetric bilinear form $(\varepsilon_k,\varepsilon_l)=\delta_{kl}s_k$, and set $\alpha_k=\varepsilon_k-\varepsilon_{k+1}$ for all $k\in \sI$. Then 
\begin{equation}
\label{eqn:finite sl Cartan}
  \big( (s_i+s_{i+1})\delta_{ij} - s_i \delta_{i,j+1} - s_j \delta_{i+1,j} \big)_{i,j \in \sI} = 
  \big( (\alpha_i,\alpha_j) \big)_{i,j \in \sI}
\end{equation}
is a Cartan matrix of the type $A$ Lie superalgebra 
$$
  \fg_{A_\Gamma} \simeq \fsl_{m|n}
$$
Let $\wI=\{0,\ldots,N-1\}$, often identified with $\BZ/N\BZ$. We endow $P_{\wA_\Gamma}=P_{A_\Gamma} \oplus \BZ\delta$ with the same symmetric bilinear form as above, extended by $(\delta,\delta)=(\delta,\varepsilon_k)=0$. Recall that $\alpha_k = \varepsilon_k - \varepsilon_{k+1}$ for $k\in \sI$ and further set $\alpha_0=\delta+\varepsilon_N-\varepsilon_1$. Then formula~\eqref{eqn:finite sl Cartan}, with indices $i,j$ considered as residues modulo $N$, is an affine Cartan matrix of the affine type $\wA$ Lie superalgebra 
\begin{equation*}
  \fg_{\wA_\Gamma} \simeq \hsl_{m|n}
\end{equation*}
We henceforth identify the root lattice $Q_{A_\Gamma} = \bigoplus_{i\in \sI} \BZ\alpha_i\subset P_{A_\Gamma}$ with $\szz$ via the assignment $\alpha_i \mapsto \bsi^i$. In particular, this endows $\szz$ with the bilinear form satisfying $(\bsi^i,\bsi^j)=(\alpha_i,\alpha_j)$, see~\eqref{eqn:finite sl Cartan}. We also have a group homomorphism 
\begin{equation}
\label{eq:i-parity}
  \szz\to \BZ/2\BZ \quad \text{given by} \quad \bsi^i \mapsto |\bsi^i|=|\alpha_i|=|i|=\tfrac{1-s_is_{i+1}}{2} \in \BZ/2\BZ
\end{equation}
Here, $|x|$ shall denote the \emph{parity} of any element. This naturally generalizes to $\wzz$. The contents of this section are unchanged (up to replacing $q\leadsto q^{-1}$) by reversing all signs of $\bs$; with this in mind the case $\bs=(1,\dots,1)$ is called the \emph{non-super} case.


\medskip 
	
\subsection{Quantum affine superalgebras}
\label{sub:quantum affine sl}
  
In what follows, we will work over $\BC$ and fix $q \in \BC^\times$ not a root of unity (alternatively, all the contents of the present paper would hold with the ground field $\BC$ replaced by $\BQ(q)$). We recall the definition of the quantum affine superalgebra $\uslG$ in the loop (Drinfeld new) realization, cf.~\cite[\S 8.5]{Y}. The idea is to start from the Chevalley generators $E_i,F_i,H_i$ of $\fg_{A_\Gamma}$, and to replace each of them by a  countable family of generators, all the while $q$-deforming the defining relations. The aforementioned relations involve the zeta functions
\begin{equation} 
\label{eqn:zeta finite sl}
\zeta^{A_{\Gamma}}_{ij}(x) = 
  \begin{cases} 
    \frac {1-xq^{2s_i}}{1-x} &\text{if } i = j \text{ and } |i|=0 \\ 
    x^{\frac 12}\frac {1}{1-x} &\text{if } i = j \text{ and } |i|=1 \\ 
    -(1-xq^{-s_i}) &\text{if } i = j+1 \\ 
    (-1)^{|i||j|}x^{-1}(1-xq^{-s_j}) &\text{if } j = i+1 \\ 
    (-1)^{|i||j|\delta_{i<j}} &\text{otherwise} 
  \end{cases} 
\end{equation}
for all $i,j \in \sI = \{1,\dots,N-1\}$. Following~\eqref{eqn:zeta intro}, we note that the linear factors in the numerators of the zeta functions are prescribed by the corresponding parameters in $\Gamma$. As in Subsection \ref{sub:quiver}, define the (half) pre-quantum affine superalgebra as
$$
  \tuslGp = \BC \langle e_{i,d} \rangle_{i \in \sI, d\in \BZ} \Big/ \left(e_i(x)e_j(y) \zeta^{A_{\Gamma}}_{ji}\left(\frac yx\right) = e_j(y) e_i(x) \zeta^{A_{\Gamma}}_{ij}\left(\frac xy\right) \right)_{i,j \in \sI}
$$
The horizontal grading $\hdeg e_{i,d}=\bsi^i$ induces a $\BZ/2\BZ$-grading via~\eqref{eq:i-parity}; we use $|x|$ to denote the parity of $x$. Throughout the paper, we use the standard notation 
\begin{equation}
\begin{split}
\label{eq:super-q-commutators}
  [x,y] &= xy-(-1)^{|x||y|}yx  \\ 
  [x,y]_a &= xy-(-1)^{|x||y|}ayx  
\end{split}
\end{equation}
for any $\szz$-graded elements $x,y$. We consider the ideal 
$$
  J^+_{A_{\Gamma}} \subset \tuslGp
$$
generated by the left-hand sides of the following super-analogues of the Drinfeld-Serre relations (see~\cite{BFK,BM}) \footnote{Note that~\cite[Theorem 8.4.1]{Y} contained the $q\leadsto q^{-1}$ version of the relation~\eqref{eqn:rel quantum affine 9}. It is a general phenomenon of many higher order relations that we shall encounter in this paper that these different-looking relations are equivalent (modulo the quadratic relations in \eqref{eqn:quad intro}), cf. Remark~\ref{rem:equivalent-C-six-relations}.}
\begin{align}
& \Sym_{z_1,z_2}\, [ e_i(z_1), [ e_i(z_2),e_{i\pm 1}(w) ]_{q^{s_i}} ]_{q^{-s_{i}}} = 0  
    \qquad \qquad \qquad \ \  \mathrm{if} \quad |i| = 0 
    \label{eqn:rel quantum affine 8} \\
& \Sym_{z_1,z_2}\, [ e_i(z_1), [ e_{i+1}(y), [ e_i(z_2) , e_{i-1}(w) ]_{q^{s_i}} ]_{q^{-s_{i}}} ]_{1} = 0  
  \qquad \mathrm{if} \quad |i|=1 
    \label{eqn:rel quantum affine 9}
\end{align}
The (half) quantum affine superalgebra is defined as 
$$
  \uslGp = \tuslGp / J^+_{A_{\Gamma}}
$$


\medskip 
	
\subsection{Quantum toroidal superalgebras}
\label{sub:quantum toroidal sl}
	
By analogy with Subsection \ref{sub:quantum affine sl}, we will now define the quantum toroidal superalgebra of type $A$. One of the special features of this algebra is that it depends on two parameters, to be denoted by $q,d$ ($d$ is a non-zero complex number that does not satisfy any non-trivial relations $q^ad^b=1$). Following~\cite[\S 4.1]{BM}, we assume that $N \geq 3$ and start from the zeta functions
\begin{equation} 
\label{eqn:zeta affine sl}
\zeta^{\wA_{\Gamma}}_{ij}(x) = 
  \begin{cases} 
    \frac {1-xq^{2s_i}}{1-x} &\text{if } i = j \text{ and } |i|=0  \\ 
    x^{\frac 12}\frac {1}{1-x} &\text{if } i = j \text{ and } |i|=1 \\ 
    -(1-xq^{-s_i}d^{s_i}) &\text{if } i = j+1 \\ 
    (-1)^{|i||j|}d^{s_{j}}x^{-1}(1-xq^{-s_j}d^{-s_j}) &\text{if } j = i+1 \\ 
    (-1)^{|i||j|\delta_{i<j}} &\text{otherwise} 
  \end{cases} 
\end{equation} 
where the equalities of indices above are viewed in $\wI = \BZ/N\BZ$ (however, note that we write $\delta_{i<j}$ to denote $1$ if $0\leq i<j <N$ and $0$ otherwise). In the language of Subsection~\ref{sub:quiver}, the corresponding quiver is
\begin{equation}
\label{eqn:quiver affine type A}
\begin{tikzcd}[row sep=3em,column sep=3em]
  && {N-1} &&& {0} && \\ \\
  \vdots &&&&&&& {1} \\ \\
  && {3} &&& {2}
  \arrow["{q^{-s_N}d^{-s_N}}", bend left, from=1-3, to=1-6]
  \arrow["{q^{-s_N}d^{s_N}}", bend left, from=1-6, to=1-3]
  \arrow["{q^{-s_1}d^{-s_1}}", bend left, from=1-6, to=3-8]
  \arrow["{q^{-s_1}d^{s_1}}", bend left, from=3-8, to=1-6]
  \arrow[bend left, from=3-1, to=1-3]
  \arrow[bend left, from=1-3, to=3-1]
  \arrow["{q^{-s_2}d^{-s_2}}", bend left, from=3-8, to=5-6]
  \arrow["{q^{-s_2}d^{s_2}}", bend left, from=5-6, to=3-8]
  \arrow[bend left, from=5-3, to=3-1]
  \arrow[bend left, from=3-1, to=5-3]
  \arrow["{q^{-s_3}d^{-s_3}}", bend left, from=5-6, to=5-3]
  \arrow["{q^{-s_3}d^{s_3}}", bend left, from=5-3, to=5-6]
  \arrow[loop above, distance=4em, dashed, from=1-3, to=1-3, looseness=4]{}{q^{2s_{N-1}}}
  \arrow[loop above, distance=4em, dashed, from=1-6, to=1-6, looseness=4]{}{q^{2s_{N}}}
  \arrow[loop right, distance=4em, dashed, from=3-8, to=3-8, looseness=4]{}{q^{2s_{1}}}
  \arrow[loop below, distance=4em, dashed, from=5-6, to=5-6, looseness=4]{}{q^{2s_{2}}}
  \arrow[loop below, distance=4em, dashed, from=5-3, to=5-3, looseness=4]{}{q^{2s_{3}}}
\end{tikzcd}
\end{equation}

\medskip

\begin{remark}
\label{rem:d-redundancy}
We note that \eqref{eqn:zeta finite sl} is simply the restriction of \eqref{eqn:zeta affine sl} to $i,j \neq 0$, followed by setting $d=1$. Though one can also introduce analogous powers of $d$ in the formulas \eqref{eqn:zeta finite sl}, this would be made redundant through the following renormalization: $e_i(z) \leadsto e_i(z \lambda_i), f_i(z) \leadsto f_i(z \lambda_i), \ph^\pm_i(z) \leadsto \ph^\pm_i(z \lambda_i)$ with $\lambda_i=d^{-s_1-\ldots-s_{i-1}}$.  
\end{remark}

\medskip
\noindent
For any $N \geq 4$ and any decorated Dynkin diagram of type $\wA_\Gamma$ (as well as $N=3$ and $\bs=(1,1,1)$) \footnote{The cases $N \in \{1,2\}$ as well as the super-cases for $N=3$ require separate treatment, which we present in Appendix~\ref{sec:small-N}.}, define the (half) pre-quantum toroidal superalgebra as
\begin{equation*}
  \tUslGp = \BC \langle e_{i,d} \rangle_{i \in \wI, d\in \BZ} \Big/ \left(e_i(x)e_j(y) \zeta^{\wA_{\Gamma}}_{ji}\left(\frac yx\right) = e_j(y) e_i(x) \zeta^{\wA_{\Gamma}}_{ij}\left(\frac xy\right) \right)_{i,j \in \wI}
\end{equation*}
We consider the ideal 
$$
  J^+_{\wA_{\Gamma}} \subset \tUslGp
$$
generated by the left-hand sides of \eqref{eqn:rel quantum affine 8}, \eqref{eqn:rel quantum affine 9}, and the special elements (see the general discussion in~\eqref{eqn:clunky})
\begin{equation}
\label{eqn:rel quantum toroidal special}
  e_{\bigcirc \text{ of \eqref{eqn:wheel special 1}}},\, e_{\bigcirc \text{ of \eqref{eqn:wheel special 2}}} \in \tUslGp 
\end{equation}
only when $m=n$. The (half) quantum toroidal superalgebra is defined as 
$$
  \UslGp = \tUslGp / J^+_{\wA_{\Gamma}}
$$


\medskip 
	
\subsection{Shuffle algebras}
\label{sub:big sl}

Because the zeta functions \eqref{eqn:zeta finite sl} involve square roots, we are compelled to redefine the big shuffle algebras of Subsection \ref{sub:quiver} as
\begin{equation*}
  \CV^+_{A_{\Gamma}} = \bigoplus_{\bn = (n_i)_{i \in \sI} \in \snn} \BC[z_{i1}^{\pm 1}, \dots, z_{in_i}^{\pm 1}]^{\sym}_{i \in \sI}  \cdot \twist
\end{equation*}
where
\begin{equation}
\label{eq:root-twist}
  \twist=\prod_{i\in \sI}^{|i|=1} \prod_{a=1}^{n_i} z_{ia}^{\frac{n_i-1}{2}}
\end{equation}
All the general theory of shuffle algebras from Section~\ref{sec:arbitrary} goes through in the present context as well. Similarly, we redefine the big shuffle algebra associated to the zeta functions \eqref{eqn:zeta affine sl} as
\begin{equation*}
  \wCV^+_{\wA_{\Gamma}} = \bigoplus_{\bn = (n_i)_{i \in \wI} \in \wnn} \BC[z_{i1}^{\pm 1}, \dots, z_{in_i}^{\pm 1}]^{\sym}_{i \in \wI}   \cdot \twist 
\end{equation*}
where we recall that we identify $\wI = \BZ/N\BZ$, and $\twist=\prod_{i\in \wI}^{|i|=1} \prod_{a=1}^{n_i} z_{ia}^{\frac{n_i-1}{2}}$.

\medskip

\begin{definition}
\label{def:wheel sl}
Consider the shuffle algebras
\begin{align*} 
  & \CV^+_{A_{\Gamma}} \supset \CS^+_{A_{\Gamma}} = \Big\{E(z_{ia})_{i \in \sI,a \geq 1} \text{ which vanishes at the wheels \eqref{eqn:wheel finite even},
    \eqref{eqn:wheel finite odd}} \Big\} \\
  & \wCV^+_{\wA_{\Gamma}} \supset \wCS^+_{\wA_{\Gamma}} = \Big\{E(z_{ia})_{i \in \wI,a \geq 1} \text{ which vanishes at the wheels \eqref{eqn:wheel affine even}, \eqref{eqn:wheel affine odd},
    \eqref{eqn:wheel special 1},  \eqref{eqn:wheel special 2}} \Big\} 
\end{align*} 
(recall that the wheels \eqref{eqn:wheel special 1}-\eqref{eqn:wheel special 2} only appear for $m=n$) and their opposites $\CS^-_{A_{\Gamma}}, \wCS^-_{\wA_{\Gamma}}$.
\end{definition}

\medskip 
\noindent 
The fact that $\CS^+_{A_{\Gamma}}$ and $\wCS^+_{\wA_{\Gamma}}$ are subalgebras (i.e.\ are preserved by shuffle product) is a particular case of Lemma \ref{lem:easy}. The following is the main result of this Section.

\medskip

\begin{theorem}
\label{thm:shuffle sl}
$\CS^+_{A_{\Gamma}}$ provides a shuffle realization of the superalgebra $\uslGp$, in the sense of Definition \ref{def:shuffle realization}. Similarly, if $m\neq n$, $\CS^+_{\wA_{\Gamma}}$ provides a shuffle realization of $\UslGp$ (we conjecture that the latter result holds even when $m=n$).
\end{theorem}

\medskip 
\noindent 
As a consequence of Theorem~\ref{thm:shuffle sl}, the discussion of Subsection~\ref{sub:double} applies to type $A$ quantum affine and toroidal superalgebras. In other words, if we define 
$$
  \uslGm = \uslGp^{\text{op}} \qquad \text{and} \qquad \UslGm = \UslGp^{\text{op}}
$$
then the particular case of \eqref{eqn:non degenerate} yields pairings
\begin{align*} 
  \uslGp \otimes \uslGm &\to \BC  \\
  \UslGp \otimes \UslGm &\to \BC 
\end{align*} 
non-degenerate in both arguments. This establishes the type $A$ and $\wA$ versions of Corollary~\ref{cor:intro-nondegeneracy}. Moreover, we have the following Drinfeld doubles
\begin{align*} 
  \uslG &= \uslGp \otimes \BC[\ph_{i,0}^{\pm },\ph_{i,1}^{\pm },\ph_{i,2}^{\pm },\dots]_{i\in \sI} \otimes \uslGm \\
  \UslG &= \UslGp \otimes \BC[\ph_{i,0}^{\pm },\ph_{i,1}^{\pm },\ph_{i,2}^{\pm },\dots]_{i \in \wI} \otimes \UslGm 
\end{align*} 
The left-hand sides of the equations above are called the type $A$ \emph{quantum affine} and \emph{quantum toroidal} superalgebras, although we note that this terminology is typically used for the further quotients of these algebras by the relations $\ph_{i,0}^+\ph_{i,0}^-=1$. In other words, the superalgebras $\uslG$ and $\UslG$ have

\medskip 

\begin{itemize}[leftmargin=0.7cm]
    
\item 
generators $e_{i,d}$, $f_{i,d}$, $\ph^\pm_{i,d'}$ with $d\in \BZ, d'\in \BN$ and $i \in \sI$ (respectively $i\in \wI$) with the parities $|e_{i,d}|=|i|$, $|f_{i,d}|=|i|$, $|\ph^\pm_{i,d'}|=0$;

\medskip 

\item 
relations in the corresponding pre-quantum affine superalgebras
\begin{equation}
\label{eqn:rel quantum affine 1}
  \left[\ph^+_{i}(x), \ph^+_{j}(y)\right] = \left[\ph^+_{i}(x), \ph^-_{j}(y)\right] = 
  \left[\ph^-_{i}(x), \ph^-_{j}(y)\right] = 0, \quad   
  \ph^\pm_{i,0} \ph^\mp_{i,0} = 1
\end{equation} 
\begin{equation}
\label{eqn:rel quantum affine 2}
  e_i(x) e_j(y) \zeta^{A_{\Gamma}}_{ji} \left(\frac yx\right) = e_j(y) e_i(x) \zeta^{A_{\Gamma}}_{ij} \left( \frac xy \right)
\end{equation}
\begin{equation}
\label{eqn:rel quantum affine 3}
  f_j(y) f_i(x) \zeta^{A_{\Gamma}}_{ji} \left(\frac yx\right) = f_i(x) f_j(y) \zeta^{A_{\Gamma}}_{ij} \left( \frac xy \right)
\end{equation}
\begin{equation} 
\label{eqn:rel quantum affine 4}
  \ph^\pm_i(x) e_j(y) = e_j(y) \ph^\pm_i(x) \frac {\zeta^{A_{\Gamma}}_{ij} 
  \Big(\frac xy \Big)}{\zeta^{A_{\Gamma}}_{ji} \Big(\frac yx \Big)} (-1)^{|i||j|}
\end{equation} 
\begin{equation} 
\label{eqn:rel quantum affine 5}
  f_j(y)\ph^\pm_i(x)  = \ph^\pm_i(x)  f_j(y) \frac {\zeta^{A_{\Gamma}}_{ij} 
  \Big(\frac xy \Big)}{\zeta^{A_{\Gamma}}_{ji} \Big(\frac yx \Big)}  (-1)^{|i||j|}
\end{equation}
\begin{equation}
\label{eqn:rel quantum affine 6}
  \Big[e_{i}(x), f_{j}(y)\Big] = \delta_{ij} \delta \left(\frac xy \right) \frac {\ph^+_i(x) - \ph^-_i(y)}{q-q^{-1}}
\end{equation} 
(respectively the same relations with $\zeta^{\wA_\Gamma}$ instead of $\zeta^{A_\Gamma}$ in type $\wA$); 

\medskip 

\item 
all applicable relations among \eqref{eqn:rel quantum affine 8} and \eqref{eqn:rel quantum affine 9} (also \eqref{eqn:rel quantum toroidal special} in type $\wA$ for $m=n$), as well as their $f$-versions in which we replace every $e_i(z)$ by $f_i(z)$; 
the $f$-version of the clunky relation $e_{\bigcirc}$ in~\eqref{eqn:clunky} is defined via 
$$
  f_{\bigcirc} = \sum_{\ell = 1}^k P_{\ell}(x_1,\dots,x_k) 
  f_{i_{\ell+1}}(x_{\ell+1}) \cdots f_{i_k}(x_k) f_{i_1}(x_1) \cdots f_{i_\ell}(x_{\ell})
$$

\end{itemize}


\medskip 
	
\subsection{Proving Theorem~\ref{thm:shuffle sl}}
\label{sub:proof}

As explained in Definition \ref{def:shuffle realization}, the proof of Theorem~\ref{thm:shuffle sl} boils down to establishing Propositions~\ref{prop:criterion 1} and \ref{prop:criterion 2} for the shuffle algebras of Definition~\ref{def:wheel sl} and the ideals $J^+_{A_{\Gamma}}$ and $J^+_{\wA_{\Gamma}}$, respectively. The former of these amounts to verifying that the shuffle algebras introduced in Definition \ref{def:wheel sl} satisfy the hypothesis of Proposition~\ref{prop:criterion 1}. This is an immediate consequence of the following result, which will be established in Subsection~\ref{sub:residues sl} using the notion of shrubs.

\medskip

\begin{lemma}
\label{lem:pairing sl}
When $F \in \CS^-_{A_{\Gamma}}$, the pairing \eqref{eqn:pairing explicit} has the property that
\begin{equation}
\label{eqn:lem pairing sl}
  \Big \langle e_{i_1,d_1}\dots e_{i_k,d_k}, F \Big \rangle = 
  \left[\text{a linear functional of }\tUpsilon^+_{A_{\Gamma}} \left(e_{i_1,d_1}\dots e_{i_k,d_k} \right) \right] 
\end{equation}
and analogously for $\wCS^-_{\wA_{\Gamma}}$ (when $m\neq n$). 
\end{lemma}

\medskip 
\noindent 
We postpone the proof of Lemma~\ref{lem:pairing sl} to Subsection~\ref{sub:residues sl}. Let us now complete the proof of Theorem~\ref{thm:shuffle sl} by verifying Proposition~\ref{prop:criterion 2} in the present setup:
\begin{equation}
\label{eqn:key pairing finite sl}
  \langle J^+_{A_{\Gamma}}, F \rangle = 0 \quad \text{for} \quad F \in \CV^-_{A_{\Gamma}} 
  \quad \Rightarrow \quad F \in \CS^-_{A_{\Gamma}}
\end{equation}
and analogously for $\wA_\Gamma$. The direct approach  was carried out in \cite[Proposition 3.5 and Example~3.9]{NSS}, and can in principle be applied to all cases in the present paper (the interested reader may find this treatment for types $A,B,D$ in Appendix~\ref{sec:delta-equalities}). However, when dealing with the type $C$ scary wheel in Theorem \ref{thm:c intro}, this approach would involve expanding various $7$-variable rational functions in different domains, and collecting the result into an appropriate linear combination of derivatives of delta functions. While this can be done on a computer, it is unpractical to be written in a theoretical paper.

\medskip 
\noindent 
With this in mind, we provide an alternative approach. Assuming $|i|=0$, consider 
\begin{equation}
\label{eqn:x}
  X_{i,i\pm 1}(z_1,z_2,w) = \Big( \mathrm{LHS\ of} \ \eqref{eqn:rel quantum affine 8} \Big) \in \tuslGp [[z^{\pm1}_1,z^{\pm 1}_2, w^{\pm 1}]]
\end{equation}
which is a formal power series 
$$
  X_{i,i\pm 1}(z_1,z_2,w)=\sum_{k_1,k_2,\ell\in \BZ} X_{i,i\pm 1}^{k_1,k_2,\ell}  z_1^{-k_1}z_2^{-k_2}w^{-\ell}
$$
with all $X_{i,i\pm 1}^{k_1,k_2,\ell}\in \tuslGp$ of horizontal degree $\bn = 2\bsi^i + \bsi^{i\pm 1}$. We recall the homomorphism
$$
  \tUpsilon^+_{A_{\Gamma}}\colon \uslGp \to \CV^+_{A_{\Gamma}}
$$
that is the particular case of \eqref{eqn:tilde upsilon} in the case at hand. The formula
\begin{equation}
\label{eqn:rational function}
  \tUpsilon^+_{A_{\Gamma}}(X_{i,i\pm 1}(z_1,z_2,w)) = 0
\end{equation}
holds, as it is tautologically equivalent to the following rational function identity 
\begin{multline*}
  \zeta^{A_{\Gamma}}_{ii}(\tfrac{z_1}{z_2}) 
  \Big(\zeta^{A_{\Gamma}}_{i,i\pm 1}(\tfrac{z_1}{w})\zeta^{A_{\Gamma}}_{i,i\pm 1}(\tfrac{z_2}{w})-(q+q^{-1})\zeta^{A_{\Gamma}}_{i,i\pm 1}(\tfrac{z_1}{w})\zeta^{A_{\Gamma}}_{i\pm 1,i}(\tfrac{w}{z_2})+\zeta^{A_{\Gamma}}_{i\pm 1,i}(\tfrac{w}{z_1})\zeta^{A_{\Gamma}}_{i\pm 1,i}(\tfrac{w}{z_2})\Big) + \\
  \zeta^{A_{\Gamma}}_{ii}(\tfrac{z_2}{z_1}) 
  \Big(\zeta^{A_{\Gamma}}_{i,i\pm 1}(\tfrac{z_1}{w})\zeta^{A_{\Gamma}}_{i,i\pm 1}(\tfrac{z_2}{w})-(q+q^{-1})\zeta^{A_{\Gamma}}_{i,i\pm 1}(\tfrac{z_2}{w})\zeta^{A_{\Gamma}}_{i\pm 1,i}(\tfrac{w}{z_1})+\zeta^{A_{\Gamma}}_{i\pm 1,i}(\tfrac{w}{z_1})\zeta^{A_{\Gamma}}_{i\pm 1,i}(\tfrac{w}{z_2})\Big) 
  = 0
\end{multline*}
(that can be easily checked either by hand or by computer). Therefore, we have
\begin{equation*}
  X_{i,i\pm 1}^ {k_1,k_2,\ell} \in \text{Ker }\tUpsilon^+_{A_{\Gamma}}
\end{equation*}
for all $k_1,k_2,\ell \in \BZ$. Due to Lemma~\ref{lem:pairing sl}, this implies that 
\begin{equation*}
  \langle X_{i,i\pm 1}^{k_1,k_2,\ell}, F \rangle = 0
\end{equation*}
for all $F \in \CS_{A_{\Gamma};-\bn,-k_1-k_2-\ell}$. Thus, for all $d \in \BZ$ the assignment
$$
  \lambda_{d}(F) = \big \langle X_{i,i\pm 1}^{0,0,d}, F \big \rangle
$$
induces a linear functional 
$$
  \lambda_{d} \colon \CV_{A_{\Gamma};-\bn,-d}/\CS_{A_{\Gamma};-\bn,-d} \rightarrow \BC
$$
However, we have 
\begin{equation}
\label{eqn:dim 1}
  \dim_{\BC} \left( \CV_{A_{\Gamma};-\bn,-d}/\CS_{A_{\Gamma};-\bn,-d} \right) = 1
\end{equation}
for all $d \in \BZ$, because the length $3$ wheel \eqref{eqn:wheel finite even} imposes one linear condition on degree $d$ homogeneous Laurent polynomials in $\bn$ variables. For any $d \in \BZ$, a straightforward computation using~\eqref{eqn:pairing explicit} (either by computer or by hand, see Appendices~\ref{sub:non-zero type a even} and~\ref{app:A-cubic} for details) shows that
$$
  F_d = z_{i\pm 1,1}^{-d} \in \CV_{A_{\Gamma};-\bn,-d}
$$
satisfies the property 
\begin{equation}
\label{eqn:not zero 1}
  \lambda_d(F_d) = 2
\end{equation}
We conclude from \eqref{eqn:dim 1} and \eqref{eqn:not zero 1} that pairing with $X_{i,i\pm 1}^{0,0,d}$ imposes a non-trivial linear condition on $\CV_{A_{\Gamma};-\bn,-d}$ that precisely cuts out $\CS_{A_{\Gamma};-\bn,-d}$:
$$
  \big \langle X_{i,i\pm 1}^{0,0,d}, F \big \rangle = 0 \text{ for } F \in \CV_{A_{\Gamma};-\bn,-d} \quad \Leftrightarrow \quad F \in \CS_{A_{\Gamma};-\bn,-d}
$$
The ``$\Rightarrow$'' implication above yields the ``$\Rightarrow$'' implication in the following equivalence
$$
  \big \langle X_{i,i\pm 1}(z_1,z_2,w), F \big \rangle = 0 \text{ for } F \in \CV_{A_{\Gamma};-\bn} \quad \Leftrightarrow \quad F \in \CS_{A_{\Gamma};-\bn}
$$
(while the ``$\Leftarrow$'' implication follows from~\eqref{eqn:rational function} and Lemma~\ref{lem:pairing sl}). As shown in \cite{N Arbitrary}, this implies that the ideal $I_{i,i\pm 1}=(X_{i,i\pm 1}^{k_1,k_2,\ell})_{k_1,k_2,\ell\in \BZ}$ satisfies 
\begin{equation}
\label{eq:duality-deg3}
  \big \langle I_{i,i\pm 1} , F \big \rangle = 0 \text{ for } F \in \CV^-_{A_{\Gamma}} \quad \Leftrightarrow \quad F \text{ vanishes at the wheel \eqref{eqn:wheel finite even} }
\end{equation}

\medskip

\begin{remark}
\label{rem:cont-irrelevance}
Let us emphasize that the precise formula for $F_d$ is immaterial to the argument above: all that matters is that it be a Laurent polynomial which is not annihilated by pairing with some linear combination of $X_{i,i\pm 1}^{k_1,k_2,\ell}$ with $k_1+k_2+\ell=d$. 
\end{remark}

\medskip
\noindent
Assuming $|i|=1$, consider 
$$
    X_{i-1,i,i+1}(z_1,z_2,w,y) = \Big( \text{LHS of \eqref{eqn:rel quantum affine 9}} \Big) \in \tuslGp [[z^{\pm 1}_1, z^{\pm 1}_2, w^{\pm 1}, y^{\pm 1}]]
$$
which is a formal power series 
$$
  X_{i-1,i,i+1}(z_1,z_2,w,y) = \sum_{k_1,k_2,a,b\in \BZ} X_{i-1,i,i+1}^{k_1,k_2,a,b} z^{-k_1}_1 z^{-k_2}_2 w^{-a} y^{-b}
$$
with all $X_{i-1,i,i+1}^{k_1,k_2,a,b}\in \tuslGp$ of degree $\bn' = \bsi^{i-1} + 2\bsi^i + \bsi^{i+1}$. Similarly to \eqref{eqn:rational function}, the following formula can be checked either by hand or on a computer:
\begin{equation}
\label{eqn:rational function 2}
  \tUpsilon^+_{A_{\Gamma}}(X_{i-1,i,i+1}(z_1,z_2,w,y)) = 0
\end{equation}
Thus, due to Lemma~\ref{lem:pairing sl}, for all $d \in \BZ$ the assignment 
$$
  \lambda'_{d}(F) = \big \langle X_{i-1,i,i+1}^{0,0,0,d}, F \big \rangle 
$$
descends to a linear functional
$$
  \lambda'_{d} \colon \CV_{A_{\Gamma};-\bn',-d}/\CS_{A_{\Gamma};-\bn',-d} \rightarrow \BC
$$
However, we have for all $d \in \BZ$
\begin{equation}
\label{eqn:dim 2}
  \dim_{\BC} \left( \CV_{A_{\Gamma};-\bn',-d}/\CS_{A_{\Gamma};-\bn',-d} \right) = 1
\end{equation}
because the length $4$ wheel \eqref{eqn:wheel finite odd} imposes one linear condition on homogeneous Laurent polynomials of degree $\bn'$. A straightforward computation using~\eqref{eqn:pairing explicit} (either by computer or by hand, see Appendices~\ref{sub:non-zero type a odd} and~\ref{app:A-quartic}) shows that
$$
  F'_d = z_{i+1,1}^{-d} \left( \sqrt{\frac {z_{i1}}{z_{i2}}} + \sqrt{\frac {z_{i1}}{z_{i2}}} \right) \in \CV_{A_{\Gamma};-\bn',-d}
$$
(note that square roots are necessary due to~\eqref{eq:root-twist}) satisfies property  
\begin{equation}
\label{eqn:not zero 2}
  \lambda'_d(F'_d) = -2(q+q^{-1})(-1)^{|i+1|(1+|i-1|)}
\end{equation}
We conclude from \eqref{eqn:dim 2} and \eqref{eqn:not zero 2} that pairing with $X_{i-1,i,i+ 1}^{0,0,0,d}$ imposes a non-trivial linear condition on $\CV_{A_{\Gamma};-\bn',-d}$ that precisely cuts out $\CS_{A_{\Gamma};-\bn',-d}$:
$$
  \big \langle X_{i-1,i,i+ 1}^{0,0,0,d}, F \big \rangle = 0 \text{ for } F \in \CV_{A_{\Gamma};-\bn',-d} \quad \Leftrightarrow \quad F \in \CS_{A_{\Gamma};-\bn',-d}
$$
Invoking \eqref{eqn:rational function 2} and Lemma~\ref{lem:pairing sl}, we conclude as before that
$$
  \big \langle X_{i-1,i,i+ 1}(z_1,z_2,w,y), F \big \rangle = 0 \text{ for } F \in \CV_{A_{\Gamma};-\bn'} \quad \Leftrightarrow \quad F \in \CS_{A_{\Gamma};-\bn'}
$$
and therefore akin to~\eqref{eq:duality-deg3}, the ideal $I_{i-1,i,i+1}=(X_{i-1,i,i+1}^{k_1,k_2,a,b})_{k_1,k_2,a,b \in \BZ}$ satisfies 
\begin{equation*}
  \big \langle I_{i-1,i,i+1} , F \big \rangle = 0 \text{ for } F \in \CV^-_{A_{\Gamma}} \quad \Leftrightarrow \quad F \text{ vanishes at the wheel \eqref{eqn:wheel finite odd} }
\end{equation*}

\medskip 
\noindent 
The affine version of the proof, namely the fact that
$$
  \langle J^+_{\wA_{\Gamma}}, F \rangle = 0 \quad \text{for} \quad F \in \CV^-_{\wA_{\Gamma}} \quad \Rightarrow \quad F \in \CS^-_{\wA_{\Gamma}}
$$
is proved just like in the finite type $A$ case treated above. This is true even for $m=n$, because the elements \eqref{eqn:rel quantum toroidal special} are precisely defined in order to be dual to the special wheels  \eqref{eqn:wheel special 1}-\eqref{eqn:wheel special 2}, as explained in Subsection \ref{sub:wheel} (for instance, the analogues of the key properties~\eqref{eqn:rational function} and \eqref{eqn:not zero 1} for the special wheel conditions were proved in \cite[Propositions 3.3 and~3.5]{N Reduced}).


\medskip 
	
\subsection{Shrubs}
\label{sub:shrubs sl}

We will prove Lemma \ref{lem:pairing sl} by explicitly calculating the linear functional which appears in \eqref{eqn:lem pairing sl}. To this end, we must recall a combinatorial tool known as a \emph{shrub}, that was introduced in \cite{N Reduced}. Consider the quiver \footnote{The quiver $\tQ$ appeared in the mathematical physics literature in the setting of the generalized conifold toric Calabi-Yau threefold, see \cite{CFIKV, FHKVW, FHMSVW, HK} and the overview in Subsection \ref{sub:intro shrub}.} $\tQ$ with

\medskip 
	
\begin{itemize}[leftmargin=0.7cm]
		
\item 
vertex set $\BZ^2$; we call $(x,y) \in \BZ^2$ odd/even based on the parity of $\frac{1 - s_{x-y}s_{x-y+1}}{2}$;
		
\medskip 
		
\item 
an arrow $(x,y) \rightarrow (x+1,y+1)$ (i.e.\ pointing northeast) if $s_{x-y} = s_{x-y+1} = 1$;
		
\medskip 
		
\item 
an arrow $(x,y) \rightarrow (x-1,y-1)$ (i.e.\ pointing southwest) if $s_{x-y} = s_{x-y+1} = -1$;

\medskip 
		
\item 
an arrow $(x,y) \rightarrow (x+1,y)$ (i.e.\ pointing right) if $s_{x-y+1} = -1$;

\medskip 
		
\item 
an arrow $(x,y) \rightarrow (x-1,y)$ (i.e.\ pointing left) if $s_{x-y} = 1$;
		
\medskip 
		
\item 
an arrow $(x,y) \rightarrow (x,y+1)$ (i.e.\ pointing up) if $s_{x-y} = -1$;
		
\medskip 
		
\item 
an arrow $(x,y) \rightarrow (x,y-1)$ (i.e.\ pointing down) if $s_{x-y+1} = 1$;
		
\end{itemize}

\medskip 
\noindent 
see Figure~\ref{fig:grid} for an example. Above, $\bs$ is the parity sequence \eqref{eqn:signs} that determines a type $\wA$ decorated Dynkin diagram $\Gamma$. We call $\tQ$ the \emph{periodic} quiver as it is preserved by translations by $(1,1)$ and $(N,0)$, and thus it descends to a quiver $Q$ on the torus
\begin{equation}
\label{eqn:translation}
  \BR^2 / \{(a+Nb,a) \,|\, a,b\in \BZ\}
\end{equation}
If we wish to deal with the case of finite type $A$ instead of affine type $\wA$, then we simply remove from $\tQ$ all vertices $(x,y)$ with $x-y \notin \{1,\dots,N-1\}$. Thus, the finite type quiver will be constrained to an infinite diagonal strip, see Figure \ref{fig:strip}.

\begin{figure}[h]
  \includegraphics[scale=1.5]{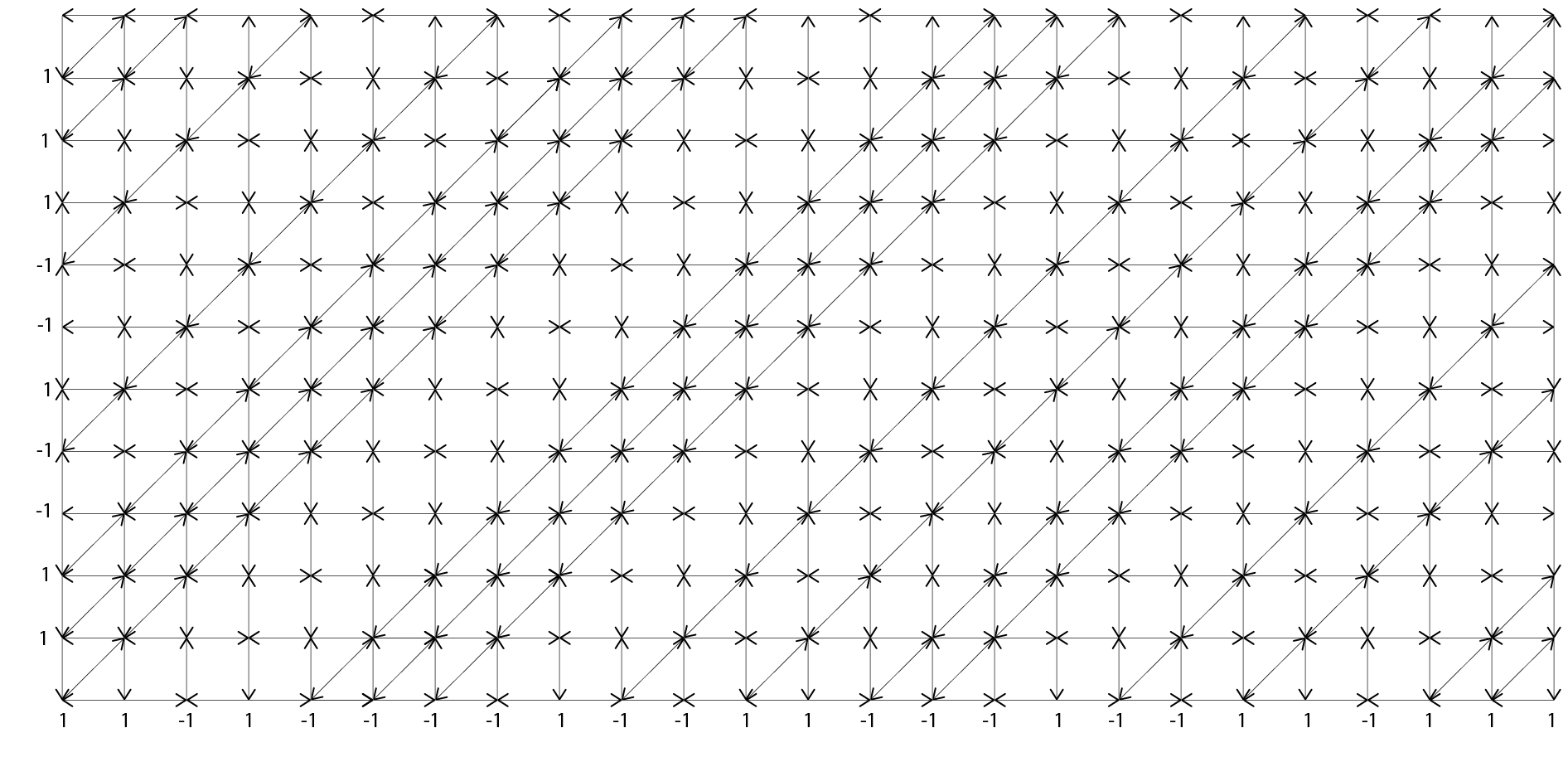} 
  \caption{Zooming in on a part of $\tQ$.}
  \label{fig:grid}
\end{figure}

\medskip 
\noindent 
The \emph{faces} of $\tQ$ are the small triangles and squares that appear in Figure~\ref{fig:grid}. Note that they are all oriented cycles, and that the quiver above is shrubby in the sense of \cite[Definition~1.5]{N Reduced}. Thus, we will freely invoke the constructions of \loccit For instance, we have the notion of a \emph{broken wheel}, which refers to the boundary of any face minus any one arrow. The \emph{mirror image} of a broken wheel refers to the broken wheel which lies on the other side of the missing arrow (in Figure \ref{fig:broken}, the blue broken wheel is the mirror image of the red one, and vice versa).

\begin{figure}[h]
  \includegraphics[scale=2]{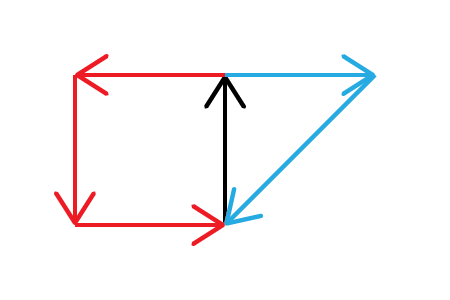} 
  \caption{Broken wheels that are mirror images of each other.}
  \label{fig:broken}
\end{figure}

\begin{definition} 
\label{def:shrub}
An \emph{(affine type) shrub} $S$ is a rooted subgraph of $\tQ$ such that:
		
\medskip 
		
\begin{enumerate}[leftmargin=1cm]
			
\item 
any vertex of $S$ can be reached from the root by following arrows in $S$;
			
\medskip 
			
\item 
$S$ does not contain any oriented cycles \footnote{As shown in \cite{N Reduced}, this is equivalent to the a priori weaker claim that $S$ does not contain the entire boundary of any face.};
			
\medskip 
			
\item 
whenever $S$ contains a broken wheel, it also contains its mirror image;
			
\medskip 
			
\item 
$S$ contains all vertices completely surrounded by arrows of $S$. 
			
\end{enumerate}
\end{definition}

\medskip
\noindent 
If the shrub is completely contained in the diagonal strip $\{ (x,y) \,\big|\,1 \leq x-y < N \}$ (see Figure~\ref{fig:strip}) then we call it a \emph{finite type shrub}. Note that even in finite type, a shrub is still required to satisfy the conditions in Definition \ref{def:shrub} in the periodic quiver $\tQ$, not merely in the aforementioned diagonal strip; for instance, the red subgraph in Figure \ref{fig:strip} is not a finite type shrub.

\medskip

\begin{figure}[h]
  \includegraphics[scale=1.5]{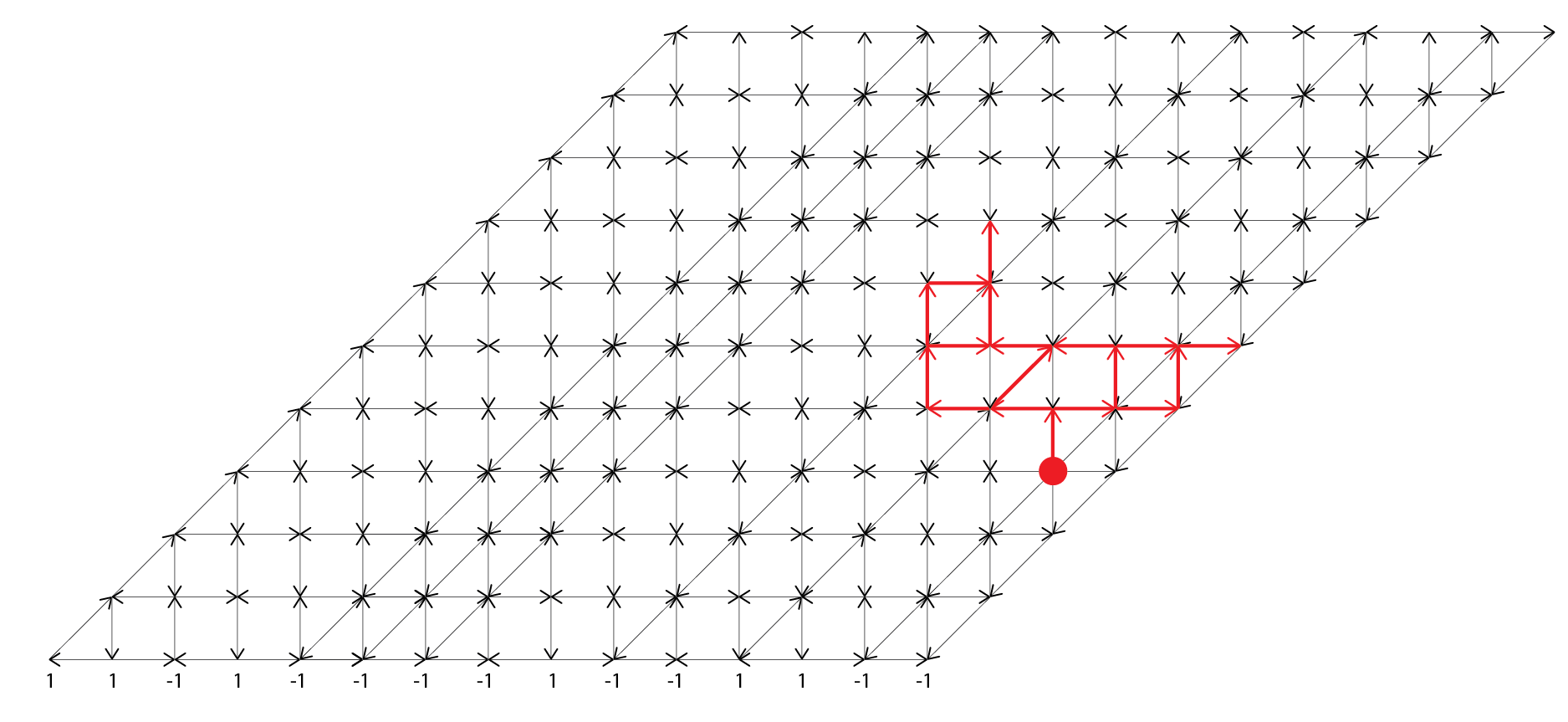} 
  \caption{
The finite type strip (for $N=16$) of a periodic quiver; note that the red subgraph is \underline{not} a shrub because the broken wheel property \emph{(3)} of Definition \ref{def:shrub} is violated at the right-most edge.}
  \label{fig:strip}
\end{figure}


\medskip 

\subsection{Addable vertices}
\label{sub:addable}

We recall from \cite[\S 3.18]{N Reduced} the notion of when a vertex $v$ is \emph{addable} to a shrub $S$, which depends on the arrows $e_1,\dots,e_k$ incoming to $v$ as depicted in Figure \ref{fig:addable}.

\begin{figure}[h]
\includegraphics[scale=2]{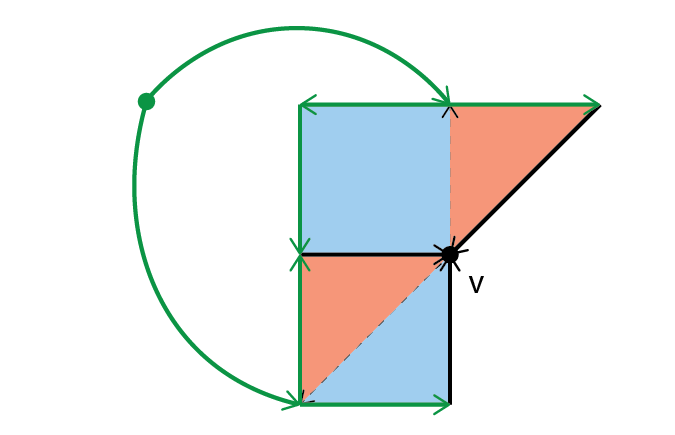} 
\caption{An addable vertex (black) to a shrub (green).}
\label{fig:addable}
\end{figure}

\medskip        
\noindent 
Specifically, in Figure \ref{fig:addable} the shrub $S$ is drawn in green and its root is the green disk. The vertex $v$ marked by the black disk has $k=3$ incoming solid black arrows, which are labeled by $e_1,e_2,e_3$ in counterclockwise order from northeast to south. By definition, we say that $v$ is addable to $S$ if for any $1\leq s<k$ one may continue the edges $e_{s}$ and $e_{s+1}$ in $S$ until they meet for the first time giving rise to sub-paths 
\begin{align*}
  p_s & \colon w \rightarrow \dots \xrightarrow{e_s} v \\
  p_{s+1} & \colon w \rightarrow \dots \xrightarrow{e_{s+1}} v
\end{align*}
(simple, non-intersecting except at the endpoints, and the region between them is minimal with respect to inclusion) so that $p_s$ and $p_{s+1}$ are broken wheels that are mirror images of each other, with interface given by a dotted arrow in Figure \ref{fig:addable}.

\medskip
\noindent 
It was shown in \cite[Proposition~3.20]{N Reduced} that $v$ is addable to $S$ if and only if $S+v$ is a shrub (by definition, we let $S+v$ consist of all vertices and arrows in $S$, together with the vertex $v$, and together with the arrows $e_1,\dots,e_k$).  Moreover, any shrub can be obtained from the root, viewed as a one-vertex shrub, by successively adding addable vertices in some not-necessarily-unique order. This recursive characterization of shrubs will be used in their classification in Proposition \ref{prop:shrubs sl}.


\medskip

\subsection{Classification}
\label{sub:classification}

We will now completely classify affine type shrubs.

\medskip

\begin{definition}
\label{def:directrix}
Fix a vertex $r \in \tQ$ called the \emph{root}. Depending on whether $r$ is even or odd, we define either $3$ or $4$ oriented paths called \emph{directrixes} as follows:
		
\medskip 
		
\begin{itemize}[leftmargin=0.5cm]
			
\item 
the $1$ or $2$ \emph{regular} directrixes: if $r$ is even, this refers to the infinite diagonal path emerging from $r$; if $r$ is odd, this refers to the two infinite horizontally-vertically alternating paths emerging from $r$. Regular directrixes point northeast-southwest;
			
\medskip 
			
\item 
the $2$ \emph{irregular} directrixes: these emerge from $r$ and stay totally within either the northwest or the southeast quadrant, and are defined as the unique succession of horizontal and vertical arrows which change direction precisely when they meet an odd vertex, and do not contain broken wheels.
			
\end{itemize}
\end{definition}

\noindent 
Since a picture is worth a thousand words, we include in  Figures \ref{fig:even irregular}-\ref{fig:odd 2} several examples of directrixes (depicted in green), which we hope will explain the general pattern. The relevance of directrixes of a certain vertex $r$ is that any shrub with root $r$ that is a tree must be an oriented rooted subtree of the union of the directrixes.

\begin{figure}[ht]
  \includegraphics[scale=2]{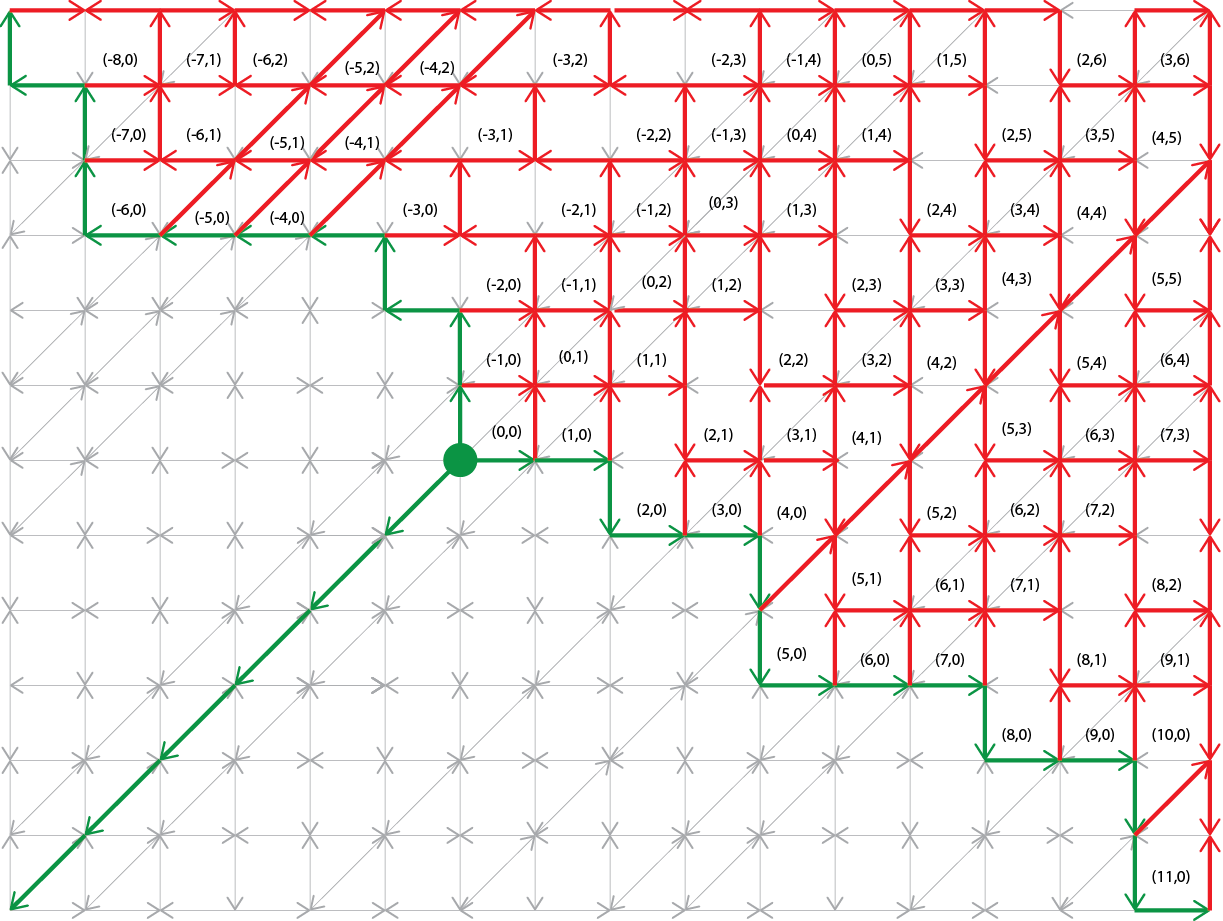} 
  \caption{Even root, irregular region.}
  \label{fig:even irregular}
\end{figure}

\begin{figure}[ht]
   \includegraphics[scale=2]{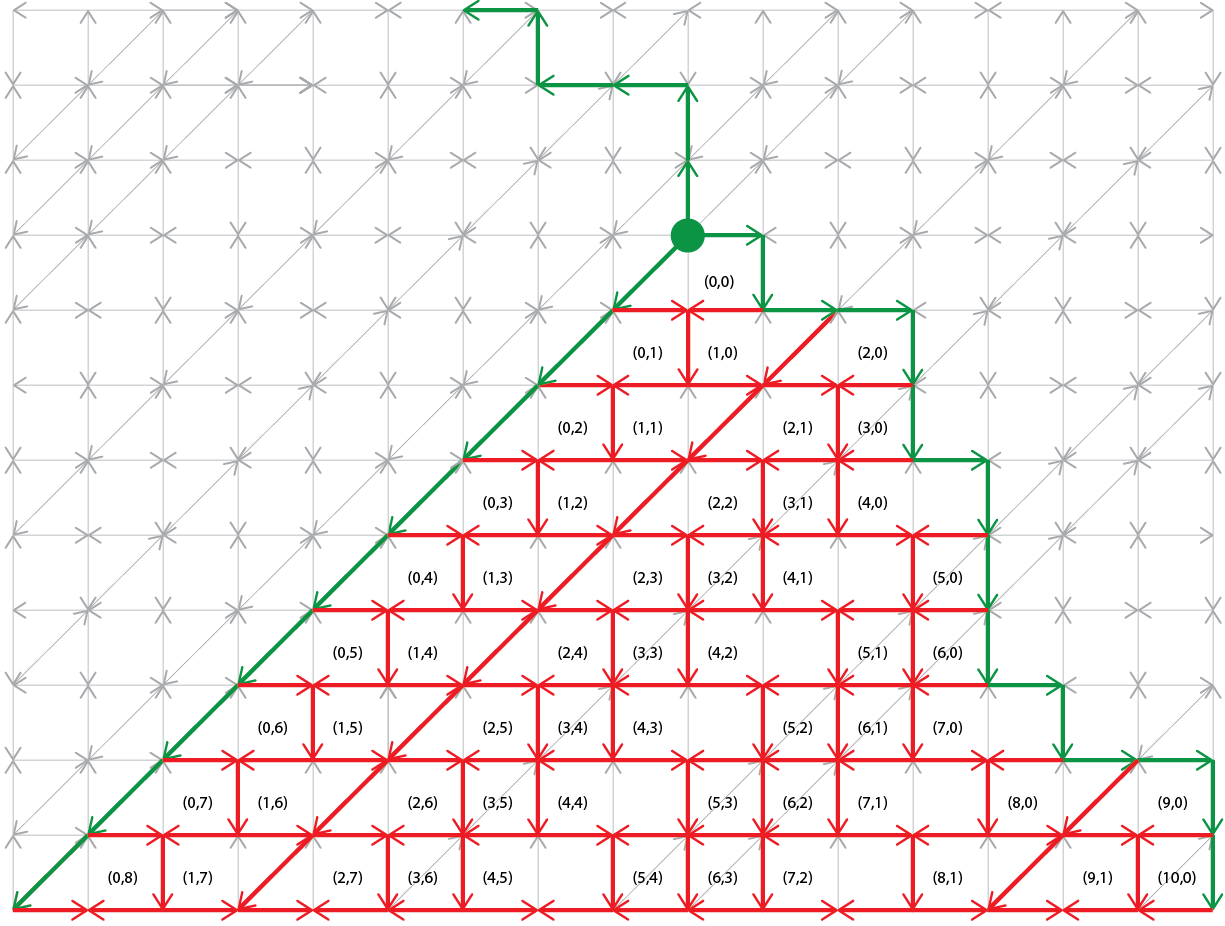} 
   \caption{Even root, regular region.}
  \label{fig:even regular}
\end{figure}

\begin{figure}[ht]
   \includegraphics[scale=2]{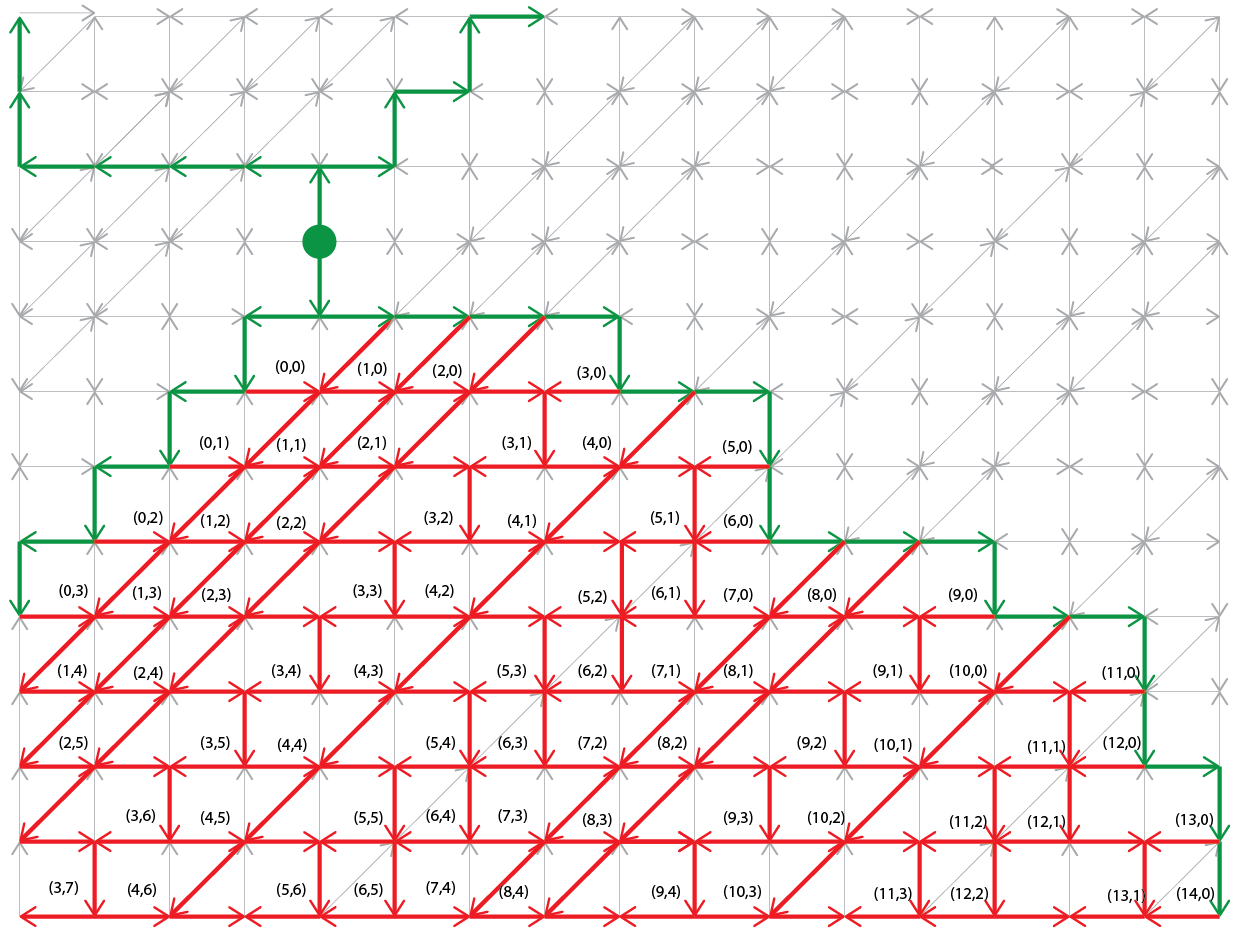} 
   \caption{Odd root (with odd vertices next to it).}	
   \label{fig:odd 1}
\end{figure}

\begin{figure}[ht]
  \includegraphics[scale=2]{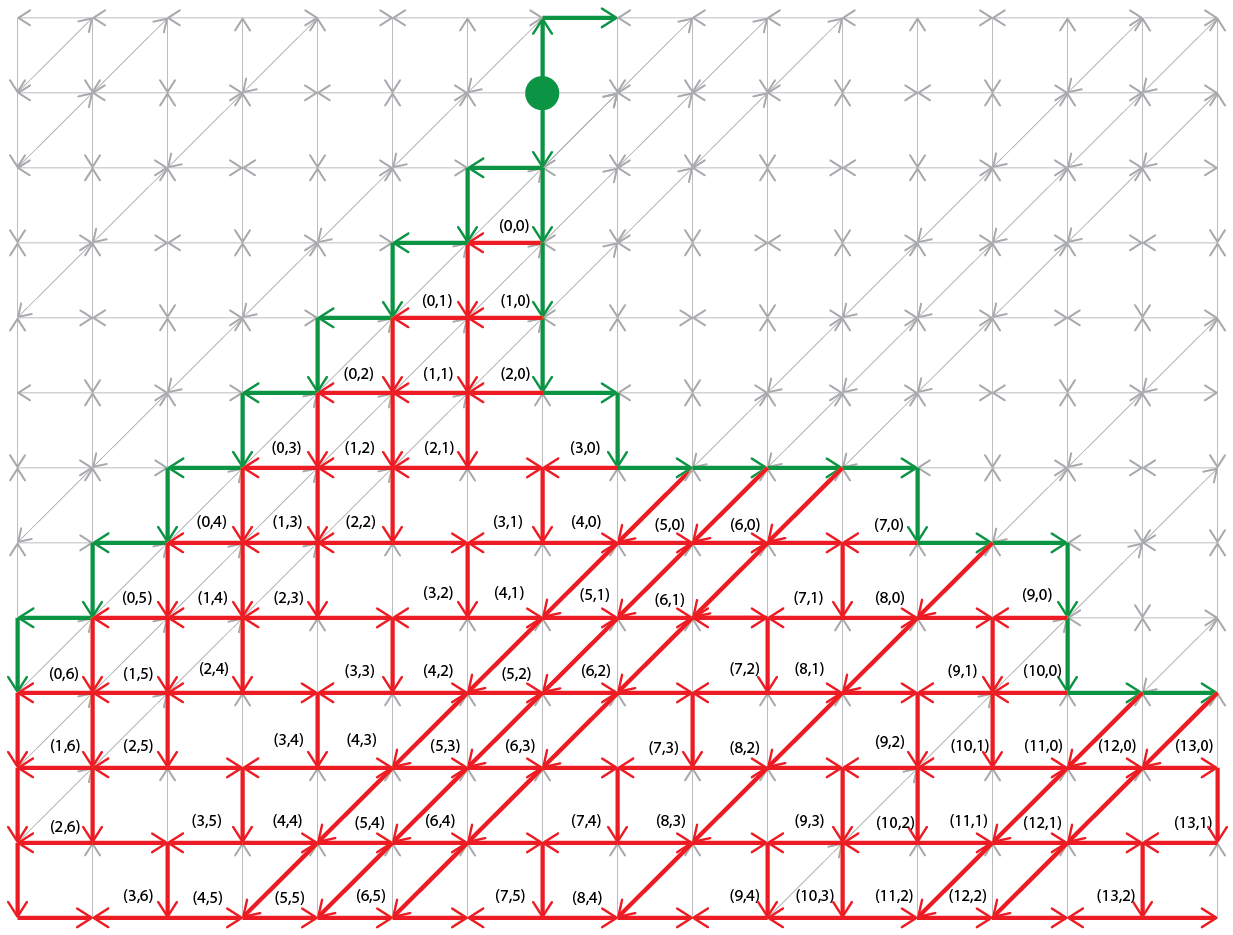} 
  \caption{Odd root (with even vertices next to it).}	
  \label{fig:odd 2}
\end{figure}

\medskip        
\noindent 
The directrixes emerging from an even/odd root divide $\tQ$ into three/four regions. In the even root case, we call the two regions adjacent to the regular directrix \emph{regular}, while the region bounded by the two irregular directrixes will be called \emph{irregular}. As indicated in Figures \ref{fig:even irregular}-\ref{fig:odd 2}, the regions are naturally tiled (periodically in the $(1,1)$ direction) with the following types of dominoes (depicted in red in the figures), each consisting of a broken wheel and its mirror image:
	
\medskip 
	
\begin{itemize}[leftmargin=0.7cm]
		
\item 
double squares, consisting of two adjacent square faces (two examples of double squares are the ones labeled $(2,0)$ and $(-2,0)$ in Figure \ref{fig:even irregular});
		
\medskip 
		
\item 
single squares, consisting of two triangular faces adjacent along a diagonal (an example of single square is the one labeled $(0,0)$ in Figure \ref{fig:even irregular});
		
\medskip 
		
\item 
lozenges, consisting of two triangular faces adjacent along a horizontal or vertical edge (an example of lozenge is the one labeled $(1,1)$ in Figure \ref{fig:odd 1});
		
\medskip 
		
\item 
trapezoids, consisting of a square face and an adjacent triangular face (two examples of trapezoids are the ones labeled $(0,0)$ and $(3,0)$ in Figure \ref{fig:odd 1}).
		
\end{itemize}

\medskip 
\noindent 
As apparent from Figures \ref{fig:even irregular}-\ref{fig:odd 2}, the dominoes in every region can be numbered in Young-diagram like fashion: the one closest to the root is denoted by $(0,0)$, and those adjacent to the irregular directrix are denoted by $(1,0)$, $(2,0)$, etc. Then as we move in the northeast-southwest direction from the domino denoted by $(a,0)$ toward the interior of the region, we start labeling the dominoes we encounter by $(a,1)$, $(a,2)$, etc. Note that in the case of the irregular region rooted at an even vertex, since it is bounded by two irregular directrixes, we will label the dominoes adjacent to them as $\dots, (-2,0), (-1,0), (0,0), (1,0), (2,0), \dots$, but then still increase the second index as we move diagonally in the northeast-southwest direction.

\medskip

\begin{proposition}
\label{prop:shrubs sl}
To obtain the full classification of shrubs rooted at some $r \in \tQ$

\medskip 

\begin{itemize}[leftmargin=0.7cm]
			
\item 
choose a finite connected subgraph $P \ni r$ of the union of the green directrixes;
			
\medskip 
			
\item 
pick a collection $C$ of dominoes in each region with the following properties:

\medskip 

\begin{itemize}[leftmargin=0.7cm] 

\item the green boundary of any domino $\{(d,0)\in C, d \in \BZ\}$ is contained in $P$;

\medskip 

\item all single squares which contain at least two edges of $P$ must lie in $C$;

\medskip 

\item if $(a,b) \in C$, then $(a',b') \in C$ for any $|a'| \leq |a|$ and $b' \leq b$;

\medskip 

\item if $(a,b) \in C$, then $(a',b') \in C$ if $|a'-a|=1$, $b'=b-1$ and $(a',b')$ is a single square; 

\end{itemize}

\medskip 

\item 
consider any set $A$ of vertical/horizontal arrows that point away from the root, and have sources lying in the union of $P$ and boundaries of dominoes in $C$. The meaning of ``point away" is to point up, right, down, left if the arrow lies in the region up, right, down, left of the root. The only situation when this is unambiguous is in the irregular region corresponding to an even root; in this case the arrow is required to point away from the regular directrix.

\end{itemize}

\medskip 
\noindent 
The union of the subgraph $P$ in the first bullet, the boundary of dominoes in the collection $C$ in the second bullet, and the set $A$ of arrows in the third bullet, forms a shrub rooted at $r$. Conversely, any shrub rooted at $r$ can be obtained in this way. 
\end{proposition}

\medskip

\begin{proof} 
It is easy to see that any union $P \cup C \cup A$ as above determines a shrub, using the recursive construction of shrubs from the end of Subsection~\ref{sub:addable}. We will now prove the converse, which states that any shrub $S$ rooted at $r$ arises in this way.

\medskip 
\noindent 
To this end, we will use the following well-known features of $\tQ$: inspired by the notion of $R$-charge of dimer models (see \cite{HHV} and \cite{Br}), we define the length of a path by weighing diagonal arrows with $1$ and weighing horizontal/vertical arrows with $\frac 12$. A path starting at $v$ and ending at $v'$ is called \emph{minimal} if its length is the smallest among all paths from $v$ and $v'$. It is a well-known fact that any two minimal paths from one vertex to another are $F$-term equivalent (\cite{FHKVW}), i.e.\ equivalent under the relation which identifies any path with the result of replacing any broken wheel by its mirror image (see Figure~\ref{fig:broken}). Moreover, paths which fail to be minimal are $F$-term equivalent with paths that contain loops around the faces of the quiver $\tQ$.

\medskip 
\noindent
Therefore, properties \emph{(2)} and \emph{(3)} in Definition \ref{def:shrub} imply that any path in a shrub must be minimal, and moreover, any minimal path from the root to a vertex of a shrub must be fully contained in the shrub. This immediately explains why no gray arrow $\alpha$ in Figure \ref{fig:even irregular} (irregular region), Figure \ref{fig:even regular} (regular region) and Figures \ref{fig:odd 1}-\ref{fig:odd 2} can lie in $S$: this is because the tail of $\alpha$ is farther from the root (in terms of length of a minimal path) than the head of $\alpha$.

\medskip 
\noindent 
Because of the previous paragraph, the intersection of $S$ with any one of the three/four regions is a subgraph of the domino tiling (i.e.\ the union of the green and red arrows in Figures \ref{fig:even irregular}-\ref{fig:odd 2}). Let $C$ be the collection of dominoes that lie completely within $S$, and we must show that it satisfies the conditions in the second bullet of Proposition \ref{prop:shrubs sl}. To this end, we observe that any domino has a single vertex which has two incoming red arrows: we call it the \emph{essential vertex} of the domino. The terminology is due to the easy-to-see claim that a domino is completely contained in a shrub $S$ if and only if its essential vertex lies in $S$. Therefore, the last two properties in the second bullet of Proposition \ref{prop:shrubs sl} are immediate consequences of the previous paragraph and the following claim: the essential vertex of a domino $(a',b')$ lies on a minimal path from the root to the essential vertex of a domino $(a,b)$ if

\medskip 

\begin{itemize}[leftmargin=0.7cm]

\item $|a'| \leq |a|$ and $b' \leq b$, or if

\medskip 

\item $|a'-a| = 1$, $b'=b-1$ and $(a',b')$ is a single square.
    
\end{itemize}

\medskip 
\noindent 
Finally, we must show that if an arrow $\alpha$ lies in $S$ but is not part of a complete domino, then it must satisfy the properties in the third bullet of Proposition \ref{prop:shrubs sl}. These arrows are easily seen to be precisely the red arrows which go head-to-head with gray arrows in Figures \ref{fig:even irregular}-\ref{fig:odd 2}, or equivalently precisely the red arrows that do not point to an essential vertex. Indeed, if an arrow $\alpha \in S$ pointed to the essential vertex of a domino, then the domino in question would have to be completely contained in $S$. The contrapositive statement is that if $\alpha$ does not belong to a domino completely contained in $S$, then it cannot point to an essential vertex, as we needed to prove. \end{proof}

\medskip 
\noindent 
The explicit characterization of all shrubs in Proposition \ref{prop:shrubs sl} immediately implies the following property, which will play a crucial role in Sections \ref{sec:osp} and \ref{sec:osp-toroidal}.

\medskip

\begin{corollary}
\label{cor:shrubs sl}
In any shrub, the vertices which lie on the extremal diagonals form ladders of the form depicted in Figure~\ref{fig:last}, that point either northeast or southwest.

\begin{figure}[h]
  \includegraphics[scale=2]{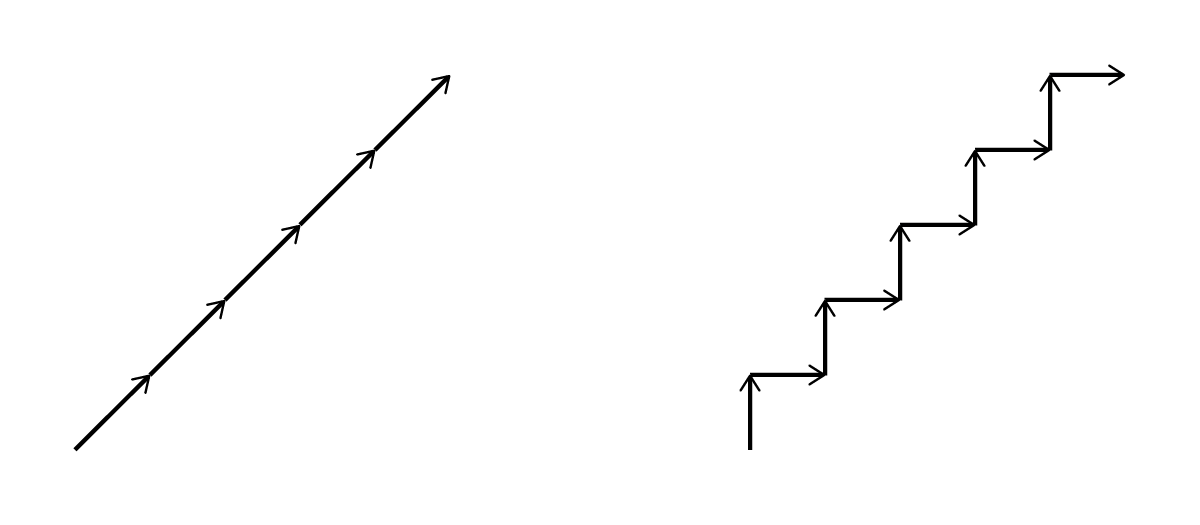} 
  \caption{Even on the left, odd on the right.}
  \label{fig:last}
\end{figure}

\end{corollary}


\medskip 

\subsection{Residues}
\label{sub:residues sl}
	
We decorate the arrows $\alpha$ of $\tQ$ with the following parameters:
	
\medskip 
	
\begin{itemize}[leftmargin=0.7cm] 
		
\item 
$t_\alpha = q^2$ to the arrows pointing northeast
		
\medskip 
		
\item 
$t_\alpha = q^{-2}$ to the arrows pointing southwest
		
\medskip 
		
\item 
$t_\alpha = qd$ to the arrows pointing right
		
\medskip 
		
\item 
$t_\alpha = q^{-1}d$ to the arrows pointing left
		
\medskip 
		
\item 
$t_\alpha = qd^{-1}$ to the arrows pointing up
		
\medskip 
		
\item 
$t_\alpha = q^{-1}d^{-1}$ to the arrows pointing down
		
\end{itemize}

\medskip 
\noindent 
(in finite type, in which the quiver $\tQ$ is reduced to the diagonal strip $\{(x,y) \,|\, 1\leq x-y < N\}$, the parameter $d$ is set equal to 1 by convention, see Remark~\ref{rem:d-redundancy}). Thus, the parameters of arrows are completely determined by their orientation: 
$$
  \text{an arrow }(x,y) \to (x',y') \text{ has parameter } q^{x'+y'-x-y}d^{|x'-x|-|y'-y|}
$$
The product of parameters around any triangular or square face (see Figure \ref{fig:grid}) is easily seen to be equal to 1. As a consequence, it is straightforward to see that for any vertex $v$ in any shrub $S$ with root $r$, the \emph{coordinate}
\begin{equation*}
  t_S(v) = t_{\alpha_1} \dots t_{\alpha_k} 
\end{equation*}
is independent on the choice of a path 
\begin{equation*}
  r \xrightarrow{\alpha_k} \dots \xrightarrow{\alpha_1} v
\end{equation*}
If $m\neq n$, it is easy to see that different vertices of the same color in a shrub have different coordinates.

\medskip

\begin{definition} 
\label{def:acceptable}
With respect to fixed positive real numbers
\begin{equation}
\label{eqn:rho}
  \rho_1,\dots,\rho_{N} \in \BR_{>0}
\end{equation}
a shrub $S$ is called \emph{acceptable} if 
\begin{equation}
\label{eqn:acceptable}
  |t_S(v)| > \frac {\rho_{\emph{col}(v)}}{\rho_{\emph{col}(r)}}
\end{equation}
for all vertices $v \in S$ other than the root $r$.
\end{definition}

\medskip 
\noindent 
We henceforth assume that the numbers \eqref{eqn:rho} are chosen generic enough so that 
\begin{equation*}
  \rho_{\text{col}(r)} |t_{\alpha_1} \dots t_{\alpha_k}| \neq \rho_{\text{col}(r')} |t_{\beta_1} \dots t_{\beta_\ell}|
\end{equation*}
for all paths
$$
r \xrightarrow{\alpha_k} \dots \xrightarrow{\alpha_1} v \xleftarrow{\beta_1} \dots \xleftarrow{\beta_{\ell}} r'
$$
with $r \neq r'$. From now on, all shrubs will be considered modulo translation by $\{(a+Nb,a)\}_{a,b \in \BZ}$, see \eqref{eqn:translation}. A \emph{shrubbery} $\mathscr{S}$ is an unordered collection of shrubs $S_1 \sqcup \dots \sqcup S_t$, and it is called acceptable if all its constituent shrubs are acceptable.

\medskip

\begin{proof} \emph{of Lemma \ref{lem:pairing sl} (sketch following \cite[Proposition 3.10, \S 3.22]{N Reduced}):} 
We shall treat the affine case $\wA_{\Gamma}$ (with $m\ne n$), while the case of $A_{\Gamma}$ is analogous. Consider any shrubbery $\mathscr{S} = S_1\sqcup \dots \sqcup S_t$. A \emph{labeling} of $\mathscr{S}$ will refer to an ordering of its vertices
\begin{equation}
\label{eqn:labeling}
  v_1, \dots, v_k
\end{equation}
such that there is a path in $\tQ$ from $v_a$ to $v_b$ only if $a>b$. Let $i_1,\dots,i_k \in \wI$ be the colors of $v_1,\dots,v_k$, and let $\bn = \bsi^{i_1}+\dots+\bsi^{i_k}$. For any $F \in \wCS_{\wA_{\Gamma};-\bn}$ and $d_1,\dots,d_k \in \BZ$, we consider the rational function 
\begin{equation}
\label{eqn:f}
  f(z_1,\dots,z_k) = 
  \frac{F(z_{i_a\bullet_a})_{a \in \{1,\dots,k\}} z_1^{d_1}\dots z_k^{d_k}}
       {\prod_{1 \leq a < b \leq k} \zeta^{\wA_{\Gamma}}_{i_bi_a} \left(\frac {z_b}{z_a} \right)}
\end{equation}
where $\bullet_1,\dots,\bullet_k$ denote the minimal positive integers such that $\bullet_a < \bullet_b$ if $a<b$ and $i_a = i_b$. It was shown in \cite{N Reduced} that for any labeled shrubbery $\mathscr{S} = S_1 \sqcup \dots \sqcup S_t$ (assume that the roots of the shrubs $S_1,\dots,S_t$ have labels $r_1,\dots,r_t$), we may define
\begin{equation}
\label{eqn:residue}
  \underset{S_1}{\text{Res}} \dots \underset{S_t}{\text{Res}} \ f \in \BC(z_{r_1}, \dots,z_{r_t})
\end{equation}
as the quantity $f^{(1)}$ produced by the following recursive procedure in $\ell = k,\dots,1$:
		
\medskip 
		
\begin{itemize}[leftmargin=0.7cm]
			
\item 
start from $f^{(k)} = f$ of \eqref{eqn:f}; 

\medskip 
			
\item 
if $v_{\ell}$ is a root of some constituent shrub $S_c \subset \mathscr{S}$, define $f^{(\ell)} = f^{(\ell+1)}$;
			
\medskip 
			
\item 
if $v_{\ell} \in S_c \subset \mathscr{S}$ is not a root, then note that $f^{(\ell+1)}$ has a simple pole at 
\begin{equation*}
  z_\ell = t_{S_c}(v_\ell) z_{r_c}
\end{equation*}
(the previous claim uses the combinatorics of shrubs, and was proved in~\cite[Proposition 3.24]{N Reduced}). Define $f^{(\ell)}$ as the residue of $f^{(\ell+1)}$ at the above pole.
			
\end{itemize}

\medskip

\begin{remark}
\label{rem:zero}
Since the zeta function $\zeta^{\wA_\Gamma}_{ii}(x)$ has a pole at $x=1$ for all $i \in \wI$, the rational function \eqref{eqn:f} has a zero at $z_a=z_b$ whenever $i_a= i_b$. Because of this, vertices of a shrub cannot ``double up" in the following sense: once a variable $z_b$ is specialized in accordance to a vertex $v \in S_c \subset \mathscr{S}$, no other variable $z_a$ with $a<b$ can be specialized to the same vertex for the same constituent shrub $S_c \subset \mathscr{S}$.
\end{remark}

\medskip 
\noindent 
As an example, consider the shrubbery consisting of the single shrub $S$ in Figure~\ref{fig:shrub}, rooted at the red disk. The southwest, southeast, northwest, northeast vertices are labeled by $v_4,v_3,v_2,v_1$ and their colors will be labeled $i_4,i_3,i_2,i_1$, respectively.

\begin{figure}[h]
  \includegraphics[scale=2]{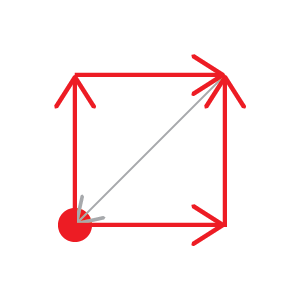} 
  \caption{A shrub $S$, and the corresponding specialization map.}
  \label{fig:shrub}
\end{figure}

\noindent
Therefore, the zeta functions are given by
\begin{align*}
  & \zeta^{\wA_{\Gamma}}_{i_4 i_3} \left(\frac {z_4}{z_3} \right) \sim (z_3 - z_4 qd), \qquad  \ \ \zeta^{\wA_{\Gamma}}_{i_2 i_1} \left(\frac {z_2}{z_1} \right) \sim (z_1 - z_2 qd)  \\
  & \zeta^{\wA_{\Gamma}}_{i_4 i_2} \left(\frac {z_4}{z_2} \right) \sim (z_2 - z_4 qd^{-1}), \qquad \zeta^{\wA_{\Gamma}}_{i_3 i_1} \left(\frac {z_3}{z_1} \right) \sim (z_1 - z_3 qd^{-1})
\end{align*}
All other zeta functions do not vanish at the point $(z_1,z_2,z_3) = (z_4q^2, z_4qd^{-1}, z_4 qd)$, which will be the main focus in the subsequent argument. Therefore, up to multiplication by a monomial and a non-zero constant that we will ignore, we have

\begin{multline*}
  \underset{S}{\text{Res}} \left( 
  \frac {F(z_1,z_2,z_3,z_4)}{\prod_{1\leq a < b \leq 4} \zeta^{\wA_{\Gamma}}_{i_bi_a}\left(\frac {z_b}{z_a}\right)}\right) = \\
  \underset{S}{\text{Res}} \left( 
  \frac {F(z_1,z_2,z_3,z_4)}{(z_1-z_2qd)(z_1-z_3qd^{-1})(z_2-z_4qd^{-1})(z_3-z_4qd)} \right) = \\ 
  \underset{z_1 = z_4q^2}{\text{Res}} \left(\frac {F(z_1,z_4qd^{-1},z_4qd,z_4)}{(z_1-z_4q^2)^2}\right)
\end{multline*}
However, because $F$ satisfies the length $3$ wheel \eqref{eqn:wheel affine even} corresponding to the two triangles in Figure \ref{fig:shrub}, the numerator of the expression above vanishes at $z_1=z_4q^2$. Thus the overall fraction has a simple pole at $z_1=z_4q^2$, as prescribed. We conclude that, up to multiplication by a non-zero constant that we will ignore, we have
\begin{equation}
\label{eqn:shrub}
  \underset{S}{\text{Res}} \left( 
  \frac {F(z_1,z_2,z_3,z_4)}{\prod_{1\leq a < b \leq 4} \zeta^{\wA_{\Gamma}}_{i_bi_a}\left(\frac {z_b}{z_a}\right)}\right) 
  = \frac {\partial F(z_1,z_4qd^{-1},z_4qd,z_4)}{\partial z_1} \Big|_{z_1 = z_4q^2}
\end{equation}

\medskip

\begin{claim}
\label{claim:independent}
\eqref{eqn:residue} only depends on the shrubbery $\mathscr{S}$, and not on the labeling \eqref{eqn:labeling}.
\end{claim}

\medskip

\begin{proof} 
The value of \eqref{eqn:residue} is given by a successive contour integral in the variables $z_1,\dots,z_k$, where the contour of $z_a$ lies outside the contour of $z_b$ if $a<b$. Thus, if there is an arrow from $v_b$ to $v_a$ in our shrubbery, then the contour of $z_a$ must lie outside the contour of $z_b$. However, the arrows determine all possible poles of the rational function \eqref{eqn:f}. Thus, if the aforementioned rational function has a pole involving two variables $z_a$ and $z_b$, then there is an arrow from $v_b$ to $v_a$ and so the contours of these variables must be in some particular order. This implies the claim: even though we are allowed to relabel the variables, the fact that the property ``$\exists$ arrow from $v_b$ to $v_a$ only if $a<b$" holds both before and after the relabeling, means that we are only required to change the order of contours of variables $z_a$ and $z_b$ which are not related by any possible poles of the rational function \eqref{eqn:f}.   
\end{proof}

\medskip
\noindent 
Let us now return to the proof of Lemma \ref{lem:pairing sl}. Formula (3.31) of \cite{N Reduced} shows that the pairing \eqref{eqn:pairing explicit} is given by the following formula
$$
  \Big \langle e_{i_1,d_1}\dots e_{i_k,d_k}, F \Big \rangle = 
  \sum^{\text{labeled acceptable}}_{\text{shrubberies }\mathscr{S} = S_1 \sqcup \dots \sqcup S_t} \int_{|z_{r_1}| = \rho_{i_{r_1}}} \dots \int_{|z_{r_t}|=\rho_{i_{r_t}}} Dz_{r_1} \dots Dz_{r_t}
$$
\begin{equation}
\label{eqn:residue formula} 
  \underset{S_1}{\text{Res}} \dots \underset{S_t}{\text{Res}} \ 
  \frac{F(z_{i_a\bullet_a})_{a \in \{1,\dots,k\}} z_1^{d_1}\dots z_k^{d_k}}
       {\prod_{1 \leq a < b \leq k} \zeta^{\wA_{\Gamma}}_{i_bi_a} \left(\frac {z_b}{z_a} \right)} 
\end{equation}
However, since both $F$ and the contours of integration are color-symmetric, the right-hand side of \eqref{eqn:residue formula} is a linear functional of
$$
  \text{Sym} \left[ \frac{z_1^{d_1}\dots z_k^{d_k}}{\prod_{1 \leq a < b \leq k} \zeta^{\wA_{\Gamma}}_{i_bi_a} \left(\frac {z_b}{z_a} \right)} \right] = 
  \frac{\twUpsilon^+_{\wA_{\Gamma}} \left(e_{i_1,d_1}\dots e_{i_k,d_k} \right)}{\prod_{1 \leq a \neq b \leq k} \zeta^{\wA_{\Gamma}}_{i_bi_a} \left(\frac {z_b}{z_a} \right)}
$$
which concludes the proof of Lemma \ref{lem:pairing sl}.

\medskip
\noindent
For future reference, let us mention that the key formula \eqref{eqn:residue formula} is proved using the following recursive procedure for $\ell$ going down from $k$ to $1$: we initialize $\mathscr{S}_k$ to be the shrubbery consisting of a single shrub with a single vertex $v_k$ of color $i_k$. At some point during the recursion, assume that the variables $\{z_b\}_{\ell < b \leq k}$ have all been specialized to $t_{S_c}(v_b)z_{r_c}$ for various constituent shrubs $\{S_c\}_{c \in \{1,\dots,t\}}$ of a shrubbery $\mathscr{S}_{\ell+1}$. As per formula \eqref{eqn:pairing explicit}, we move the contour of integration of the variable $z_{\ell}$ from infinity to the circle of radius $\rho_{i_{\ell}}$. If the contour makes it all the way to this circle, then we declare the variable $z_{\ell}$ to be specialized to the root of a new shrub:
\begin{equation}
\label{eqn:option 1}
  \mathscr{S}_{\ell} = \mathscr{S}_{\ell+1} \sqcup \Big(\text{one-vertex shrub } \{v_{\ell}\} \text{ of color }i_{\ell} \Big)
\end{equation}
Otherwise (by the residue theorem) the variable $z_{\ell}$ would be specialized to one of the apparent poles of the rational function \eqref{eqn:f}, which are of the form
\begin{equation}
\label{eqn:pole}
  z_{\ell} = z_b t_\alpha 
\end{equation}
for various arrows $\alpha \colon i_b \rightarrow i_{\ell}$, and various $b \in \{\ell+1,\dots,k\}$. If we let $v_{\ell} \in \tQ$ denote the vertex whose coordinate equals that of $v_b$ times $t_{\alpha}$ (with $b$ and $\alpha$ as in \eqref{eqn:pole}), then the denominator of \eqref{eqn:f} produces a pole of order 
$$
  \# = \left| \Big\{ \text{arrows from }v_{\ell+1},\dots,v_k \in \mathscr{S}_{\ell+1} \text{ to }v_{\ell} \Big\} \right| 
$$
However, as proved in \cite[Proposition 3.21]{N Reduced}, the combinatorics of shrubs ensures that the numerator of \eqref{eqn:f} vanishes (because $F$ satisfies the wheel conditions \eqref{eqn:wheel affine even} and \eqref{eqn:wheel affine odd}, i.e.\ vanishes when its variables are specialized according to any triangular or square face) up to order 
$$
  \begin{cases} 
    = \# - 1 &\text{if }v_{\ell} \text{ is addable to some constituent shrub }S_c \subset \mathscr{S}_{\ell+1} \\
    \geq \# &\text{otherwise} 
  \end{cases}
$$
To see that the first option above holds in the example of Figure \ref{fig:addable}, assume that $v = v_\ell$ is the vertex marked by the black disk and $S_c$ is the green shrub. We have $\#=3$ in this case, but the fact that $F$ satisfies the length $3$ wheel conditions \eqref{eqn:wheel affine even} implies that the numerator of \eqref{eqn:f} vanishes to order $2=3-1$ (this is because of the two shaded red faces in Figure \ref{fig:addable}, which impose two independent vanishing conditions as they only intersect at the black vertex; one could also use the two shaded blue faces instead to reach the same conclusion).

\medskip 
\noindent 
We thus conclude that the apparent pole \eqref{eqn:pole} could be an actual pole (and thus contributes non-trivially to the right-hand side of \eqref{eqn:residue formula}) only if the vertex $v_{\ell}$ is addable to $S_c \subset \mathscr{S}_{\ell+1}$, in which case the pole in question is simple. We let
\begin{equation}
\label{eqn:option 2}
  \mathscr{S}_{\ell} = \Big(\mathscr{S}_{\ell+1} \text{ with }S_c \text{ replaced by }S_c + v_{\ell} \Big)
\end{equation}
and continue the recursive procedure until we get down to $\ell=1$. The shrubbery $\mathscr{S} = \mathscr{S}_1$ obtained via \eqref{eqn:option 1} or \eqref{eqn:option 2} at every step of this recursive algorithm is the one appearing in the right-hand side of \eqref{eqn:residue formula}. 
\end{proof}


\medskip 

\subsection{Asps (\underline{a}cceptable \underline{s}pecialization \underline{p}attern\underline{s})}
\label{sub:specialization patterns}

For further	use, let us note that in the proof of Lemma \ref{lem:pairing sl}, the full combinatorics of shrubs (i.e.\ the classification of Proposition \ref{prop:shrubs sl}) was not needed. Instead, the only essential feature we used is that a (labeled) acceptable shrub encodes the information necessary to specialize Laurent polynomials $F$ at various values
\begin{equation*}
  \big\{z_\ell \mapsto x_{\ell} \in \BC^\times \big\}_{\ell \in \{1,\dots,k\}}
\end{equation*}
(where each $z_{\ell}$ is a variable of color $i_\ell$ of $F$), which is defined recursively in $\ell$, ranging from $k$ down to $1$. These specializations have the following features:

\medskip 

\begin{itemize}[leftmargin=0.7cm]

\item 
if there are $\#$ factors of the form $z_{\ell} - z_o \frac {x_{\ell}}{x_o}$ in the numerator of the expression
$$
  \prod_{o = \ell+1}^k \zeta^{\wA_\Gamma}_{i_{o}i_{\ell}} \left( \frac {z_o}{z_{\ell}} \right)
$$
then any $F \in \CS^-_{\wA_{\Gamma}}$ has the property that the specialization 
$$
  F|_{z_o \mapsto x_o, \ell < o \leq k}
$$
vanishes to order $\geq \#-1$ at $z_{\ell} = x_{\ell}$; 

\medskip

\item 
for a fixed choice of positive real numbers $\{\rho_i\}_{i \in \wI}$, we require that
\begin{equation}
\label{eq:acceptability}
  |x_k| =  \rho_{i_k}, \qquad |x_{\ell}| > \rho_{i_{\ell}}  \quad \forall\, 1\leq \ell < k 
\end{equation}
in order to mimic the acceptability condition \eqref{eqn:acceptable}. 

\end{itemize}

\medskip 
\noindent 
In types other than $A$ and $\wA$ (corresponding to an arbitrary quiver with vertex set $I$ and parameters as in Subsection \ref{sub:quiver}), we formalize the above discussion as follows. The main idea is that the two bullets above allow one to recursively define the map
$$
  F(z_1,\dots,z_k) \leadsto \underset{S}{\text{Res}} \left( \frac {F(z_1,\dots,z_k)}{\prod_{1\leq a < b \leq k} \zeta^{\wA_\Gamma}_{i_bi_a}\left(\frac {z_b}{z_a}\right)} \right) \in \BK(z_k)
$$
for any shrub $S$ (recall \eqref{eqn:shrub} for an example). In the case of a shrubbery consisting of $t$ shrubs, a similar procedure produces a rational function in $t$ variables. The analogous construction for a general quiver is given by the following.

\medskip

\begin{definition}
\label{def:specialization pattern}
A specialization pattern is a collection of so-called \emph{vertices}
$$
  v_\ell = ( i_{\ell}, x_{\ell} ) \in I \times \BK^\times 
$$
for $\ell \in \{1,\dots,k\}$, which satisfies the following property for all $\ell$: if there are $\#$ arrows $\{i_o \rightarrow i_{\ell}\}_{\ell < o \leq k}$ in the quiver with parameter equal to $\frac {x_{\ell}}{x_{o}}$, then any Laurent polynomial $F$ in the corresponding shuffle algebra has the property that 
\begin{equation}
\label{eqn:spec}
  F|_{z_o \mapsto x_o, \ell < o \leq k}
\end{equation}
vanishes to order $\geq \#-1$ at $z_{\ell} = x_{\ell}$, with the variables labeled as in \eqref{eqn:f}. The vertex $v_k$ is called the \emph{root} of the specialization pattern. 
\end{definition}

\medskip

\begin{definition}
\label{def:acceptable specialization pattern} 
Fix a choice of positive real numbers $\{\rho_i\}_{i \in I}$. A specialization pattern is called \emph{acceptable} (abbreviated \emph{asp}) if \eqref{eq:acceptability} holds.
\end{definition}

\medskip 
\noindent 
With this in mind, we observe that the key argument that allowed the proof of Lemma \ref{lem:pairing sl} is the following: the set of shrubs is a collection of asps which is \emph{closed under extensions}. The first requirement for a collection $\mathcal{C}$ of asps to be closed under extensions is that it should contain all one-vertex asps. The second requirement for $\mathcal{C}$ to be closed under extensions is the following: suppose we have chosen vertices $v_{\ell+1},\dots,v_k$ that determine an asp in $\mathcal{C}$, and pick an arbitrary vertex $v_{\ell} = (i_{\ell},x_{\ell})$ that satisfies the acceptability condition \eqref{eq:acceptability}. If $\#$ is the number of arrows $\{i_o \rightarrow i_{\ell}\}_{\ell < o \leq k}$ in the quiver with parameter equal to $\frac {x_{\ell}}{x_{o}}$, then one of the following holds:

\medskip 

\begin{enumerate}[leftmargin=1cm]

\item 
the vertices $\{v_{\ell},\dots,v_k\}$ also form an asp in $\mathcal{C}$;

\medskip 

\item  
any $F$ in the corresponding shuffle algebra has the property that the specialization \eqref{eqn:spec} vanishes to order $\geq \#$ at $z_{\ell} = x_{\ell}$.

\end{enumerate}

\medskip 
\noindent 
When using the terminology of asps in Section \ref{sec:osp}, we will often use the phrase ``go" or ``travel" from an asp $\{v_{\ell+1},\dots,v_k\}$ to a vertex $v_{\ell}$ in the situation of option (1) above, in order to suggestively present the concept of extending asps. Alternatively, in option (2), we would say that various wheel conditions on $F$ ``disallow'' or ``prohibit'' us from going/traveling from an asp $\{v_{\ell+1},\dots,v_k\}$ to the vertex $v_{\ell}$.

\medskip

\begin{remark} 
\label{rem:asp}
Thus, we may summarize the proof of Lemma \ref{lem:pairing sl} as saying that ``asps underlie shrubs", i.e.\ shrubs are merely a combinatorial way to encode asps in super types $A$ and $\wA$, and to prove that they are closed under extensions. Thus, the need to use shrubs is simply a consequence of the complicated combinatorics of the periodic quiver $\tQ$ of Figure \ref{fig:grid} for general parity sequences. In the non-super types $A$ and $\wA$ studied in \cite{E, FO, N Toroidal, Ts}, the parity sequence was $(1,\dots,1)$ and the periodic quiver was the one of Figure \ref{fig:C3 cover}. In this particular case, by choosing the absolute values of $q$ and $d$ appropriately, one could ensure that the only asps are given by straight diagonal lines and the full combinatorics of shrubs is not necessary.
\end{remark}


\bigskip 
    
\section{Quantum affine superalgebras of types $B, C, D$}
\label{sec:osp}
	
\medskip 

In this Section, we will give a case-by-case rundown of types $B,C,D$, and give the shuffle realization for each of them. Note that our organization into the three types $B,C,D$ follows~\cite[\S 3.1]{Y1} (it is suggested by the analogy with the shape of non-super Dynkin diagrams) and is different from the standard reference~\cite{K}. Let $\pI = \{1,\dots,N\} = \sI\cup\{N\}$ be the indexing set of simple roots in types $B,C,D$.


\subsection{Type $D$}
\label{sub:d}
	
Define the pre-quantum affine superalgebra as (cf.~\eqref{eqn:quad intro})
\begin{equation*}
  \tudG = \BC \Big \langle e_{i,d}, f_{i,d}, \ph_{i,d'}^\pm \Big \rangle_{i \in \pI, d \in \BZ, d' \geq 0} \Big/ 
  \Big (  \text{relations \eqref{eqn:rel quantum affine 1}-\eqref{eqn:rel quantum affine 6}}   \Big )
\end{equation*}
where the relations above are defined with respect to the zeta functions $\zeta^{D_{\Gamma}}_{ij}(x)$ that are encoded by the decorated Dynkin diagram $\Gamma$. We further define the quantum affine superalgebra $\udG$ by imposing higher order relations, depending on $\Gamma$.

\medskip
\noindent
\underline{Case 1}
\begin{equation*}
\begin{tikzcd}
  &&&& {N-1} \\
  \cdots && {N-2} \\
  &&&& N
  \arrow[no head, from=2-1, to=2-3]
  \arrow["q^{\pm 1}", no head, from=2-3, to=1-5]
  \arrow["q^{\mp 2}", swap, no head, from=1-5, to=1-5, looseness=4]
  \arrow["q^{\pm 1}"', no head, from=2-3, to=3-5]
  \arrow["q^{\mp 2}", swap, no head, from=3-5, to=3-5, looseness=4]
  \arrow[dashed, no head, from=2-3, to=2-3, looseness=4]
\end{tikzcd}
\end{equation*}
The body $\Gb$ formed by the vertices $\{1,\dots,N-2\}$ is a type $A$ decorated Dynkin diagram, and $|N-1|=0=|N|$. The zeta functions $\zeta^{D_{\Gamma}}_{ij}(x)$ are defined by
\begin{equation}
\label{eq:zeta-D-case1}
\begin{split}
  & \zeta^{D_{\Gamma}}_{ij}(x)=\zeta^{A_{\Gb}}_{ij}(x) \qquad \forall\, i,j\in \Gb  \\
  & \zeta^{D_{\Gamma}}_{N-1,N-1}(x) = \zeta^{D_{\Gamma}}_{NN}(x)=\frac{1-xq^{\mp 2}}{1-x} \\
  & \zeta^{D_{\Gamma}}_{N-2,N-1}(x) = \zeta^{D_{\Gamma}}_{N-2,N}(x) = \frac{1-xq^{\pm 1}}{x} \\
  & \zeta^{D_{\Gamma}}_{N-1,N-2}(x) = \zeta^{D_{\Gamma}}_{N,N-2}(x) = -(1-xq^{\pm 1}) \\
  & \zeta^{D_{\Gamma}}_{N-1,N}(x) = \zeta^{D_{\Gamma}}_{N,N-1}(x) = 1 \\
  & \zeta^{D_{\Gamma}}_{N-1,i}(x)=\zeta^{D_{\Gamma}}_{Ni}(x)=\zeta^{D_{\Gamma}}_{i,N-1}(x)=\zeta^{D_{\Gamma}}_{iN}(x)=1 
    \qquad \forall\, i \in \{1,\dots,N-3\}
\end{split}
\end{equation}
The corresponding quantum affine superalgebra of type $D_\Gamma$ is 
\begin{equation*}
\udG = \tudG \Big / \Big (\text{relations \eqref{eqn:rel quantum affine 8}-\eqref{eqn:rel quantum affine 9}} \text{ and their }  f\text{-versions} \Big )
\end{equation*}
where we consider \eqref{eqn:rel quantum affine 8} for $i,i\pm 1 \leq N-1$ as well as $(i,i\pm 1)$ replaced by $(N-2,N)$ or $(N,N-2)$, consider~\eqref{eqn:rel quantum affine 9} for $i-1,i,i+1\leq N-1$ as well as $(i-1,i,i+1)$ replaced by $(N-3,N-2,N)$. Note, however, that we do not impose~\eqref{eqn:rel quantum affine 9} for $\{i-1,i,i+1\}=\{N-2,N-1,N\}$. 
For the neck parameter $q^{\pm 1} = q^{-1}$, this recovers the algebra denoted by $U^D_q(\widehat{\fg}_{\bs}) |_{C=1}$ in~\cite[Definition 5.1]{BFK} corresponding to $\fosp(2m|2n)$ with $N=m+n$ and the parity sequence $\bs$ satisfying $s_{N-1}=1, s_{N}=1$ (which is conjectured to be isomorphic to $U_q^{DJ}(\fg_{\widehat{X}_\bs})|_{c=1}$, see \cite[Conjecture 5.1]{BFK}).

\medskip
\noindent
\medskip
\noindent
\underline{Case 2}
\begin{equation*}
\begin{tikzcd}
  &&&& {N-1} \\
  \cdots && {N-2} \\
  &&&& N
  \arrow[no head, from=2-1, to=2-3]
  \arrow["q^{\pm 1}", no head, from=2-3, to=1-5]
  \arrow["q^{\pm 1}"', no head, from=2-3, to=3-5]
  \arrow["{q^{\mp 2}}", no head, from=1-5, to=3-5]
  \arrow[dashed, no head, from=2-3, to=2-3, looseness=4]
\end{tikzcd}
\end{equation*}

\noindent
The body $\Gb$ formed by the vertices $\{1,\dots,N-2\}$ is a type $A$ decorated Dynkin diagram, and $|N-1|=1=|N|$. The zeta functions $\zeta^{D_{\Gamma}}_{ij}(x)$ are defined by
\begin{equation}
\label{eq:zeta-D-case2}
\begin{split}
  & \zeta^{D_{\Gamma}}_{ij}(x)=\zeta^{A_{\Gb}}_{ij}(x) \qquad \forall\, i,j\in \Gb \\
  & \zeta^{D_{\Gamma}}_{N-1,N-1}(x) = \zeta^{D_{\Gamma}}_{NN}(x)=\frac{x^{\frac{1}{2}}}{1-x} \\
  & \zeta^{D_{\Gamma}}_{N-2,N-1}(x) = \zeta^{D_{\Gamma}}_{N-2,N}(x) = (-1)^{|N-2|} \frac{1-xq^{\pm 1}}{x} \\
  & \zeta^{D_{\Gamma}}_{N-1,N-2}(x) = \zeta^{D_{\Gamma}}_{N,N-2}(x) = - (1-xq^{\pm 1}) \\ 
  & \zeta^{D_{\Gamma}}_{N-1,N}(x)=-\frac{1-xq^{\mp 2}}{x}, \quad \zeta^{D_{\Gamma}}_{N,N-1}(x)=-(1-xq^{\mp 2}) \\
  & \zeta^{D_{\Gamma}}_{N-1,i}(x)=\zeta^{D_{\Gamma}}_{Ni}(x)=1, \quad \zeta^{D_{\Gamma}}_{i,N-1}(x)=\zeta^{D_{\Gamma}}_{iN}(x)=(-1)^{|i|} \qquad \forall\, i\in \{1,\dots,N-3\} 
\end{split}
\end{equation}
The corresponding quantum affine superalgebra of type $D$ is 
\begin{equation*}
\udG = \tudG \Big / \Big (\text{relations \eqref{eqn:rel quantum affine 8}-\eqref{eqn:rel quantum affine 9}, \eqref{eqn:rel cd 1}} 
  \text{ and their }  f\text{-versions} \Big )
\end{equation*}
where \eqref{eqn:rel quantum affine 8}-\eqref{eqn:rel quantum affine 9} are understood as in Case 1 above, and 
\begin{equation}
\label{eqn:rel cd 1}
  [ [ e_{N-2}(w) , e_{N-1}(y) ]_{q^{\mp 1}},  e_N(z) ]_{q^{\pm 1}} - [ [ e_{N-2}(w) , e_N(z) ]_{q^{\mp 1}} ,  e_{N-1}(y) ]_{q^{\pm 1}} = 0
\end{equation}
For the neck parameter $q^{\pm 1} = q$, this recovers the algebra denoted by $U^D_q(\widehat{\fg}_{\bs}) |_{C=1}$ in~\cite[Definition 5.1]{BFK} corresponding to $\fosp(2m|2n)$ with $N=m+n$ and the parity sequence $\bs$ satisfying $s_{N-1}=-1, s_{N}=1$ (which is conjectured to be isomorphic to $U_q^{DJ}(\fg_{\widehat{X}_\bs})|_{c=1}$, see \cite[Conjecture 5.1]{BFK}). In particular, relation~\eqref{eqn:rel cd 1} recovers~\cite[(5.15)]{BFK} (though \loccit lists six relations, all of them are tautologically equivalent).~\footnote{In fact, the $q^{\pm 1}\leadsto q^{\mp 1}$ version of this relation can alternatively be used, cf.~Remark~\ref{rem:equivalent-C-six-relations}.}

\medskip
\noindent 
The big shuffle algebra is (with $\twist$ defined similarly to~\eqref{eq:root-twist})
$$
  \CV^+_{D_{\Gamma}} = \bigoplus_{\bn = (n_i)_{i \in \pI} \in \pnn} \BC[z_{i1}^{\pm 1},\dots,z_{in_i}^{\pm 1}]_{i \in \pI}^{\sym} \cdot \twist
$$
made into an associative algebra by formula \eqref{eqn:shuffle product} with $\zeta^{D_{\Gamma}}$ instead of $\zeta$. Let
$$
  \CS^+_{D_{\Gamma}} \subset \CV^+_{D_{\Gamma}}
$$
be the set of color-symmetric Laurent polynomials (times $\twist$) which vanish at the length $3$ and $4$ wheels \eqref{eqn:wheel finite even}-\eqref{eqn:wheel finite odd} for all applicable colors as explained above, together with the additional length $3$ wheels \eqref{eqn:wheel d 1} in Case 2. According to Lemma~\ref{lem:easy}, the aforementioned wheel conditions are preserved by the shuffle product and thus $\CS^+_{D_{\Gamma}}$ is an algebra. We will now prove Theorem \ref{thm:d intro}, which states that the homomorphism 
$$
  \tUpsilon^+_{D_{\Gamma}}\colon \tudGp \to \CV^+_{D_{\Gamma}}, \qquad  e_{i,d}\mapsto z^d_{i1} ,\quad \forall\, i\in \pI, d\in \BZ
$$ 
of~\eqref{eqn:tilde upsilon} induces an algebra isomorphism 
$$
  \Upsilon^+_{D_{\Gamma}}\colon \udGp \iso \CS^+_{D_{\Gamma}} 
$$
We will denote by $\CS^-_{D_{\Gamma}} \subset \CV^-_{D_{\Gamma}}$ the same sets as $\CS^+_{D_{\Gamma}} \subset \CV^+_{D_{\Gamma}}$, but endowed with the opposite algebra structure.

\medskip

\begin{proof} \emph{of Theorem \ref{thm:d intro}:} 
Since inverting all edge parameters in a decorated Dynkin diagram has the effect of swapping $q \leftrightarrow q^{-1}$ in all associated quantum algebras, we may assume without loss of generality that the neck parameter is $q$. We will follow the outline of Subsection \ref{sub:plan}, in which one needs to prove the following key facts: 
\begin{equation}
\label{eqn:fact 1 cd 1}
  \Big \langle \Big( \text{LHS of \eqref{eqn:rel quantum affine 8}-\eqref{eqn:rel quantum affine 9} and \eqref{eqn:rel cd 1}} \Big), F \Big \rangle = 0 \quad \text{for} \quad F \in \CV^-_{D_{\Gamma}} 
  \quad \Rightarrow \quad F \in \CS^-_{D_{\Gamma}} 
\end{equation} 
(we note that relation \eqref{eqn:rel cd 1} applies only to Case 2) and 
\begin{equation}
\label{eqn:fact 2 d}
  \Big \langle e_{i_1,d_1}\dots e_{i_k,d_k}, F \Big \rangle = 
  \left[\text{a linear functional of }\tUpsilon^+_{D_{\Gamma}} \left(e_{i_1,d_1}\dots e_{i_k,d_k} \right) \right] 
\end{equation}
for any $F\in \CS^-_{D_{\Gamma}}$.

\medskip
\noindent
To establish \eqref{eqn:fact 1 cd 1}, we follow the process laid out in the proof of~\eqref{eqn:key pairing finite sl}. By analogy with \eqref{eqn:x}, we consider the corresponding power series 
$$
  X(w,y,z) = \sum_{a,b,c\in \BZ} X^{a,b,c} w^{-a} y^{-b} z^{-c} = \Big( \text{LHS of \eqref{eqn:rel cd 1}} \Big) 
$$
in Case 2, with all $X^{a,b,c}\in \tudGp$ of degree $\bn=\bsi^{N-2}+\bsi^{N-1}+\bsi^N$. A direct calculation (either by hand or by computer) easily verifies the following analogue of~\eqref{eqn:rational function}:
$$
  \tUpsilon^+_{D_{\Gamma}}(X(w,y,z)) = 0
$$
Moreover, we have the following analogue of \eqref{eqn:dim 1} and~\eqref{eqn:dim 2}:
\begin{equation*}
  \dim_{\BC} \left(\CV_{D_{\Gamma};-\bn,-d}/\CS_{D_{\Gamma};-\bn,-d} \right) = 2
\end{equation*}
because there are two different length $3$ wheels in~\eqref{eqn:wheel d 1}. However, the following analogues of~\eqref{eqn:not zero 1} and~\eqref{eqn:not zero 2} can be computed in a straightforward fashion (either by computer or by hand, see Appendices~\ref{sub:non-zero type d case 2} and~\ref{app:D-fork-odd} for details):
\begin{equation}
\label{eqn:not zero d}
\begin{split}  
  \big \langle X^{a,b,c}, z^{-a}_{N-2,1} z^{-b-1}_{N-1,1} z^{-c+1}_{N1} \big \rangle &= 1 \\
  \big \langle X^{a,b,c}, z^{-a}_{N-2,1} z^{-b}_{N-1,1} z^{-c}_{N1} \big \rangle &= -1  \\ 
  \big \langle X^{a,b,c}, z^{-a}_{N-2,1} z^{-b+1}_{N-1,1} z^{-c-1}_{N1} \big \rangle &= -q^2-1-q^{-2} 
\end{split}
\end{equation}
Because of these formulas, we conclude that for any fixed $a,b,c \in \BZ$, the two linear functionals $\langle X^{a,b,c},-\rangle$ and $\langle X^{a,b+1,c-1},-\rangle$ take linearly independent values when evaluated on the Laurent polynomials $z^{-a}_{N-2,1} z^{-b}_{N-1,1} z^{-c}_{N1}$ and $z^{-a}_{N-2,1} z^{-b-1}_{N-1,1} z^{-c+1}_{N1}$. Thus, just like in the proof of~\eqref{eqn:key pairing finite sl}, we conclude that 
$$
  \big \langle X(w,y,z), F \big \rangle = 0 \text{ for } F \in \CV_{D_{\Gamma};-\bn} \quad \Rightarrow \quad F \in \CS_{D_{\Gamma};-\bn}
$$
This allows to complete the proof of~\eqref{eqn:fact 1 cd 1} just like in the proof of~\eqref{eqn:key pairing finite sl}.

\medskip 
\noindent 
In order to prove \eqref{eqn:fact 2 d}, we will adapt the proof of Lemma \ref{lem:pairing sl}. As explained in Subsection \ref{sub:specialization patterns}, we only need to describe the asps of Definitions \ref{def:specialization pattern}-\ref{def:acceptable specialization pattern} in the case at hand, and to show that they are closed under extensions. To this end, we define acceptability with respect to positive real numbers \eqref{eqn:rho} such that 
\begin{equation}
\label{eqn:greater d}
  \rho_1 , \dots, \rho_{N-2} \ll \rho_{N-1} = \rho_N, 
\end{equation}
meaning that while the relative order of the numbers $\rho_1,\dots,\rho_{N-2}$ is irrelevant, $\rho_{N-1}=\rho_N$ should be much larger than all of them \footnote{Specifically, we need $|\rho_N| = |\rho_{N-1}| \gg |\rho_1| |q^{\pm k\varpi}|,\dots,|\rho_{N-2}| |q^{\pm k\varpi}|$, where $\varpi$ is a suitably large constant and $k$ is the number which appears in \eqref{eqn:fact 2 d}.}. Because of the assumption \eqref{eqn:greater d}, any asp rooted at color $\leq N-2$ will actually be constrained to vertices of colors $1,\dots,N-2$, and is thus nothing more than a finite type shrub. As for asps rooted at color $N-1$ (the case of color $N$ is analogous), their full classification is depicted on the left side of Figures \ref{fig:asps D 1}-\ref{fig:asps D 2} and we will now explain how to read the pictures. The parameter decorating any arrow is completely determined by its endpoints according to the rule
$$
  (\text{lattice point }(x,y)) \xrightarrow{q^{x'+y'-x-y}} (\text{lattice point }(x',y'))
$$
(the lattice is normalized so that the vertical arrows  in the pictures on the left carry the parameter $q$). The root of the asp is depicted by the blue disk, and the arrows indicate the behavior of the asps in colors $N-2$ and $N-1,N$. The black arrows carry the same information as a finite type shrub $S$ rooted at the black disk (indicated on the right side of Figures \ref{fig:asps D 1}-\ref{fig:asps D 2}), corresponding to a type $A$ decorated Dynkin diagram with vertices
\begin{equation}
\label{eqn:small d}
  1,\dots,N-2, *
\end{equation}
where $*$ is an odd vertex. The blue arrows in Figures \ref{fig:asps D 1}-\ref{fig:asps D 2} just show how one can extend the finite type shrub $S$ in order to obtain a type $D$ asp; the only condition is that there be at least as many southwest-pointing blue arrows as southwest-pointing black curved arrows. Note that whenever two solid arrows come into a vertex $v$, there is always a dotted arrow coming out of that vertex: this indicates that the apparent double pole corresponding to $v$ in the residue \eqref{eqn:residue} is actually reduced to a simple pole due to various wheel conditions, much like the example of Figure \ref{fig:shrub} in type $A$. Finally, the fact that the asps in Figures \ref{fig:asps D 1}-\ref{fig:asps D 2} are closed under extensions is clear from the positions of the red $\textcolor{red}{\boldsymbol{\times}}$, which indicate vertices where we cannot extend an asp due to the following wheel conditions (by analogy with Remark~\ref{rem:zero}, one cannot extend an asp to an already existing vertex):

\medskip 

\begin{itemize}[leftmargin=0.7cm]

\item 
the $\textcolor{red}{\boldsymbol{\times}}$ southwest of the root are forbidden due to length $3$ wheels they would form with a dotted arrow connected vertex of color $N-2$ and its northeast neighbor of color $N-2$ (or length $4$ wheel if $N-2$ is odd, where we also adjoin one vertex of color $N-3$ that is located on the curved arrow);

\medskip 

\item 
the $\textcolor{red}{\boldsymbol{\times}}$ northeast of the root are forbidden due to length $3$ wheels they would form with a dotted arrow connected vertex of color $N-2$ and the vertex underneath the latter.

\end{itemize}

\begin{figure}[h]
  \includegraphics[scale=0.75]{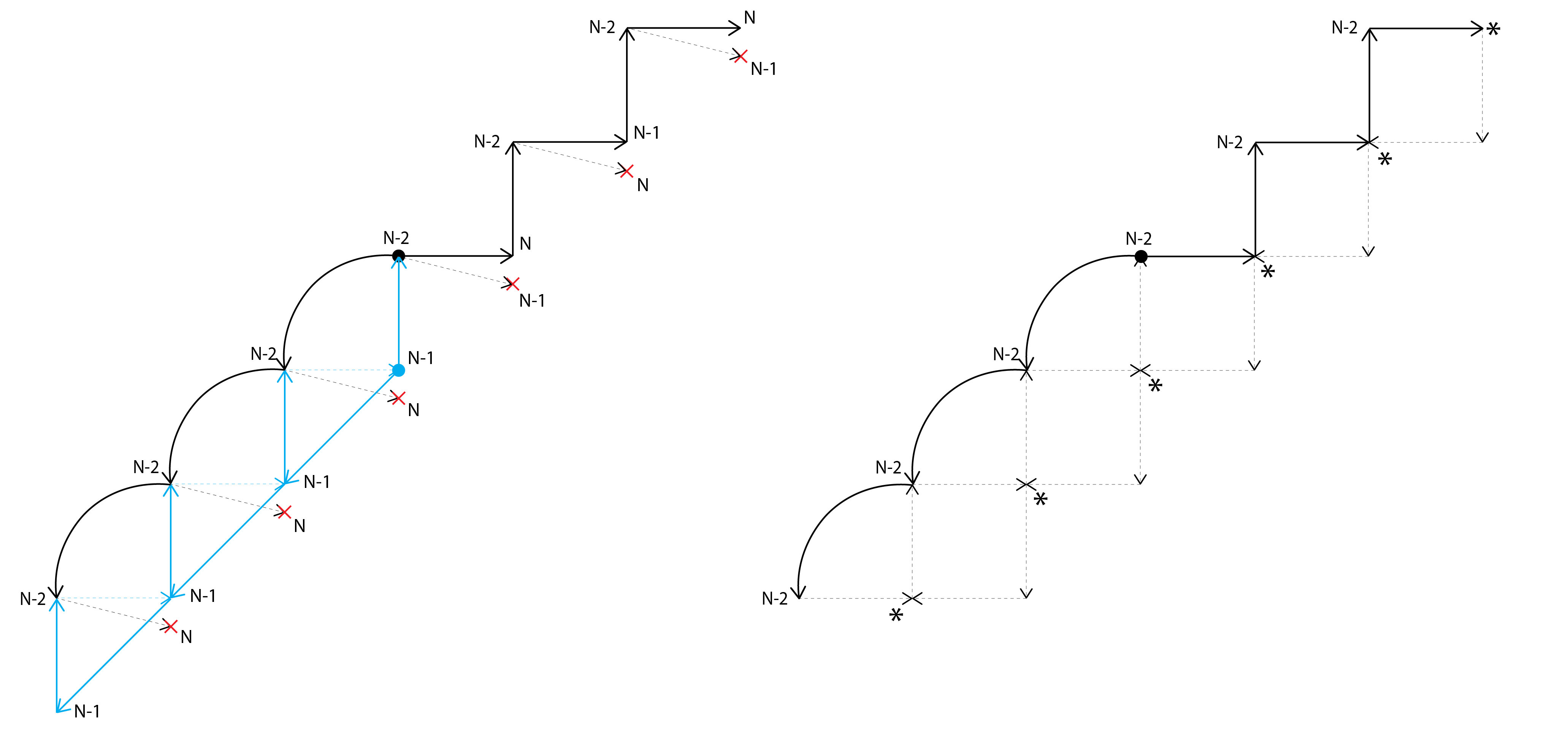} 
  \caption{Type $D$ Case 1: asps rooted at the blue disk on the left, the corresponding finite type shrub on the vertex set \eqref{eqn:small d} on the right. The curved arrows indicate a straight southwest pointing arrow if $|N-2|=0$ or a left-then-down zig-zag if $|N-2|=1$.}
  \label{fig:asps D 1}
\end{figure}

\begin{figure}[h]
  \includegraphics[scale=0.75]{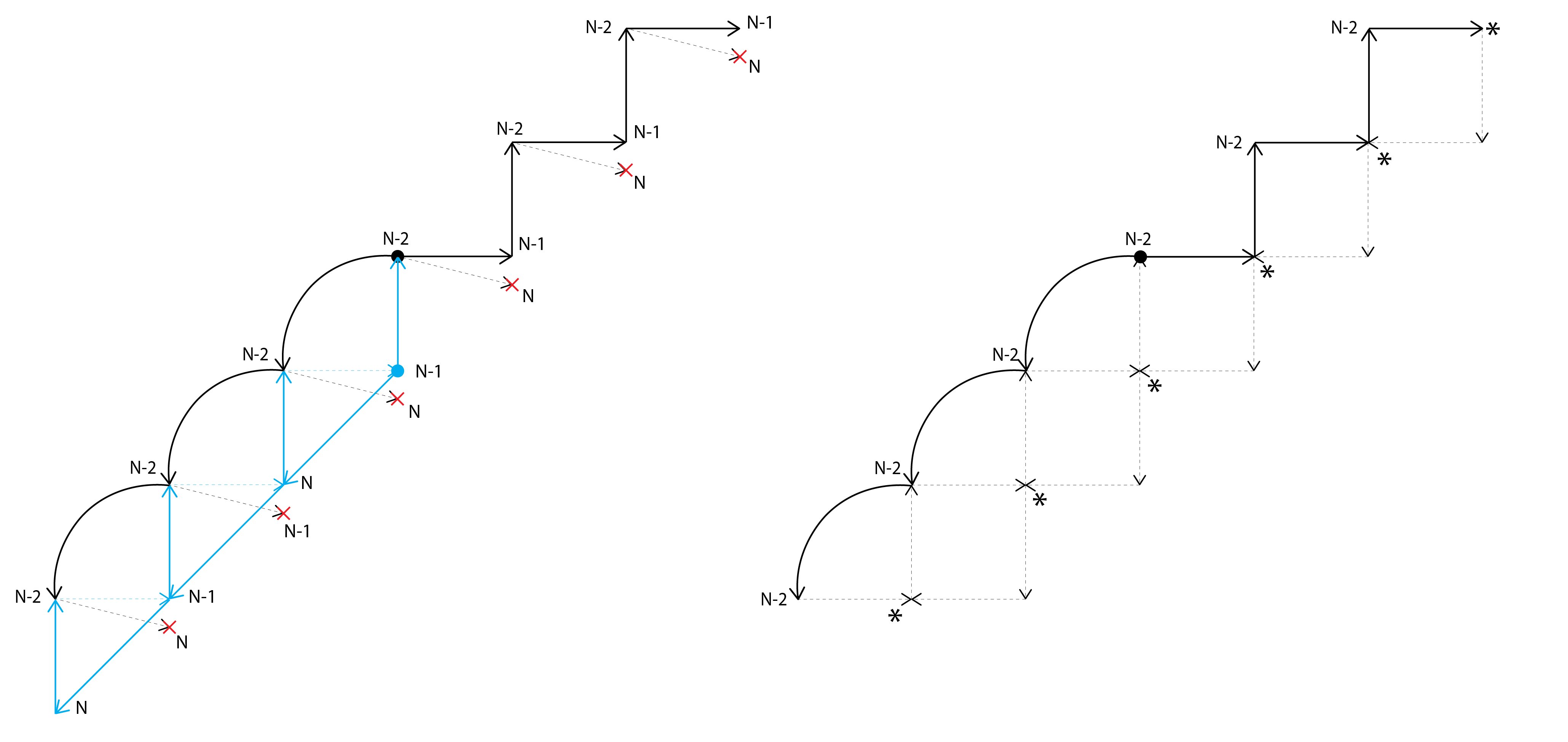} 
  \caption{Type $D$ Case 2: asps rooted at the blue disk on the left, the corresponding finite type shrub on the vertex set \eqref{eqn:small d} on the right. The curved arrows indicate a straight southwest pointing arrow if $|N-2|=0$ or a left-then-down zig-zag if $|N-2|=1$.}
  \label{fig:asps D 2}
\end{figure}

\medskip 
\noindent 
Note that we never invoked the acceptability condition \eqref{eqn:greater d}: to be precise, the notion of ``asps" only involves those pictures among Figures \ref{fig:asps D 1}-\ref{fig:asps D 2} such that for every vertex $v$ of color $N-1$ or $N$ other than the blue disk, the product of arrow parameters from the blue disk to $v$ has absolute value strictly greater than $1$. By cleverly choosing the absolute value of $q$ itself, one could drastically reduce the number of asps. For instance, if we choose $|q|>1$ (respectively $|q|<1$), then a type $D$ asp cannot contain any color $N-1,N$ vertices southwest (respectively northeast) of the blue disk. The reason why we do not make this simplifying assumption at the present moment is that we will need to avoid it for the toroidal cases that will be treated in Section \ref{sec:osp-toroidal}. 
\end{proof}


\medskip 
	
\subsection{Type $B$}
\label{sub:b}

Define the pre-quantum affine superalgebra as (cf.~\eqref{eqn:quad intro})
\begin{equation*}
  \tubG = \BC \Big \langle e_{i,d}, f_{i,d}, \ph_{i,d'}^\pm \Big \rangle_{i \in \pI, d \in \BZ, d' \geq 0} \Big/ 
  \Big (   \text{relations \eqref{eqn:rel quantum affine 1}-\eqref{eqn:rel quantum affine 6}}   \Big )
\end{equation*}
where the relations above are defined with respect to the zeta functions $\zeta^{B_{\Gamma}}_{ij}(x)$ that are encoded by the decorated Dynkin diagram $\Gamma$. We further define the quantum affine superalgebra $\ubG$ by imposing higher order relations, depending on $\Gamma$.

\medskip 
\noindent 
\underline{Case 1}
\begin{equation*}
\begin{tikzcd}
  \cdots && {N-2} && {N-1} && N
  \arrow["q^{\pm 1}"', no head, from=1-7, to=1-5]
  \arrow[no head, from=1-5, to=1-3]
  \arrow[no head, from=1-3, to=1-1]
  \arrow["q^{\mp 1}"', no head, from=1-7, to=1-7, looseness=4]
  \arrow[dashed, no head, from=1-5, to=1-5, looseness=4]
  \arrow[dashed, no head, from=1-3, to=1-3, looseness=4]
\end{tikzcd}
\end{equation*}
The body $\Gb$ formed by the vertices $\{1,\dots,N-1\}$ is a type $A$ decorated Dynkin diagram, and $|N|=0$. The zeta functions $\zeta^{B_{\Gamma}}_{ij}(x)$ are defined by
\begin{equation}
\label{eq:zeta-B-case1}
\begin{split}
  & \zeta^{B_{\Gamma}}_{ij}(x)=\zeta^{A_{\Gb}}_{ij}(x) \qquad \forall\, i,j\in \Gb \\
  & \zeta^{B_{\Gamma}}_{NN}(x)=\frac{1-xq^{\mp 1}}{1-x} \\
  & \zeta^{B_{\Gamma}}_{N-1,N}(x)=\frac{1-xq^{\pm 1}}{x}, \quad \zeta^{B_{\Gamma}}_{N,N-1}(x)=-(1-xq^{\pm 1}) \\
  & \zeta^{B_{\Gamma}}_{Ni}(x)=1, \quad \zeta^{B_{\Gamma}}_{iN}(x)=1 \qquad \forall\, i \in \{1,\dots,N-2\}
\end{split}
\end{equation}
The corresponding quantum affine superalgebra of type $B$ is 
\begin{equation*}
\ubG = \tubG \Big / \Big (\text{relations \eqref{eqn:rel quantum affine 8}-\eqref{eqn:rel quantum affine 9}},~\eqref{eqn:rel b plus} 
  \text{ and their }  f\text{-versions} \Big )
\end{equation*}
where \eqref{eqn:rel quantum affine 8}-\eqref{eqn:rel quantum affine 9} apply to all indices with the exception of \eqref{eqn:rel quantum affine 8} for $(N,N-1)$, when we instead impose the following relation:
\begin{equation}
\label{eqn:rel b plus}
  \mathrm{Sym}_{z_1,z_2,z_3}\,  [ e_N(z_1), [ e_N(z_2), [ e_N(z_3) ,e_{N-1}(w) ]_{q^{\mp 1}} ]_1 ]_{q^{\pm 1}} = 0 
\end{equation}
(this relation is easily seen to be invariant under $q\leadsto q^{-1}$).
For the neck parameter $q^{\pm 1} = q^{-1}$, this recovers the algebra denoted by $U^D_q(\widehat{\fg}_{\bs}) |_{C=1}$ in~\cite[Definition~5.1]{BFK} corresponding to $\fosp(2m+1|2n)$ with $N=m+n$ and the parity sequence $\bs$ satisfying $s_{N}=1$ (which is conjectured to be isomorphic to $U_q^{DJ}(\fg_{\widehat{X}_\bs})|_{c=1}$, see \cite[Conjecture~5.1]{BFK}).

\medskip
\noindent 
\underline{Case 2}
\begin{equation*}
\begin{tikzcd}
  \cdots && {N-2} && {N-1} && N 
  \arrow["q^{\pm 1}"', no head, from=1-7, to=1-5]
  \arrow[no head, from=1-5, to=1-3]
  \arrow[no head, from=1-3, to=1-1]
  \arrow["q^{\pm 1}", no head, from=1-7, to=1-7, looseness=4]
  \arrow["q^{\mp 2}"', no head, from=1-7, to=1-7, looseness=8]
  \arrow[dashed, no head, from=1-5, to=1-5, looseness=4]
  \arrow[dashed, no head, from=1-3, to=1-3, looseness=4]
\end{tikzcd}
\end{equation*}
The body $\Gb$ formed by the vertices $\{1,\dots,N-1\}$ is a type $A$ decorated Dynkin diagram, and $|N|=1$. The zeta functions $\zeta^{B_{\Gamma}}_{ij}(x)$ are defined by
\begin{equation}
\label{eq:zeta-B-case2}
\begin{split}
  & \zeta^{B_{\Gamma}}_{ij}(x)=\zeta^{A_{\Gb}}_{ij}(x) \qquad \forall\, i,j\in \Gb \\
  & \zeta^{B_{\Gamma}}_{NN}(x)=x^{-\frac{1}{2}} \frac{(1-xq^{\mp 2})(1-xq^{\pm 1})}{1-x} \\
  & \zeta^{B_{\Gamma}}_{N-1,N}(x)=(-1)^{|N-1|} \frac{1-xq^{\pm 1}}{x}, \quad \zeta^{B_{\Gamma}}_{N,N-1}(x)=-(1-xq^{\pm 1}) \\
  & \zeta^{B_{\Gamma}}_{Ni}(x)=1, \quad \zeta^{B_{\Gamma}}_{iN}(x)=(-1)^{|i|} \qquad \forall\, i \in \{1,\dots,N-2\} 
\end{split}
\end{equation}
The corresponding quantum affine superalgebra of type $B$ is 
\begin{equation*}
\ubG = \tubG \Big / \Big (\text{relations \eqref{eqn:rel quantum affine 8}-\eqref{eqn:rel quantum affine 9}, \eqref{eqn:rel b 1}-\eqref{eqn:rel b 2}} 
  \text{ and their }  f\text{-versions} \Big )
\end{equation*}
where \eqref{eqn:rel quantum affine 8}-\eqref{eqn:rel quantum affine 9} are understood as in Case 1 and we add the following two relations: 
\begin{equation}
\label{eqn:rel b 1}
  \mathrm{Sym}_{z_1,z_2,z_3}\,  z_3 [ e_N(z_1), [ e_N(z_2), e_N(z_3) ]_{q^{\pm 1}} ]_{q^{\pm 2}}  = 0 
\end{equation}
\begin{multline}
\label{eqn:rel b 2}
   \mathrm{Sym}_{z_1,z_2}\,  z_1\Big( q^{\mp 1}[[e_N(z_1),e_N(z_2)]_{q^{\mp 1}},e_{N-1}(w)]_{q^{\pm 2}}  \\ + (q+q^{-1}) [[e_{N-1}(w),e_N(z_1)]_{q^{\mp 1}},e_N(z_2)]_{1} \Big) = 0 
\end{multline}
In the ``$f$-version'' of relations~\eqref{eqn:rel b 1}-\eqref{eqn:rel b 2}, we replace every $e_i(z)$ with $f_i(z)$ but also replace each prefactor $z_k$ with $z_k^{-1}$. For the neck parameter $q^{\pm 1} = q$, this recovers the algebra denoted by $U^D_q(\widehat{\fg}_{\bs}) |_{C=1}$ in~\cite[Definition 5.1]{BFK} corresponding to $\fosp(2m+1|2n)$ with $N=m+n$ and the parity sequence $\bs$ satisfying $s_{N}=-1$ (which is conjectured to be isomorphic to $U_q^{DJ}(\fg_{\widehat{X}_\bs})|_{c=1}$, see \cite[Conjecture 5.1]{BFK}). In particular, we note that~\cite[(5.10)]{BFK} is precisely~\eqref{eqn:rel b 1}, while \cite[(5.12)]{BFK} is equivalent to~\eqref{eqn:rel b 2}, and \cite[(5.11)]{BFK} is equivalent to~\eqref{eqn:quad intro} for $i=j=N$ and the above choice of $\zeta^{B_{\Gamma}}_{NN}(x)$.

\medskip
\noindent
The big shuffle algebra is (with $\twist$ defined similarly to~\eqref{eq:root-twist})
$$
  \CV^+_{B_{\Gamma}} = \bigoplus_{\bn = (n_i)_{i \in \pI} \in \pnn} \BC[z_{i1}^{\pm 1},\dots,z_{in_i}^{\pm 1}]_{i \in \pI}^{\sym} \cdot \twist
$$
made into an associative algebra by formula \eqref{eqn:shuffle product} with $\zeta^{B_{\Gamma}}$ instead of $\zeta$. Let
$$
  \CS^+_{B_{\Gamma}} \subset \CV^+_{B_{\Gamma}}
$$
be the set of color-symmetric Laurent polynomials (times $\twist$) which satisfy the length $3$ and $4$ wheel conditions \eqref{eqn:wheel finite even} and \eqref{eqn:wheel finite odd} for all applicable colors except for $(i,i-1)=(N,N-1)$ in Case 1, in which case we impose instead the length $4$ wheel condition~\eqref{eqn:wheel b 1}, together with the additional length 3 wheel condition~\eqref{eqn:wheel b 2} in Case 2. According to Lemma~\ref{lem:easy}, all these wheel conditions are preserved by the shuffle product and thus $\CS^+_{B_{\Gamma}}$ is an algebra. We will now prove Theorem \ref{thm:b intro}, which states that the algebra homomorphism 
$$
  \tUpsilon^+_{B_{\Gamma}}\colon \tubGp \to \CV^+_{B_{\Gamma}}, \qquad e_{i,d}\mapsto z^d_{i1} , \quad \forall\, i\in \pI, d\in \BZ
$$
of~\eqref{eqn:tilde upsilon} induces an isomorphism 
$$
  \Upsilon^+_{B_{\Gamma}}\colon \ubGp \iso \CS^+_{B_{\Gamma}}
$$
We will denote by $\CS^-_{B_{\Gamma}} \subset \CV^-_{B_{\Gamma}}$ the same sets as $\CS^+_{B_{\Gamma}} \subset \CV^+_{B_{\Gamma}}$, but endowed with the opposite algebra structure.

\medskip

\begin{proof}\emph{of Theorem \ref{thm:b intro}:} 
Since inverting all edge parameters in a decorated Dynkin diagram has the effect of swapping $q \leftrightarrow q^{-1}$ in all associated quantum algebras, we may assume without loss of generality that the neck parameter is $q$. We will follow the outline of Subsection \ref{sub:plan}, in which one needs to prove the following key facts: 
\begin{multline}
\label{eqn:fact 1 b}
  \Big \langle \Big( \text{LHS of \eqref{eqn:rel quantum affine 8}-\eqref{eqn:rel quantum affine 9} and \eqref{eqn:rel b plus}/\eqref{eqn:rel b 1}-\eqref{eqn:rel b 2}} \Big), F \Big \rangle = 0 
  \quad \text{for} \quad F \in \CV^-_{B_{\Gamma}} \\ \quad \Rightarrow \quad F \in \CS^-_{B_{\Gamma}}
\end{multline} 
(above, the slash reflects the fact that relation \eqref{eqn:rel b plus} applies only to Case 1, while relations \eqref{eqn:rel b 1}-\eqref{eqn:rel b 2} apply only to Case 2) and 
\begin{equation}
\label{eqn:fact 2 b}
  \Big \langle e_{i_1,d_1}\dots e_{i_k,d_k}, F \Big \rangle = 
  \left[\text{a linear functional of }\tUpsilon^+_{B_{\Gamma}} \left(e_{i_1,d_1}\dots e_{i_k,d_k} \right) \right] 
\end{equation}
for any $F\in \CS^-_{B_{\Gamma}}$.

\medskip
\noindent
To establish \eqref{eqn:fact 1 b}, we follow the process laid out in the proof of~\eqref{eqn:key pairing finite sl}.  
By analogy with \eqref{eqn:x}, we consider the corresponding power series 
$$
  X_{N-1,N}(z_1,z_2,z_3,w) = \sum_{k_1,k_2,k_3,\ell\in \BZ} X_{N-1,N}^{k_1,k_2,k_3,\ell} z^{-k_1}_1 z^{-k_2}_2 z^{-k_3}_3 w^{-\ell} = 
  \Big( \text{LHS of \eqref{eqn:rel b plus}} \Big) 
$$
in Case 1, and 
\begin{align*}
  X_{N,N-1}(z_1,z_2,w) & = \sum_{k_1,k_2,\ell\in \BZ} X_{N,N-1}^{k_1,k_2,\ell} z^{-k_1}_1 z^{-k_2}_2 w^{-\ell} = \Big( \text{LHS of \eqref{eqn:rel b 2}} \Big) \\
  X_N(z_1,z_2,z_3) &= \sum_{k_1,k_2,k_3\in \BZ} X_N^{k_1,k_2,k_3} z^{-k_1}_1 z^{-k_2}_2 z^{-k_3}_3 = \Big( \text{LHS of \eqref{eqn:rel b 1}} \Big) 
\end{align*}
in Case 2. A direct calculation (either by hand or by computer) easily verifies the respective analogues of~\eqref{eqn:rational function} for the series above:
\begin{equation*}
\begin{split}
  & \tUpsilon^+_{B_{\Gamma}}(X_{N-1,N}(z_1,z_2,z_3,w)) = 0 \quad  \mathrm{in\ Case}\ 1   \\
  & \tUpsilon^+_{B_{\Gamma}}(X_{N,N-1}(z_1,z_2,w)) = 0  \quad  \mathrm{in\ Case}\ 2   \\
  & \tUpsilon^+_{B_{\Gamma}}(X_{N}(z_1,z_2,z_3)) = 0 \quad  \mathrm{in\ Case}\ 2 
\end{split}
\end{equation*}
For example, the last two equalities are equivalent to the rational function identities:
\begin{multline}
\label{eq:B-type-defining last 2}
  \Sym_{z_1,z_2} \Bigg[ (z_1q^{-1}-z_2) \zeta^{B_{\Gamma}}_{NN} \left(\frac{z_1}{z_2}\right) \cdot 
  \Bigg( \zeta^{B_{\Gamma}}_{N,N-1}\left(\frac{z_1}{w}\right) \zeta^{B_{\Gamma}}_{N,N-1}\left(\frac{z_2}{w}\right) -  \\
   (-1)^{|N-1|} (q+q^{-1}) \zeta^{B_{\Gamma}}_{N,N-1}\left(\frac{z_1}{w}\right) \zeta^{B_{\Gamma}}_{N-1,N}\left(\frac{w}{z_2}\right) + 
  \zeta^{B_{\Gamma}}_{N-1,N}\left(\frac{w}{z_1}\right) \zeta^{B_{\Gamma}}_{N-1,N}\left(\frac{w}{z_2}\right) \Bigg) \Bigg] = 0
\end{multline}
\begin{equation}
\label{eq:B-type-defining last one}
  \Sym_{z_1,z_2,z_3} \left[ (z_3+z_2(q-q^2)-z_1q^3) 
  \zeta^{B_{\Gamma}}_{NN}\left(\frac{z_1}{z_2}\right) \zeta^{B_{\Gamma}}_{NN}\left(\frac{z_1}{z_3}\right) \zeta^{B_{\Gamma}}_{NN}\left(\frac{z_2}{z_3}\right) \right] = 0
\end{equation}

\medskip
\noindent
The key facts to check are the following analogues of~\eqref{eqn:not zero 1} and~\eqref{eqn:not zero 2} which can be computed in a straightforward way (either by computer or by hand, see Appendices~\ref{app:B-quartic},~\ref{app:B-cubic-1},~\ref{app:B-cubic-2}, ~\ref{sub:non-zero type b case 1},~\ref{sub:non-zero type b usual wheel}, and~\ref{sub:non-zero type b cubic wheel}  combined with Remark~\ref{rem:cont-irrelevance}): 
\begin{equation}
\label{eq:const-B-1st}
  \left \langle X_{N-1,N}^{0,0,0,d}, z^{-d}_{N-1,1} \right \rangle = 6
\end{equation}
\begin{equation}
\label{eq:const-B-2nd}
  \left \langle X_{N,N-1}^{0,0,d},  \frac {z_{N-1,1}^{-d}}{\sqrt{z_{N1}z_{N2}}}  \right \rangle = 2q^{-1}
\end{equation}
\begin{multline}
\label{eq:const-B-3rd}
  \left \langle X_{N}^{0,0,d}, z^{-d-1}_{N 1} + z^{-d-1}_{N 2} + z^{-d-1}_{N 3} \right \rangle  = \\
   \frac{-2q^{3-2d}(1+q^{d}+q^{2d})(1+q^{d+1}+q^{2d+2})}{1+q+q^2} 
\end{multline}
(the square roots in~\eqref{eq:const-B-2nd} are necessary due to the corresponding $\twist$, the $B$-version of~\eqref{eq:root-twist}). This allows to complete the proof of \eqref{eqn:fact 1 b} just like in the proof of~\eqref{eqn:key pairing finite sl}.

\medskip 
\noindent 
In order to prove \eqref{eqn:fact 2 b}, we will adapt the proof of Lemma \ref{lem:pairing sl}. As explained in Subsection \ref{sub:specialization patterns}, we only need to describe the asps of Definitions \ref{def:specialization pattern}-\ref{def:acceptable specialization pattern} in the case at hand, and to show that they are closed under extensions. To this end, we define acceptability with respect to positive real numbers \eqref{eqn:rho} such that 
\begin{equation}
\label{eqn:greater b}
  \rho_1 , \dots,  \rho_{N-1} \ll \rho_N, 
\end{equation}
meaning that while the relative order of $\rho_1,\dots,\rho_{N-1}$ is irrelevant, $\rho_N$ should be much larger than all of them. Because of the assumption \eqref{eqn:greater b}, any asp rooted at color $\leq N-1$ will actually be constrained to vertices of colors $1,\dots,N-1$, and is thus nothing more than a finite type shrub. As for asps rooted at color $N$, their full classification is depicted in Figures \ref{fig:asps B 1}-\ref{fig:asps B 2} and we will now explain how to read the pictures. The parameter decorating any arrow is completely determined by its endpoints according to the rule
$$
  (\text{lattice point }(x,y)) \xrightarrow{q^{x'+y'-x-y}} (\text{lattice point }(x',y'))
$$
(the lattice is normalized so that the vertical arrows  in the pictures on the left carry the parameter $q$). The root of the asp is depicted by the blue disk, and the arrows indicate the behavior of the asps in colors $N-1$ and $N$. The black arrows carry the same information as two arbitrary finite type shrubs $S_\bullet$ and $S_\circ$ with vertices
\begin{equation}
\label{eqn:small b}
  1,\dots,N-1, *
\end{equation}
where $*$ is an odd vertex. The aforementioned shrubs $S_\bullet$ and $S_\circ$ are rooted at the black solid and hollow disks, respectively, and because the parameters of their roots differ by $q$ they do not interact with each other \footnote{Here, the word ``interact" means that there are no arrows between vertices of $S_\bullet$ and $S_\circ$; this is because in type $A$, arrows between vertices of the same color carry parameters $q^{\pm 2}$ while arrows between vertices of neighboring colors carry parameters $q^{\pm 1}$.}. The colored arrows in Figures \ref{fig:asps B 1}-\ref{fig:asps B 2} indicate how one can extend the shrubs $S_\bullet$ and $S_\circ$ in order to obtain a type $B$ asp; note that whenever two solid arrows come into a vertex $v$, there is always a dotted arrow coming out of that vertex: this indicates that the apparent double pole corresponding to $v$ in the residue \eqref{eqn:residue} is actually reduced to a simple pole due to various wheel conditions, much like the example of Figure \ref{fig:shrub} in type $A$.

\begin{figure}[h]
  \includegraphics[scale=0.75]{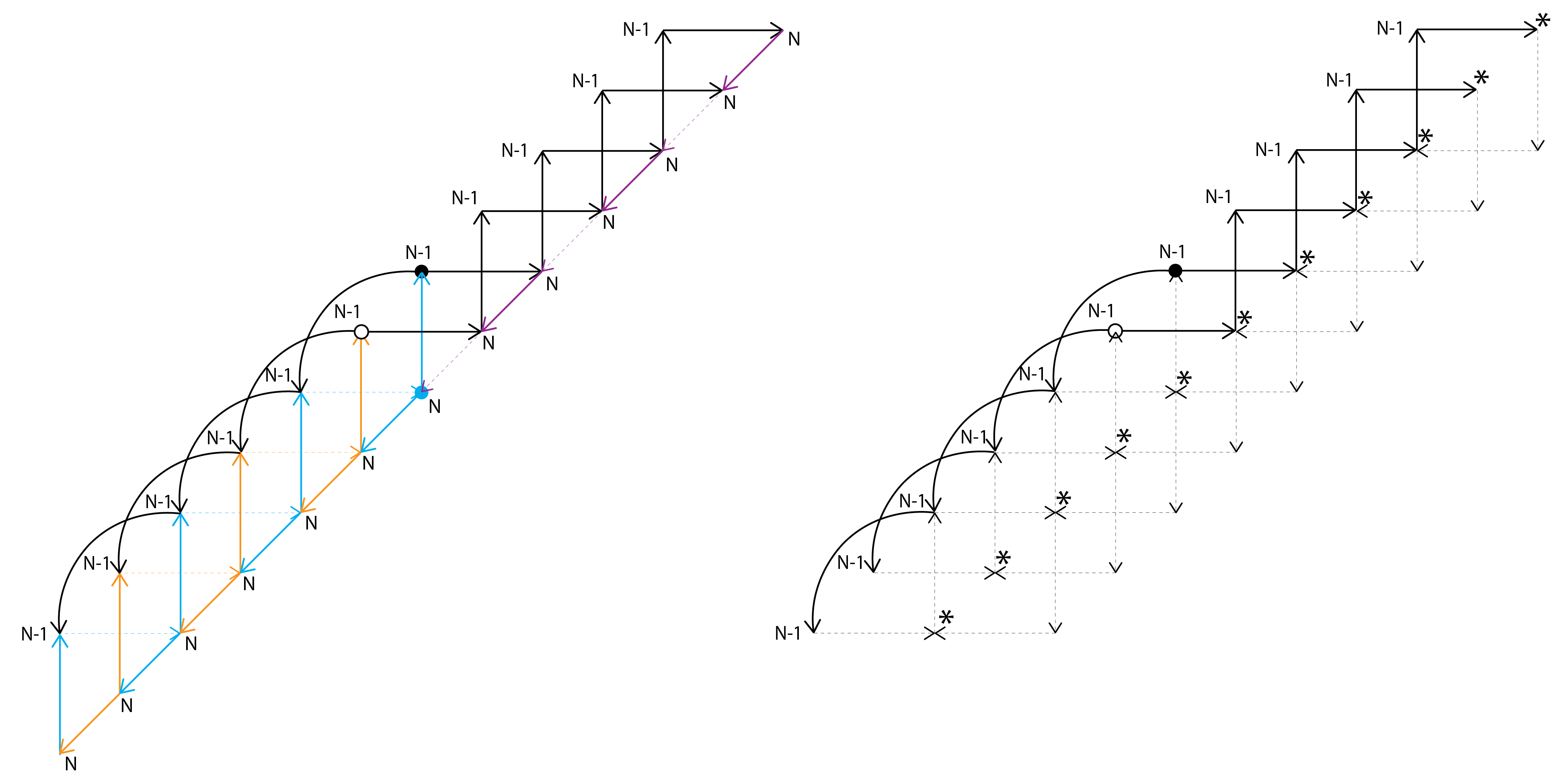} 
  \caption{Type $B$ Case 1: asps rooted at the blue disk on the left, the corresponding finite type non-interacting shrubs $S_\bullet$ and $S_\circ$ on the vertex set \eqref{eqn:small b} on the right. The curved arrows indicate a straight southwest pointing arrow if $|N-1|=0$ or a left-then-down zig-zag if $|N-1|=1$.}
  \label{fig:asps B 1}
\end{figure}

\begin{figure}[h]
  \includegraphics[scale=0.75]{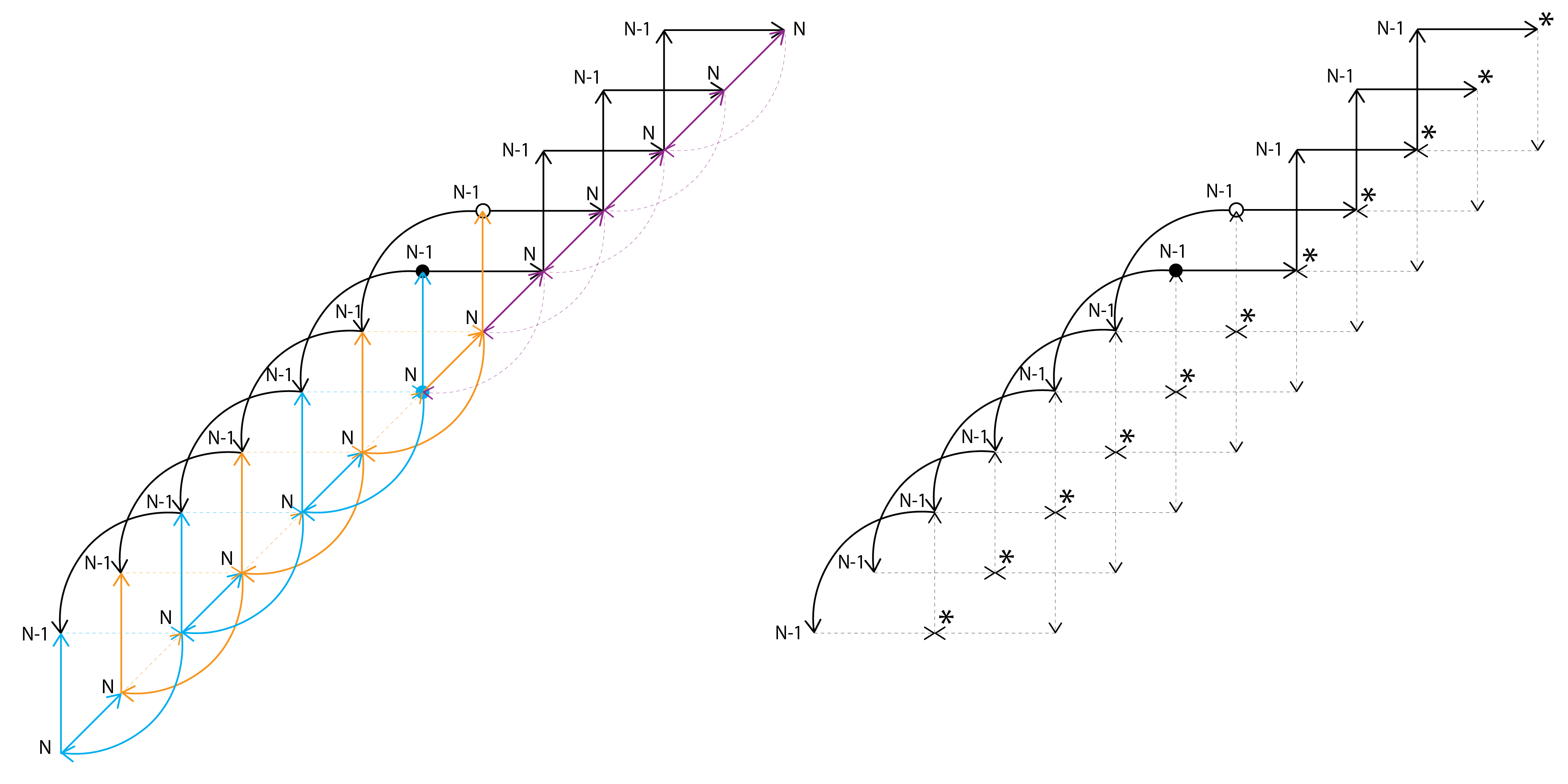} 
  \caption{Type $B$ Case 2: asps rooted at the blue disk on the left, the corresponding finite type non-interacting shrubs $S_\bullet$ and $S_\circ$ on the vertex set \eqref{eqn:small b} on the right. The curved arrows indicate a straight southwest pointing arrow if $|N-1|=0$ or a left-then-down zig-zag if $|N-1|=1$.}
  \label{fig:asps B 2}
\end{figure}

\medskip 
\noindent
Note that we never invoked the acceptability condition \eqref{eqn:greater b}: to be precise, the notion of ``asps" only involves those pictures among Figures \ref{fig:asps B 1}-\ref{fig:asps B 2} such that for every vertex $v$ of color $N$ other than the blue disk, the product of arrow parameters from the blue disk to $v$ has absolute value strictly greater than $1$. By cleverly choosing the absolute value of $q$ itself, one could drastically reduce the number of asps. For instance, if we choose $|q|>1$ in Case 1 (respectively $|q|<1$ in Case~2), then a type $B$ asp cannot contain any color $N$ vertices southwest (respectively northeast) of the blue disk; with either choice, the hollow disk (and hence the entire shrub $S_\circ$) would thus be prohibited. The reason why we do not make this simplifying assumption at the present moment is that we will need to avoid it for the toroidal cases that will be treated in Section \ref{sec:osp-toroidal}. 
\end{proof}


\medskip
	
\subsection{Type $C$}
\label{sub:c}
	
Define the pre-quantum affine superalgebra as (cf.~\eqref{eqn:quad intro}) 
\begin{equation*}
  \tucG = \BC \Big \langle e_{i,d}, f_{i,d}, \ph_{i,d'}^\pm \Big \rangle_{i \in \pI, d \in \BZ, d' \geq 0} \Big/ 
  \Big (  \text{relations \eqref{eqn:rel quantum affine 1}-\eqref{eqn:rel quantum affine 6}} \Big )
\end{equation*}
where the relations above are defined with respect to the zeta functions $\zeta^{C_{\Gamma}}_{ij}(x)$ that are encoded by the decorated Dynkin diagram $\Gamma$. We further define the quantum affine superalgebra $\ucG$ by imposing higher order relations, depending on $\Gamma$.

\medskip
\noindent
\underline{Case 1} 
\begin{equation*}
\begin{tikzcd}
  \cdots &&  {N-2} && {N-1} && N
  \arrow["q^{\pm 2}"', no head, from=1-7, to=1-5]
  \arrow["q^{\pm 1}"', no head, from=1-5, to=1-3] 
  \arrow[no head, from=1-3, to=1-1]
  \arrow["q^{\mp 4}"', no head, from=1-7, to=1-7, looseness=4]
  \arrow["q^{\mp 2}"', no head, from=1-5, to=1-5, looseness=4]
  \arrow[dashed, no head, from=1-3, to=1-3, looseness=4]
\end{tikzcd}
\end{equation*}
The body $\Gb$ formed by the vertices $\{1,\dots,N-1\}$ is a type $A$ decorated Dynkin diagram, and $|N|=0$. The zeta functions $\zeta^{C_{\Gamma}}_{ij}(x)$ are defined by
\begin{equation}
\label{eq:zeta C-type}
\begin{split}
  & \zeta^{C_{\Gamma}}_{ij}(x)=\zeta^{A_{\Gb}}_{ij}(x) \qquad \forall\, i,j\in \Gb \\
  & \zeta^{C_{\Gamma}}_{NN}(x)=\frac{1-xq^{\mp 4}}{1-x} \\
  & \zeta^{C_{\Gamma}}_{N-1,N}(x)=\frac{1-xq^{\pm 2}}{x}, \quad  \zeta^{C_{\Gamma}}_{N,N-1}(x)=-(1-xq^{\pm 2}) \\
  & \zeta^{C_{\Gamma}}_{Ni}(x)=\zeta^{C_{\Gamma}}_{iN}(x)=1 \qquad \forall\, i\in \{1,\dots,N-2\}
\end{split}
\end{equation}
The corresponding quantum affine superalgebra of type $C$ is 
\begin{equation*}
\ucG = \tucG \Big / \Big (\text{relations \eqref{eqn:rel quantum affine 8}-\eqref{eqn:rel quantum affine 9}},~\eqref{eqn:rel b plus-plus} 
  \text{ and their }  f\text{-versions} \Big )
\end{equation*}
where we consider \eqref{eqn:rel quantum affine 8} for $i,i\pm 1 \leq N-1$ or $(i,i-1)=(N,N-1)$, consider~\eqref{eqn:rel quantum affine 9} for $i-1,i,i+1\leq N-1$, and impose the following relation:
\begin{equation}
\label{eqn:rel b plus-plus}
  \mathrm{Sym}_{z_1,z_2,z_3}\,  [ e_{N-1}(z_1), [ e_{N-1}(z_2), [ e_{N-1}(z_3) ,e_{N}(w) ]_{q^{\mp 2}} ]_{1} ]_{q^{\pm 2}} = 0 
\end{equation}  
(this relation is easily seen to be invariant under $q\leadsto q^{-1}$).
For the neck parameter $q^{\pm 2} = q^{2}$, this recovers the algebra denoted by $U^D_q(\widehat{\fg}_{\bs}) |_{C=1}$ in~\cite[Definition 5.1]{BFK} corresponding to $\fosp(2m|2n)$ with $N=m+n$ and the parity sequence $\bs$ satisfying $s_{N-1}=s_N=-1$ (which is conjectured to be isomorphic to $U_q^{DJ}(\fg_{\widehat{X}_\bs})|_{c=1}$, see \cite[Conjecture 5.1]{BFK}).

\medskip
\noindent
\underline{Case 2} 
\begin{equation*}
\begin{tikzcd}
  \cdots && N-3 && {N-2} && {N-1} && N
  \arrow["q^{\pm 2}"', no head, from=1-9, to=1-7]
  \arrow["q^{\mp 1}"', no head, from=1-7, to=1-5]
  \arrow["q^{\pm 1}"', no head, from=1-5, to=1-3]
  \arrow[no head, from=1-3, to=1-1]
  \arrow["q^{\mp 4}"', no head, from=1-9, to=1-9, looseness=4]
  \arrow[dashed, no head, from=1-3, to=1-3, looseness=4]
\end{tikzcd}
\end{equation*}
The body $\Gb$ formed by the vertices $\{1,\dots,N-1\}$ is a type $A$ decorated Dynkin diagram, $|N|=0$, and the zeta functions $\zeta^{C_{\Gamma}}_{ij}(x)$ are defined by~\eqref{eq:zeta C-type}. The corresponding quantum affine superalgebra of type $C$ is 
\begin{equation*}
\ucG = \tucG \Big / \Big (\text{relations \eqref{eqn:rel quantum affine 8}-\eqref{eqn:rel quantum affine 9}},~\eqref{eqn:rel cd 2 1} 
  \text{ and their }  f\text{-versions} \Big )
\end{equation*}
where we consider \eqref{eqn:rel quantum affine 8} and~\eqref{eqn:rel quantum affine 9} for the same indices as in Case 1, and impose the following degree $6$ relation:
\begin{multline}
\label{eqn:rel cd 2 1}
  \mathrm{Sym}_{z_1,z_2} \Sym_{w_1,w_2,w_3} \\
  [[[e_{N-2}(z_1),e_{N-1}(w_1)]_{q^{\pm 1}},[[e_{N-2}(z_2), e_{N-1}(w_2)]_{q^{\pm 1}},e_{N}(y)]_{q^{\mp 2}}]_1,e_{N-1}(w_3)]_1 = 0
\end{multline}
For the neck parameter $q^{\pm 2} = q^{2}$, this recovers the algebra denoted by $U^D_q(\widehat{\fg}_{\bs}) |_{C=1}$ in~\cite[Definition 5.1]{BFK} corresponding to $\fosp(2m|2n)$ with $N=m+n$ and the parity sequence $\bs$ satisfying $s_{N-2}=-1$, $s_{N-1}=1$, $s_{N}=-1$ (which is conjectured to be isomorphic to $U_q^{DJ}(\fg_{\widehat{X}_\bs})|_{c=1}$, see \cite[Conjecture 5.1]{BFK}).
We also recall some of explicit zeta values that will be needed later: 
\begin{equation}
\label{eq:C-zeta-case2}
\begin{split}
  & \zeta^{C_{\Gamma}}_{N-2,N-2}(x)=\zeta^{C_{\Gamma}}_{N-1,N-1}(x) = \frac{x^{\frac{1}{2}}}{1-x}  \\
  & \zeta^{C_{\Gamma}}_{N-1,N}(x)=\frac{1-xq^{\pm 2}}{x}, \quad \zeta^{C_{\Gamma}}_{N,N-1}(x)=-(1-xq^{\pm 2}) \\
  & \zeta^{C_{\Gamma}}_{N-2,N-1}(x)=-\frac{1-xq^{\mp 1}}{x}, \quad \zeta^{C_{\Gamma}}_{N-1,N-2}(x)=-(1-xq^{\mp 1})
\end{split}
\end{equation}

\medskip

\begin{remark}\label{rem:equivalent-C-six-relations}
(a) We note that the published version of~\cite[(5.13)]{BFK} contains a small typo, as they imposed 
\begin{multline*}
\mathrm{Sym}_{z_1,z_2} \mathrm{Sym}_{w_1,w_2,w_3} \\
  [[e_{N-2}(z_1),e_{N-1}(w_1)]_{q^{\pm 1}},[[e_{N-2}(z_2),e_{N-1}(w_2)]_{q^{\pm 1}},[e_{N-1}(w_3),e_{N}(y)]_{q^{\mp 2}}]_{q^{\mp 1}} ]_{q^{\pm 1}} = 0
\end{multline*}
(which is incompatible with~\cite[(3.7)]{BFK}) instead of the required \eqref{eqn:rel cd 2 1}.

\medskip
\noindent
(b) Alternatively, either of the following relations is equivalent to~\eqref{eqn:rel cd 2 1}:
\begin{multline}
\label{eqn:rel cd 2 1 corrected}
  \mathrm{Sym}_{z_1,z_2} \mathrm{Sym}_{w_1,w_2,w_3} \\
  [[e_{N-2}(z_1),e_{N-1}(w_1)]_{q^{\pm 1}},[[e_{N-2}(z_2),e_{N-1}(w_2)]_{q^{\pm 1}}, [e_{N}(y),e_{N-1}(w_3)]_{q^{\mp 2}} ]_{q^{\mp 1}} ]_{q^{\pm 1}} = 0
\end{multline}
\begin{multline}
\label{eqn:rel cd 2 1 another}
  \mathrm{Sym}_{z_1,z_2} \mathrm{Sym}_{w_1,w_2,w_3} \\
  [[e_{N-2}(z_1), e_{N-1}(w_1)]_{q^{\pm 1}},  [[[e_{N-2}(z_2),e_{N-1}(w_2)]_{q^{\pm 1}},e_{N}(y)]_{q^{\mp 2}},e_{N-1}(w_3)]_{q^{\mp 1}} ]_{q^{\pm 1}} = 0
\end{multline}

\noindent
(c) It is in general non-trivial to show that different higher order relations, such as \eqref{eqn:rel cd 2 1},  \eqref{eqn:rel cd 2 1 corrected},  \eqref{eqn:rel cd 2 1 another} and their $q\leadsto q^{-1}$ versions are equivalent (modulo the quadratic relations in \eqref{eqn:quad intro} with $\zeta^{C_{\Gamma}}$ instead of $\zeta$). However, Theorem~\ref{thm:c intro} and the discussion of Section~\ref{sec:arbitrary} provide a natural framework for establishing this kind of equivalence.
\end{remark}

\medskip
\noindent
\underline{Case 3} 
\begin{equation*}
\begin{tikzcd}
  \cdots && N-3 && {N-2} && {N-1} && N
  \arrow["q^{\pm 2}"', no head, from=1-9, to=1-7]
  \arrow["q^{\mp 1}"', no head, from=1-7, to=1-5]
  \arrow["q^{\mp 1}"', no head, from=1-5, to=1-3]
  \arrow[no head, from=1-3, to=1-1]
  \arrow["q^{\mp 4}"', no head, from=1-9, to=1-9, looseness=4]
  \arrow["q^{\pm 2}"', no head, from=1-5, to=1-5, looseness=4]
  \arrow[dashed, no head, from=1-3, to=1-3, looseness=4]
\end{tikzcd}
\end{equation*}
The body $\Gb$ formed by the vertices $\{1,\dots,N-1\}$ is a type $A$ decorated Dynkin diagram, $|N|=0$, and the zeta functions $\zeta^{C_{\Gamma}}_{ij}(x)$ are defined by~\eqref{eq:zeta C-type}. The corresponding quantum affine superalgebra of type $C$ is 
\begin{equation*}
  \ucG = \tucG \Big / \Big (\text{relations \eqref{eqn:rel quantum affine 8}-\eqref{eqn:rel quantum affine 9}},~\eqref{eqn:rel cd 2 2} 
  \text{ and their }  f\text{-versions} \Big )
\end{equation*}
where we consider \eqref{eqn:rel quantum affine 8} and~\eqref{eqn:rel quantum affine 9} for the same indices as in Case 1, and impose the following degree $7$ relation: \footnote{In fact, the $q^{\pm 1} \leadsto q^{\mp 1}$ version of this relation can alternatively be used, cf.~Remark~\ref{rem:equivalent-C-six-relations}.}
\begin{multline}
\label{eqn:rel cd 2 2}
  \mathrm{Sym}_{z_1,z_2} \Sym_{w_1,w_2,w_3} \,  
  [[[[[[ e_{N-3}(y), e_{N-2}(z_1) ]_{q^{\pm 1}}, e_{N-1}(w_1) ]_{q^{\pm 1}}, \\ e_N(u) ]_{q^{\mp 2}},  e_{N-1}(w_2) ]_{q^{\mp 1}}, e_{N-2}(z_2) ]_{q^{\pm 1}}, e_{N-1}(w_3) ]_{1} = 0
\end{multline}
For the neck parameter $q^{\pm 2} = q^{2}$, this recovers the algebra denoted by $U^D_q(\widehat{\fg}_{\bs}) |_{C=1}$ in~\cite[Definition 5.1]{BFK} corresponding to $\fosp(2m|2n)$ with $N=m+n$ and the parity sequence $\bs$ satisfying $s_{N-2}=1$, $s_{N-1}=1$, $s_{N}=-1$ (which is conjectured to be isomorphic to $U_q^{DJ}(\fg_{\widehat{X}_\bs})|_{c=1}$, see \cite[Conjecture 5.1]{BFK}).
We also recall some of explicit zeta values that will be needed later:
\begin{equation}
\label{eq:C-zeta-case3}
\begin{split}
  & \zeta^{C_{\Gamma}}_{N-2,N-2}(x)=\frac{1-xq^{\pm 2}}{1-x}, \quad \zeta^{C_{\Gamma}}_{N-1,N-1}(x) = \frac{x^{\frac{1}{2}}}{1-x}  \\
  & \zeta^{C_{\Gamma}}_{N-3,N-2}(x)=\zeta^{C_{\Gamma}}_{N-2,N-1}(x)=\frac{1-xq^{\mp 1}}{x}, \quad 
    \zeta^{C_{\Gamma}}_{N-1,N}(x)=\frac{1-xq^{\pm 2}}{x} \\
  & \zeta^{C_{\Gamma}}_{N-2,N-3}(x)=\zeta^{C_{\Gamma}}_{N-1,N-2}(x)=-(1-xq^{\mp 1}), \quad 
    \zeta^{C_{\Gamma}}_{N,N-1}(x)=-(1-xq^{\pm 2}) \\
\end{split}
\end{equation}

\medskip
\noindent
The big shuffle algebra is (with $\twist$ of~\eqref{eq:root-twist})
$$
  \CV^+_{C_{\Gamma}} = \bigoplus_{\bn = (n_i)_{i \in \pI} \in \pnn} \BC[z_{i1}^{\pm 1},\dots,z_{in_i}^{\pm 1}]_{i \in \pI}^{\sym} \cdot \twist
$$
made into an associative algebra by formula \eqref{eqn:shuffle product} with $\zeta^{C_{\Gamma}}$ instead of $\zeta$. Let
$$
  \CS^+_{C_{\Gamma}} \subset \CV^+_{C_{\Gamma}}
$$
be the set of color-symmetric Laurent polynomials (times $\twist$) which satisfy the length $3$ wheel conditions \eqref{eqn:wheel finite even} for $i,i\pm 1<N$ or $(i,i-1)=(N,N-1)$, the length $4$ wheel conditions \eqref{eqn:wheel finite odd} for $i\leq N-2$, the length $4$ wheel condition~\eqref{eqn:wheel c 1} in Case 1, the length $6$ wheel condition~\eqref{eqn:wheel c 2} in Case 2, and the length $7$ scary wheel condition~ \eqref{eqn:wheel c 3} in Case 3. By definition, a Laurent polynomial 
$$
  E\Big(z_{N-3,1},z_{N-2,1},z_{N-2,2},z_{N-1,1},z_{N-1,2},z_{N-1,3},z_{N1},\dots \Big)
$$
(where the $\dots$ stand for other variables that will not be important to us) is said to satisfy the length 7 scary wheel condition, if 
\begin{equation}
\label{eqn:rigorous}
\begin{split}
  & (z_{N-2,2}-xq^{\pm 1})^2 \quad \text{divides} \\ 
  & \qquad R \mid_{(z_{N-3,1},z_{N-2,1},z_{N-1,1},z_{N-1,2},z_{N-1,3},z_{N1}) \mapsto 
    (xq^{\pm 2},xq^{\pm 3},x,xq^{\pm 4},xq^{\pm 2},xq^{\pm 2})}
\end{split}
\end{equation}
This is the rigorous definition of the symbol $\boxed{2}$ in \eqref{eqn:wheel c 3}. The fact that $\CS^+_{C_{\Gamma}}$ is a subalgebra is a consequence of Lemma \ref{lem:easy} together with the following result.

\medskip

\begin{lemma}
\label{lem:7}
The set of Laurent polynomials which satisfy the length $7$ scary wheel \eqref{eqn:wheel c 3} (see \eqref{eqn:rigorous}) together with the length $3$ wheels corresponding to the two small triangles in the top-left part of the diagram \eqref{eqn:wheel c 3},  is closed under the shuffle product.
\end{lemma}

\medskip

\begin{proof} 
Without loss of generality, we shall assume that the neck parameter is $q^{2}$. 
As in the proof of Lemma~\ref{lem:easy}, we must consider each summand in the symmetrization that defines $E * E'$, and consider all ways of distributing the variables $z_{N-3,1},z_{N-2,1},z_{N-2,2},z_{N-1,1},z_{N-1,2},z_{N-1,3},z_{N1}$ among $E$ and $E'$. If all of the variables are placed in either $E$ or $E'$, then the fact that both $E$ and $E'$ individually satisfy the length $7$ scary wheel condition implies that the corresponding summand of $E * E'$ also does. Suppose first that the variable $z_{N-1,1}$ is plugged into $E$: because of the zeta factors determined by the arrows in \eqref{eqn:wheel c 3}, the variables $z_{N1}, z_{N-1,2}, z_{N-2,1}, z_{N-1,3}, z_{N-3,1}$ must all be plugged into $E$ in order to have a chance of a non-zero result under the specialization from~\eqref{eqn:rigorous}. This implies that the last variable $z_{N-2,2}$ must be plugged into $E'$, but then the divisibility by $(z_{N-2,2}-xq)^2$ is due to two arrows from $z_{N-3,1}$ and $z_{N-1,3}$ into $z_{N-2,2}$.

\medskip 
\noindent 
The alternative is that the variable $z_{N-1,1}$ is plugged into $E'$, and we shall now look at the placement of $z_{N-2,1}$. If $z_{N-2,1}$ is plugged into $E$, then by the above argument the only nontrivial case is when $z_{N-3,1},z_{N-1,3}, z_{N-2,2}$ are also plugged into $E$; in this case, the resulting specialization of $E$ is already divisible by $z_{N-2,2}-xq$ (due to either of the length $3$ wheel conditions on $E$ in the variables $\{z_{N-2,1},z_{N-2,2},z_{N-3,1}\}$ or $\{z_{N-2,1},z_{N-2,2},z_{N-1,3}\}$) and we furthermore get another linear factor $z_{N-2,2}-xq$ due to the arrow from $z_{N-2,2}$ to $z_{N-1,1}$. Let us now consider the case when $z_{N-2,1}$ is plugged into $E'$, so that the variables $z_{N1}$ and $z_{N-1,2}$ are also plugged into $E'$ to have a chance of producing a non-zero result under the specialization from~\eqref{eqn:rigorous}. If the variable $z_{N-2,2}$ is plugged into $E$, then again we get $(z_{N-2,2}-xq)^2$ due to two arrows from $z_{N-2,2}$ to $z_{N-2,1}$ and $z_{N-1,1}$. On the other hand, if $z_{N-2,2}$ is plugged into $E'$, then either one or both of the variables $z_{N-1,3},z_{N-3,1}$ must be plugged into $E$. If both of them occur in $E$, then the two arrows from them to $z_{N-2,2}$ contribute $(z_{N-2,2}-xq)^2$. If $z_{N-3,1}$ is in $E$ but $z_{N-1,3}$ is in $E'$ (the other case is symmetric), then the length $3$ wheel condition in the variables $\{z_{N-2,1},z_{N-2,2},z_{N-1,3}\}$ on $E'$ implies that (the resulting specialization of) $E'$ is divisible by $z_{N-2,2}-xq$ at the last step while we get another factor of $z_{N-2,2}-xq$ due to the arrow from $z_{N-3,1}$ to $z_{N-2,2}$. 
\end{proof}

\medskip 
\noindent  
We will now prove Theorem \ref{thm:c intro}, which states that 
the algebra homomorphism 
$$
  \tUpsilon^+_{C_{\Gamma}}\colon \tucGp \to \CV^+_{C_{\Gamma}}, \qquad e_{i,d}\mapsto z^d_{i1} , \quad \forall\, i\in \pI, d\in \BZ
$$ 
of~\eqref{eqn:tilde upsilon} induces an isomorphism 
$$
  \Upsilon^+_{C_{\Gamma}}\colon \ucGp \iso  \CS^+_{C_{\Gamma}} 
$$
We will denote by $\CS^-_{C_{\Gamma}} \subset \CV^-_{C_{\Gamma}}$ the same sets as $\CS^+_{C_{\Gamma}} \subset \CV^+_{C_{\Gamma}}$, but endowed with the opposite algebra structure.

\medskip

\begin{proof} \emph{of Theorem \ref{thm:c intro}:} 
Since inverting all edge parameters in a decorated Dynkin diagram has the effect of swapping $q \leftrightarrow q^{-1}$ in all associated quantum algebras, we may assume without loss of generality that the neck parameter is $q^2$. We will follow the outline of Subsection \ref{sub:plan}, in which one needs to prove the following key facts: 
\begin{multline}
\label{eqn:fact 1 cd 2}
  \Big \langle \Big( \text{LHS of \eqref{eqn:rel quantum affine 8}-\eqref{eqn:rel quantum affine 9} and~\eqref{eqn:rel b plus-plus}/~\eqref{eqn:rel cd 2 1}/~\eqref{eqn:rel cd 2 2}} \Big) , F \Big \rangle = 0 
  \quad \text{for} \quad F \in \CV^-_{C_{\Gamma}} \\ \quad \Rightarrow \quad F \in \CS^-_{C_{\Gamma}} 
\end{multline} 
(above, the slash reflects the fact that relation \eqref{eqn:rel b plus-plus} applies to Case 1, \eqref{eqn:rel cd 2 1} applies to Case 2, and \eqref{eqn:rel cd 2 2} applies to Case 3) and 
\begin{equation}
\label{eqn:fact 2 c}
  \Big \langle e_{i_1,d_1}\dots e_{i_k,d_k}, F \Big \rangle = 
  \left[\text{a linear functional of }\tUpsilon^+_{C_{\Gamma}} \left(e_{i_1,d_1}\dots e_{i_k,d_k} \right) \right] 
\end{equation}
for any $F\in \CS^-_{C_{\Gamma}}$.

\medskip
\noindent
To establish \eqref{eqn:fact 1 cd 2}, we follow the process laid out in the proof of~\eqref{eqn:key pairing finite sl}. By analogy with \eqref{eqn:x}, we consider the corresponding power series 
$$
  X(z_1,z_2,z_3,w) = \sum_{k_1,k_2,k_3,\ell\in \BZ} X^{k_1,k_2,k_3,\ell} z^{-k_1}_1 z^{-k_2}_2 z^{-k_3}_3 w^{-\ell} = 
  \Big( \text{LHS of \eqref{eqn:rel b plus-plus}} \Big) 
$$
in Case 1, 
\begin{multline*}
  X(z_1,z_2,w_1,w_2,w_3,y) = \\
  \sum_{k_1,k_2,\ell_1,\ell_2,\ell_3,a\in \BZ} X^{k_1,k_2,\ell_1,\ell_2,\ell_3,a} z^{-k_1}_1 z^{-k_2}_2 w_1^{-\ell_1} w_2^{-\ell_2} w_3^{-\ell_3} y^{-a} 
  = \Big( \text{LHS of \eqref{eqn:rel cd 2 1}} \Big) 
\end{multline*}
in Case 2, and 
\begin{multline*}
  X(z_1,z_2,w_1,w_2,w_3,y,u) = \\
  \sum_{k_1,k_2,\ell_1,\ell_2,\ell_3,a,b\in \BZ} X^{k_1,k_2,\ell_1,\ell_2,\ell_3,a,b} z^{-k_1}_1 z^{-k_2}_2 w_1^{-\ell_1} w_2^{-\ell_2} w_3^{-\ell_3} y^{-a} u^{-b} 
  = \Big( \text{LHS of \eqref{eqn:rel cd 2 2}} \Big) 
\end{multline*}
in Case 3. A direct calculation (on a computer, as the time investment to do the calculation by hand would be too great and not worth it in our opinion; the relevant values of zeta functions are presented in~\eqref{eq:C-zeta-case2} and~\eqref{eq:C-zeta-case3}) verifies the respective analogues of~\eqref{eqn:rational function} for them:
\begin{equation*}
\begin{split}
  & \tUpsilon^+_{C_{\Gamma}}(X(z_1,z_2,z_3,w)) = 0 \quad \mathrm{in\ Case}\ 1 \\
  & \tUpsilon^+_{C_{\Gamma}}(X(z_1,z_2,w_1,w_2,w_3,y)) = 0 \quad \mathrm{in\ Case}\ 2 \\
  & \tUpsilon^+_{C_{\Gamma}}(X(z_1,z_2,w_1,w_2,w_3,y,u)) = 0 \quad \mathrm{in\ Case}\ 3
\end{split}
\end{equation*}

\medskip
\noindent
It remains to establish the analogues of~\eqref{eqn:not zero 1} and~\eqref{eqn:not zero 2}, which we present below 
(verified either by computer or by hand, see details in Appendices~\ref{sub:non-zero type c case 1},~\ref{sub:non-zero type c case 2},~\ref{sub:non-zero type c case 3},~\ref{app:C-quartic}): 
\begin{equation}
\label{eqn:not zero c 1}
  \left \langle X^{0,0,0,d}, z^{-d}_{N1} \right \rangle = -6
\end{equation}
\begin{equation}
\label{eqn:not zero c 2}
  \left \langle X^{0,0,0,0,0,d}, \frac {z_{N1}^{1-d}}{\sqrt{z_{N-2,1}z_{N-2,2}}} \right \rangle = 12 
\end{equation}
\begin{equation}
\label{eqn:not zero c 3}
  \left \langle X^{0,0,0,0,0,0,d}, \frac {z_{N-3,1}^2(z_{N-2,1}q - z_{N-2,2}q^{-1})(z_{N-2,2}q - z_{N-2,1}q^{-1})}{z_{N-2,1}^2z_{N-2,2}^2z_{N1}^{d}} \right \rangle = 12 
\end{equation}
(we note that the square roots in~\eqref{eqn:not zero c 2} are necessary due to the corresponding $\twist$ from~\eqref{eq:root-twist}). Let us first explain why in \eqref{eqn:not zero c 3} we compute the pairing with the element
\begin{multline*}
  F_d\Big(z_{N-3,1},z_{N-2,1},z_{N-2,2},z_{N-1,1},z_{N-1,2},z_{N-1,3},z_{N 1}\Big) = \\ 
 \frac {z_{N-3,1}^2(z_{N-2,1}q - z_{N-2,2}q^{-1})(z_{N-2,2}q - z_{N-2,1}q^{-1})}{z_{N-2,1}^2z_{N-2,2}^2z_{N1}^{d}} 
 \in \CV_{C_{\Gamma};-\bn,-d}
\end{multline*}
for $\bn = \bsi^{N-3}+2\bsi^{N-2}+3\bsi^{N-1}+\bsi^N$. Because the scary wheel \eqref{eqn:wheel c 3} involves a second order vanishing condition on Laurent polynomials, we have
$$
  \dim_{\BC} \left(\CV_{C_{\Gamma};-\bn,-d} / \CS_{C_{\Gamma};-\bn,-d}\right) = 2
$$
for all $d \in \BZ$. However, if we let $\CV_{C_{\Gamma};-\bn,-d} \supset \CT_{C_{\Gamma};-\bn,-d}$ denote the subset of color-symmetric Laurent polynomials which vanish to the first order at the scary wheel \eqref{eqn:wheel c 3}, we note that
$$
  \dim_{\BC} \left(\CT_{C_{\Gamma};-\bn,-d} / \CS_{C_{\Gamma};-\bn,-d}\right) = 1
$$
We have $F_d \in \CT_{C_{\Gamma};-\bn,-d}$, and because of this, the argument from the proof of~\eqref{eqn:key pairing finite sl} shows that
$$
  \left \langle X^{0,0,0,0,0,0,d},F \right \rangle = 0 \text{ for } F \in \CT_{C_{\Gamma};-\bn,-d} \quad \Rightarrow \quad F \in \CS_{C_{\Gamma};-\bn,-d}
$$
However,~\eqref{eqn:key pairing finite sl} itself shows that
$$
  \big \langle (\text{ideal generated by LHS of \eqref{eqn:rel quantum affine 8}}),F \big \rangle = 0 \text{ for } 
  F \in \CV_{C_{\Gamma};-\bn,-d} \quad \Rightarrow \quad F \in \CT_{C_{\Gamma};-\bn,-d}
$$
Combining the two displays above shows the required fact that
\begin{multline*}
  \big \langle  (\text{ideal generated by } X^{k_1,k_2,\ell_1,\ell_2,\ell_3,a,b} \text{ and LHS of \eqref{eqn:rel quantum affine 8}-\eqref{eqn:rel quantum affine 9}}), F \big \rangle = 0 \\ 
  \text{ for } F \in \CV_{C_{\Gamma};-\bn}  \qquad \Rightarrow \qquad F \in \CS_{C_{\Gamma};-\bn}
\end{multline*}
This allows to complete the proof of \eqref{eqn:fact 1 cd 2} just like in the proof of~\eqref{eqn:key pairing finite sl}.

\medskip 
\noindent 
In order to prove \eqref{eqn:fact 2 c}, we will adapt the proof of Lemma \ref{lem:pairing sl}. As explained in Subsection \ref{sub:specialization patterns}, we only need to describe the asps of Definitions \ref{def:specialization pattern}-\ref{def:acceptable specialization pattern} in the case at hand, and to show that they are closed under extensions. To this end, we define acceptability with respect to positive real numbers \eqref{eqn:rho} such that 
\begin{equation}
\label{eqn:greater c}
  \rho_1, \dots, \rho_{N-1} \ll \rho_N, 
\end{equation}
meaning that while the relative order of $\rho_1,\dots,\rho_{N-1}$ is irrelevant, $\rho_N$ should be much larger than all of them. Because of the assumption \eqref{eqn:greater c}, any asp rooted at color $\leq N-1$ will actually be constrained to vertices of colors $1,\dots,N-1$, and is thus nothing more than a finite type shrub. As for asps rooted at color $N$, their full classification is depicted in Figures \ref{fig:asps C 1}-\ref{fig:asps C 2} and we will now explain how to read the pictures. The parameter decorating any arrow is completely determined by its endpoints according to the rule
$$
  (\text{lattice point }(x,y)) \xrightarrow{q^{x'+y'-x-y}} (\text{lattice point }(x',y'))
$$
(the lattice is normalized so that the vertical arrows in the pictures on the left carry the parameter $q^2$). The root of the asp is depicted by the blue disk, and the arrows indicate the behavior of the asp in colors $N-2,N-1,N$. The black arrows carry the same information as a shrub $S$ rooted at the black disk, corresponding to a type $A$ decorated Dynkin diagram with vertices
\begin{equation}
\label{eqn:small c}
  1,\dots,N-1, *
\end{equation}
where $*$ is an even vertex. The colored arrows in Figures \ref{fig:asps C 1}-\ref{fig:asps C 2} just show how one can extend the type $A$ shrub $S$ in order to obtain a type $C$ asp; the only restriction is that if there are $k$ (respectively $k'$) color $N$ vertices an even (respectively odd) number of steps southwest of the blue disk and $\ell$ (respectively $\ell'$) color $N-1$ vertices an even (respectively odd) number of steps southwest of the black disk, we should have $k \geq \ell$ and $k' < \ell'$.

\begin{figure}[h]
  \includegraphics[scale=1]{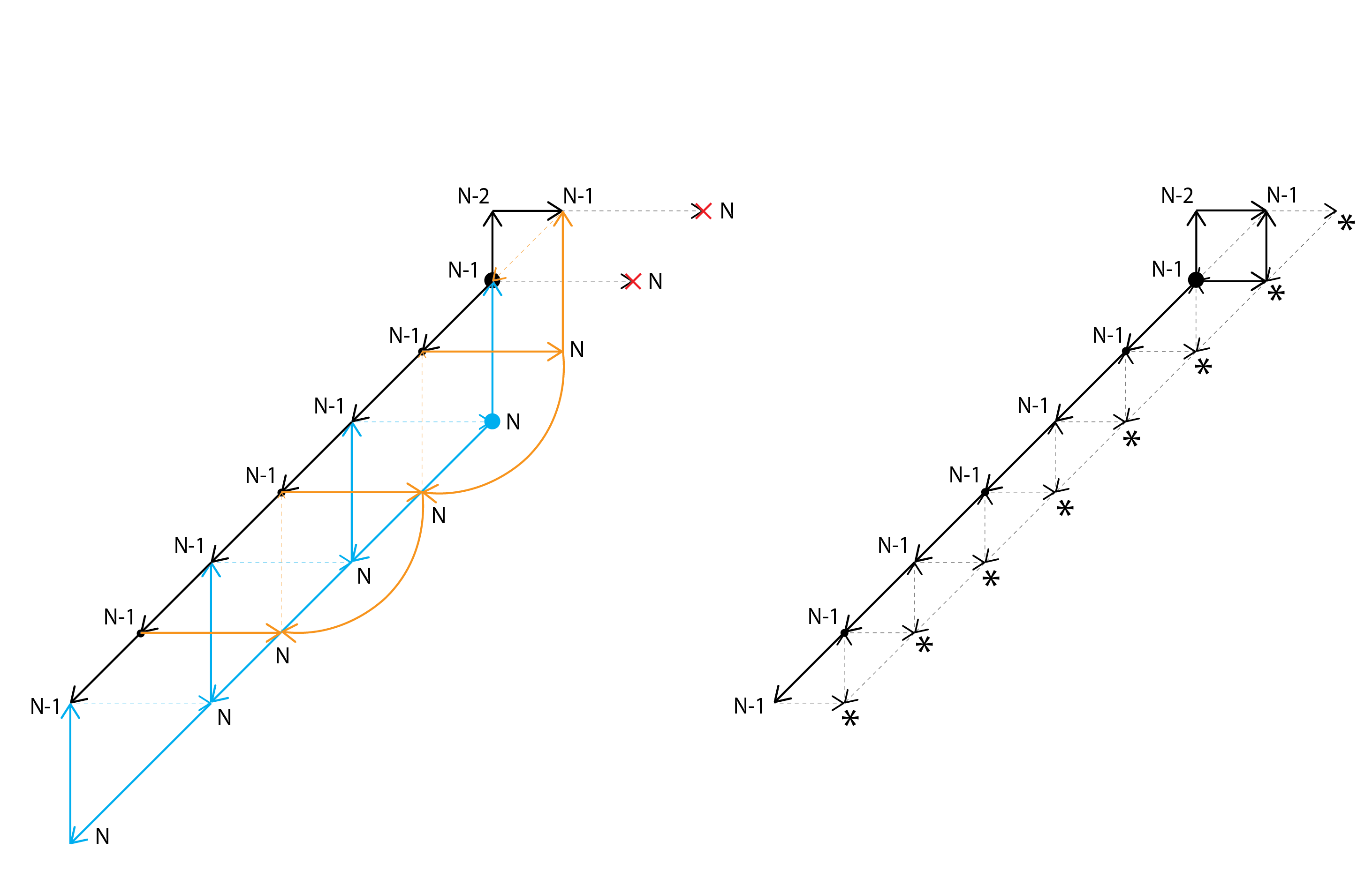} 
  \caption{Type $C$ Case 1: asps rooted at the blue disk on the left, the corresponding finite type shrub on the vertex set \eqref{eqn:small c} on the right (the small square on the top of the picture on the right cannot be extended further due to the $\textcolor{red}{\boldsymbol{\times}}$ in the picture on the left).}
  \label{fig:asps C 1} 
\end{figure}

\begin{figure}[h]
  \includegraphics[scale=1]{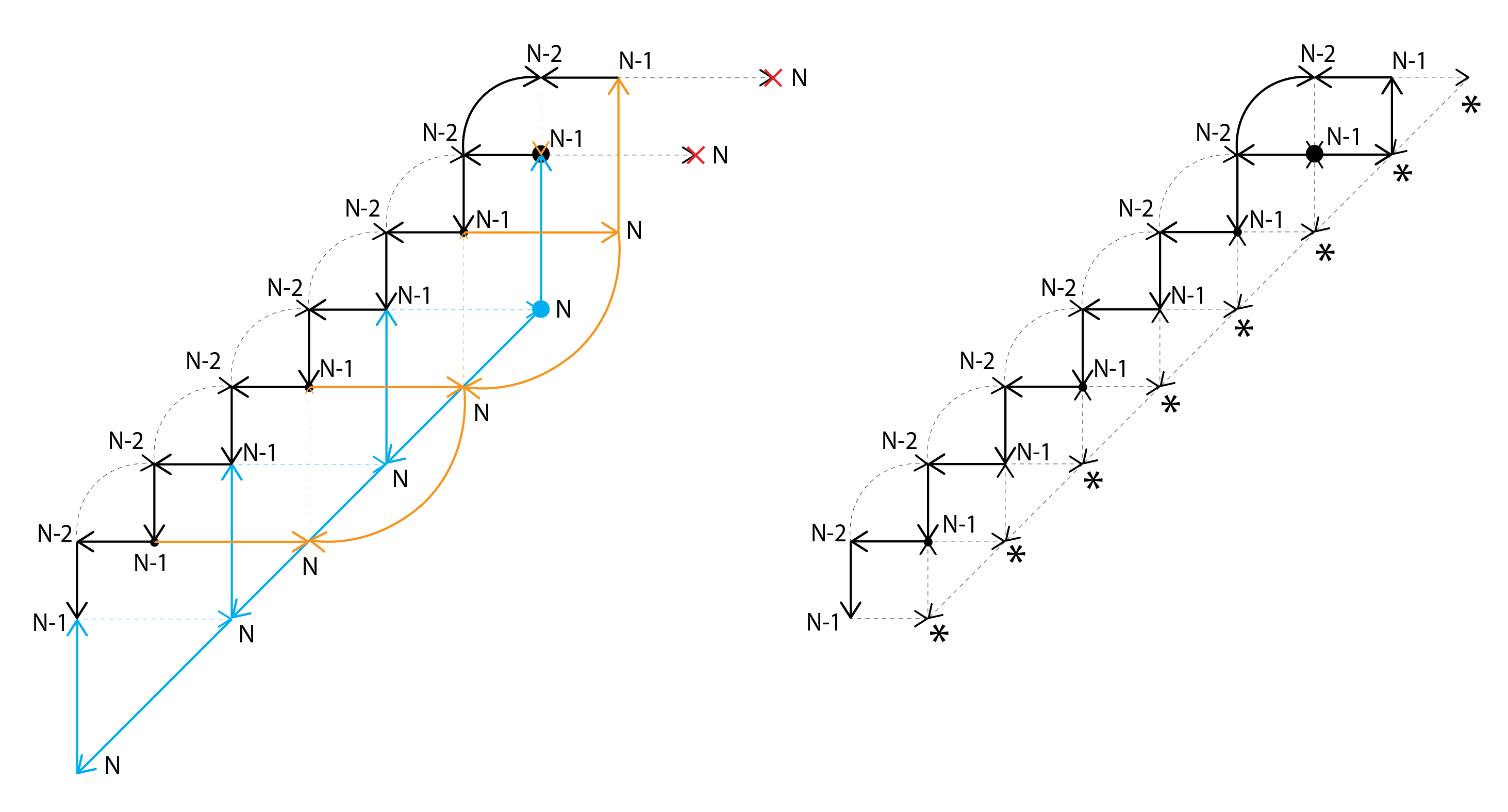} 
  \caption{Type $C$ Cases 2-3: asps rooted at the blue disk on the left, the corresponding finite type shrub on the vertex set \eqref{eqn:small c} on the right. The curved arrows indicate a straight northeast pointing arrow if $|N-2|=0$ (Case 3) or an up-then-right zig-zag if $|N-2|=1$ (Case 2).}
  \label{fig:asps C 2}
\end{figure}

\medskip
\noindent 
Note that whenever two solid arrows come into a vertex $v$, there is always a dotted arrow coming out of that vertex: this indicates that the apparent double pole corresponding to $v$ in the residue \eqref{eqn:residue} is actually reduced to a simple pole due to various wheel conditions, much like the example of Figure \ref{fig:shrub} in type $A$. This is a quite non-trivial fact in the case of Figure~\ref{fig:asps C 2}, which is depicted in Figure \ref{fig:scary}: the apparent double pole caused by the two arrows incoming into the color $N-2$ vertices on top is reduced to a simple pole due to the wheel conditions \eqref{eqn:wheel c 2} and \eqref{eqn:wheel c 3}; in the latter case, this is due to Claim \ref{claim:seven}. Finally, the red $\textcolor{red}{\boldsymbol{\times}}$ indicate vertices that cannot be part of any asp due to the length $3$ wheels they would form with a dotted arrow connected vertex of color $N-1$ and the vertex of color $N$ underneath it.

\begin{figure}[h]
  \includegraphics[scale=1.5]{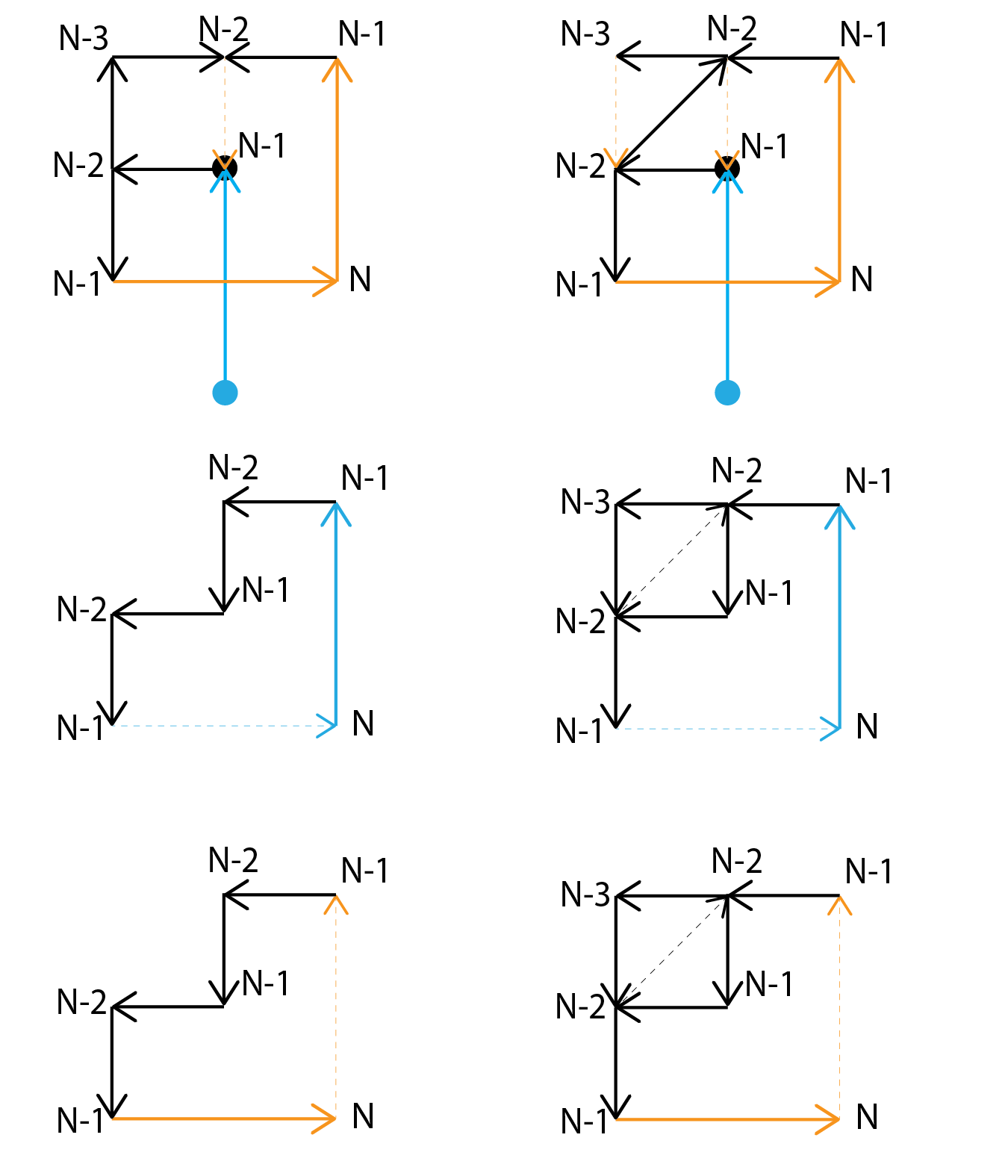} 
  \caption{The importance of length $6$ and $7$ wheels near the black disk or its southwestern neighbors in Figure \ref{fig:asps C  2} (Case 2 on the left, Case 3 on the right).}
  \label{fig:scary}
\end{figure}

\medskip 
\noindent
Note that we never invoked the acceptability condition \eqref{eqn:greater c}: to be precise, the notion of ``asps" only involves those pictures among Figures \ref{fig:asps C 1}-\ref{fig:asps C 2} such that for every vertex $v$ of color $N$ other than the blue disk, the product of arrow parameters from the blue disk to $v$ has absolute value strictly greater than $1$. By cleverly choosing the absolute value of $q$ itself, one could drastically reduce the number of asps. For instance, if we choose $|q|>1$, a type $C$ asp can not contain any of the color $N$ vertices southwest of the blue disk, and thus also can not contain any of the color $N-1$ vertices at least two steps southwest of the black disk. The reason why we do not make this simplifying assumption at the present moment is that we will need to avoid it for the toroidal cases that will be treated in Section \ref{sec:osp-toroidal}. 
\end{proof}

\medskip

\begin{claim}
\label{claim:seven}
Consider any partition
$$
  \Big\{z_{N-3,1}, z_{N-2,1}q^{\mp 1}, z_{N-2,2}q^{\pm 1}, z_{N-1,1}q^{\pm 2} ,z_{N-1,2}q^{\mp 2},z_{N-1,3},z_{N1}\Big\} 
  = A \sqcup B
$$
such that $z_{N-2,1}q^{\mp 1} \in A$ and $z_{N-2,2}q^{\pm 1} \in B$. Then any 
$$
  F\Big(z_{N-3,1}, z_{N-2,1}, z_{N-2,2}, z_{N-1,1},z_{N-1,2},z_{N-1,3},z_{N1}, \dots \Big) \in \CS^-_{C_{\Gamma}}
$$
has the property that
\begin{equation}
\label{eqn:divisibility}
  (x-y)^2 \text{  divides  } F \Big|_{(\text{variables in }A) \leadsto x, (\text{variables in }B) \leadsto y} 
\end{equation}
\end{claim}

\medskip 
\noindent 
Note that definition \eqref{eqn:rigorous} corresponds to \eqref{eqn:divisibility} with $B = \{z_{N-2,2}q^{\pm 1}\}$. The relevance of the condition in Claim \ref{claim:seven} is that various choices of $A$ and $B$ ensure various vanishings of color-symmetric Laurent polynomials at the vertices of various asps. For example, the pictures on the right in Figure~\ref{fig:scary} correspond respectively to
\begin{equation}\label{eq:B-set}
\begin{split}
  & B = \{z_{N-2,2}q^{\pm 1}, z_{N-1,1}q^{\pm 2}, z_{N-1,2}q^{\mp 2}, z_{N-1,3}, z_{N1}\} \\
  & B = \{z_{N-2,2}q^{\pm 1}, z_{N-1,1}q^{\pm 2}\} \\
  & B = \{z_{N-2,2}q^{\pm 1}, z_{N-1,1}q^{\pm 2}, z_{N1}\}
\end{split}
\end{equation}

\medskip

\begin{proof} \emph{of Claim~\ref{claim:seven}:} 
The only key feature of $F \in \CS^-_{C_{\Gamma}}$ that we will use is that it vanishes at two length $3$ wheels corresponding to the small triangles in \eqref{eqn:wheel c 3}. To make the presentation clearer, let us rename the variables according to
\begin{multline*}
  \Big(z_{N-3,1}, z_{N-2,1}q^{\mp 1}, z_{N-2,2}q^{\pm 1}, z_{N-1,1}q^{\pm 2} ,z_{N-1,2}q^{\mp 2}, z_{N-1,3},z_{N1}\Big) \\ = (c,a,b,s,t,d,u)
\end{multline*}
while using $\dots$ for other variables that will not be important to us. With this in mind, it suffices to prove the following claim: for any Laurent polynomial $F$ which vanishes when $a=b=c$ and when $a=b=d$, conditions \eqref{eqn:divisibility} are equivalent to each other for any partition
\begin{equation}
\label{eq:partition}
  \{a,b,c,d,s,t,u\} = A \sqcup B
\end{equation}
such that $a \in A$ and $b \in B$. Any homogeneous element $F \in \CS^-_{C_\Gamma}$ satisfies $F(x,x,x,x,x,x,x,\dots)=0$ due to the length $3$ wheel conditions, and therefore 
$$
  \sum_{\gamma \in \{a,b,c,d,s,t,u\}} \frac {\partial F}{\partial \gamma}(x,x,x,x,x,x,x,\dots) = 0
$$
Hence, the divisibility condition \eqref{eqn:divisibility} is equivalent to
\begin{equation}
\label{eqn:divisibility foiled}
  \sum_{\gamma \in A} \frac {\partial F}{\partial \gamma}(x,x,x,x,x,x,x,\dots) =  
  \sum_{\gamma \in B} \frac {\partial F}{\partial \gamma}(x,x,x,x,x,x,x,\dots)  = 0
\end{equation}
However, for any set of variables $C \subset \{a,b,c,d,s,t,u\}$ which contains either $\{a,b,c\}$ or $\{a,b,d\}$, the length $3$ wheel conditions imply that
\begin{multline*}
  F|_{(\text{variables in }C)\leadsto x} = 0 \quad \Rightarrow \quad 
  \sum_{\gamma \in C} \frac {\partial F}{\partial \gamma}\Big|_{(\text{variables in }C)\leadsto x} = 0 \quad \\ 
  \Rightarrow \quad \sum_{\gamma \in C} \frac {\partial F}{\partial \gamma} (x,x,x,x,x,x,x,\dots) = 0
\end{multline*}
Therefore, we have 
\begin{equation*}
\begin{split}
  & \frac {\partial F}{\partial \gamma} (x,x,x,x,x,x,x,\dots) = 0 \qquad \forall\, \gamma \in \{c,d,s,t,u\}, \\
  & \frac {\partial F}{\partial a} (x,x,x,x,x,x,x,\dots) + \frac {\partial F}{\partial b} (x,x,x,x,x,x,x,\dots) = 0
\end{split}
\end{equation*}
This implies that the divisibility conditions \eqref{eqn:divisibility}, interpreted as~\eqref{eqn:divisibility foiled}, are in fact equivalent to 
$$
  \frac {\partial F}{\partial a} (x,x,x,x,x,x,x,\dots) = \frac {\partial F}{\partial b} (x,x,x,x,x,x,x,\dots) = 0
$$
for any partitions~\eqref{eq:partition} with $A \ni a$ and $B \ni b$. Since the choice $B = \{b\}$ precisely corresponds to $F \in \CS^-_{C_{\Gamma}}$, we conclude that \eqref{eqn:divisibility} holds for all $A$ and $B$ as above.
\end{proof}


\bigskip 

\section{Quantum toroidal superalgebras of types $B, C, D$}
\label{sec:osp-toroidal}

\medskip 

In this Section, we will consider quantum loop superalgebras associated to affine $B,C,D$ type decorated Dynkin diagrams (and their generalizations, as we will explain in Subsection \ref{sub:extended}) and prove the shuffle realization from Theorem \ref{thm:toroidal intro}.


\medskip 

\subsection{Extended decorated diagrams}
\label{sub:extended}

Recall that the decorated Dynkin diagrams of finite type $B,C,D$ consist of gluing together the following:

\medskip 

\begin{itemize}[leftmargin=0.7cm]

\item 
a \emph{body} $\Gb$, i.e.\ a type $A$ decorated Dynkin diagram with vertices $1,\dots,N-1$ (in types $B,C$) or $1,\dots,N-2$ (in type $D$). We call the largest vertex the~\emph{neck};

\medskip 

\item 
a \emph{head} $\Gh$, i.e.\ an edge from the neck to vertex $N$ (in types $B,C$) or a fork from the neck to vertices $N-1,N$ (in type $D$), see the diagrams \eqref{eqn:type b 1 head}-\eqref{eqn:type d 2 head}.

\end{itemize}

\medskip 
\noindent
Similarly, an \emph{extended decorated Dynkin diagram} $\Gamma$ of affine type $B,C,D$ consists of a type $A$ decorated Dynkin diagram $\Gb$ (a body) to which we glue $\Gh',\Gh''$ (two heads) at the leftmost and rightmost ends, respectively. If $X,Y \in \{B,C,D\}$ are the types of these heads, then we shall refer to all objects associated to $\Gamma$ as being of type
\begin{equation}
\label{eqn:extended diagram}
  \widehat{XY}_{\Gamma} 
\end{equation}
An extended decorated Dynkin diagram has two \emph{neck} vertices where the body is connected to the two heads, and thus two neck parameters connecting the respective necks with the outermost vertices. As in finite types $B,C,D$, the existence and parameter of the loop on each neck (and hence its parity) is determined by comparing the neck parameter and the parameter of the edge in the body adjacent to the neck. The existence of the latter is due to the following.

\medskip 
\noindent 
\textbf{Assumption:} the two necks of an extended decorated Dynkin diagram are different, and moreover in type $CC$, the body has at least one vertex between the two necks.

\medskip 
\noindent
All decorated Dynkin diagrams associated to type $B,C,D$ affine Lie superalgebras are of the form \eqref{eqn:extended diagram}, but our notion of extended decorated Dynkin diagrams is actually more general. For instance, it includes situations which are prohibited for affine superalgebras, such as type $BB$, or type $CC$ with neck parameters involving different signs (cf.~\cite[\S1.4]{Y} that provides Dynkin diagrams of type $XY$ with $X,Y\in \{B,C,D\}$ that correspond to twisted affine Lie superalgebras). The vertex set of the diagram $\Gamma$ will be denoted by 
$$
  \tI = \{0,\dots,N\}
$$
The following are possible extended decorated Dynkin diagrams:
\begin{equation}
\label{eqn:dynkin bb}
\text{Type $BB$:} \begin{tikzcd}
  0 && 1 & \cdots & {N-1} && N
  \arrow["q^{\pm 1}"', no head, from=1-7, to=1-5]
  \arrow[no head, from=1-5, to=1-4]
  \arrow[no head, from=1-4, to=1-3]
  \arrow["q^{\pm' 1}"', no head, from=1-3, to=1-1]
  \arrow["q^{\mp 1}"', no head, from=1-7, to=1-7, looseness=3.3]
  \arrow[dashed, no head, from=1-5, to=1-5, looseness=3.3]
  \arrow[dashed, no head, from=1-3, to=1-3, looseness=4.5]
  \arrow["q^{\mp' 1}"', no head, from=1-1, to=1-1, looseness=4.5]
\end{tikzcd}
\end{equation}
\begin{equation}
\label{eqn:dynkin bc}
\text{Type $BC$:} \begin{tikzcd}
  0 && 1 & \cdots & {N-1} && N
  \arrow["q^{\pm 2}"', no head, from=1-7, to=1-5]
  \arrow[no head, from=1-5, to=1-4]
  \arrow[no head, from=1-4, to=1-3]
  \arrow["q^{\pm' 1}"', no head, from=1-3, to=1-1]
  \arrow["q^{\mp 4}"', no head, from=1-7, to=1-7, looseness=3.3]
  \arrow[dashed, no head, from=1-5, to=1-5, looseness=3.3]
  \arrow[dashed, no head, from=1-3, to=1-3, looseness=4.5]
  \arrow["q^{\mp' 1}"', no head, from=1-1, to=1-1, looseness=4.5]
\end{tikzcd}
\end{equation}
\begin{equation}
\label{eqn:dynkin bd}
\text{Type $BD$:} \quad \begin{tikzcd}
  &&&&& N-1 \\
  0 && 1 & \cdots & {N-2} \\
  &&&&& N
  \arrow[no head, from=2-5, to=2-4]
  \arrow[no head, from=2-4, to=2-3]
  \arrow[dashed, no head, from=2-4, to=2-3]
  \arrow["q^{\pm' 1}"', no head, from=2-3, to=2-1]
  \arrow["q^{\pm 1}"', no head, from=1-6, to=2-5]
  \arrow["q^{\pm 1}"', swap, no head, from=3-6, to=2-5]
  \arrow["q^{\mp 2}"', no head, from=3-6, to=1-6]
  \arrow["q^{\mp' 1}"', no head, from=2-1, to=2-1, looseness=4.5]
  \arrow[dashed, no head, from=2-3, to=2-3, looseness=4.5]
  \arrow[dashed, no head, from=2-5, to=2-5, looseness=3.3]
\end{tikzcd}
\end{equation}
\begin{equation}
\label{eqn:dynkin cc}
\text{Type $CC$:} \begin{tikzcd}
  0 && 1 & \cdots & {N-1} && N
  \arrow["q^{\pm 2}"', no head, from=1-7, to=1-5]
  \arrow[no head, from=1-5, to=1-4]
  \arrow[no head, from=1-4, to=1-3]
  \arrow["q^{\pm' 2}"', no head, from=1-3, to=1-1]
  \arrow["q^{\mp 4}"', no head, from=1-7, to=1-7, looseness=3.3]
  \arrow[dashed, no head, from=1-5, to=1-5, looseness=3.3]
  \arrow[dashed, no head, from=1-3, to=1-3, looseness=4]
  \arrow["q^{\mp' 4}"', no head, from=1-1, to=1-1, looseness=4]
\end{tikzcd}
\end{equation}
\begin{equation}
\label{eqn:dynkin cd}
\text{Type $CD$:} \quad \begin{tikzcd}
  &&&&& N-1 \\
  0 && 1 & \cdots & {N-2} \\
  &&&&& N
  \arrow[no head, from=2-5, to=2-4]
  \arrow[no head, from=2-4, to=2-3]
  \arrow[dashed, no head, from=2-4, to=2-3]
  \arrow["q^{\pm' 2}"', no head, from=2-3, to=2-1]
  \arrow["q^{\pm 1}"', no head, from=1-6, to=2-5]
  \arrow["q^{\pm 1}"', swap, no head, from=3-6, to=2-5]
  \arrow["q^{\mp 2}"', no head, from=3-6, to=1-6]
  \arrow["q^{\mp' 4}"', no head, from=2-1, to=2-1, looseness=4.5]
  \arrow[dashed, no head, from=2-3, to=2-3, looseness=4.5]
  \arrow[dashed, no head, from=2-5, to=2-5, looseness=3.3]
\end{tikzcd}
\end{equation}
\begin{equation}
\label{eqn:dynkin dd}
\text{Type $DD$:} \quad \begin{tikzcd}
0 &&&& N-1 \\
& 2 & \cdots & {N-2} & \\
1 &&&& N
  \arrow["q^{\pm' 1}", no head, from=1-1, to=2-2]
  \arrow["q^{\pm' 1}", swap, no head, from=3-1, to=2-2]
  \arrow["q^{\pm 1}", swap, no head, from=1-5, to=2-4]
  \arrow["q^{\pm 1}", no head, from=3-5, to=2-4]
  \arrow["q^{\mp 2}", swap, no head, from=3-5, to=1-5]
  \arrow["q^{\mp' 2}", no head, from=3-1, to=1-1]
  \arrow[no head, from=2-4, to=2-3]
  \arrow[no head, from=2-3, to=2-2]
  \arrow[dashed, no head, from=2-2, to=2-2, looseness=4.5]
  \arrow[dashed, no head, from=2-4, to=2-4, looseness=3.3]
\end{tikzcd}
\end{equation}
The signs $\pm'$ and $\pm$ are independent of each other. In types $B$ and $D$ we may replace either of the heads represented above (which are \eqref{eqn:type b 1 head} and \eqref{eqn:type d 2 head}) by \eqref{eqn:type b 2 head} and \eqref{eqn:type d 1 head}, respectively. One also has types $CB$, $DB$ and $DC$ by reflecting the above diagrams left-to-right.

\medskip

\begin{definition}
\label{def:special}
Let $Q \in q^{\BZ}$ be the product of all the horizontal edge parameters in $\Gamma$ (for each type $D$ head, we include the corresponding neck parameter exactly once in the product), times a single factor of $q^{\pm 1}$ in those types $BB$ where both heads have a loop parameter of $q^{\pm 1}$. 

\medskip 
\noindent 
An extended decorated Dynkin diagram is called \emph{special} if it is of types $BB, CC, DD$, $CD, DC$ and $Q = 1$, and \emph{non-special} otherwise.
\end{definition}

\medskip
\noindent 
Special extended decorated Dynkin diagrams are the type $B,C,D$ analogues of type $\wA$ with $m=n$, which we have already seen to involve long wheels \eqref{eqn:wheel special 1}-\eqref{eqn:wheel special 2}, dually resulting in clunky relations \eqref{eqn:rel quantum toroidal special} that are not easy to write down. We will see the analogous behavior for special extended decorated Dynkin diagrams in Subsection~\ref{sub:wheel toroidal}.


\medskip 

\subsection{Quantum toroidal superalgebras in types $B,C,D$}
\label{sub:toroidal}

For the remainder of the present Section, we fix a non-special extended decorated Dynkin diagram $\Gamma$ as in~\eqref{eqn:extended diagram}. We note that cutting off the head containing $0$ (or $0,1$ if $X$ is of type $D$) provides a decorated Dynkin diagram of type $Y_{\Gar}$, where $\Gar$ denotes $\Gamma$ with its leftmost head $\Gh'$ removed, while leaving behind the neck. Likewise, cutting off $N$ (or $N-1,N$ if $Y$ is of type $D$) provides a decorated Dynkin diagram of type $X_{\Gal}$, where $\Gal$ denotes $\Gamma$ with its rightmost head $\Gh''$ removed, while leaving behind the neck. For such finite types, we wrote down the corresponding zeta factors in Section~\ref{sec:osp}, which are compatible across the body:
$$
  \zeta^{X_{\Gal}}_{ij}(x) = \zeta^{Y_{\Gar}}_{ij}(x) = \zeta^{A_{\Gb}}_{ij}(x) \qquad \forall\, i,j\in \Gb
$$
With this in mind, we introduce zeta functions of type~\eqref{eqn:extended diagram} as follows:
\begin{equation}
\label{eq:zeta-toroidal}    
\begin{split}
  & \zeta^{\wXYG}_{ij} (x) =  \zeta^{Y_{\Gar}}_{ij}(x) \qquad \forall\, i,j\in \Gar \\
  & \zeta^{\wXYG}_{ij} (x) =  \zeta^{X_{\Gal}}_{ij}(x) \qquad \forall\, i,j\in \Gal \\
  & \zeta^{\wXYG}_{ji} (x) =1 , \quad \zeta^{\wXYG}_{ij} (x) = (-1)^{|i||j|} \quad \forall\, 
    i\in \Gh'\backslash\{\mathrm{neck}\}, 
    j\in \Gh''\backslash\{\mathrm{neck}\}
\end{split}
\end{equation}
with the parity of the nodes in the heads defined explicitly in Section~\ref{sec:osp}.

\medskip
\noindent
We define the pre-quantum toroidal superalgebra of type $\wXYG$ as before:
\begin{equation*}
  \widetilde{U}_q(L\fg_{\wXYG}) = \BC \Big \langle e_{i,d}, f_{i,d}, \ph_{i,d'}^\pm \Big \rangle_{i \in \tI, d \in \BZ, d' \geq 0} \Big/ 
  \Big ( \text{relations \eqref{eqn:rel quantum affine 1}-\eqref{eqn:rel quantum affine 6}} \Big )
\end{equation*}
where the relations above are defined with respect to the zeta functions of~\eqref{eq:zeta-toroidal}. To keep our notations short, we shall use $\big\{ \relR \big\}$ to denote all the relations from Section~\ref{sec:osp} that enter the definition of $U_q(L\fg_{Y_{\Gar}})$ as an explicit quotient of $\widetilde{U}_q(L\fg_{Y_{\Gar}})$, and likewise we shall use $\big\{ \relL \big\}$ to denote all similar relations entering the definition of $U_q(L\fg_{X_{\Gal}})$ as an explicit quotient of $\widetilde{U}_q(L\fg_{X_{\Gal}})$. With this in mind, we make the following definition.

\medskip

\begin{definition}
\label{def:toroidal}
For a non-special extended decorated Dynkin diagram $\Gamma$, we define the quantum toroidal superalgebra of type $\wXYG$ as
\begin{equation}
\label{eqn:toroidal}
  U_q(L\fg_{\wXYG}) = \widetilde{U}_q(L\fg_{\wXYG}) \Big/ \Big( \relR , \relL \Big)
\end{equation}
\end{definition}


\medskip 

\subsection{Wheel conditions for extended decorated Dynkin diagrams}
\label{sub:wheel toroidal}

Consider the big shuffle algebra associated to~\eqref{eq:zeta-toroidal}:
$$
  \CV^+_{\wXYG}= \bigoplus_{\bn \in \tnn} \BC[z_{i1},\dots,z_{in_i}]^{\sym}_{i \in \tI} \cdot \twist
$$
with $\twist$ defined analogously to~\eqref{eq:root-twist}. Recall from Subsection~\ref{sub:shuffle intro} all possible wheels of types $Y_{\Gar}$ and $X_{\Gal}$, both of which include the length $3$ and $4$ wheels over the body $A_{\Gb}$. We will refer to the totality of all such wheels as type $\wXYG$ wheels.

\medskip

\begin{definition}
\label{def:extended shuffle}    
For a non-special extended decorated Dynkin diagram $\Gamma$, we define 
\begin{equation*}
  \CS^+_{\wXYG} = \Big\{ F \ \mathrm{satisfying\ all\ type} \ \wXYG \ \mathrm{wheel \ conditions}  \Big \} \subset  \CV^+_{\wXYG} 
\end{equation*}
\end{definition}

\medskip 
\noindent 
Special extended decorated Dynkin diagrams require additional wheel conditions (and therefore dually, the quantum toroidal superalgebras \eqref{eqn:toroidal} require additional relations, which we do not discuss in the present paper), akin to \eqref{eqn:wheel special 1}-\eqref{eqn:wheel special 2}. For example:

\medskip

\begin{itemize}[leftmargin=0.7cm]

\item 
the type $CC$ diagram 
$$
\begin{tikzcd}
	0 && 1 && 2 && 3 && 4
	\arrow["q^2", no head, from=1-1, to=1-3]
    \arrow["q", no head, from=1-3, to=1-5]
    \arrow["q^{-1}", no head, from=1-5, to=1-7]
	\arrow["q^{-2}", no head, from=1-7, to=1-9]
    \arrow["q^{-4}"', no head, from=1-1, to=1-1, looseness=4]
    \arrow["q^{-2}"', no head, from=1-3, to=1-3, looseness=4]
    \arrow["q^{2}"', no head, from=1-7, to=1-7, looseness=4]
    \arrow["q^4"', no head, from=1-9, to=1-9, looseness=4]
\end{tikzcd}
$$
requires the wheel 
$$
\begin{tikzcd}
	{z_{01}} && {z_{11}} & {z_{21}} \\
	&&& {z_{31}} \\
	{z_{12}} \\
	{z_{22}} & {z_{32}} && {z_{41}}
	\arrow["{q^2}", from=1-1, to=1-3]
	\arrow["q", from=1-3, to=1-4]
	\arrow["{q^{-1}}", from=1-4, to=2-4]
	\arrow["{q^{-2}}", from=2-4, to=4-4]
	\arrow["{q^2}", from=3-1, to=1-1]
	\arrow["q", from=4-1, to=3-1]
	\arrow["{q^{-1}}", from=4-2, to=4-1]
	\arrow["{q^{-2}}", from=4-4, to=4-2]
\end{tikzcd}
$$

\item 
the type $DD$ diagram 
$$
\begin{tikzcd}
0 &&&&&& 5 \\
& 2 && 3 && 4 & \\
1 &&&&&& 6
  \arrow["q", no head, from=2-2, to=1-1]
  \arrow["q", swap, no head, from=2-2, to=3-1]
  \arrow["q^{- 1}", no head, from=1-7, to=2-6]
  \arrow["q^{- 1}", swap, no head, from=3-7, to=2-6]
  \arrow["q^{2}", swap, no head, from=3-7, to=1-7]
  \arrow["q^{- 2}", no head, from=3-1, to=1-1]
  \arrow["q^{-1}", swap, no head, from=2-6, to=2-4]
  \arrow["q", swap, no head, from=2-4, to=2-2]
  \arrow["q^{-2}"', no head, from=2-2, to=2-2, looseness=4]
  \arrow["q^{2}"', no head, from=2-6, to=2-6, looseness=4]
\end{tikzcd}
$$
requires the wheel (the two $\boxed{2}$ in the diagram below refer to the fact that one needs to impose order $2$ vanishing conditions on Laurent polynomials, akin to~\eqref{eqn:rigorous}, which we will not need to make explicit in the present paper)
$$
\begin{tikzcd}
	& {z_{11}} && {z_{22}} &&& {z_{32}} \\
	{z_{01}} & \boxed{2} \\
	\\
	{z_{21}} &&&&&& {z_{42}} \\
	\\
	&&&&& \boxed{2} & {z_{51}} \\
	{z_{31}} &&& {z_{41}} && {z_{61}}
	\arrow["q", from=1-2, to=1-4]
	\arrow["q", from=1-4, to=1-7]
	\arrow["{q^{-1}}", from=1-7, to=4-7]
	\arrow["q", swap, from=2-1, to=1-4]
	\arrow["q", swap, from=4-1, to=1-2]
	\arrow["q", from=4-1, to=2-1]
	\arrow["{q^{-1}}", from=4-7, to=6-7]
	\arrow["{q^{-1}}"', from=4-7, to=7-6]
	\arrow["{q^{-1}}"', from=6-7, to=7-4]
	\arrow["q", from=7-1, to=4-1]
	\arrow["{q^{-1}}"', from=7-4, to=7-1]
	\arrow["{q^{-1}}", from=7-6, to=7-4]
\end{tikzcd}
$$

\item 
the type $BB$ diagram 
$$
\begin{tikzcd}
  0 && 1 && 2 && 3
  \arrow["q"', no head, from=1-7, to=1-5]
  \arrow["q^{-1}", swap, no head, from=1-5, to=1-3]
  \arrow["q"', no head, from=1-3, to=1-1]
  \arrow["q^{- 1}"', no head, from=1-7, to=1-7, looseness=4]
  \arrow["q^{- 1}"', no head, from=1-1, to=1-1, looseness=4]
\end{tikzcd}
$$
requires the wheel 
$$
\begin{tikzcd}
	&& {z_{01}} &&&&&& \\
	\\
	{z_{02}} &&&& {z_{12}} \\
	\\
	&& {z_{11}} &&&& {z_{21}} \\
	\\
	&&&& {z_{22}} &&&& {z_{31}} \\
	\\
	&&&&&& {z_{32}}
	\arrow["{q^{-1}}"', from=1-3, to=3-1,]
    \arrow["q", from=3-1, to=3-5, pos=0.7]
	\arrow["{q^{-1}}"', from=3-5, to=7-5, pos=0.7]
	\arrow["q", from=5-3, to=1-3,  pos=0.7]
	\arrow["{q^{-1}}", from=5-7, to=5-3, swap, pos=0.7]
	\arrow["q", from=7-5, to=7-9, pos=0.7]
	\arrow["{q^{-1}}", from=7-9, to=9-7]
	\arrow["q", from=9-7, to=5-7, pos=0.7]
\end{tikzcd}
$$

\end{itemize}

\medskip
\noindent
None of these \emph{special} wheel conditions is implied by the finite type wheel conditions on the two heads or the body of the respective extended decorated Dynkin diagram.


\medskip 

\subsection{The shuffle realization for toroidal superalgebras}
\label{sub:proof toroidal}

We now recall Theorem \ref{thm:toroidal intro}, which states that for a non-special extended decorated Dynkin diagram $\Gamma$, the algebra homomorphism 
$$
  \tUpsilon^+_{\wXYG}\colon \widetilde{U}^+_q(L\fg_{\wXYG}) \to \CV^+_{\wXYG}, \qquad e_{i,d}\mapsto z^d_{i1} , \quad \forall\, i\in \tI, d\in \BZ
$$
of~\eqref{eqn:tilde upsilon} induces an isomorphism 
$$
  \Upsilon^+_{\wXYG}\colon U^+_q(L\fg_{\wXYG}) \iso \CS^+_{\wXYG}
$$

\medskip

\begin{proof} 
\emph{of Theorem \ref{thm:toroidal intro}:} 
The proof of the toroidal analogues of~\eqref{eqn:fact 1 cd 1},~\eqref{eqn:fact 1 b},~\eqref{eqn:fact 1 cd 2} is completely identical to the proofs already given. It therefore remains to prove the analogue of Lemma \ref{lem:pairing sl} which, as we explained in Subsection \ref{sub:specialization patterns}, boils down to describing the asps and showing that they are closed under extensions.

\medskip
\noindent
We will first deal with the case when the diagram is of type $CC, CD$ or $DD$. In this case, we define the acceptability condition with respect to positive real numbers
\begin{equation}
\label{eqn:greater toroidal}
  \rho_0 \gg \rho_1, \dots, \rho_{N-1} \ll \rho_N = \rho_0
\end{equation}
(if the left-most head is of type $D$, we replace the above by $\rho_0 = \rho_1 \gg \rho_2$; if the right-most head is of type $D$ then we replace above by $\rho_{N-2}\ll \rho_{N-1}=\rho_N=\rho_0$). The relative order of the $(\rho_i)_{i\in \Gb}$ will be irrelevant. Finally, we assume that the product $Q$  from Definition~\ref{def:special} satisfies
\begin{equation}
\label{eqn:Q}
  |Q| < 1
\end{equation}
which is possible due to the assumption that $\Gamma$ is non-special and $Q\in q^{\BZ}$. Because of the choice \eqref{eqn:greater toroidal}, any asp rooted at a color in the body $\Gb$ will be constrained to the body, and is thus nothing more than a finite type shrub. It therefore remains to classify asps rooted at the right-most head $N$; the case of the root being at the left-most head $0$ is analogous. These are indicated in Figures \ref{fig:even head} and \ref{fig:odd head} in the cases when the left-most neck is even or odd, respectively. We note that these figures are drawn in the case when the right-most neck is even and of type $D$; the other cases will require small modifications at colors $N-2,N-1,N$, but will not affect the behavior at colors $0,1,2$, which is where the subsequent argument plays out.

\begin{figure}[h]
  \includegraphics[scale=0.6]{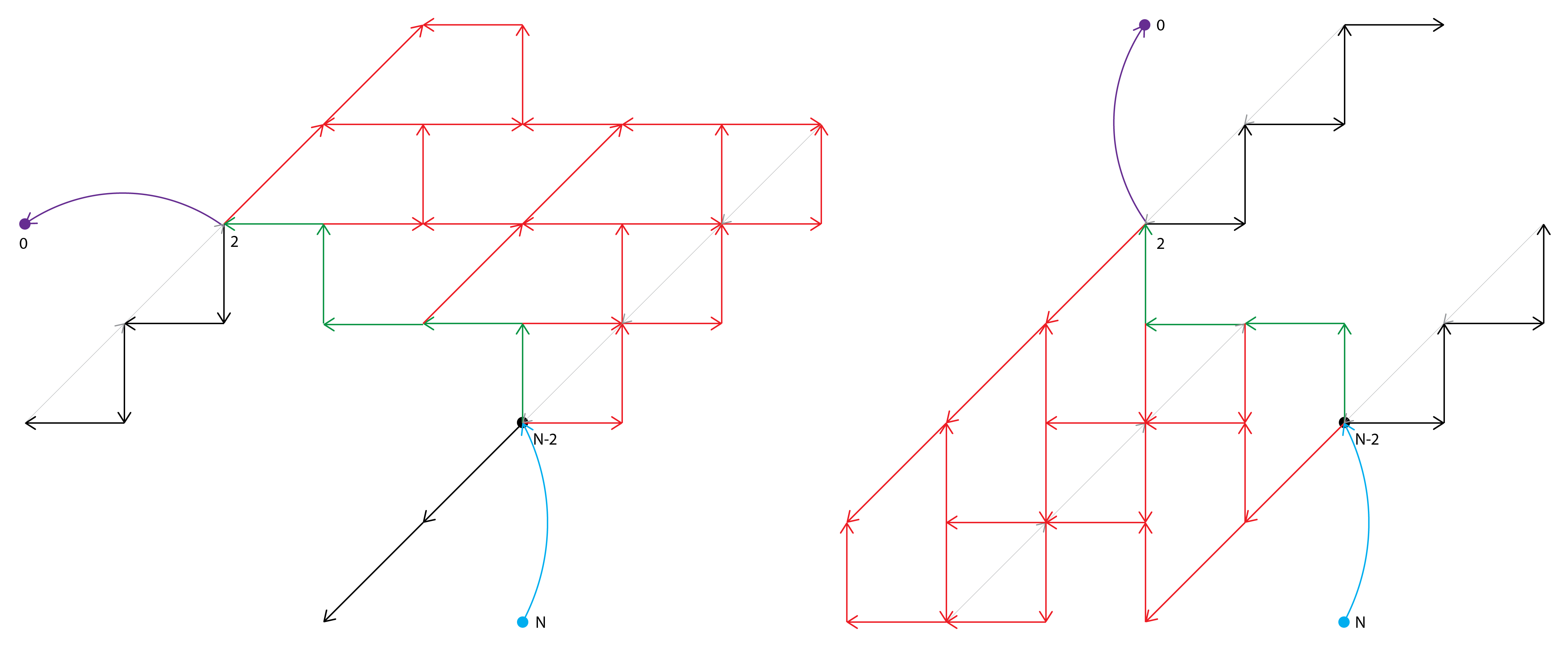} 
  \caption{The possible asps rooted at the right-most head, for a type $DD$ diagram with neck $|N-2|=0$ (the case of $|N-2|=1$, as well as types $CC$, $CD$, $DC$ are analogous), if the neck $|2|=0$.}
  \label{fig:even head}
\end{figure}

\begin{figure}[h]
  \includegraphics[scale=0.6]{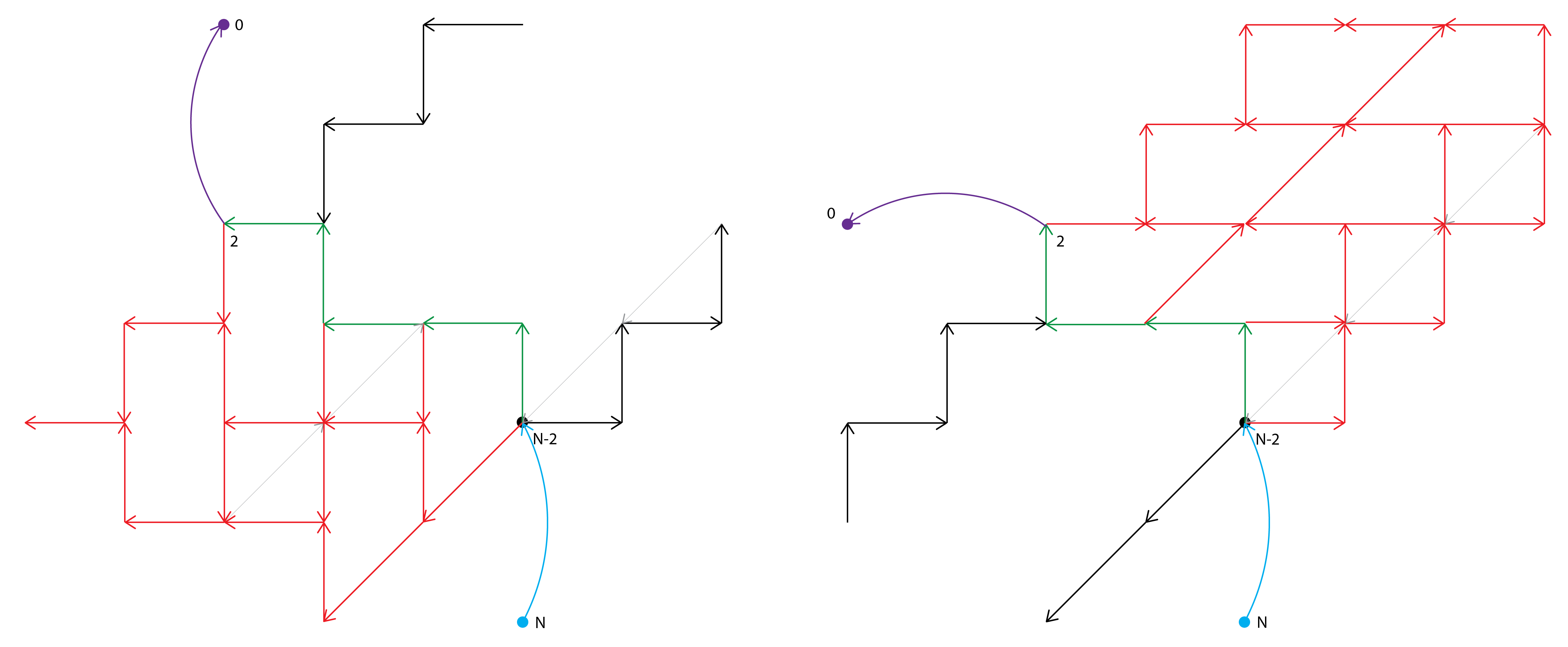} 
  \caption{The possible asps rooted at the right-most head, for a type $DD$ diagram with neck $|N-2|=0$ (the case of $|N-2|=1$, as well as types $CC$, $CD$, $DC$ are analogous), if the neck $|2|=1$.}
  \label{fig:odd head}
\end{figure}

\medskip 
\noindent 
Let us explain how to read Figures \ref{fig:even head}-\ref{fig:odd head}: strictly speaking, the subsequent explanation only applies when the left-most head of color $0$ is of type $D$, but we will explain afterwards how to modify the argument for a type $C$ left head. As shown in the proof of Theorem~\ref{thm:d intro} (respectively Theorem~\ref{thm:c intro}), any type $D$ (respectively type $C$) asp rooted at the color $N$ blue disk consists of a chain of color $N$, $N-1$ vertices (respectively color $N$ vertices) attached to a finite type shrub $S$ rooted at the color $N-2$ (respectively $N-1$) black disk above; note that in type $B$, we actually obtain two non-interacting shrubs instead of $S$, and this case will be discussed at the end of the proof. The behavior of the shrub $S$ between the two necks precisely follows the classification in Proposition \ref{prop:shrubs sl}: in particular, the green arrows represent the irregular directrix that points out from the black disk toward color $2$. The shrub $S$ is depicted in red in Figures \ref{fig:even head}-\ref{fig:odd head}, and we claim that it is constrained only to one side of the green directrix: 

\medskip 

\begin{itemize}[leftmargin=0.7cm]

\item 
when the left-most neck $|2|=0$ (depicted in Figure \ref{fig:even head}), this is because the shrub condition would be violated if $S$ contained any of the black zig-zags at colors $2$ and $3$; this forbids the shrub from containing any color $2$ vertices on the respective side of the green directrix;

\medskip 

\item 
when the left-most neck $|2|=1$ (depicted in Figure \ref{fig:odd head}), this is because Corollary \ref{cor:shrubs sl} does not allow 
the bounding ladders (zig-zags or straight rays) of any shrub to point in opposite directions (as do the two black ladders in the figure), given that one of these ladders lies on the regular directrix of the shrub $S$.

\end{itemize}

\medskip 
\noindent
Therefore, the only vertices of color $2$ in the shrub $S$ are the vertex $v$ lying at the end of the green directrix (indicated by ``$2$" in Figures \ref{fig:even head} and \ref{fig:odd head}) and its neighbors either only northeast or only southwest, as indicated in the respective figures. If we wished to extend such a shrub to the head of color $0$ (the case of color $1$ is analogous) of the extended decorated Dynkin diagram, which is indicated by a violet curved arrow in the figures, length $3$ wheel conditions (if $|2|=0$) or length $4$ wheel conditions (if $|2|=1$) only allow such an extension to happen from the vertex $v$. However, in this case, \eqref{eqn:Q} prohibits the extension due to the acceptability condition \eqref{eqn:greater toroidal}. We conclude that no asp can involve both heads of the extended decorated Dynkin diagram, and thus the classification of asps in the case at hand is nothing more than that of Theorems \ref{thm:d intro} and \ref{thm:c intro}.

\medskip 
\noindent 
Let us now assume the left-most head is of type $C$. In this case, we can extend the shrub $S$ toward color 0 not only from the end-point of the green directrix, but also from its immediate neighbor on the red diagonal (Case 1) or from its same-color neighbor on the red zig-zag (Cases 2-3). This is depicted in Figure \ref{fig:neighbor}, which also explains why such an asp could not be extended any further from color $0$: from the violet disk we cannot go diagonally to the cyan $\textcolor{cyan}{\boldsymbol{\times}}$ due to length $3$ wheel condition, not directly down to the orange $\textcolor{orange}{\boldsymbol{\times}}$ due to length $4$, $6$ or $7$ wheel condition depending on whether the head is Case 1, Case 2 or Case 3 (the latter is quite non-trivial and is based on Claim~\ref{claim:seven} applied to $B = \{z_{N-2,2}q^{\pm 1}, z_{N-1,1}q^{\pm 2}, z_{N-1,2}q^{\mp 2}, z_{N1}\}$).

\begin{figure}[h]
  \includegraphics[scale=0.75]{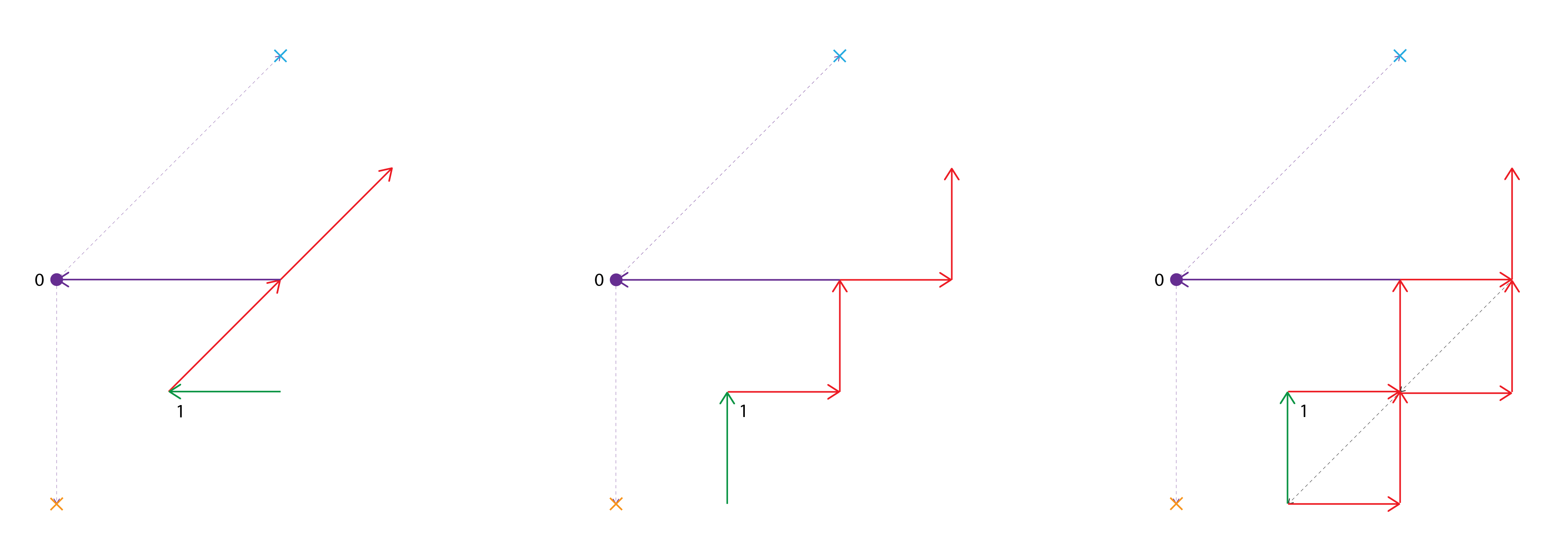} 
  \caption{Reaching a type $C$ head (Case 1 on the left, Case 2 in the middle, Case 3 on the right); the green arrow is the northwestern most green arrow on the directrix of Figures~\ref{fig:even head} and~\ref{fig:odd head}. If the violet arrow points vertically, the treatment is analogous.} 
  \label{fig:neighbor}
\end{figure}

\medskip 
\noindent 
We will now deal with the case of type $BC$ and $BD$ diagrams, all of which are non-special. Without loss of generality, we assume that the left-most head is of type $B$, and that the corresponding vertex $0$ has a loop with parameter $q$ (in Case 2, it also has a loop with parameter $q^{-2}$). With this in mind, we define the acceptability condition with respect to positive real numbers 
\begin{equation}
\label{eqn:greater toroidal 2}
  \rho_0 \gg \rho_1, \dots, \rho_{N-1} \ll \rho_N = \frac {\rho_0}r
\end{equation}
(if the right-most head is of type $D$, we replace $\rho_{N-1} \ll \rho_N$ above by $\rho_{N-2} \ll \rho_{N-1} = \rho_N=\frac {\rho_0}{r}$) for some positive real number $r$ which satisfies the property:

\medskip 

\begin{itemize}[leftmargin=0.7cm]

\item $r=1$ if $Q \neq 1 \neq Qq$; 

\medskip 

\item $|q^{-1}| < r < 1$ if $Qq = 1$, which requires us to assume $|q|>1$;

\medskip 

\item $|q| < r < 1$ if $Q=1$, which requires us to assume $|q|<1$.
    
\end{itemize}

\medskip 
\noindent 
If $Q \neq 1 \neq Qq$, then we may choose the absolute value of $q$ such that $|Q|, |Qq| < 1$. With these choices in mind, let us consider any asp which starts at either head (asps which start in the body of the diagram are dealt with as in the previous case). Just like in the case of types $CC, CD, DD$ treated above, such an asp can never reach the other head; the only exception is if the other head is of type $C$, in which case our treatment should follow the one surrounding Figure \ref{fig:neighbor}. In any case, such asps are closed under extensions. On the other hand, if either $Q=1$ or $Qq=1$, it is elementary to show that the choices in the above bulleted list ensure that
\begin{equation}
\label{eqn:uno}
  |Q| < r \quad \text{ or } \quad |Qq|<r
\end{equation}
and 
\begin{equation}
\label{eqn:dos}
  |Q| < 1/r \quad \text{and} \quad  |Qq| < 1/r
\end{equation}
Property \eqref{eqn:uno} together with the acceptability condition~\eqref{eqn:greater toroidal 2} ensures that any shrub which starts at the right-most (type $C$ or $D$) head either can never reach the left-most (type $B$) head at color $0$, or $Q=1$ holds and the shrub can reach color $0$ only at the violet disk in Figure \ref{fig:bc head}. However, from there the asp cannot extend anywhere: the violet $\textcolor{violet}{\boldsymbol{\times}}$ is prohibited due to the assumption $|Qq|< r$ and the acceptability condition~\eqref{eqn:greater toroidal 2}, while the orange $\textcolor{orange}{\boldsymbol{\times}}$ and cyan $\textcolor{cyan}{\boldsymbol{\times}}$ are prohibited by length $3$ and $4$ wheel conditions. Meanwhile, property \eqref{eqn:dos} ensures that any shrub which starts at the left-most head (type $B$) can never reach the right-most (type $C$ or $D$) head at color $N$, according to the characterization of type $B$ asps in the proof of Theorem \ref{thm:b intro}.

\begin{figure}[h]
  \includegraphics[scale=1]{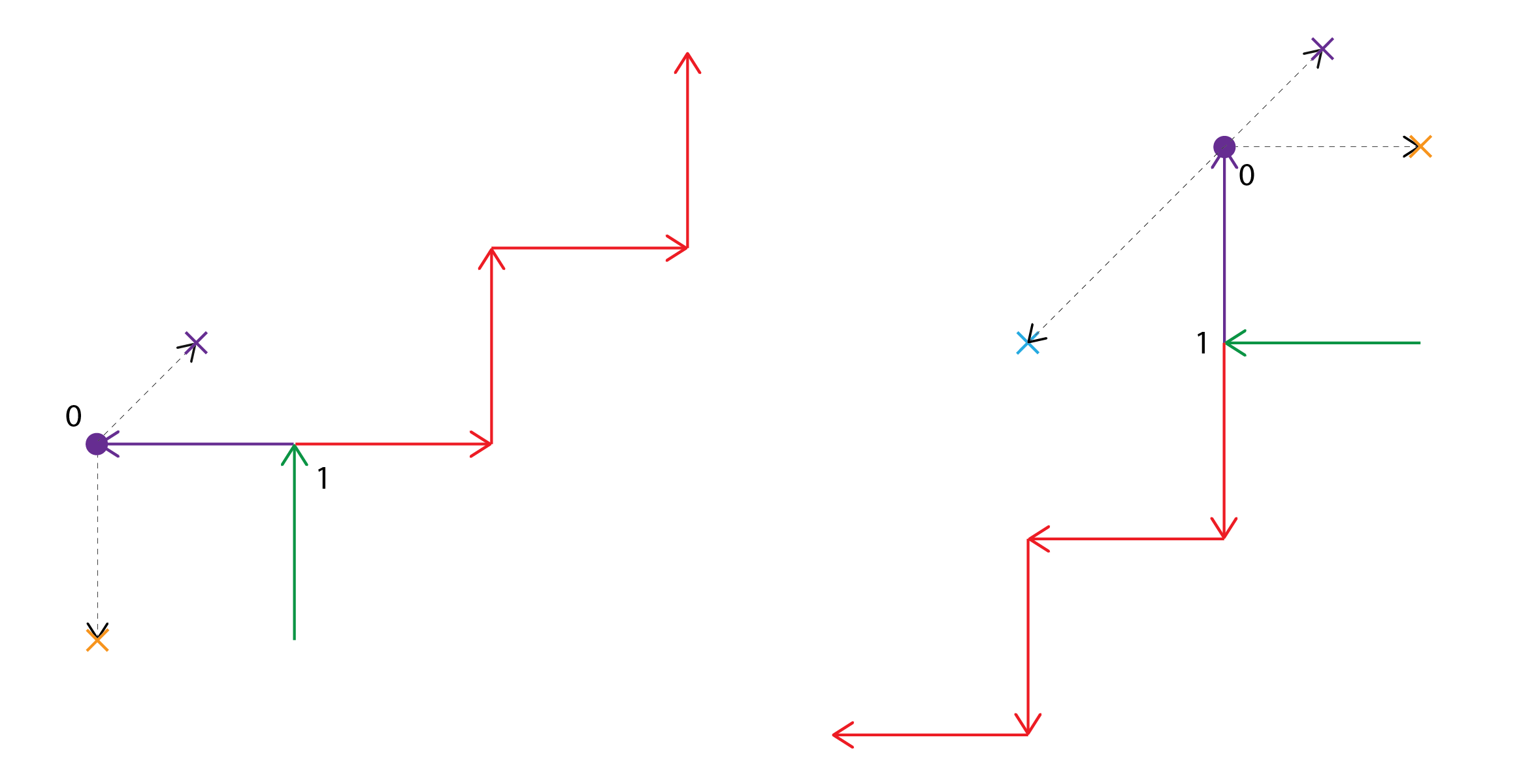} 
  \caption{Reaching a type $B$ head (Case 1 on the left, Case 2 on the right) from a type $C$ or $D$ root; the green arrow is the northwestern most green arrow on the directrix of Figures~\ref{fig:even head} and~\ref{fig:odd head}. We do not depict the various reflections of the diagram, or its natural analogue when $|1|=0$, which are treated analogously.}
  \label{fig:bc head}
\end{figure}

\medskip 
\noindent 
Finally, we consider the case when both heads are of type $B$, and we will assume that the heads have loop parameters $q$ and $q' \in \{q,q^{-1}\}$ in vertices $N$ and $0$, respectively (there may also be loops with parameters $q^{-2}$ at $N$ and $(q')^{-2}$ at $0$). We define the acceptability condition with respect to positive real numbers $\rho_i$ as in \eqref{eqn:greater toroidal}, and recall that the diagram $\Gamma$ being non-special implies that
\begin{align*} 
  Qq \neq 1 & \text{ if } q = q' \\
  Q \neq 1 & \text{ if } q \neq q'
\end{align*}
In the former case, we assume $|Qq| < 1$, which means that an asp that starts at vertex $N$ (the case of asps starting at $0$ is analogous)  can only reach color $0$ at the violet disk in Figure \ref{fig:bb head 1}. However, from there we cannot move anywhere else: the violet $\textcolor{violet}{\boldsymbol{\times}}$ is prohibited due to the assumption $|Qq|<1$, while the orange $\textcolor{orange}{\boldsymbol{\times}}$ and cyan $\textcolor{cyan}{\boldsymbol{\times}}$ are prohibited by length $3$ or $4$ wheel conditions. Note that when $|1|=1$, the source of the green arrow preceding the violet one would exactly provide the fourth variable in the aforementioned length $4$ wheel.

\begin{figure}[h]
  \includegraphics[scale=0.75]{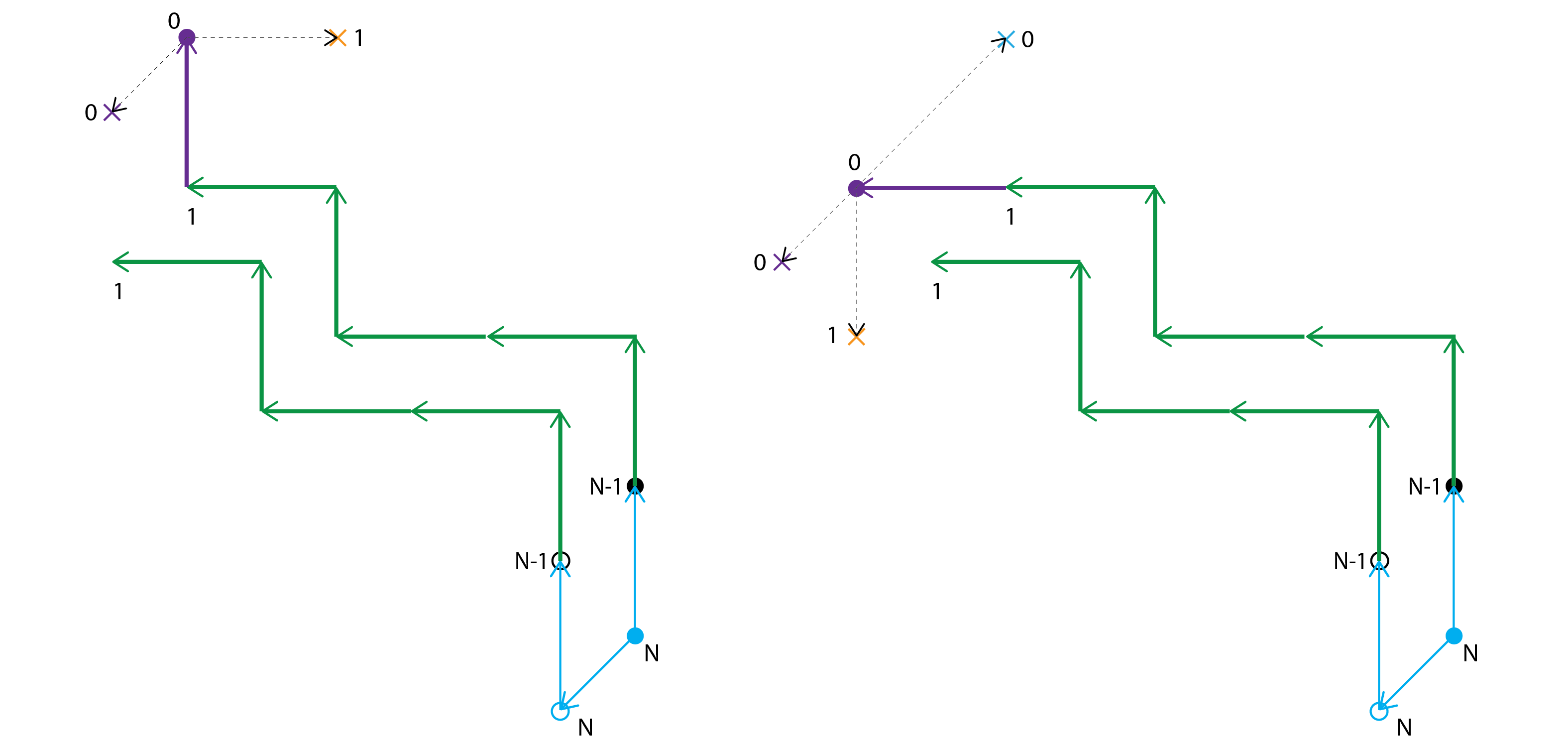} 
  \caption{Reaching a type $B$ head from a type $B$ root (Case 1 on the left and Case 2 on the right). We do not depict the various reflections of the above diagrams, which are analogous.}
  \label{fig:bb head 1}
\end{figure}

\medskip 
\noindent 
We are left with the case $q \neq q'$. As $Q\ne 1$, we may assume without loss of generality that $Q\in q^{\BZ_{<0}}$ (the case $Q\in q^{\BZ_{>0}}=(q')^{\BZ_{<0}}$ is treated analogously). We shall now assume that $|q|>1$, so that $|Q|, |q'| < 1$. 
This means that an asp starting at vertex $0$ can never reach color $N$, while the asp that starts at vertex $N$ can only reach color $0$ at the hollow violet disk in Figure \ref{fig:bb head 2}. However, from there we cannot move anywhere else: the violet $\textcolor{violet}{\boldsymbol{\times}}$ is prohibited due to the assumption $|Q|<1$, while the orange $\textcolor{orange}{\boldsymbol{\times}}$ and cyan $\textcolor{cyan}{\boldsymbol{\times}}$ are prohibited by length $3$ or $4$ wheel conditions. Note that when $|1|=1$, the source of the green arrow preceding the violet one would exactly provide the fourth variable in the aforementioned length $4$ wheel.

\begin{figure}[h]
  \includegraphics[scale=0.75]{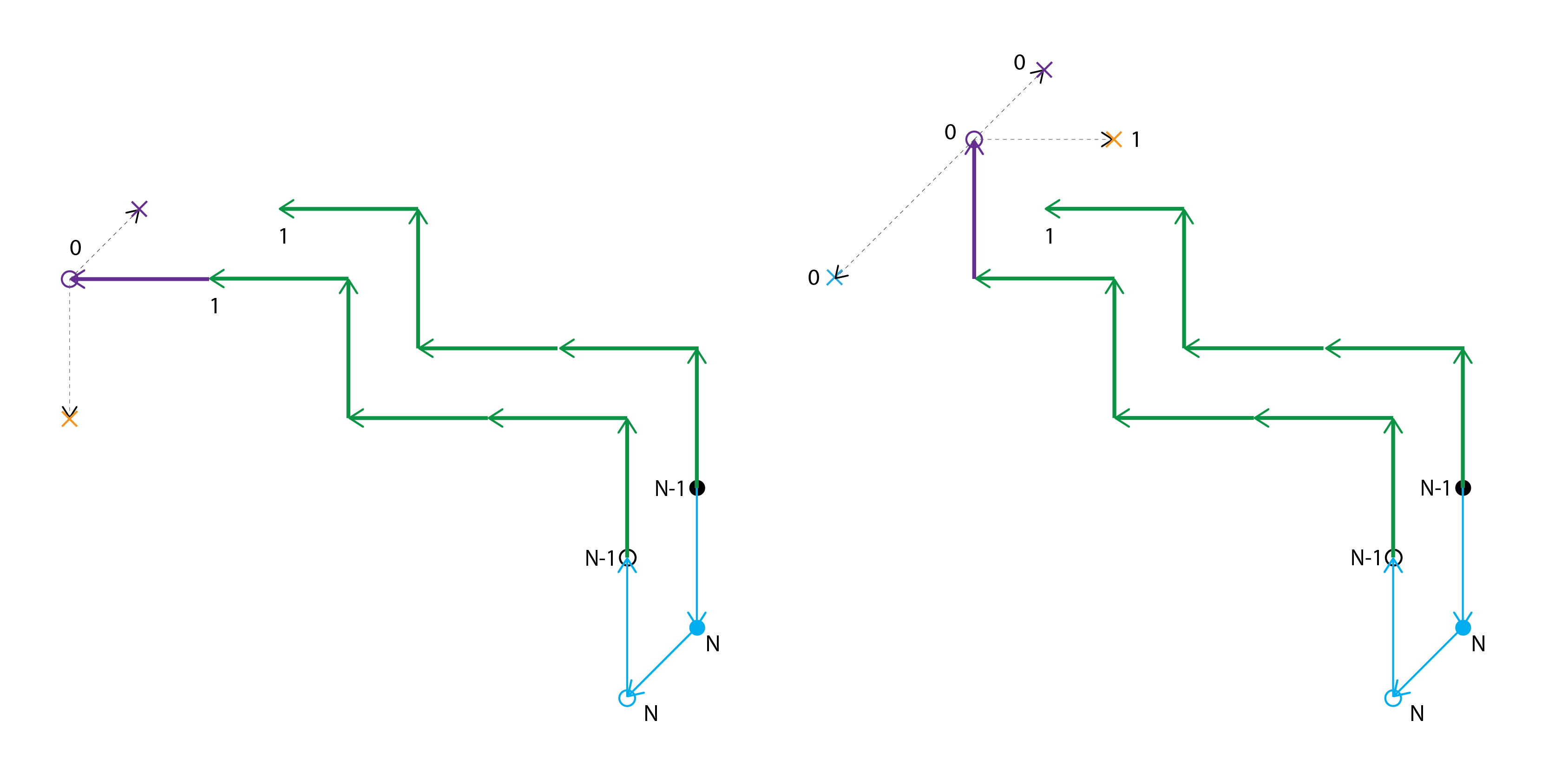} 
  \caption{Reaching a type $B$ head from a type $B$ root (Case 1 on the left and Case 2 on the right). We do not depict the various reflections of the above diagrams, which are analogous.}
  \label{fig:bb head 2}
\end{figure}

\end{proof}


\bigskip

\appendix


\section{Small rank cases}\label{sec:small-N}

\noindent
The quiver approach provides a uniform way to define $\UslG$ for $N\leq 3$, while their standard definitions in literature have more complicated structure. We set
$$ 
  q_1=q^{-1}d,\quad q_2=q^2,\quad q_3=q^{-1}d^{-1}
$$


\medskip

\subsection{$N=1$} 
We may assume $\bs = (1)$ by the invariance $\bs \leftrightarrow -\bs$ discussed at the end of Subsection \ref{sub:super lie}, and shall ignore the unique index $0 \in \widehat{I}_1$ from the notation. The corresponding algebra is commonly known as the quantum toroidal algebra of $\fgl_1$, see~\cite{M}. The corresponding quiver $Q$ is depicted in Figure~\ref{fig:c3}, and the zeta function is given by 
\begin{equation}
\label{eqn:zeta affine gl n=1}
  \zeta(x)=x^{-1}\frac{(1-xq_1)(1-xq_2)(1-xq_3)}{1-x} 
\end{equation}
resulting in the following quadratic relation (cf.~\eqref{eqn:quad intro}):
$$
  e(z)e(w)(z-wq_1)(z-wq_2)(z-wq_3)=-e(w)e(z)(w-zq_1)(w-zq_2)(w-zq_3)
$$
Define the pre-quantum toroidal algebra of $\fgl_{1}$ by
$$
  \widetilde{U}_{q,d}(\widehat{\widehat{\fgl}}_{1}) =  
  \BC \Big \langle e_{d}, f_{d}, \ph_{d'}^\pm \Big \rangle_{d \in \BZ, d' \geq 0} \Big/ 
  \Big ( \text{relations \eqref{eqn:rel quantum affine 1}-\eqref{eqn:rel quantum affine 6}} \Big )
$$
with the relations above defined with respect to the zeta function $\zeta(x)$ of~\eqref{eqn:zeta affine gl n=1}. The corresponding quantum toroidal algebra of $\mathfrak{gl}_{1}$ is then defined by 
\begin{equation*}
  U_{q,d}(\widehat{\widehat{\fgl}}_{1}) = \widetilde{U}_{q,d}(\widehat{\widehat{\fgl}}_{1}) \Big / 
  \Big (\text{clunky relations } e_{\bigcirc} \text{ of~\eqref{eqn:clunky}, and their }  f\text{-versions} \Big )
\end{equation*}
associated with the length $3$ wheels $\bigcirc$ formed by two possible triangular faces of~$\tQ$ in Figure~\ref{fig:c3}:
\begin{equation}
\label{eq:wheel-pic-gl1}
		\begin{tikzcd}
			{z_{3}} && {z_{2}} &&&& {z_{2}} \\ \\
			{z_{1}} &&&& {z_{1}} && {z_{3}}
			\arrow["{q_{3}}"', from=1-1, to=3-1]
			\arrow["{q_{1}}"', from=1-3, to=1-1]
			\arrow["{q_{3}}", from=1-7, to=3-7]
			\arrow["{q_{2}}"', from=3-1, to=1-3]
			\arrow["{q_{2}}", from=3-5, to=1-7]
			\arrow["{q_{1}}", from=3-7, to=3-5]
		\end{tikzcd}
\end{equation}
As pointed out in~\cite{N Reduced, NSS}, this recovers the usual quantum toroidal algebra of $\mathfrak{gl}_1$ (see~\cite{M}) whose degree $3$ relation 
\begin{equation}
\label{eq:Miki}
  \Sym_{z_1,z_2,z_3}\, \frac{z_2}{z_3} [e(z_1), [e(z_2),e(z_3)]] = 0
\end{equation}
is equivalent to the above $e_{\bigcirc}$ modulo the quadratic relations~\eqref{eqn:quad intro}. As noted in \loccitt, the relation~\eqref{eq:Miki} is dual to both length $3$ wheels above: algebraically this is reflected by the equality~\eqref{eq:delta-gl1} and its consequence~\eqref{eq:wheel-gl1}, see Section~\ref{sec:arbitrary}.


\medskip

\subsection{$N=2$ and $\bs=(1,1)$} 

The corresponding quiver dictates the following choice   
\begin{equation*} 
  \zeta_{ij}(x) \sim  
  \begin{cases} 
    \frac {1-xq_2}{1-x} &\text{if } i = j  \\ 
    (1-xq_1)(1-xq_3) &\text{if } i \ne j
  \end{cases} 
\end{equation*}
To match with the standard quadratic relations for the quantum toroidal $\mathfrak{gl}_2$ (see~\cite{FJMM}): 
\begin{equation*}
\begin{split}
  e_i(z)e_i(w)(z-wq_2) &= - e_i(w)e_i(z)(w-zq_2), \ \qquad \qquad \forall \ i\in \wItwo \\
  e_i(z)e_j(w)(z-wq_1)(z-wq_3) &= e_j(w)e_i(z)(w-zq_1)(w-zq_3), \quad \forall \ i\ne j
\end{split}
\end{equation*}
we make the following explicit choice: 
\begin{equation} 
\label{eqn:zeta affine sl n=2}
  \zeta_{ij}(x) =
  \begin{cases} 
    \frac {1-xq_2}{1-x} &\text{if } i = j  \\ 
    x^{-1}(1-xq_1)(1-xq_3) &\text{if } i\ne j
  \end{cases} 
\end{equation}
Define the pre-quantum toroidal algebra of $\fgl_{2}$ by
$$
  \widetilde{U}_{q,d}(\widehat{\widehat{\fgl}}_{2}) =  \BC \Big \langle e_{i,d}, f_{i,d}, \ph_{i,d'}^\pm \Big \rangle_{i\in \wItwo, d \in \BZ, d' \geq 0} \Big/ 
  \Big (  \text{relations \eqref{eqn:rel quantum affine 1}-\eqref{eqn:rel quantum affine 6}} \Big )
$$
with the relations above defined with respect to the zeta functions $\zeta_{ij}(x)$ of~\eqref{eqn:zeta affine sl n=2}. The corresponding quantum toroidal algebra of $\mathfrak{gl}_{2}$ is then defined by 
\begin{equation*}
  U_{q,d}(\widehat{\widehat{\fgl}}_{2}) = \widetilde{U}_{q,d}(\widehat{\widehat{\fgl}}_{2}) \Big / 
  \Big (\text{clunky relations } e_{\bigcirc} \text{ of~\eqref{eqn:clunky}, and their }  f\text{-versions} \Big )
\end{equation*}
associated with the length $3$ wheels $\bigcirc$ formed by the corresponding triangular faces of~$\tQ$:
\begin{equation}
\label{eq:wheel-pic-gl2}
		\begin{tikzcd}
			{z_{j1}} && {z_{i2}} &&&& {z_{i2}} \\ \\
			{z_{i1}} &&&& {z_{i1}} && {z_{j1}}
			\arrow["{q_{3}}"', from=1-1, to=3-1]
			\arrow["{q_{1}}"', from=1-3, to=1-1]
			\arrow["{q_{3}}", from=1-7, to=3-7]
			\arrow["{q_{2}}"', from=3-1, to=1-3]
			\arrow["{q_{2}}", from=3-5, to=1-7]
			\arrow["{q_{1}}", from=3-7, to=3-5]
		\end{tikzcd}
\end{equation}
This recovers the usual quantum toroidal algebra of $\mathfrak{gl}_2$ (see~\cite{FJMM}) featuring the following degree $4$ relation (see~\cite{C}) for $i\ne j\in \{0,1\}$:
\begin{multline}
\label{eq:3Serre-sl2}
   \Sym_{z_1,z_2}\, \Big[
    e_i(z_1) e_i(z_2) e_j(w) (z_1(1+q_2^{-1}) - w(q_1+q_3))  \\
    -  e_i(z_1) e_j(w) e_i(z_2) (z_1(1+q_2) + z_2(1+q_2^{-1}) - w(q_1+q_3+q^{-1}_1+q^{-1}_3)) \\
     + e_j(w) e_i(z_1) e_i(z_2) e_j(w) (z_2(1+q_{2}) - w(q^{-1}_1+q^{-1}_3))   \Big] = 0 
\end{multline}
This relation is dual to both length $3$ wheels in~\eqref{eq:wheel-pic-gl2}: algebraically this is reflected by the equality~\eqref{eq:delta-gl2} and its consequence~\eqref{eq:wheel-gl2}, see the discussion in Section~\ref{sec:arbitrary}.


\medskip

\subsection{$N=2$ and $\bs=(1,-1)$} 
In this case, the key assumption~\eqref{eq:assumption} fails as we have a length $2$ path (part of the square face) starting from color $1$ going one step to the right and then one step down ending again in color $1$ and having product of arrow parameters equal to $1$. With this in mind, we redefine the arrow parameters by:
$$
  t_{\uparrow}=q^{-1}d^{-1}=q_3, \quad t_{\downarrow}=qd=q_3^{-1}, \quad t_{\leftarrow}=q^{-1}d=q_1, \quad t_{\rightarrow}=qd^{-1}=q^{-1}_1
$$
thus resulting in the following zeta functions
\begin{equation*} 
  \bar{\zeta}_{ij}(x) \sim  
  \begin{cases} 
    \frac{1}{1-x} &\text{if } i = j  \\ 
    (1-xq_3)(1-xq^{-1}_3) &\text{if } i=0,  j=1 \\
    (1-xq_1)(1-xq^{-1}_1) &\text{if } i=1,  j=0 
  \end{cases} 
\end{equation*}
We shall make the particular normalization 
\begin{equation} 
\label{eqn:zeta affine sl 11}
  \bar{\zeta}_{ij}(x) =
  \begin{cases} 
    x^{\frac{1}{2}} \frac {1}{1-x} &\text{if } i = j  \\ 
    -x^{-1}(1-xq_3)(1-xq^{-1}_3) &\text{if } i=0,  j=1 \\
    x^{-1}(1-xq_1)(1-xq^{-1}_1) &\text{if } i=1,  j=0 
  \end{cases} 
\end{equation}
so that the quadratic relations of~\eqref{eqn:quad intro}, with $\bar{\zeta}_{ij}(x)$ instead of $\zeta_{ij}(x)$,  match those in~\cite[\S 4.4]{NW}:
\begin{equation*}
\begin{split}
  e_i(z)e_i(w) &= - e_i(w)e_i(z), \qquad \forall\, i\in \wItwo \\
  e_1(z)e_0(w)(z-wq_3)(z-wq^{-1}_3) &= - e_0(w)e_1(z)(w-zq_1)(w-zq^{-1}_1) 
\end{split}
\end{equation*}
We define the pre-quantum toroidal superalgebra of $\mathfrak{gl}_{1|1}$ by
$$
  \widetilde{U}_{q,d}(\widehat{\widehat{\fgl}}_{1|1}) =  
  \BC \Big \langle e_{i,d}, f_{i,d}, \ph_{i,d'}^\pm \Big \rangle_{i\in \wItwo, d \in \BZ, d' \geq 0} \Big/ 
  \Big ( \text{relations \eqref{eqn:rel quantum affine 1}-\eqref{eqn:rel quantum affine 6}} \Big )
$$
with the relations above defined with respect to the zeta functions $\bar{\zeta}_{ij}(x)$ of~\eqref{eqn:zeta affine sl 11}. The corresponding quantum toroidal superalgebra of $\mathfrak{gl}_{1|1}$ is then defined by 
\begin{equation*}
  U_{q,d}(\widehat{\widehat{\fgl}}_{1|1}) = \widetilde{U}_{q,d}(\widehat{\widehat{\fgl}}_{1|1}) \Big / 
  \Big (\text{clunky relations } e_{\bigcirc} \text{ of~\eqref{eqn:clunky}, and their }  f\text{-versions} \Big )
\end{equation*}
associated with the length $4$ wheels $\bigcirc$ formed by the following square faces of~$\tQ$:
\begin{equation}
\label{eqn:wheels 1|1}
		\begin{tikzcd}
			{z_{11}} && {z_{01}} \\ \\
			{z_{02}} && {z_{12}}
			\arrow["{q_1^{-1}}", from=1-1, to=1-3]
			\arrow["{q_3^{-1}}", from=1-3, to=3-3]
			\arrow["{q_3}", from=3-1, to=1-1]
			\arrow["{q_1}", from=3-3, to=3-1]
		\end{tikzcd} \qquad  \qquad
        \begin{tikzcd}
			{z_{02}} && {z_{11}} \\ \\
			{z_{12}} && {z_{01}}
			\arrow["{q_1}", swap, to=1-1, from=1-3]
			\arrow["{q_3}", swap, to=1-3, from=3-3]
			\arrow["{q_3^{-1}}", swap, to=3-1, from=1-1]
			\arrow["{q_1^{-1}}", swap, to=3-3, from=3-1]
		\end{tikzcd}
\end{equation}
By the natural analogue of Theorem~\ref{thm:shuffle sl}, we have the shuffle realization
\begin{equation}
\label{eqn:1|1}
U^+_{q,d}(\widehat{\widehat{\fgl}}_{1|1}) \simeq \Big \{ E(z_{ia})_{i \in \{0,1\}, a \geq 1} 
\text{ which vanishes at the wheels \eqref{eqn:wheels 1|1}} \Big\} 
\end{equation}
To see this, consider (as in Subsection~\ref{sub:specialization patterns}) the collection of asps given by all zig-zags which alternate between colors $0$ and $1$ via the edges decorated with the parameters $q_3^{\pm 1}$ and $q_1^{\pm' 1}$, for all signs $\pm, \pm'$. It is easy to prove that this collection of asps is closed under extensions, a claim which we leave as an exercise to the reader. This implies the isomorphism \eqref{eqn:1|1} by analogy with the proof of Theorem~\ref{thm:shuffle sl}.

\medskip

\subsection{$N=3$ and $\bs=(1,1,-1)$} \footnote{All other parity sequences can be obtained from this one using cyclic permutations and $\bs \leftrightarrow -\bs$, except the non-super parity sequence $(1,1,1)$ to which Theorem~\ref{thm:shuffle sl} applies.} In this case, we alter our definition of $\UslG$ by replacing the degree $4$ relations~\eqref{eqn:rel quantum affine 9} with the degree $4$ clunky relations $e_{\bigcirc}$ of~\eqref{eqn:clunky} associated with the square faces of the corresponding quiver~$\tQ$: 
\begin{equation}
\label{eq:BM-wheels}
		\begin{tikzcd}
			{z_{21}} && {z_{02}} \\ \\
			{z_{01}} && {z_{11}}
			\arrow["{q_3^{-1}}", from=1-1, to=1-3]
			\arrow["{q_3}", from=1-3, to=3-3]
			\arrow["{q_1^{-1}}", from=3-1, to=1-1]
			\arrow["{q_1}", from=3-3, to=3-1]
		\end{tikzcd} \qquad  \qquad
        \begin{tikzcd}
			{z_{11}} && {z_{21}} \\ \\
			{z_{22}} && {z_{01}}
			\arrow["{q_1}", swap, to=1-1, from=1-3]
			\arrow["{q_1^{-1}}", swap, to=1-3, from=3-3]
			\arrow["{q_3}", swap, to=3-1, from=1-1]
			\arrow["{q_3^{-1}}", swap, to=3-3, from=3-1]
		\end{tikzcd}
\end{equation}
In this Subsection, we show that this definition is equivalent to that of~\cite{BM}:
\begin{equation}
\label{eq:BM-iso}    
  \UslG \simeq \UslGBM 
\end{equation}

\medskip
\noindent
The algebra $\UslGBM$ featured in~\eqref{eq:BM-iso} was defined akin to $\UslG$ but rather imposing the following degree $5$ relation (and its $f$-version): \footnote{In fact, \cite{BM} imposed instead the equivalent $q\leadsto q^{-1}$ version, cf.~Remark~\ref{rem:equivalent-C-six-relations}.}
\begin{multline}
\label{eq:BM-relation-0}
   \Sym_{z_1,z_2} \Sym_{w_1,w_2} \Big( [ e_0(z_1), e_2(w_1), e_0 (z_2),  e_2(w_2), e_1(y) ]_{q^{-1},1,1,q} \\ - 
    [ e_2(w_1),  e_0(z_1),  e_2 (w_2),  e_0(z_2), e_1(y) ]_{q^{-1},1,1,q} \Big) = 0
\end{multline}
where we use the standard conventions for iterated $q$-commutators:
\begin{equation*}
  [x_1,\ldots,x_{k-1},x_k]_{\lambda_1,\ldots,\lambda_{k-1}} = [x_1, \dots [x_{k-1},x_k]_{\lambda_1}\dots ]_{\lambda_{k-1}}  
\end{equation*}
Henceforth, we shall iteratively use the following (super-deformed-)Jacobi identity:
\begin{equation}
\label{eq:Jacobi}
  [x,y,z]_{\lambda,\mu} = [[x,y]_\nu,z]_{\lambda\mu/\nu} + (-1)^{|x||y|}\nu [y,[x,z]_{\mu/\nu}]_{\lambda/\nu} 
  \quad \forall\, \lambda,\mu,\nu \in \BC^\times 
\end{equation}
We also consider $H_{2 1}=(\ph^+_{2 0})^{-1}\ph^+_{2 1}$, so that by~\eqref{eqn:rel quantum affine 4} we have for any $\ell\in \BZ$:
\begin{equation}
\label{eq:aux-comm}
\begin{split}
  & \ph^+_{2 0} e_{2 \ell} = e_{2 \ell} \ph^+_{2 0}, \qquad 
    [H_{2 1},e_{2 \ell}]_{1} = 0 \\
  & \ph^+_{2 0} e_{0 \ell} = q e_{0 \ell} \ph^+_{2 0}, \qquad 
    [H_{2 1},e_{0 \ell}]_{1} = d^{-1}(q-q^{-1}) e_{0,\ell+1} \\
  & \ph^+_{2 0} e_{1 \ell} = q^{-1} e_{1 \ell} \ph^+_{2 0}, \quad 
    [H_{2 1},e_{1 \ell}]_{1} = -d^{-1}(q-q^{-1})  e_{1,\ell+1}
\end{split}
\end{equation}
This also implies the following useful formula (for any $y$ and $\lambda\in \BC^\times$, $\ell\in \BZ$):
\begin{equation}
\label{eq:cartan-left}
  [e_{0 \ell}, \ph^+_{2 0}y]_{\lambda} = q^{-1}\ph^+_{2 0} [e_{0 \ell},y]_{q\lambda}, \qquad 
  [e_{2 \ell}, \ph^+_{2 0}y]_{\lambda} = \ph^+_{2 0} [e_{2 \ell},y]_{\lambda} 
\end{equation}

\medskip
\noindent
Consider the degree $5$ relation obtained by evaluating the $y^{-k}$-coefficients in~\eqref{eq:BM-relation-0}:
\begin{equation}   
\label{eq:deg 5 special}
  [ e_{00}, e_{20}, e_{00}, e_{20}, e_{1k}]_{q^{-1},1,1,q}  \, -  [ e_{20}, e_{00}, e_{20}, e_{00}, e_{1k} ]_{q^{-1},1,1,q} \,  = 0
\end{equation}
Applying $[(q-q^{-1})f_{2 1},-]$ to this equality and combining~\eqref{eqn:rel quantum affine 6} with~\eqref{eq:Jacobi}, we obtain 
\begin{multline}
\label{eq:4-term}
   [ e_{00}, e_{20}, e_{00}, \ph^+_{21}, e_{1k} ]_{q^{-1},1,1,q} \, + 
   [ e_{20}, e_{00}, \ph^+_{21}, e_{00}, e_{1k} ]_{q^{-1},1,1,q} \\
   + [ e_{00}, \ph^+_{21}, e_{00}, e_{20}, e_{1k} ]_{q^{-1},1,1,q} \, +
   [ \ph^+_{21}, e_{00}, e_{20}, e_{00}, e_{1k} ]_{q^{-1},1,1,q} \, = 0
\end{multline}
We shall now evaluate each term in the LHS using~\eqref{eq:Jacobi}-\eqref{eq:cartan-left}.

\smallskip
\noindent
$\bullet$ 
First, applying~\eqref{eq:aux-comm} we obtain 
$$
  [\ph^+_{21}, e_{1k}]_{q^{-1}} = \ph^+_{20} [H_{21},e_{1k}]_1 = - d^{-1}(q-q^{-1}) \ph^+_{20} e_{1,k+1}
$$ 
Invoking~\eqref{eq:cartan-left}, we can thus evaluate the entire first summand in the LHS of~\eqref{eq:4-term}:
\begin{multline}
\label{eq:term-1}
  [ e_{00}, e_{20}, e_{00}, \ph^+_{21}, e_{1k} ]_{q^{-1},1,1,q} = \\
  -q^{-2}d^{-1}(q-q^{-1}) \ph^+_{20} [e_{00} , e_{20} , e_{00} , e_{1,k+1}]_{q,1,q^2}
\end{multline}

\noindent
$\bullet$
To evaluate the second term in~\eqref{eq:4-term}, we start with the following calculation: 
\begin{multline*}
  [\ph^+_{2 1} , e_{00} , e_{1k}]_{q^{-1},1} = \ph^+_{2 0} [H_{2 1} , e_{00} , e_{1k}]_{q^{-1},1} = \\
  d^{-1}(q-q^{-1}) \ph^+_{2 0} ([e_{01},e_{1k}]_{q^{-1}} - [e_{00},e_{1,k+1}]_{q^{-1}})
\end{multline*}
where we used~\eqref{eq:Jacobi} with $\nu=1$ in the last equality. Invoking~\eqref{eq:cartan-left}, we get 
\begin{multline}
\label{eq:term-2}
  [ e_{20}, e_{00}, \ph^+_{21}, e_{00}, e_{1k} ]_{q^{-1},1,1,q} = \\
  q^{-1}d^{-1}(q-q^{-1}) \ph^+_{2 0} 
  \big( [e_{20} , e_{00} , e_{01} , e_{1 k}]_{q^{-1},q,q} - [e_{20} , e_{00} , e_{00} , e_{1,k+1}]_{q^{-1},q,q} \big)
\end{multline}

\noindent
$\bullet$
We likewise evaluate the third term in~\eqref{eq:4-term} to obtain:
\begin{multline}
\label{eq:term-3}
  [ e_{00}, \ph^+_{21}, e_{00}, e_{20}, e_{1k} ]_{q^{-1},1,1,q} = \\
  q^{-1}d^{-1}(q-q^{-1}) \ph^+_{2 0} 
  \big( [e_{00} , e_{01} , e_{20} , e_{1 k} ]_{q^{-1},1,q^2} - [ e_{00} , e_{00} , e_{20} , e_{1,k+1} ]_{q^{-1},1,q^2} \big)
\end{multline}

\noindent
$\bullet$
A similar iterated application of~\eqref{eq:Jacobi} evaluates the last summand of \eqref{eq:4-term}:
\begin{multline}
\label{eq:term-4}
  [ \ph^+_{21}, e_{00}, e_{20}, e_{00}, e_{1k} ]_{q^{-1},1,1,q}  = 
  d^{-1}(q-q^{-1}) \ph^+_{2 0} \cdot \\ 
  \Big( [ e_{01} , e_{20} , e_{00} , e_{1 k}]_{q^{-1},1,1} + [ e_{00} , e_{20} , e_{01} , e_{1 k} ]_{q^{-1},1,1} - [ e_{00} , e_{20} , e_{00} , e_{1,k+1} ]_{q^{-1},1,1} \Big)
\end{multline}

\medskip
\noindent
Adding together the RHS of \eqref{eq:term-1}-\eqref{eq:term-4}, gives us the following degree $4$ relation: 
\begin{multline}
\label{eq:deg4-special}
  [e_{01} , e_{20} , e_{00} , e_{1 k}]_{q^{-1},1,1} + [e_{00} , e_{20} , e_{01} , e_{1 k}]_{q^{-1},1,1} - [e_{00} , e_{20} , e_{00} , e_{1,k+1}]_{q^{-1},1,1} \\
  + q^{-1} [e_{00} , e_{01} , e_{20} , e_{1 k}]_{q^{-1},1,q^2} - q^{-1} [e_{00} , e_{00} , e_{20} , e_{1,k+1}]_{q^{-1},1,q^2} \\
  + q^{-1} [e_{20} , e_{00} , e_{01} , e_{1 k}]_{q^{-1},q,q} - q^{-1} [e_{20} , e_{00} , e_{00} , e_{1,k+1}]_{q^{-1},q,q}  \\
  - q^{-2} [e_{00} , e_{20} , e_{00} , e_{1,k+1}]_{q,1,q^2} = 0
\end{multline}
for each $k\in \BZ$. If $X_k$ denotes the LHS of~\eqref{eq:deg4-special} then a direct calculation shows that
$$
  \tUpsilon^+_{\wA_{\Gamma}}(X_k) = 0
$$
We also have the following analogue of~\eqref{eqn:not zero 1} and~\eqref{eqn:not zero 2}:
\begin{equation*}
  \langle X_k, F_k \rangle = q_1 + q_3  
  \qquad \mathrm{for} \qquad F_k=\frac{z^{-k}_{11}}{\sqrt{z_{01}z_{02}}}\in \CV_{\wA_{\Gamma};-2\bsi^0-\bsi^1-\bsi^2,-k-1}
\end{equation*}
Therefore, just like in the proof of~\eqref{eq:duality-deg3}, we obtain that 
\begin{equation*}
  \big \langle (X_k)_{k\in \BZ} , F \big \rangle = 0 \text{ for } F \in \CV^-_{\wA_{\Gamma}} \quad \Leftrightarrow \qquad 
  F \text{ vanishes at the right wheel in \eqref{eq:BM-wheels} }
\end{equation*}
Noting the natural symmetry of $\UslG$ swapping $d\leftrightarrow d^{-1}$ and currents $e_0(z)\leftrightarrow e_2(z), f_0(z)\leftrightarrow f_2(z), \ph^\pm_0(z)\leftrightarrow \ph^\pm_2(z)$, we also obtain the following degree $4$ relation:
\begin{multline}
\label{eq:deg4-special-2}
  [e_{21} , e_{00} , e_{20} , e_{1 k}]_{q^{-1},1,1} + [e_{20} , e_{00} , e_{21} , e_{1 k}]_{q^{-1},1,1} - [e_{20} , e_{00} , e_{20} , e_{1,k+1}]_{q^{-1},1,1} \\
  + q^{-1} [e_{20} , e_{21} , e_{00} , e_{1 k}]_{q^{-1},1,q^2} - q^{-1} [e_{20} , e_{20} , e_{00} , e_{1,k+1}]_{q^{-1},1,q^2} \\
  + q^{-1} [e_{00} , e_{20} , e_{21} , e_{1 k}]_{q^{-1},q,q} - q^{-1} [e_{00} , e_{20} , e_{20} , e_{1,k+1}]_{q^{-1},q,q}  \\
  - q^{-2} [e_{20} , e_{00} , e_{20} , e_{1,k+1}]_{q,1,q^2} = 0
\end{multline}
(which can be directly obtained by applying $[(q-q^{-1})f_{01},-]$ to~\eqref{eq:deg 5 special}). Therefore, just as above, if $X'_k$ denotes the LHS of~\eqref{eq:deg4-special-2}, then we obtain 
\begin{equation*}
  \big \langle (X'_k)_{k\in \BZ} , F \big \rangle = 0 \text{ for } F \in \CV^-_{\wA_{\Gamma}} \quad \Leftrightarrow \quad 
  F \text{ vanishes at the left wheel in \eqref{eq:BM-wheels} }
\end{equation*}
This establishes the main isomorphism~\eqref{eq:BM-iso} of this Subsection, following the approach outlined in Subsection~\ref{sub:proof} in type $A$.

\medskip
\noindent
In fact, one can derive from~\eqref{eq:deg4-special} and~\eqref{eq:deg4-special-2} more general degree $4$ relations. To this end, define elements $\{H_{i,\pm r}\}_{i\in \wIthree}^{r>0}$ by 
$(\ph^\pm_{i,0})^{-1} \ph^\pm_i(z)  = \exp \left( \pm \sum_{r>0} H_{i,\pm r} z^{\mp r}\right)$ 
It follows from~\eqref{eqn:rel quantum affine 4} that these elements commute with $e_{j,\ell}$ via
\begin{equation*}
  [H_{i,r}, e_{j,\ell}] = \frac{d^{-rM^{\bs}_{ij}}(q^{rA^{\bs}_{ij}} - q^{-rA^{\bs}_{ij}})}{r} \cdot e_{j,\ell+r} \qquad \forall\, r\ne 0, \ell\in \BZ, i,j\in \wIthree
\end{equation*}
with $A^\bs_{ij}=(\bsi^i,\bsi^j)$ and $M^{\bs}_{ij}=s_i\delta_{i,j+1}-s_j\delta_{i+1,j}$ as in~\cite{BM}. 
Since the corresponding $3\times 3$ matrix $\big( d^{-rM^{\bs}_{ij}}(q^{rA^{\bs}_{ij}} - q^{-rA^{\bs}_{ij}}) \big)_{i,j\in\wIthree}$ is non-degenerate for any $r\ne 0$, one can therefore find $H^{\perp}_{i,r} \in \mathrm{span}_{\BC} \{H_{0,r}, H_{1,r}, H_{2,r}\}$ such that 
\begin{equation}
\label{eq:shift-cartan}
  [H^{\perp}_{i,r}, e_{j,\ell}] = \delta_{ij} e_{j,\ell+r} \quad \forall\, r\ne 0, \ell\in \BZ, i,j\in \wIthree
\end{equation}
We refer to $[H^\perp_{i,r},-]$ as \emph{shift operators} $\tau_{r\bsi^i}$ for each $i\in \wIthree, r\in \BZ$, whereas $\tau_{0}=\mathrm{Id}$.

\medskip
\noindent
Applying the operators $\tau_{r'\bsi^2}$ as well as $\tau_{r_1\bsi^0}\tau_{r_2\bsi^0} - \tau_{(r_1+r_2)\bsi^0}$ to~\eqref{eq:deg4-special}, multiplying each result by $z_1^{-r_1}z_2^{-r_2}w^{-r'}y^{-k}$ and summing over all $r_1,r_2,r',k\in \BZ$ implies
\begin{equation*}
\begin{split}
  \Sym_{z_1,z_2} \Big(
  & (z_1+z_2-y) [e_{0}(z_1) , e_{2}(w) , e_{0}(z_2) , e_{1}(y)]_{q^{-1},1,1}  \\
  & + q^{-1}(z_2-y) [e_{0}(z_1) , e_{0}(z_2) , e_{2}(w) , e_{1}(y)]_{q^{-1},1,q^2} \\
  & + q^{-1}(z_2-y) [e_{2}(w) , e_{0}(z_1) , e_{0}(z_2) , e_{1}(y)]_{q^{-1},q,q} \\
  & - q^{-2} y  [e_{0}(z_1) , e_{2}(w) , e_{0}(z_2) , e_{1}(y)]_{q,1,q^2} \Big) = 0
\end{split}
\end{equation*}
Completely analogously, one derives from~\eqref{eq:deg4-special-2} the following relation:
\begin{equation*}
\begin{split}
  \Sym_{w_1,w_2} \Big(
  & (w_1+w_2-y) [e_{2}(w_1) , e_{0}(z) , e_{2}(w_2) , e_{1}(y)]_{q^{-1},1,1}  \\
  & + q^{-1}(w_2-y) [e_{2}(w_1) , e_{2}(w_2) , e_{0}(z) , e_{1}(y)]_{q^{-1},1,q^2} \\
  & + q^{-1}(w_2-y) [e_{0}(z) , e_{2}(w_1) , e_{2}(w_2) , e_{1}(y)]_{q^{-1},q,q} \\
  & - q^{-2} y  [e_{2}(w_1) , e_{0}(z) , e_{2}(w_2) , e_{1}(y)]_{q,1,q^2} \Big) = 0
\end{split}
\end{equation*}


\bigskip 

\section{Non-zero pairings}\label{non-zero pairings}

The purpose of this appendix is to evaluate certain pairings $\langle X,F \rangle$ that played crucial roles in our proof of~\eqref{eqn:key pairing finite sl} and its type $B,C,D$ versions from the proofs of Theorems \ref{thm:d intro}, \ref{thm:b intro}, \ref{thm:c intro}. Here, $X$ denotes certain linear combinations of coefficients of the higher order relations defining the corresponding quantum affine superalgebras, while $F$ is a ``test function'' from the corresponding big shuffle algebras.


\medskip 

\subsection{}
\label{sub:non-zero type a even}

We start with the case of $|i|=0$ in type $A$. In the proof of~\eqref{eqn:key pairing finite sl}, we considered the zero-th modes of the series involved, namely
$$
  X = [ e_{i,0}, [ e_{i,0},e_{i\pm 1,0} ]_{q^{s_i}} ]_{q^{-s_i}} = \tfrac{1}{2} X^{0,0,0}_{i,i\pm 1} 
$$
We let 
$$
  F = 1 \in \CV_{A_{\Gamma};-2\bsi^i-\bsi^{i\pm 1}}
$$
and a key part of the computation was to show that the pairing $\langle X,F\rangle$ is non-zero (we note that in the aforementioned proof, $F$ was multiplied by the monomial $z^{-d}_{i\pm 1,1}$, which perfectly balanced out a shift in the indices of $e_{i\pm 1,d}$ that appeared in $X^{0,0,d}_{i,i\pm 1}$). Then (the type $A_{\Gamma}$ analogue of) formula \eqref{eqn:pairing explicit} implies that
\begin{equation}
\label{eqn:a non-zero 1}
  \langle X,F \rangle = \sum_{i_1,i_2,i_3} (\pm q^*) \int_{|z_1|\gg |z_2|\gg |z_3|} \frac {Dz_1 Dz_2 Dz_3}{\zeta^{A_{\Gamma}}_{i_2i_1}\left(\frac {z_2}{z_1}\right)\zeta^{A_{\Gamma}}_{i_3i_2}\left(\frac {z_3}{z_2}\right)\zeta^{A_{\Gamma}}_{i_3i_1}\left(\frac {z_3}{z_1}\right)} 
\end{equation}
with $\textbf{i}=(i_1,i_2,i_3) \in \{(i,i,i\pm 1), (i,i\pm 1,i), (i\pm 1,i,i)\}$. Henceforth, $*$ denotes some integers that will not be crucial until the very end. Due to \eqref{eqn:zeta finite sl}, we have
$$
  \lim_{x \rightarrow 0} \zeta^{A_{\Gamma}}_{i+1,i}(x) = -1, \qquad 
  \lim_{x \rightarrow 0} \zeta^{A_{\Gamma}}_{i,i-1}(x) = -1, \qquad 
  \lim_{x \rightarrow 0} \zeta^{A_{\Gamma}}_{ii}(x) = 1
$$
while 
\begin{equation}
\label{eq:infty-limits}
  \lim_{x \rightarrow 0} \zeta^{A_{\Gamma}}_{i,i+1}(x) = \infty, \qquad 
  \lim_{x \rightarrow 0} \zeta^{A_{\Gamma}}_{i-1,i}(x) = \infty 
\end{equation}  
Thus, in the limit as $|z_1| \gg |z_2| \gg |z_3|$, the integral in \eqref{eqn:a non-zero 1} vanishes unless $\textbf{i} \in \{(i-1,i,i), (i,i,i+1)\}$. We conclude that only one of the three terms that make up $X$ produces a non-zero contribution to the right-hand side of \eqref{eqn:a non-zero 1}, which is easily seen to equal  $1(-1)^2(\pm q^*)=1$. Thus $\langle X,F \rangle = 1$, establishing formula~\eqref{eqn:not zero 1}.


\medskip 

\subsection{}
\label{sub:non-zero type a odd}

Next, we consider the case of $|i|=1$ in type $A$. In the proof of~\eqref{eqn:key pairing finite sl}, we considered the zero-th modes of the series involved, namely 
\begin{align*}
  X' = 
  [ e_{i,0}, [ e_{i+1,0}, [ e_{i,0} , e_{i-1,0} ]_{q^{s_i}}]_{q^{-s_i}}]_1 = \tfrac{1}{2} X^{0,0,0,0}_{i-1,i,i+1}
\end{align*}
We also set $p = |i-1|$, $p' = |i+1|$. We let 
$$
  F' = \sqrt{\frac {z_{i1}}{z_{i2}}} + \sqrt{\frac {z_{i2}}{z_{i1}}} \in \CV_{A_{\Gamma};-\bsi^{i-1}-2\bsi^i-\bsi^{i + 1}}
$$
(note that in the aforementioned proof, $F$ was multiplied by the monomial $z^{-d}_{i+1, 1}$ which perfectly balanced out a shift in the indices of $e_{i+1,d}$ that appeared in $X^{0,0,0,d}$) and the goal is to prove the non-zero-ness of the following quantity:
\begin{equation}
\label{eqn:a non-zero 2}
  \langle X',F' \rangle = \sum_{i_1,i_2,i_3,i_4} (\pm q^*) \int_{|z_1|\gg |z_2|\gg |z_3| \gg |z_4|} \frac {\left(\sqrt{\frac {z_{a'}}{z_{b'}}}+\sqrt{\frac {z_{b'}}{z_{a'}}}\right) Dz_1 Dz_2 Dz_3 Dz_4}{\prod_{1 \leq a < b \leq 4} \zeta^{A_{\Gamma}}_{i_bi_a}\left(\frac {z_b}{z_a}\right)} 
\end{equation}
where $\textbf{i}=(i_1,i_2,i_3,i_4) \in \{(i,i+1,i,i-1),(i,i+1,i-1,i),(i,i,i-1,i+1),(i,i-1,i,i+1),(i+1,i,i-1,i),(i+1,i-1,i,i),(i,i-1,i+1,i),(i-1,i,i+1,i)\}$. Let $\{a'<b', c', d'\} = \{1,2,3,4\}$ denote the indices with the property that $i_{a'} = i_{b'} = i$, $i_{c'}=i-1$, and $i_{d'}=i+1$. Similar to the previous Subsection, we have 
\begin{equation*}
\begin{split}
  & \lim_{x \rightarrow 0} \zeta^{A_{\Gamma}}_{i+1,i}(x) = -1 ,\qquad  
    \lim_{x \rightarrow 0} \zeta^{A_{\Gamma}}_{i,i-1}(x) = -1 \\ 
  & \lim_{x \rightarrow 0} \zeta^{A_{\Gamma}}_{i+1,i-1}(x) = 1 ,\qquad 
    \lim_{x \rightarrow 0} \zeta^{A_{\Gamma}}_{i-1,i+1}(x) = (-1)^{pp'}  
\end{split}
\end{equation*}
as well as ~\eqref{eq:infty-limits} and 
$$
  \lim_{x \rightarrow 0} x^{-\frac 12} \zeta_{ii}^{A_{\Gamma}}(x) = 1
$$
Because of the latter property, we have 
\begin{equation*}
  \frac {\sqrt{\frac {z_{a'}}{z_{b'}}}+\sqrt{\frac {z_{b'}}{z_{a'}}}}{\zeta_{ii}^{A_{\Gamma}} \left(\frac {z_{b'}}{z_{a'}}\right)} 
  = \frac {z_{a'}}{z_{b'}} \left[1 + O   \left(\frac {z_{b'}}{z_{a'}}\right) \right] 
  = \frac {z_{a'}}{z_{b'}} + O(1) 
\end{equation*}
in the limit as $|z_{a'}|\gg |z_{b'}|$. However, due to the zeta factors $\zeta^{A_{\Gamma}}_{i-1,i}$ and $\zeta^{A_{\Gamma}}_{i,i+1}$ in the denominator, it is easy to see that the integrand behaves like $O(z^{-1}_1)$ or $O(z^0_1 z^{-1}_2)$, forcing the corresponding integral in \eqref{eqn:a non-zero 2} to vanish, in all cases except for
$$
  \textbf{i}=(i_1,i_2,i_3,i_4) \in \{(i,i-1,i,i+1),(i-1,i,i+1,i), (i,i-1,i+1,i)\}
$$
For $\textbf{i}=(i,i-1,i,i+1)$, the corresponding integral equals 
$$
  \int_{|z_1|\gg |z_2|\gg |z_3| \gg |z_4|} 
  \frac { -(-1)^p  (\frac{z_2}{z_3} + o(1)) Dz_1 Dz_2 Dz_3 Dz_4}
        {(1-\frac{z_2}{z_1}q^{-s_i})(1-\frac{z_4}{z_1}q^{s_i})(1-\frac{z_3}{z_2}q^{-s_i})(1-\frac{z_4}{z_3}q^{s_i})} = 
  - (-1)^p q^{-s_i} 
$$
while  $(\pm q^*)=(-1)^{p+p'+pp'}$. For $\textbf{i}=(i-1,i,i+1,i)$, the corresponding integral is
$$
  \int_{|z_1|\gg |z_2|\gg |z_3| \gg |z_4|} 
  \frac { -(-1)^{p'}  (\frac{z_2}{z_3} + o(1)) Dz_1 Dz_2 Dz_3 Dz_4}
        {(1-\frac{z_2}{z_1}q^{-s_i})(1-\frac{z_4}{z_1}q^{-s_i})(1-\frac{z_3}{z_2}q^{s_i})(1-\frac{z_4}{z_3}q^{s_i})} = 
  - (-1)^{p'} q^{s_i}
$$
while $(\pm q^*)=(-1)^{pp'}$. For $\textbf{i}=(i,i-1,i+1,i)$, the corresponding integral is
$$
  \int_{|z_1|\gg |z_2|\gg |z_3| \gg |z_4|} 
  \frac { (-1)^{p+p'}  ( \frac {z_2}{z_3} + o(1)) Dz_1 Dz_2 Dz_3 Dz_4}
        {(1-\frac{z_2}{z_1} q^{-s_i})(1-\frac{z_3}{z_1}q^{s_i})(1-\frac{z_4}{z_2}q^{-s_i})(1-\frac{z_4}{z_3}q^{s_i})} = 0
$$
Hence $\langle X',F' \rangle = -(-1)^{p'(1+p)} (q^{s_i}+q^{-s_i}) = -(-1)^{p'(1+p)} (q+q^{-1})$, establishing~\eqref{eqn:not zero 2}.


\medskip

\subsection{}
\label{sub:non-zero type b case 1}

Consider type $B$ Case 1. Without loss of generality, we shall assume that the neck parameter equals $q^{-1}$. In the proof of Theorem \ref{thm:b intro}, we considered the zero-th mode of the series involved, namely 
\begin{align*}
  X &= 
  [e_{N,0}, [e_{N,0}, [e_{N,0},e_{N-1,0}]_{q}]_1]_{q^{-1}} = \tfrac{1}{6} X^{0,0,0,0}_{N-1,N} 
\end{align*}
We let
$$
  F = 1 \in \CV_{B_{\Gamma};-3\bsi^N - \bsi^{N-1}}
$$
(in the aforementioned proof, $F$ was multiplied by the monomial $z^{-d}_{N-1, 1}$ which perfectly balanced out a shift in the indices of $e_{N-1,d}$ that appeared in $X_{N-1,N}^{0,0,0,d}$). Then (the type $B_{\Gamma}$ analogue of) formula \eqref{eqn:pairing explicit} implies that
\begin{equation}
\label{eqn:b non-zero}
  \langle X,F \rangle = \sum_{i_1,i_2,i_3,i_4} (\pm q^*) \int_{|z_1|\gg |z_2|\gg |z_3| \gg |z_4|} \frac {Dz_1 Dz_2 Dz_3 Dz_4}{\prod_{1 \leq a < b \leq 4} \zeta^{B_{\Gamma}}_{i_bi_a}\left(\frac {z_b}{z_a}\right)} 
\end{equation}
where $\textbf{i}=(i_1,i_2,i_3,i_4) \in \{(N,N,N,N-1), (N,N,N-1,N), (N,N-1,N,N), (N-1,N,N,N)\}$. Since
$$
  \lim_{x \rightarrow 0} \zeta^{B_{\Gamma}}_{N,N-1}(x) = -1, \qquad 
  \lim_{x \rightarrow 0} \zeta^{B_{\Gamma}}_{NN}(x) = 1, \qquad 
  \lim_{x \rightarrow 0} \zeta^{B_{\Gamma}}_{N-1,N}(x) = \infty
$$
it is clear that only one out of the four summands in the right-hand side of \eqref{eqn:b non-zero}, namely the one with $\textbf{i} = (N-1,N,N,N)$, produces the non-zero contribution $(-1)^3\cdot (\pm q^*)=1$. Hence $\langle X , F \rangle = 1$, establishing~\eqref{eq:const-B-1st}.


\medskip    

\subsection{}
\label{sub:non-zero type d case 2}

Consider type $D$ Case 2. Without loss of generality, we shall assume that the neck parameter equals $q$. In the proof of Theorem \ref{thm:d intro},  we considered the zero-th mode of the series involved, namely 
\begin{align*}
  X 
  &= [ [ e_{N-2,0} , e_{N-1,0} ]_{q^{-1}} ,  e_{N,0} ]_{q} - [ [ e_{N-2,0} , e_{N,0} ]_{q^{-1}} ,  e_{N-1,0}  ]_{q}
\end{align*}
For $\bn = \bsi^{N-2} + \bsi^{N-1} + \bsi^N$, we consider $F,F',F''\in \CV_{D_{\Gamma};-\bn}$ defined by
$$
  F = \frac {z_{N1}}{z_{N-1,1}} , \qquad F'= 1 , \qquad F''= \frac {z_{N-1,1}}{z_{N1}}
$$
(in the aforementioned proof $F,F',F''$ were all multiplied by $z^{-a}_{N-2,1}z^{-b}_{N-1,1}z^{-c}_{N1}$ which perfectly balanced out a shift in the indices of $e_{N-2,a}, e_{N-1,b}, e_{N,c}$ in $X^{a,b,c}$). Then (the type $D_{\Gamma}$ analogue of) formula \eqref{eqn:pairing explicit} implies that
\begin{equation}
\label{eqn:d non-zero}
  \langle X,F \rangle = \sum_{i_1,i_2,i_3} (\pm q^*) \int_{|z_1|\gg |z_2|\gg |z_3|} \frac {\frac {z_{b'}}{z_{a'}} Dz_1 Dz_2 Dz_3}{\prod_{1 \leq a < b \leq 3} \zeta^{D_{\Gamma}}_{i_bi_a}\left(\frac {z_b}{z_a}\right)} 
\end{equation}
\begin{equation}
\label{eqn:d non-zero prime}
  \langle X,F' \rangle = \sum_{i_1,i_2,i_3} (\pm q^{*'}) \int_{|z_1|\gg |z_2|\gg |z_3|} \frac {Dz_1 Dz_2 Dz_3}{\prod_{1 \leq a < b \leq 3} \zeta^{D_{\Gamma}}_{i_bi_a}\left(\frac {z_b}{z_a}\right)} 
\end{equation}
\begin{equation}
\label{eqn:d non-zero prime prime}
  \langle X,F'' \rangle = \sum_{i_1,i_2,i_3} (\pm q^{*''}) \int_{|z_1|\gg |z_2|\gg |z_3|} \frac {\frac {z_{a'}}{z_{b'}} Dz_1 Dz_2 Dz_3}{\prod_{1 \leq a < b \leq 3} \zeta^{D_{\Gamma}}_{i_bi_a}\left(\frac {z_b}{z_a}\right)} 
\end{equation}
where $\textbf{i}=(i_1,i_2,i_3) \in \{(N-2,N-1,N), (N-1,N-2,N), (N, N-2,N-1),  (N, N-1,N-2)\} \sqcup \{\text{previous set of triples with }N \leftrightarrow N-1\}$ and $\{a',b'\} \subset \{1,2,3\}$ denote those indices with the property that $i_{a'} = N-1$, $i_{b'} = N$. Since
\begin{equation}
\label{eq:D-limits}
\begin{split}
  &\lim_{x \rightarrow 0} \zeta^{D_{\Gamma}}_{N,N-1}(x) = \lim_{x \rightarrow 0} \zeta^{D_{\Gamma}}_{N-1,N-2}(x) = \lim_{x \rightarrow 0} \zeta^{D_{\Gamma}}_{N,N-2}(x) = -1 \\
  &\lim_{x \rightarrow 0} \zeta^{D_{\Gamma}}_{N-1,N}(x) = \lim_{x \rightarrow 0} \zeta^{D_{\Gamma}}_{N-2,N-1}(x) =  \lim_{x \rightarrow 0} \zeta^{D_{\Gamma}}_{N-2,N}(x) = \infty
\end{split}
\end{equation}
it is clear that only one out of the eight summands in the right-hand side of \eqref{eqn:d non-zero prime}, namely the one with $\textbf{i} = (N-2,N-1,N)$, produces the non-zero contribution $(-1)^3 (\pm q^{*'}) = -1$. Thus $\langle X,F' \rangle = -1$, proving the second formula of~\eqref{eqn:not zero d}.

\medskip 
\noindent 
As for \eqref{eqn:d non-zero}, due to $\frac {z_{b'}}{z_{a'}}$ in the numerator, the integral may be non-zero only if $a' > b'$, i.e.\ the index $N$ appears before  $N-1$ in the sequence $\textbf{i}$. However, in this case, the zeta factor $\zeta^{D_{\Gamma}}_{N-1,N}(z_{a'}/z_{b'})$ in the denominator results in 
$$
  \frac { \frac{z_{b'}}{z_{a'}} }{\zeta^{D_{\Gamma}}_{N-1,N} (\frac {z_{a'}}{z_{b'}})} = \frac { -1 } {1- \frac{z_{a'}}{z_{b'}} q^{-2}} = -1 + o(1)
$$
Invoking~\eqref{eq:D-limits}, we see that only the summand corresponding to $\textbf{i} = (N-2,N,N-1)$ in \eqref{eqn:d non-zero} produces a non-zero contribution, equal to $(-1)^3(\pm q^*)=1$. Therefore $\langle X,F \rangle = 1$, proving the first formula of~\eqref{eqn:not zero d}.

\medskip 
\noindent 
As for \eqref{eqn:d non-zero prime prime}, due to $\frac {z_{a'}}{z_{b'}}$ in the numerator the integral may be non-zero only if $b' > a'$, i.e.\ the index $N-1$ appears before $N$ in the sequence $\textbf{i}$. For $\textbf{i}=(N-2,N-1,N)$, the corresponding integral equals 
$$
  \int_{|z_1|\gg |z_2|\gg |z_3|} \frac {-\frac{z_2}{z_3} Dz_1 Dz_2 Dz_3}
  {(1-\frac{z_2}{z_1}q) (1-\frac{z_3}{z_1}q) (1-\frac{z_3}{z_2}q^{-2})} = - q^{-2}
$$
while $(\pm q^{*''})=1$. For $\textbf{i}=(N-1,N-2,N)$, the corresponding integral is
$$
  \int_{|z_1|\gg |z_2|\gg |z_3|} \frac {(-1)^{|N-2|} \frac{z_2}{z_3} Dz_1 Dz_2 Dz_3}
  {(1-\frac{z_2}{z_1}q) (1-\frac{z_3}{z_1}q^{-2}) (1-\frac{z_3}{z_2}q)} =  (-1)^{|N-2|} q
$$
while the corresponding terms $(\pm q^{*''})$ contribute $-(-1)^{|N-2|}q-(-1)^{|N-2|}q^{-1}$ (we note that $\textbf{i}=(N-1,N-2,N)$ appears twice among the eight summands constituting $X$). For $\textbf{i}=(N-1,N,N-2)$, the corresponding integral equals 
$$
  \int_{|z_1|\gg |z_2|\gg |z_3|} \frac { - \frac{z^2_3}{z_2^2} Dz_1 Dz_2 Dz_3}
  {(1-\frac{z_2}{z_1}q^{-2}) (1-\frac{z_3}{z_1}q) (1-\frac{z_3}{z_2}q)} = 0
$$
Hence $\langle X , F'' \rangle = -q^2 - 1 - q^{-2}$, proving the last formula of~\eqref{eqn:not zero d}.


\medskip 

\subsection{}
\label{sub:non-zero type c case 1}

Consider type $C$ Case 1. Without loss of generality, we shall assume that the neck parameter equals $q^{2}$. In the proof of Theorem \ref{thm:c intro}, we considered the zero-th mode of the series involved, namely 
\begin{align*} 
  X 
  &= [ e_{N-1,0}, [ e_{N-1,0}, [ e_{N-1,0} ,e_{N,0} ]_{q^{-2}}]_1]_{q^2} = \tfrac{1}{6} X^{0,0,0,0} 
\end{align*} 
We let
$$
  F = 1 \in \CV_{C_{\Gamma};-3\bsi^{N-1}-\bsi^N}
$$
(in the aforementioned proof, $F$ was multiplied by the monomial $z^{-d}_{N1}$ which perfectly balanced out a shift in the indices of $e_{N,d}$ that appeared in $X^{0,0,0,d}$). Then (the type $C_{\Gamma}$ analogue of) formula \eqref{eqn:pairing explicit} implies that
\begin{equation}
\label{eqn:c non-zero case 1}
  \langle X,F \rangle = \sum_{i_1,i_2,i_3,i_4} (\pm q^*) \int_{|z_1|\gg |z_2|\gg |z_3| \gg |z_4|} \frac {Dz_1 Dz_2 Dz_3 Dz_4}{\prod_{1 \leq a < b \leq 4} \zeta^{C_{\Gamma}}_{i_bi_a}\left(\frac {z_b}{z_a}\right)} 
\end{equation}
where $\textbf{i}=(i_1,i_2,i_3,i_4) \in \{(N-1,N-1,N-1,N), (N-1,N-1,N,N-1), (N-1,N,N-1,N-1), (N,N-1,N-1,N-1)\}$. Since
$$
  \lim_{x \rightarrow 0} \zeta^{C_{\Gamma}}_{N,N-1}(x) = -1, \qquad 
  \lim_{x \rightarrow 0} \zeta^{C_{\Gamma}}_{N-1,N}(x) = \infty, \qquad 
  \lim_{x \rightarrow 0} \zeta^{C_{\Gamma}}_{N-1,N-1}(x) = 1
$$
it is clear that only one out of the four summands in the right-hand side of \eqref{eqn:c non-zero case 1} produces a non-zero contribution, namely the one with  $\textbf{i}=(N-1,N-1,N-1,N)$, and that contribution is $(-1)^3(\pm q^*)=-1$. Hence $\langle X,F\rangle = -1$, establishing~\eqref{eqn:not zero c 1}.


\medskip 

\subsection{}
\label{sub:non-zero type c case 2}

Consider type $C$ Case 2. Without loss of generality, we shall assume that the neck parameter equals $q^{2}$. In the proof of Theorem \ref{thm:c intro}, we considered the zero-th mode of the series involved, namely 
\begin{align*}
  X 
  & = [[[ e_{N-2,0}, e_{N-1,0} ]_{q},  [[ e_{N-2,0}, e_{N-1,0} ]_{q} , e_{N,0} ]_{q^{-2}}]_1 , e_{N-1,0} ]_1 
    = \tfrac{1}{12} X^{0,0,0,0,0,0} 
\end{align*}
We let 
$$
  F = \frac {z_{N1}}{\sqrt{z_{N-2,1}z_{N-2,2}}} \in \CV_{C_{\Gamma};-2\bsi^{N-2}-3\bsi^{N-1}-\bsi^N}
$$
(note that in the aforementioned proof $F$ was multiplied by the monomial $z^{-d}_{N1}$ which perfectly balanced out a shift in the indices of $e_{N,d}$ that appeared in $X^{0,0,0,0,0,d}$). Then (the type $C_{\Gamma}$ analogue of) formula \eqref{eqn:pairing explicit} implies that
\begin{equation}
\label{eqn:c non-zero case 2}
  \langle X,F \rangle = \sum_{i_1,\dots,i_6} (\pm q^*) \int_{|z_1|\gg \dots \gg |z_6|} \frac {\frac {z_{c'}}{\sqrt{z_{a'}z_{b'}}} Dz_1 \dots Dz_6}{\prod_{1 \leq a < b \leq 6} \zeta^{C_{\Gamma}}_{i_bi_a}\left(\frac {z_b}{z_a}\right)} 
\end{equation}
where $\textbf{i}=(i_1,\dots,i_6)$ goes over certain permutations of $(N-2,N-2,N-1,N-1,N-1,N)$ and $\{a' < b',c'\} \subset \{1,\dots,6\}$ denote the indices such that $i_{a'} = i_{b'} = N-2$ and $i_{c'} = N$. We also let $\{d'<e'<f'\} \subset \{1,\dots,6\}$ denote those indices such that $i_{d'} = i_{e'} = i_{f'} = N-1$. Invoking \eqref{eq:C-zeta-case2}, we see that 
\begin{align*} 
  & \lim_{x \rightarrow 0} \zeta^{C_{\Gamma}}_{N,N-1}(x) = \lim_{x \rightarrow 0} \zeta^{C_{\Gamma}}_{N-1,N-2}(x)  = -1 \\
  & \lim_{x \rightarrow 0} \zeta^{C_{\Gamma}}_{N,N-2}(x) = \lim_{x \rightarrow 0} \zeta^{C_{\Gamma}}_{N-2,N}(x)  = 1 \\
  & \lim_{x \rightarrow 0} \zeta^{C_{\Gamma}}_{N-1,N}(x) = \lim_{x \rightarrow 0} \zeta^{C_{\Gamma}}_{N-2,N-1}(x) = \infty \\
  & \lim_{x \rightarrow 0} x^{-\frac{1}{2}} \zeta^{C_{\Gamma}}_{N-1,N-1}(x) = 
    \lim_{x \rightarrow 0} x^{-\frac{1}{2}} \zeta^{C_{\Gamma}}_{N-2,N-2}(x) = 1
\end{align*}
so that the integrand in \eqref{eqn:c non-zero case 2} is equal to
$$
  \frac {z_{c'}}{\sqrt{z_{a'}z_{b'}}} \sqrt{\frac {z_{a'}}{z_{b'}}} \sqrt{\frac {z_{d'}}{z_{e'}}} \sqrt{\frac {z_{d'}}{z_{f'}}} \sqrt{\frac {z_{e'}}{z_{f'}}} 
  = \frac {z_{c'}z_{d'}}{z_{b'} z_{f'}}
$$
times terms which are either $1$, $-1$, or $0$ in the limit $|z_1| \gg \dots \gg |z_6|$. If $i_1 \in \{N-1,N\}$, i.e.\ $1 \in \{d',c'\}$, then the zeta factors in the denominator of \eqref{eqn:c non-zero case 2} produce a pole in the variable $z_1$, forcing the entire integral to vanish. Thus we may assume $i_1 = N-2 \Rightarrow a'=1$, which also forces $i_2 = N-1 \Rightarrow d'=2$, due to the structure of the iterated $q$-commutator defining $X$. If $i_3=N \Rightarrow c'=3$, then the zeta factors $\zeta^{C_{\Gamma}}_{N-1,N}(z_{e'}/z_{c'}), \zeta^{C_{\Gamma}}_{N-1,N}(z_{f'}/z_{c'}), \zeta^{C_{\Gamma}}_{N-2,N-1}(z_{b'}/z_{d'})$ in the denominator of \eqref{eqn:c non-zero case 2} produce a pole in the variable $z_3$, forcing the integral to vanish. If $i_3 = N-1$, then the structure of the iterated $q$-commutator defining $X$ implies $\textbf{i} = (N-2,N-1,N-1,N-2,N,N-1)$, for which the zeta factors $\zeta^{C_{\Gamma}}_{N-2,N-1}(z_{b'}/z_{d'}), \zeta^{C_{\Gamma}}_{N-2,N-1}(z_{b'}/z_{e'}), \zeta^{C_{\Gamma}}_{N-1,N}(z_{f'}/z_{c'})$ in the denominator will create a pole in $z_3$, forcing the integral to vanish. If $i_3=N-2$, then we have $\textbf{i} = (N-2,N-1,N-2,N-1,N,N-1)$ and the resulting contribution is $1(\pm  q^*)=1$. Thus $\langle X,F \rangle = 1$, establishing formula~\eqref{eqn:not zero c 2}.


\medskip 

\subsection{}
\label{sub:non-zero type c case 3}

Consider type $C$ Case 3. Without loss of generality, we shall assume that the neck parameter equals $q^{2}$. In the proof of Theorem \ref{thm:c intro}, we considered the zero-th mode of the series involved, namely 
\begin{align*}
  X 
  &= [[[[[[ e_{N-3,0}, e_{N-2,0}]_q , e_{N-1,0} ]_q , e_{N,0} ]_{q^{-2}} , e_{N-1,0} ]_{q^{-1}}, e_{N-2,0} ]_q, e_{N-1,0} ]_1 \\
  &= \tfrac{1}{12} X^{0,0,0,0,0,0,0} 
\end{align*}
For $\bn = \bsi^{N-3} + 2\bsi^{N-2} + 3\bsi^{N-1} + \bsi^N$, we let 
$$
  F = \frac {z_{N-3,1}^2 (z_{N-2,1}q - z_{N-2,2}q^{-1})(z_{N-2,2}q - z_{N-2,1}q^{-1})}{z_{N-2,1}^2z_{N-2,2}^2} \in \CV_{C_{\Gamma};-\bn}
$$
(note that in the aforementioned proof, $F$ was multiplied by the monomial $z^{-d}_{N1}$ which perfectly balanced out a shift in the indices of $e_{N,d}$ featuring in $X^{0,0,0,0,0,0,d}$). Then (the type $C_{\Gamma}$ analogue of) formula \eqref{eqn:pairing explicit} implies that
\begin{equation}
\label{eqn:c non-zero case 3}
  \langle X,F \rangle = \sum_{i_1,\dots,i_7} (\pm q^*) \int_{|z_1|\gg \dots \gg |z_7|} \frac {\frac {z^2_{g'}(z_{a'}q-z_{b'}q^{-1})(z_{b'}q-z_{a'}q^{-1})}{z^2_{a'}z^2_{b'}} Dz_1 \dots Dz_7}{\prod_{1 \leq a < b \leq 7} \zeta^{C_{\Gamma}}_{i_bi_a}\left(\frac {z_b}{z_a}\right)} 
\end{equation}
where $\textbf{i}=(i_1,\dots,i_7)$ goes over certain permutations of $(N-3,N-2,N-2,N-1,N-1,N-1,N)$ and $\{a' < b',d'<e'<f', c',g'\} = \{1,\dots,7\}$ denote the indices such that $i_{a'} = i_{b'} = N-2$, $i_{d'} = i_{e'} = i_{f'} = N-1$, $i_{c'} = N$, and $i_{g'}=N-3$. Invoking \eqref{eq:C-zeta-case3}, we see that 
\begin{align*} 
  & \lim_{x \rightarrow 0} \zeta^{C_{\Gamma}}_{N,N-1}(x) = \lim_{x \rightarrow 0} \zeta^{C_{\Gamma}}_{N-1,N-2}(x) = 
    \lim_{x \rightarrow 0} \zeta^{C_{\Gamma}}_{N-2,N-3}(x)  = -1 \\
  & \lim_{x \rightarrow 0} \zeta^{C_{\Gamma}}_{N,N-2}(x) = \lim_{x \rightarrow 0} \zeta^{C_{\Gamma}}_{N,N-3}(x) = 
    \lim_{x \rightarrow 0} \zeta^{C_{\Gamma}}_{N-2,N}(x) = \lim_{x \rightarrow 0} \zeta^{C_{\Gamma}}_{N-3,N}(x) = 1 \\  
  & \lim_{x \rightarrow 0} \zeta^{C_{\Gamma}}_{N-2,N-2}(x) = 1, \quad  \lim_{x \rightarrow 0} \zeta^{C_{\Gamma}}_{N-3,N-1}(x) = (-1)^{|N-3|}, 
    \quad \lim_{x \rightarrow 0} x^{-\frac{1}{2}}\zeta^{C_{\Gamma}}_{N-1,N-1}(x) = 1 \\
  & \lim_{x \rightarrow 0} \zeta^{C_{\Gamma}}_{N-1,N}(x) = \lim_{x \rightarrow 0} \zeta^{C_{\Gamma}}_{N-2,N-1}(x) = 
    \lim_{x \rightarrow 0} \zeta^{C_{\Gamma}}_{N-3,N-2}(x) = \infty
\end{align*}
so that the integrand in \eqref{eqn:c non-zero case 3} is equal to
$$
  \frac {z^2_{g'}(z_{a'}q-z_{b'}q^{-1})(z_{b'}q-z_{a'}q^{-1})}{z^2_{a'}z^2_{b'}} \cdot \frac {z_{d'}}{z_{f'}}
$$
times terms which are either $1$, $-1$, or $0$ in the limit $|z_1| \gg \dots \gg |z_7|$. If $i_1=N-2$, then the zeta factor $\zeta^{C_{\Gamma}}_{N-3,N-2}(z_{g'}/z_{a'})$ in the denominator produces a pole in $z_1$ forcing the integral to vanish. If $i_1=N-1$, then the zeta factors $\zeta^{C_{\Gamma}}_{N-2,N-1}(z_{a'}/z_{d'}), \zeta^{C_{\Gamma}}_{N-2,N-1}(z_{b'}/z_{d'})$ in the denominator produce a pole in $z_1$ forcing the integral to vanish. If $i_1=N$, then the zeta factors $\zeta^{C_{\Gamma}}_{N-1,N}(z_{d'}/z_{c'})$, $\zeta^{C_{\Gamma}}_{N-1,N}(z_{e'}/z_{c'})$, $\zeta^{C_{\Gamma}}_{N-1,N}(z_{f'}/z_{c'})$ in the denominator produce a pole in $z_1$ forcing the integral to vanish again. For $i_1=N-3$, the structure of the iterated $q$-commutator defining $X$ forces $\textbf{i} = (N-3,N-2,N-1,N,N-1,N-2,N-1)$ while the corresponding $(\pm q^*)=1$. For that choice, direct calculations of the integral give 
\begin{align*}
  \langle X,F \rangle  
  &=\int_{|z_1|\gg \dots \gg |z_7|} \frac{q^{-1} (z_2-z_6)(z_2q-z_6q^{-1})z_1^2 z_2^{-2}z_4^{-2} Dz_1\dots Dz_7}
  {\left[\substack{(1-\frac{z_2}{z_1}q^{-1})(1-\frac{z_6}{z_1}q^{-1})(1-\frac{z_3}{z_2}q^{-1})(1-\frac{z_5}{z_2}q^{-1})(1-\frac{z_7}{z_2}q^{-1})(1-\frac{z_4}{z_3}q^{2})\\
  (1-\frac{z_6}{z_3}q^{-1})(1-\frac{z_5}{z_4}q^{2})(1-\frac{z_7}{z_4}q^{2})(1-\frac{z_6}{z_5}q^{-1})(1-\frac{z_7}{z_6}q^{-1})}\right]} 
  \\
  &=\int_{|z_2|\gg \dots \gg |z_7|} \frac{q^{-3} (z^3_2-z^3_6)(z_2q-z_6q^{-1})z_2^{-2}z_4^{-2} Dz_2 \dots Dz_7}
  {\left[\substack{(1-\frac{z_3}{z_2}q^{-1})(1-\frac{z_5}{z_2}q^{-1})(1-\frac{z_7}{z_2}q^{-1})(1-\frac{z_4}{z_3}q^{2})\\
  (1-\frac{z_6}{z_3}q^{-1})(1-\frac{z_5}{z_4}q^{2})(1-\frac{z_7}{z_4}q^{2})(1-\frac{z_6}{z_5}q^{-1})(1-\frac{z_7}{z_6}q^{-1})}\right]} 
  \\
  &=\int_{|z_3|\gg \dots \gg |z_7|} \frac{q^{-4} z_4^{-2} (z_3^2+z_5^2+z_7^2+z_3z_5+z_3z_7+z_5z_7) Dz_3 \dots Dz_7}
  {(1-\frac{z_4}{z_3}q^{2})(1-\frac{z_6}{z_3}q^{-1})(1-\frac{z_5}{z_4}q^{2})(1-\frac{z_7}{z_4}q^{2})(1-\frac{z_6}{z_5}q^{-1})(1-\frac{z_7}{z_6}q^{-1})} 
  \\
  &=\int_{|z_4|\gg \dots \gg |z_7|} \frac{Dz_4 \dots Dz_7}
  {(1-\frac{z_5}{z_4}q^{2})(1-\frac{z_7}{z_4}q^{2})(1-\frac{z_6}{z_5}q^{-1})(1-\frac{z_7}{z_6}q^{-1})} = 1
\end{align*}
Thus $\langle X,F \rangle = 1$, establishing formula~\eqref{eqn:not zero c 3}.


\medskip    

\subsection{}
\label{sub:non-zero type b usual wheel}

Consider type $B$ Case 2. Without loss of generality, we shall assume that the neck parameter equals $q$. Formula \eqref{eqn:rel b 2}, see also~\eqref{eq:B-type-defining last 2}, compels us to take 
\begin{multline*}
  X = - (-1)^{|N-1|} (q+q^{-1}) (q^{-1} e_{N,1} e_{N-1,0} e_{N,0} - e_{N,0} e_{N-1,0} e_{N,1}) \\
  + (q^{-1} e_{N,1} e_{N,0} - e_{N,0} e_{N,1} ) e_{N-1,0} + e_{N-1,0} ( q^{-1} e_{N,1} e_{N,0} - e_{N,0} e_{N,1} ) 
  = \tfrac{1}{2} X^{0,0,0}_{N,N-1} 
\end{multline*}
and set 
\begin{align*}
  F = \frac{1}{\sqrt{z_{N1} z_{N2}}} \in \CV_{B_{\Gamma}; - \bsi^{N-1} - 2\bsi^N} 
\end{align*} 
(note that in the aforementioned proof, $F$ was multiplied by the monomial $z^{-d}_{N-1,1}$ which perfectly balanced out a shift in the indices of $e_{N-1,d}$ featuring in $X_{N,N-1}^{0,0,d}$). Then (the type $B_{\Gamma}$ analogue of) formula \eqref{eqn:pairing explicit} implies that
\begin{equation}
\label{eqn:b non-zero usual deg3}
  \langle X,F \rangle = \sum_{i_1,i_2,i_3} \mathrm{coef}_{i_1,i_2,i_3} \int_{|z_1|\gg |z_2|\gg |z_3|} \frac { \frac{z_{a'}q^{-1}-z_{b'}}{\sqrt{z_{a'}z_{b'}}}   Dz_1 Dz_2 Dz_3}{\prod_{1 \leq a < b \leq 3} \zeta^{B_{\Gamma}}_{i_bi_a}\left(\frac {z_b}{z_a}\right)} 
\end{equation}
where $\textbf{i}=(i_1,i_2,i_3) \in \{(N,N,N-1), (N,N-1,N), (N-1, N, N)\}$, $\{a'<b', c'\} = \{1,2,3\}$ denote indices with $i_{a'} = i_{b'} = N$, $i_{c'}=N-1$, and $\mathrm{coef}_{N,N,N-1}=1,\mathrm{coef}_{N-1,N,N}=1, \mathrm{coef}_{N,N-1,N} = - (-1)^{|N-1|} (q+q^{-1})$. Since
\begin{equation*}
  \lim_{x \rightarrow 0} \zeta^{B_{\Gamma}}_{N,N-1}(x) = -1, \qquad  
  \lim_{x \rightarrow 0} \zeta^{B_{\Gamma}}_{N-1,N}(x) = \infty, \qquad 
  \lim_{x \rightarrow 0} x^{\frac{1}{2}}\zeta^{B_{\Gamma}}_{NN}(x) = 1 
\end{equation*}
we see that the integrand in \eqref{eqn:b non-zero usual deg3} is equal to
$$
  q^{-1} - \frac{z_{b'}}{z_{a'}} 
$$
times terms which are either $1$, $-1$, or $0$ in the limit $|z_1| \gg |z_2| \gg |z_3|$. For $i_1=N$, the zeta factor $\zeta^{B_{\Gamma}}_{N-1,N}(z_{c'}/z_{a'})$ produces a pole in the variable $z_1$, forcing the integral to vanish. Finally, for $\textbf{i}=(N-1,N,N)$, the corresponding integral is
$$
  \int_{|z_1|\gg |z_2|\gg |z_3|} \frac { q^{-1} + o(1) } {( 1- \frac{z_2}{z_1}q ) ( 1-\frac{z_3}{z_1}q ) } 
  \frac{1-\frac{z_3}{z_2}}{(1-\frac{z_3}{z_2}q)(1-\frac{z_3}{z_2}q^{-2})} Dz_1 Dz_2 Dz_3 = q^{-1}
$$
Therefore $\langle X,F \rangle = q^{-1}$, establishing formula \eqref{eq:const-B-2nd}.


\medskip    

\subsection{}
\label{sub:non-zero type b cubic wheel}

Consider type $B$ Case 2. Without loss of generality, we shall assume that the neck parameter equals $q$. We note that $X_N(z_1,z_2,z_3)$ involves only a single current $e_N(z)$ and thus one cannot easily absorb the parameter $d$ as in the previous Subsections. To get around this, we consider a certain linear combination of elements $X^{a,b,c}_N$ and a different choice of test functions $F\in \CV_{B_{\Gamma}; -3\bsi^N}$, see Remark~\ref{rem:cont-irrelevance}.

\medskip
\noindent
$\bullet$ For $d=a+b+c=3k$ with $k\in \BZ$, we consider 
$$
  X = 6 ( -q^3 e_{N,k+1}e^2_{N,k} + (q-q^2) e_{N,k} e_{N,k+1} e_{N,k} + e^2_{N,k} e_{N,k+1} ) =X^{k,k,k}_N
$$ 
and set 
\begin{align*}
  F = (z^{-1}_{N1} + z^{-1}_{N2} + z^{-1}_{N3})\cdot (z_{N1}z_{N2}z_{N3})^{-k} \in \CV_{B_{\Gamma}; -3\bsi^N} 
\end{align*} 
so that (the type $B_{\Gamma}$ analogue of) formula \eqref{eqn:pairing explicit}, cf.~\eqref{eq:B-type-defining last one}, implies that
\begin{equation}
\label{eqn:b non-zero new deg3}
  \langle X,F \rangle = \int_{|z_1|\gg |z_2|\gg |z_3|} 
  \frac { 6(-z_1q^3 + z_2(q-q^2) + z_3)(z^{-1}_{1} + z^{-1}_{2} + z^{-1}_{3})  Dz_1 Dz_2 Dz_3}
        {\prod_{1 \leq a < b \leq 3} \zeta^{B_{\Gamma}}_{i_bi_a}\left(\frac {z_b}{z_a}\right)} 
\end{equation}
As  $\lim_{x \rightarrow 0} x^{\frac{1}{2}}\zeta^{B_{\Gamma}}_{NN}(x) =1$, we see that the integrand in \eqref{eqn:b non-zero new deg3} is equal to
$$
  6\frac{z_3}{z_1} (-z_1q^3 + z_2(q-q^2) + z_3) (z^{-1}_{1} + z^{-1}_{2} + z^{-1}_{3}) = -6q^3 + o(1)
$$
times terms which are $1$ in the limit $|z_1| \gg |z_2| \gg |z_3|$, so that the integral is $-6q^3$. This establishes 
\begin{equation}
\label{eq:const-B-alt-1}
  \langle X,F \rangle = -6q^{3}
\end{equation}

\medskip
\noindent
$\bullet$ For $d=a+b+c=3k-1$ with $k\in \BZ$, we consider 
$$ 
  X=X^{k-1,k,k}_N + X^{k,k-1,k}_N + X^{k,k,k-1}_N
$$
and set 
\begin{align*}
  F = (z_{N1}z_{N2}z_{N3})^{-k} \in \CV_{B_{\Gamma}; -3\bsi^N} 
\end{align*} 
so that the pairing $\langle X,F \rangle$ is given by the same integral as in the previous bullet: 
\begin{equation*}
  \langle X,F \rangle = -6q^{3}
\end{equation*}

\medskip
\noindent
$\bullet$ For $d=a+b+c=3k+1$ with $k\in \BZ$, we consider 
$$ 
  X=X^{k+1,k,k}_N + X^{k,k+1,k}_N + X^{k,k,k+1}_N
$$
and set 
\begin{align*}
  F = (z_{N1} + z_{N2} + z_{N3})\cdot (z_{N1}z_{N2}z_{N3})^{-k-1} \in \CV_{B_{\Gamma}; -3\bsi^N} 
\end{align*} 
so that (the type $B_{\Gamma}$ analogue of) formula \eqref{eqn:pairing explicit} implies that 
\begin{equation*}
  \langle X,F \rangle = 6 \int_{|z_1|\gg |z_2|\gg |z_3|} 
  \frac { (-z_1q^3 + z_2(q-q^2) + z_3)(z_{1} + z_{2} + z_{3})^2 Dz_1 Dz_2 Dz_3}
        {z_1z_2z_3\cdot \prod_{1 \leq a < b \leq 3} \zeta^{B_{\Gamma}}_{i_bi_a}\left(\frac {z_b}{z_a}\right)} 
\end{equation*}
As  $\lim_{x \rightarrow 0} x^{\frac{1}{2}}\zeta^{B_{\Gamma}}_{NN}(x) =1$, we see that the integrand above is equal to
$$
  -q^3\frac{z_1}{z_2} + (q-q^2-2q^3) + o(1)
$$
times terms which are $1$ in the limit $|z_1| \gg |z_2| \gg |z_3|$. The contribution of the part with $(q-q^2-2q^3)$ gives $(q-q^2-2q^3)$, while the contribution of the part with $-q^3\frac{z_1}{z_2}$ gives $-q^3(-1+q+q^{-2})=-q+q^3-q^4$. Thus 
\begin{equation*}
  \langle X,F \rangle = -q^2(1+q+q^2)
\end{equation*}


\bigskip

\section{Delta function equalities}\label{sec:delta-equalities}

The goal of this Appendix is to show explicitly that the relations one imposes to obtain quantum affine superalgebras from pre-quantum ones are indeed dual to the corresponding wheel conditions, as discussed in Section~\ref{sec:arbitrary}. In the case of non-super types $A$, $B$, $C$, $D$ this essentially goes back to~\cite{E}. All the calculations in this Appendix are based on the standard equality
\begin{equation}
\label{eq:delta-1}
  \delta(x) = \frac{1}{1-x} + \frac{1}{x-1}   
\end{equation}
with the terms in the right-hand side understood as formal power series:
$$
  \frac{1}{1-x}=\sum_{d\geq 0} x^d, \qquad \frac{1}{x-1}=\sum_{d>0} x^{-d}
$$  
Invoking~\eqref{eqn:pairing explicit}, we note that each nontrivial linear factor $(1-xt)$ entering $\zeta_{i_b i_a}(x)$ yields the factor $\frac{1}{1-tz_b/z_a}$ in $\zeta^{-1}_{i_bi_a}(z_b/z_a)$ understood as a power-series expansion in the domain $|z_a|\gg|z_b|$. We shall utilize~\eqref{eq:delta-1} to re-expand 
\begin{equation}
\label{eq:reexpansion}
  \frac{1}{1-\frac {tz_b}{z_a}} = \delta\left( \frac {tz_b}{z_a}\right)-\frac{1}{\frac {tz_b}{z_a}-1}
\end{equation}
in the domain $|z_a|\ll|z_b|$, which will buy us valuable delta function identities. Every time we apply~\eqref{eq:reexpansion}, we pick up either the delta function or the expansion of the same rational function in the correct domain. We always apply~\eqref{eqn:key} whenever picking up a delta function, so as to decrease the number of variables. In this Appendix, we will use the notation $|z_1|\sim \ldots\sim |z_k|$ while $|q^{\pm 1}|\gg 1$ to indicate the domain $(z_1,\ldots,z_k)\in (\BC^\times)^k$ such that $|z_a/z_b| \ll |q^{\pm 1}|$ for all $1\leq a,b\leq k$.


\medskip 

\subsection{}
\label{app:A-cubic}
The evaluation of $\langle X_{i,i+1}(z_1,z_2,w), F \rangle$ for $X_{i,i+1}(z_1,z_2,w)$ of~\eqref{eqn:x} and $F\in \CV_{A_{\Gamma};-2\bsi^i-\bsi^{i+1}}$ essentially goes back to~\cite[\S2]{E}, but we shall recall it since it serves as a prototypical example for the rest of this Appendix. Applying~\eqref{eqn:pairing explicit}, we obtain 
\begin{multline*}
  \Big\langle X_{i,i+1}(z_1,z_2,w) , F \Big\rangle = F \cdot \\
  \Sym_{z_1,z_2}\, \frac{1}{\zeta^{A_{\Gamma}}_{ii}(\frac{z_{2}}{z_1})}
  \Bigg( \frac{1}{\zeta^{A_{\Gamma}}_{i+1,i}(\frac{w}{z_1})\zeta^{A_{\Gamma}}_{i+1,i}(\frac{w}{z_2})} 
    - \frac{q^{s_i}+q^{-s_i}}{\zeta^{A_{\Gamma}}_{i+1,i}(\frac{w}{z_1})\zeta^{A_{\Gamma}}_{i,i+1}(\frac{z_2}{w})}     
    + \frac{1}{\zeta^{A_{\Gamma}}_{i,i+1}(\frac{z_1}{w})\zeta^{A_{\Gamma}}_{i,i+1}(\frac{z_2}{w})}\Bigg)
\end{multline*}
The expression in the second line above vanishes as a rational function, but the reader should remember that each factor is expanded via~\eqref{eq:delta-1} into the power series in different domains. We shall assume that $|q^{s_i}| \gg 1$ while $|z_1|\sim |z_2|\sim |w|$, so that we have to re-expand using~\eqref{eq:delta-1} only the terms $\frac{1}{1-q^{2s_i}z_2/z_1}$ and $\frac{1}{1-q^{2s_i}z_1/z_2}$. Picking up $\delta\left(q^{2s_i}z_2/z_1\right)$ allows us to replace $z_1$ by $z_2q^{2s_i}$ in the remaining factors, and similarly for $\delta\left(q^{2s_i} z_1/z_2 \right)$. A direct calculation then provides the formula: 
\begin{equation}
\label{eq:A-Enriguez}
  \Big\langle X_{i,i+1}(z_1,z_2,w) , F \Big\rangle = 
  \Sym_{z_1,z_2}\,   \delta\left( \frac{z_1}{q^{2s_i}z_2} \right) \delta\left( \frac{w}{q^{-s_i} z_1}  \right) \cdot F(z_1,z_2,w)
\end{equation}
where we plug $z_1,z_2,w$ instead of $z_{i1}, z_{i2}, z_{i+1,1}$. In particular, we have
\begin{equation*}
  \langle X_{i,i+1}(z_1,z_2,w) , F \rangle=0 \quad \Leftrightarrow \quad F \text{ vanishes at the wheel \eqref{eqn:wheel finite even}}
\end{equation*}
Formula~\eqref{eq:A-Enriguez} also implies the equality~\eqref{eqn:not zero 1} from the proof of~\eqref{eqn:key pairing finite sl}:
$$
  \langle X_{i,i+1}^{0,0,d}, z_{i+1,1}^{-d} \rangle = 2
$$
The case of $\langle X_{i,i-1}(z_1,z_2,w) , - \rangle$ is completely analogous; we leave it to the reader.


\medskip 

\subsection{}
\label{app:A-quartic}

Next, we consider the pairing of $X_{i-1,i,i+1}(z_1,z_2,w,y)=(\mathrm{LHS\ of}\ \eqref{eqn:rel quantum affine 9})$ with $F\in \CV_{\wA_{\Gamma}; -\bsi^{i-1}-2\bsi^i-\bsi^{i+1}}$ (the case of finite type $A$ follows by setting $d=1$). Without loss of generality, we shall assume that $s_i=1$. To simplify the exposition, we also assume that $|i-1|=0=|i+1|$ (in the general case, one simply gains an extra sign $(-1)^{|i+1|(1+|i-1|)}$). Let $q_1=q^{-1}d, q_3=q^{-1}d^{-1}$. Then we have 
\begin{multline*}
  \Big\langle X_{i-1,i,i+1}(z_1,z_2,w,y) , F \Big\rangle = \\
  F \cdot \Sym_{z_1,z_2}\,  
  \left[ \Bigg( \tfrac{1}{\zeta^{\wA_{\Gamma}}_{ii}(\frac{z_2}{z_1})\zeta^{\wA_{\Gamma}}_{i+1,i}(\frac{y}{z_1})\zeta^{\wA_{\Gamma}}_{i-1,i}(\frac{w}{z_1})} + \tfrac{1}{\zeta^{\wA_{\Gamma}}_{ii}(\frac{z_1}{z_2})\zeta^{\wA_{\Gamma}}_{i,i+1}(\frac{z_1}{y}) \zeta^{\wA_{\Gamma}}_{i,i-1}(\frac{z_1}{w})} \Bigg) \right. \times \\
  \left. \Bigg( \tfrac{1}{\zeta^{\wA_{\Gamma}}_{i,i+1}(\frac{z_2}{y}) \zeta^{\wA_{\Gamma}}_{i-1,i}(\frac{w}{z_2})}  
   - \tfrac{q}{\zeta^{\wA_{\Gamma}}_{i,i+1}(\frac{z_2}{y}) \zeta^{\wA_{\Gamma}}_{i,i-1}(\frac{z_2}{w})}
   - \tfrac{q^{-1}}{\zeta^{\wA_{\Gamma}}_{i+1,i}(\frac{y}{z_2}) \zeta^{\wA_{\Gamma}}_{i-1,i}(\frac{w}{z_2})} 
   + \tfrac{1}{\zeta^{\wA_{\Gamma}}_{i+1,i}(\frac{y}{z_2}) \zeta^{\wA_{\Gamma}}_{i,i-1}(\frac{z_2}{w})}\Bigg) \right]
\end{multline*}
The above $\Sym_{z_1,z_2}\,(\dots)$ vanishes as a rational function, but the reader should remember that each factor is expanded via~\eqref{eq:delta-1} into the power series in different domains. We shall assume that $|q_1|, |q_3| \gg 1$ while $|z_1|\sim |z_2|\sim |w| \sim |y|$, so that we have to re-expand using~\eqref{eq:delta-1} only the terms $\frac{1}{1-q_3 w/z_a}$ and $\frac{1}{1-q_1 z_a/w}$ for $a=1,2$. Picking up two $\delta$-factors $\delta(q_3w/z_1) \delta(q_3w/z_2)$ or $\delta(q_1z_1/w)\delta(q_1z_2/w)$ will always produce $0$, due to the factor $z_1-z_2$ arising from $1/\zeta^{\wA_{\Gamma}}_{ii}(z_1/z_2)$ or $1/\zeta^{\wA_{\Gamma}}_{ii}(z_2/z_1)$. On the other hand, picking up only a single $\delta$-function, e.g.\ $\delta\left(q_3w/z_1\right)$ and replacing $w$ by $z_1 q_3^{-1}$ in the rest of terms, will also result in $0$ (all the corresponding power series are expansions of the same-named rational functions that add up to $0$ in the common domain $|z_1|\sim |z_2| \sim |y|$ for  $|q_1|, |q_3| \gg 1$). Finally, picking up two $\delta$-factors $\delta(q_3w/z_1) \delta(q_1z_2/w)$, the rest of the terms add up to:
\begin{multline*}
  \frac{-q^2(1-q^2)q_3^{-1}z_1y^{-1}}{(1-q_1^{-1}\frac{y}{z_1})(1-q_1^{-1}q_3^{-2}\frac{z_1}{y})} -
  \frac{1-q^2}{(1-q_1^{-1}\frac{y}{z_1})(1-q_3\frac{y}{z_1})} \\ 
  - \frac{q_3^{-2}(1-q^2)z_1^2y^{-2}}{(1-q_3^{-1}\frac{z_1}{y})(1-q_1^{-1}q_3^{-2}\frac{z_1}{y})} -
  \frac{q_3^{-1}(1-q^2)z_1y^{-1}}{(1-q_1^{-1}\frac{y}{z_1})(1-q_1^{-1}q_3^{-2}\frac{z_1}{y})}
\end{multline*}
The above sum vanishes as a rational function, but the reader should remember that each factor is expanded via~\eqref{eq:delta-1} into the power series in different domains. Assuming $|q_1|, |q_3| \gg 1$, we see that we have to re-expand using~\eqref{eq:delta-1} only the single term $\frac{1}{1-q_3 y/z_1}$, which produces $-\delta(q_3y/z_1)$. Therefore, 
\begin{multline}
\label{eq:a-deg4-equality}
  \Big\langle X_{i-1,i,i+1}(z_1,z_2,w,y) , F \Big\rangle =\\ 
   - \Sym_{z_1,z_2}\,  \delta\left( \frac{z_1}{q_3w} \right) \delta\left( \frac{y}{q_3^{-1}z_1} \right) \delta\left( \frac{z_2}{q_1^{-1}y} \right) \cdot F(z_1,z_2,w,y)
\end{multline}
where we plug $z_1,z_2,w,y$ instead of $z_{i1}, z_{i2}, z_{i-1,1}, z_{i+1,1}$. In particular, we have 
$$
  \langle X_{i-1,i,i+1}(z_1,z_2,w,y) , F \rangle=0 \quad \Leftrightarrow \quad F \ \text{ vanishes at the wheel~\eqref{eqn:wheel affine odd}}
$$
Formula~\eqref{eq:a-deg4-equality} also implies the equality~\eqref{eqn:not zero 2} from the proof of~\eqref{eqn:key pairing finite sl}:
$$
   \left\langle X_{i-1,i,i+1}^{0,0,0,d}, z_{i+1,1}^{-d}\left(\sqrt{\frac{z_{i1}}{z_{i2}}} + \sqrt{\frac{z_{i2}}{z_{i1}}}\right) \right\rangle  =  -2(q+q^{-1}) 
$$


\medskip 

\subsection{}
\label{app:B-quartic}

Consider type $B$ Case 1. Without loss of generality, we shall assume that the neck parameter equals $q^{-1}$. For $F\in \CV_{B_{\Gamma};-3\bsi^N-\bsi^{N-1}}$, the evaluation of the pairing $\langle X_{N-1,N}(z_1,z_2,z_3,w) , F \rangle$ with the series $X_{N-1,N}(z_1,z_2,z_3,w)$ from the proof of Theorem~\ref{thm:b intro} proceeds analogously:
\begin{multline*}
  \Big\langle X_{N-1,N}(z_1,z_2,z_3,w) , F \Big\rangle = F \cdot
  \Sym_{z_1,z_2,z_3}\, \Bigg[ \frac{(z_1-z_2)(z_1-z_3)(z_2-z_3)}{(z_1-z_2 q)(z_1-z_3 q)(z_2-z_3 q)}\times \\
  \Bigg( \frac{w^3}{(z_1-wq^{-1})(z_2-wq^{-1})(z_3-wq^{-1})} + \frac{(q+1+q^{-1})w^3}{(z_1-wq^{-1})(z_2-wq^{-1})(w-z_3q^{-1})} \\ 
  + \frac{(q+1+q^{-1})w^3}{(z_1-wq^{-1})(w-z_2q^{-1})(w-z_3q^{-1})} + \frac{w^3}{(w-z_1q^{-1})(w-z_2q^{-1})(w-z_3q^{-1})}\Bigg) \Bigg]
\end{multline*}
According to~\cite[Proposition 4]{E} (with $m=2$  and $q^2$ replaced by $q^{-1}$), we obtain 
\begin{multline}
\label{eq:B-Enriguez}
  \Big\langle X_{N-1,N}(z_1,z_2,z_3,w) , F \Big\rangle = \\
  \Sym_{z_1,z_2,z_3}\,  \delta\left( \frac{z_2}{qz_1} \right) \delta\left( \frac{z_3}{qz_2} \right) \delta\left( \frac{w}{q^{-1}z_3} \right) \cdot F(z_1,z_2,z_3,w)
\end{multline}
where we plug $z_1,z_2,z_3,w$ instead of $z_{N 1}, z_{N 2}, z_{N 3}, z_{N-1,1}$. In particular, we have 
\begin{equation*}
  \langle X_{N-1,N}(z_1,z_2,z_3,w) , F \rangle=0 \quad \Leftrightarrow \quad F \text{ vanishes at the wheel \eqref{eqn:wheel b 1}}
\end{equation*}
Formula~\eqref{eq:B-Enriguez} also implies the equality~\eqref{eq:const-B-1st} from the proof of Theorem~\ref{thm:b intro}:
$$
  \langle X_{N-1,N}^{0,0,0,d}, z^{-d}_{N-1,1} \rangle = 6
$$


\medskip 

\subsection{}
\label{app:B-cubic-1}

Consider type $B$ Case 2. Without loss of generality, we shall assume that the neck parameter equals $q$. Let $X_N(z_1,z_2,z_3)$ be the LHS of~\eqref{eqn:rel b 1} as used in the proof of Theorem~\ref{thm:b intro}. Then for any $F\in \CV_{B_{\Gamma}; -3\bsi^N}$, we have 
\begin{equation*}
  \Big\langle X_N(z_1,z_2,z_3) , F \Big\rangle = 
  \Sym_{z_1,z_2,z_3}\,  \frac{-z_1q^3+z_2(q-q^2)+z_3}{\zeta^{B_{\Gamma}}_{NN}(\frac{z_2}{z_1}) \zeta^{B_{\Gamma}}_{NN}(\frac{z_3}{z_1}) \zeta^{B_{\Gamma}}_{NN}(\frac{z_3}{z_2})}  \cdot F
\end{equation*}
The above $\Sym_{z_1,z_2,z_3}\,(\dots)$ vanishes as a rational function, but the reader should remember that each factor is expanded via~\eqref{eq:delta-1} into the power series in different domains. We shall assume that $|q|\gg 1$ while $|z_1|\sim |z_2|\sim |z_3|$, so that we have to re-expand using~\eqref{eq:delta-1} only the terms $\frac{1}{1-qz_a/z_b}$ for $a\ne b$. Picking up two $\delta$-factors 
$$ 
  \delta\left(\frac{z_b}{qz_a}\right) \delta\left(\frac{z_b}{qz_c}\right) \quad \mathrm{or} \quad 
  \delta\left(\frac{z_a}{qz_b}\right)\delta\left(\frac{z_c}{qz_b}\right)
$$
for $\{a,b,c\}=\{1,2,3\}$ will always produce $0$, due to the factor $z_a-z_c$ arising from $1/\zeta^{B_{\Gamma}}_{NN}(z_a/z_c)$ or $1/\zeta^{B_{\Gamma}}_{NN}(z_c/z_a)$. On the other hand, picking only a single $\delta$-function, e.g.\ $\delta\left(qz_2/z_1\right)$ and replacing $z_2$ by $z_1 q^{-1}$ in the rest of terms, will also result in $0$ (all the corresponding power series are expansions of the same-named rational functions that add up to $0$ in the common domain $|z_2|\sim |z_3| < |z_1|$ for  $|q|\gg 1$). Evaluating the remaining terms with $\delta\left(qz_b/z_a\right) \delta\left(qz_c/z_b\right)$ for $\{a,b,c\}=\{1,2,3\}$, we obtain 
\begin{multline}
\label{eq:b1-equality}
  \Big\langle X_N(z_1,z_2,z_3) , F \Big\rangle =\\ 
   \frac{-q^5}{1+q+q^{2}} \cdot \Sym_{z_1,z_2,z_3}\,  z_3 
   \delta\left( \frac{z_1}{qz_2} \right) \delta\left( \frac{z_2}{qz_3} \right)  \cdot F(z_1,z_2,z_3)
\end{multline}
where we plug $z_1,z_2,z_3$ instead of $z_{N1}, z_{N2}, z_{N3}$. In particular, we have 
$$
  \langle X_N(z_1,z_2,z_3) , F \rangle=0 \quad \Leftrightarrow \quad F \ \text{ vanishes at the wheel~\eqref{eqn:wheel b 2}}
$$
Formula~\eqref{eq:b1-equality} also implies the equality~\eqref{eq:const-B-3rd} from the proof of Theorem~\ref{thm:b intro}:
$$
   \langle X_{N}^{0,0,d}, z^{-d-1}_{N 1} + z^{-d-1}_{N 2} + z^{-d-1}_{N 3} \rangle  =  \frac{-2q^{3-2d}(1+q^{d}+q^{2d})(1+q^{d+1}+q^{2d+2})}{1+q+q^2} 
$$
In the particular case $d=0$, this is compatible with the formula~\eqref{eq:const-B-alt-1}.


\medskip 

\subsection{}
\label{app:B-cubic-2}

Consider type $B$ Case 2. Without loss of generality, we shall assume that the neck parameter equals $q$. Let $X_{N,N-1}(z_1,z_2,w)$ be the LHS of~\eqref{eqn:rel b 2} as used in the proof of Theorem~\ref{thm:b intro}, so that for any $F\in \CV_{B_{\Gamma}; -\bsi^{N-1}-2\bsi^N}$:
\begin{multline*}
  \Big\langle X_{N,N-1}(z_1,z_2,w) , F \Big\rangle = F\cdot \Sym_{z_1,z_2}\, \Bigg[ \frac{z_1 q^{-1} - z_2}{\zeta^{B_{\Gamma}}_{NN}(\frac{z_2}{z_1})} \times \\
  \Bigg( \frac{1}{\zeta^{B_{\Gamma}}_{N-1,N}(\frac{w}{z_1}) \zeta^{B_{\Gamma}}_{N-1,N}(\frac{w}{z_2})}
  + \frac{1}{\zeta^{B_{\Gamma}}_{N,N-1}(\frac{z_1}{w}) \zeta^{B_{\Gamma}}_{N,N-1}(\frac{z_2}{w})} 
  - \frac{(-1)^{|N-1|}(q+q^{-1})}{\zeta^{B_{\Gamma}}_{N-1,N}(\frac{w}{z_1}) \zeta^{B_{\Gamma}}_{N,N-1}(\frac{z_2}{w})} \Bigg)  \Bigg]
\end{multline*}
The above $\Sym_{z_1,z_2}\, (\dots)$ vanishes as a rational function, but the reader should remember that each factor is expanded via~\eqref{eq:delta-1} into the power series in different domains. We shall assume that $|q|\ll 1$ while $|z_1|\sim |z_2|\sim |w|$, so that we have to re-expand using~\eqref{eq:delta-1} only the terms $\frac{1}{1-q^{-2}z_2/z_1}$ and $\frac{1}{1-q^{-2}z_1/z_2}$. Picking up $\delta\left(q^{-2}z_2/z_1\right)$ allows us to replace $z_1$ by $z_2q^{-2}$ in the remaining factors, and likewise for $\delta\left(q^{-2}z_1/z_2\right)$. The direct calculation provides the following formula:
\begin{equation}
\label{eq:b2-equality}
  \Big\langle X_{N,N-1}(z_1,z_2,w) , F \Big\rangle = \Sym_{z_1,z_2}\, \frac{z_2}{q^2} \delta\left( \frac{z_2}{q^{2}z_1} \right) \delta\left( \frac{w}{q^{-1}z_2} \right)  \cdot F(z_1,z_2,w)
\end{equation}
where we plug $z_1,z_2,w$ instead of $z_{N1}, z_{N2}, z_{N-1,1}$. In particular, we have 
$$
  \langle X_{N,N-1}(z_1,z_2,w) , F \rangle=0 \quad \Leftrightarrow \quad F \ \text{ vanishes at the wheel~\eqref{eqn:wheel finite even} in colors } N-1,N,N
$$
Formula~\eqref{eq:b2-equality} also implies the equality~\eqref{eq:const-B-2nd} from the proof of Theorem~\ref{thm:b intro}:
$$
  \left\langle X_{N,N-1}^{0,0,d}, \frac{z^{-d}_{N-1,1}}{\sqrt{z_{N1} z_{N2}}}  \right\rangle = 2q^{-1}
$$


\medskip 

\subsection{}
\label{app:D-fork-odd}

Consider type $D$ Case 2. Without loss of generality, we shall assume that the neck parameter equals $q$. Let $X(w,y,z)$ be the LHS of~\eqref{eqn:rel cd 1} as used in the proof of Theorem~\ref{thm:d intro}, so that for any $F\in \CV_{D_{\Gamma}; -\bsi^{N-2}-\bsi^{N-1}-\bsi^N}$:
\begin{multline*}
  \Big\langle X(w,y,z) , F \Big\rangle = F \times \\
  \Bigg( - \frac{1}{\zeta^{D_{\Gamma}}_{N,N-2}(\frac{z}{w})\zeta^{D_{\Gamma}}_{N-1,N-2}(\frac{y}{w})\zeta^{D_{\Gamma}}_{N-1,N}(\frac{y}{z})} 
    + \frac{1}{\zeta^{D_{\Gamma}}_{N-2,N}(\frac{w}{z})\zeta^{D_{\Gamma}}_{N-2,N-1}(\frac{w}{y})\zeta^{D_{\Gamma}}_{N,N-1}(\frac{z}{y})} \\
    + \frac{1}{\zeta^{D_{\Gamma}}_{N,N-2}(\frac{z}{w})\zeta^{D_{\Gamma}}_{N-1,N-2}(\frac{y}{w})\zeta^{D_{\Gamma}}_{N,N-1}(\frac{z}{y})} 
    - \frac{1}{\zeta^{D_{\Gamma}}_{N-2,N}(\frac{w}{z})\zeta^{D_{\Gamma}}_{N-2,N-1}(\frac{w}{y})\zeta^{D_{\Gamma}}_{N-1,N}(\frac{y}{z})} \\
    + \frac{(-1)^{|N-2|}(q+q^{-1})}{\zeta^{D_{\Gamma}}_{N-2,N}(\frac{w}{z})\zeta^{D_{\Gamma}}_{N-1,N-2}(\frac{y}{w})\zeta^{D_{\Gamma}}_{N-1,N}(\frac{y}{z})}
    - \frac{(-1)^{|N-2|}(q+q^{-1})}{\zeta^{D_{\Gamma}}_{N,N-2}(\frac{z}{w})\zeta^{D_{\Gamma}}_{N-2,N-1}(\frac{w}{y})\zeta^{D_{\Gamma}}_{N,N-1}(\frac{z}{y})} \Bigg)
\end{multline*}
The expression in above $(\dots)$ vanishes as a rational function, but each factor is expanded via~\eqref{eq:delta-1} in different domains. We shall assume that $|q|\ll 1$ while $|y|\sim |z|\sim |w|$, so that we have to re-expand using~\eqref{eq:delta-1} only the terms $\frac{1}{1-q^{-2} y/z}$ and $\frac{1}{1- q^{-2} z/y}$. Picking up $\delta\left( q^{-2}y/z \right)$ allows us to replace $z$ by $yq^{-2}$ in the remaining factors, and likewise for $\delta\left(q^{-2}z/y\right)$. A direct calculation then shows that
\begin{multline}
\label{eq:D-fork-odd-delta}
  \big\langle X(w,y,z) , F \big\rangle = \\ 
  \left( \frac{q^2}{1-q^2} \delta\left( \frac{z}{q^{-2}y} \right) \delta\left( \frac{w}{qz} \right) - 
         \frac{1}{1-q^2} \delta\left( \frac{y}{q^{-2}z} \right) \delta\left( \frac{w}{qy}\right) \right) \cdot F(w,y,z)
\end{multline}
where we plug $w,y,z$ instead of $z_{N-2,1}, z_{N-1,1}, z_{N1}$. In particular, we have 
\begin{equation}
\label{eq:D-fork-odd-conclusion}
  \langle X(w,y,z) , F \rangle=0 \quad \Leftrightarrow \quad F \ \text{ vanishes at both wheels~\eqref{eqn:wheel d 1}}
\end{equation}
Formula~\eqref{eq:D-fork-odd-delta} also implies the equalities \eqref{eqn:not zero d} used in the proof of Theorem~\ref{thm:d intro}.


\medskip 

\subsection{}
\label{app:C-quartic}

Consider type $C$ Case 1. Without loss of generality, we shall assume that the neck parameter equals $q^2$. For any $F\in \CV_{C_{\Gamma};-3\bsi^{N-1}-\bsi^N}$, the evaluation of $\langle X(z_1,z_2,z_3,w) , F \rangle$ with the series $X(z_1,z_2,z_3,w)$ from the proof of Theorem~\ref{thm:c intro} is very similar to the one in Subsection~\ref{app:B-quartic}: 
\begin{multline*}
  \Big\langle X(z_1,z_2,z_3,w) , F \Big\rangle = F \cdot
  \Sym_{z_1,z_2,z_3}\, \Bigg[ \frac{(z_1-z_2)(z_1-z_3)(z_2-z_3)}{(z_1-z_2 q^{-2})(z_1-z_3 q^{-2})(z_2-z_3 q^{-2})}\times \\
  \Bigg( - \frac{z_1z_2z_3}{(z_1-wq^{2})(z_2-wq^{2})(z_3-wq^{2})} - \frac{(q^2+1+q^{-2})z_1z_2z_3}{(z_1-wq^{2})(z_2-wq^{2})(w-z_3q^{2})} \\ 
  - \frac{(q^2+1+q^{-2})z_1z_2z_3}{(z_1-wq^{2})(w-z_2q^{2})(w-z_3q^{2})} - \frac{z_1z_2z_3}{(w-z_1q^{2})(w-z_2q^{2})(w-z_3q^{2})}\Bigg) \Bigg]
\end{multline*}
Invoking~\cite[Proposition 4]{E} again, we thus obtain 
\begin{multline}
\label{eq:C-Enriguez}
  \Big\langle X(z_1,z_2,z_3,w) , F \Big\rangle = \\
  - \Sym_{z_1,z_2,z_3}\,  \delta\left( \frac{z_2}{q^{-2}z_1} \right) \delta\left( \frac{z_3}{q^{-2}z_2} \right) \delta\left( \frac{w}{q^{2}z_3} \right) \cdot F(z_1,z_2,z_3,w)
\end{multline}
where we plug $z_1,z_2,z_3,w$ instead of $z_{N-1,1}, z_{N-1,2}, z_{N-1,3}, z_{N1}$. In particular 
\begin{equation*}
  \langle X(z_1,z_2,z_3,w) , F \rangle=0 \quad \Leftrightarrow \quad F \text{ vanishes at the wheel \eqref{eqn:wheel c 1}}
\end{equation*}
Formula~\eqref{eq:C-Enriguez} also implies the equality \eqref{eqn:not zero c 1} from the proof of Theorem~\ref{thm:c intro}:
$$ 
  \langle X^{0,0,0,d}, z^{-d}_{N1} \rangle = -6
$$


\medskip 

\subsection{}
\label{app:gl1-cubic}

Let $X(z_1,z_2,z_3)$ be the LHS of~\eqref{eq:Miki} . Direct calculations, similar to the ones presented above, then prove the following formula for any $F\in \CV_{-3\bsi}$:
\begin{multline}
\label{eq:delta-gl1}
  \Big\langle X(z_1,z_2,z_3) , F \Big\rangle = \\
  \Sym_{z_1,z_2,z_3}\, \Bigg( \gamma_1 \delta\left( \frac{z_2}{q_2 z_1} \right) \delta\left( \frac{z_3}{q_1 z_2}\right)
  + \gamma_2 \delta\left( \frac{z_2}{q_2 z_1} \right) \delta\left( \frac{z_3}{q_3 z_2}\right) \Bigg) \cdot  F(z_1,z_2,z_3) 
\end{multline}
with certain nonzero constants $\gamma_1$ and $\gamma_2$. Therefore, we obtain:
\begin{equation}
\label{eq:wheel-gl1}
  \langle X(z_1,z_2,z_3) , F \rangle = 0 \quad \Leftrightarrow \quad F \text{ vanishes at both wheels~\eqref{eq:wheel-pic-gl1}}
\end{equation}


\medskip 

\subsection{}
\label{app:gl2-cubic}

Let $X(z_1,z_2,w)$ be the LHS of~\eqref{eq:3Serre-sl2} for $i\ne j\in \{0,1\}$. Direct calculations then prove for $F\in \CV_{-2\bsi^i-\bsi^j}$:
\begin{multline}
\label{eq:delta-gl2}
  \Big\langle X(z_1,z_2,w) , F \Big\rangle = \\
  \Sym_{z_1,z_2}\, \Bigg( \gamma'_1 \delta\left( \frac{z_2}{q_2 z_1} \right) \delta\left( \frac{w}{q_1 z_2}\right)
  + \gamma'_2 \delta\left( \frac{z_2}{q_2 z_1} \right) \delta\left( \frac{w}{q_3 z_2}\right) \Bigg) \cdot  w F(z_1,z_2,w) 
\end{multline}
with certain nonzero constants $\gamma'_1$ and $\gamma'_2$. Therefore, we obtain:
\begin{equation}
\label{eq:wheel-gl2}
  \langle X(z_1,z_2,w) , F \rangle=0 \quad \Leftrightarrow \quad F \text{ vanishes at both wheels~\eqref{eq:wheel-pic-gl2}}
\end{equation}


\end{document}